\documentclass[nosumlimits,twoside]{amsart}
\usepackage{amsfonts, amsmath, amssymb}
\usepackage[bookmarksnumbered,plainpages,hypertex]{hyperref}
\usepackage{graphicx}
\usepackage{float}
\usepackage{srcltx}
\usepackage[all]{xy}
\usepackage{version}

\newtheorem{theorem}{\sc Theorem}[section]
\newtheorem{proposition}[theorem]{\sc Proposition}
\newtheorem{notation}[theorem]{\sc Notation}

\newtheorem{lemma}[theorem]{\sc Lemma}
\newtheorem{corollary}[theorem]{\sc Corollary}
\theoremstyle{definition}
\newtheorem{definition}[theorem]{\sc Definition}

\newtheorem{example}[theorem]{\sc Example}

\theoremstyle{remark}
\newtheorem{remark}[theorem]{\sc Remark}

\newtheorem{claim}[theorem]{}

\newcommand\id{\mathrm{Id}}
\newcommand\M{\mathcal{M}}
\newcommand\N{\mathcal{N}}
\newcommand\Alg{\mathrm{Alg}}
\newcommand\Bialg{\mathrm{Bialg}}
\newcommand\BrBialg{\mathrm{BrBialg}}
\newcommand\BrAlg{\mathrm{BrAlg}}
\newcommand\Br{\mathrm{Br}}
\newcommand\Lie{\mathrm{Lie}}
\newcommand\BrLie{\mathrm{BrLie}}
\newcommand\vet{\mathfrak{M}}
\newcommand\I{\mathbb{I}}
\newcommand\U{\mathcal{U}}
\newcommand{\duplica}[1]{#1\otimes #1}

\allowdisplaybreaks

\begin{document}
\title{Milnor-Moore Categories and Monadic Decomposition}
\author{Alessandro Ardizzoni}
\address{%
\parbox[b]{\linewidth}{University of Turin, Department of Mathematics ``G. Peano'', via
Carlo Alberto 10, I-10123 Torino, Italy}}
\email{alessandro.ardizzoni@unito.it}
\urladdr{www.unito.it/persone/alessandro.ardizzoni}
\author{Claudia Menini}
\address{University of Ferrara, Department of Mathematics, Via Machiavelli
35, Ferrara, I-44121, Italy}
\email{men@unife.it}
\urladdr{web.unife.it/utenti/claudia.menini}
\subjclass[2010]{Primary 18C15; Secondary 17B75}
\thanks{This paper was written while both authors were members of GNSAGA.
The first author was partially supported by the research grant ``Progetti di
Eccellenza 2011/2012'' from the ``Fondazione Cassa di Risparmio di Padova e
Rovigo''.}

\begin{abstract}
In this paper Hom-Lie algebras, Lie color algebras, Lie superalgebras and
other type of generalized Lie algebras are recovered by means of an iterated
construction, known as monadic decomposition of functors, which is based on
Eilenberg-Moore categories. To this aim we introduce the notion of
Milnor-Moore category as a monoidal category for which a Milnor-Moore type
Theorem holds. We also show how to lift the property of being a Milnor-Moore
category whenever a suitable monoidal functor is given and we apply this
technique to provide examples.
\end{abstract}

\keywords{Monads, Milnor-Moore Category, Generalized Lie Algebras}
\maketitle
\tableofcontents

\section*{Introduction}

Let $\left( L:\mathcal{B}\rightarrow \mathcal{A},R:\mathcal{A}\rightarrow
\mathcal{B}\right) $ be an adjunction with unit $\eta $ and counit $\epsilon
$. Then $\left( RL,R\epsilon L,\eta \right) $ is a monad on $\mathcal{B}$
and one can consider the Eilenberg-Moore category ${_{RL}\mathcal{B}}$
associated to this monad and the so-called \emph{comparison functor} $K:%
\mathcal{A}\rightarrow {_{RL}\mathcal{B}}$ which is defined by $KX:=\left(
RX,R\epsilon X\right) $ and $Kf:=Rf.$ This gives the diagram
\begin{equation*}
\xymatrixrowsep{15pt}\xymatrixcolsep{2cm}\xymatrix{\mathcal{A}%
\ar@<.5ex>[d]^{R}&\mathcal{A}\ar@<.5ex>[d]^{K}\ar[l]_{\mathrm{Id}_{%
\mathcal{A}}}\\
\mathcal{B}\ar@<.5ex>@{.>}[u]^{L}&{_{RL}\mathcal{B}}\ar[l]_{_{RL}U}}
\end{equation*}%
where the undashed part commutes. In the case when $K$ itself has a left
adjoint $\Lambda $ one can repeat this construction starting from the new
adjunction $(\Lambda ,K)$. Going on this way one possibly obtains a diagram
of the form
\begin{equation*}
\xymatrixrowsep{15pt}\xymatrixcolsep{2cm}\xymatrix{
\mathcal{A}\ar@<.5ex>[d]^{R_0}&\mathcal{A}\ar@<.5ex>[d]^{R_1}\ar[l]_{%
\mathrm{Id}_{\mathcal{A}}}&\mathcal{A}\ar@<.5ex>[d]^{R_2}\ar[l]_{%
\mathrm{Id}_{\mathcal{A}}}&\dots \ar[l]_{\mathrm{Id}_{\mathcal{A}}}\\
\mathcal{B}_0\ar@<.5ex>@{.>}[u]^{L_0}&\mathcal{B}_1\ar@<.5ex>@{.>}[u]^{L_1}
\ar[l]_{U_{0,1}}&\mathcal{B}_2 \ar@<.5ex>@{.>}[u]^{L_2}
\ar[l]_{U_{1,2}}&\dots\ar[l]_{U_{2,3}}}
\end{equation*}%
where it is more convenient to relabel $(L,R)$ and $(\Lambda ,K)$ as $%
(L_{0},R_{0})$ and $(L_{1},R_{1})$ respectively. If there is a minimal $N\in
\mathbb{N}$ such that $L_{N}$ is full and faithful, then $R$ is said to have%
\emph{\ monadic decomposition of monadic length }$N$. This is equivalent to
require that the forgetful functor $U_{N,N+1}$ is a category isomorphism and
no $U_{n,n+1}$ has this property for $0\leq n\leq N-1$ (see e.g. \cite[%
Remark 2.4]{AGM-MonadicLie1}). In \cite[Theorem 3.4]{AGM-MonadicLie1}, we
investigated the particular case
\begin{equation*}
\xymatrixrowsep{15pt}\xymatrixcolsep{40pt}\xymatrix{\Bialg_\vet%
\ar@<.5ex>[d]^{P}&\Bialg_\vet
\ar@<.5ex>[d]^{P_1}\ar[l]_{\id_{\Bialg_\vet}}&\Bialg_\vet
\ar@<.5ex>[d]^{P_2}\ar[l]_{\id_{\Bialg_\vet}}\\
\vet\ar@<.5ex>@{.>}[u]^{\overline{T}}&{\vet_1}\ar[l]_{U_{0,1}}%
\ar@<.5ex>@{.>}[u]^{\overline{T}_1}&{\vet_2}\ar[l]_{U_{1,2}}%
\ar@<.5ex>@{.>}[u]^{\overline{T}_2}}
\end{equation*}%
where $\mathfrak{M}$ denotes the category of vector spaces over a fixed base
field $\Bbbk $, $\mathrm{Bialg}_{\mathfrak{M}}$ is the category of $\Bbbk $%
-bialgebras, $\overline{T}$ is the tensor bialgebra functor (the barred
notation serves to distinguish this functor from the tensor algebra functor $%
T:\mathfrak{M}\rightarrow \mathrm{Alg}_{\mathfrak{M}}$ which goes into $%
\Bbbk $-algebras) and $P$ is the primitive functor which assigns to each $%
\Bbbk $-bialgebra its space of primitive elements. We proved that this $P$
has a monadic decomposition of monadic length at most $2$. Moreover, when $%
\mathrm{char}\left( \Bbbk \right) =0,$ for every $V_{2}=\left( \left( V,\mu
\right) ,\mu _{1}\right) \in \mathfrak{M}_{2}$ one can define $\left[ x,y%
\right] :=\mu \left( xy-yx\right) $ for every $x,y\in V.$ Then $\left( V,%
\left[ -,-\right] \right) $ is an ordinary Lie algebra and $\overline{T}%
_{2}V_{2}=TV/\left( xy-yx-\left[ x,y\right] \mid x,y\in V\right) $ is the
corresponding universal enveloping algebra. This suggests a connection
between the category $\mathfrak{M}_{2}$ and the category $\mathrm{Lie}_{%
\mathfrak{M}}$ of Lie $\Bbbk $-algebras. It is then natural to expect the
existence of a category equivalence $\Lambda $ such that the following
diagram
\begin{equation*}
\xymatrixcolsep{1cm}\xymatrixrowsep{0.40cm}\xymatrix{\Bialg_\vet%
\ar@<.5ex>[dd]^{P}&&\Bialg_\vet\ar@<.5ex>[dd]^{P_1}|(.30)\hole\ar[ll]_{%
\mathrm{Id}_{\Bialg_\vet}}&&\Bialg_\vet
\ar@<.5ex>[dd]^{P_2}\ar[ll]_{\mathrm{Id}_{\Bialg_\vet}}\ar[dl]|{%
\mathrm{Id}_{\Bialg_\vet}}\\
&&&\Bialg_\vet\ar@<.5ex>[dd]^(.30){\mathcal{P}}\ar[ulll]^(.70){\mathrm{Id}_{%
\Bialg_\vet}}\\
\vet\ar@<.5ex>@{.>}[uu]^{\overline{T}}&&\vet_1\ar@<.5ex>@{.>}[uu]^{
\overline{T}_1}|(.70)\hole \ar[ll]_{U_{0,1}}&&\vet_2
\ar@<.5ex>@{.>}[uu]^{\overline{T}_2}
\ar[ll]_(.30){U_{1,2}}|\hole\ar[dl]^{{\Lambda}} \\
&&&\Lie_\vet\ar@<.5ex>@{.>}[uu]^(.70){{\mathcal{\overline{U}}}}\ar[ulll]^{H_%
\Lie}}
\end{equation*}%
commutes in its undashed part, where $H_{\mathrm{Lie}}$ denotes the
forgetful functor, $\overline{\mathcal{U}}$ the universal enveloping
bialgebra functor and $\mathcal{P}$ the corresponding primitive functor.

A first investigation showed that, in order to solve the problem above, it
is more natural to work with braided $\Bbbk $-vector spaces instead of
ordinary $\Bbbk $-vector spaces and to replace the categories $\mathfrak{M}$%
, $\mathrm{Bialg}_{\mathfrak{M}}$ and $\mathrm{Lie}_{\mathfrak{M}}$ with
their braided analogues $\mathrm{Br}_{\mathfrak{M}}$, $\mathrm{BrBialg}_{%
\mathfrak{M}}$ and $\mathrm{BrLie}_{\mathfrak{M}}$ consisting of braided
vector spaces, braided bialgebras and braided Lie algebras respectively. We
were further led to substitute $\mathfrak{M}$ with an arbitrary monoidal
category $\mathcal{M}$. We point out that, in order to produce a braided
analogue of the universal enveloping algebra which further carries a braided
bialgebra structure, the assumption that the underlying braiding is
symmetric is also needed. Thus let $\mathrm{Br}_{\mathcal{M}}^{s}$, $\mathrm{%
BrBialg}_{\mathcal{M}}^{s}$ and $\mathrm{BrLie}_{\mathcal{M}}^{s}$ be the
analogue of $\mathrm{Br}_{\mathcal{M}}$, $\mathrm{BrBialg}_{\mathcal{M}}$
and $\mathrm{BrLie}_{\mathcal{M}}$ consisting of objects with symmetric
braiding. Let $\overline{T}_{\mathrm{Br}}^{s}:\mathrm{Br}_{\mathcal{M}%
}^{s}\rightarrow \mathrm{BrBialg}_{\mathcal{M}}^{s}$ be the symmetric
braided tensor bialgebra functor and let $P_{\mathrm{Br}}^{s}$ be its right
adjoint, the primitive functor. We seek for a condition for $P_{\mathrm{Br}%
}^{s}$ to have monadic decomposition of monadic length at most $2$. On the
other hand, in this setting, the functor $P_{\mathrm{Br}}^{s}$ induces a
functor $\mathcal{P}_{\mathrm{Br}}^{s}:\mathrm{BrBialg}_{\mathcal{M}%
}^{s}\rightarrow \mathrm{BrLie}_{\mathcal{M}}^{s}$ which comes out to have a
left adjoint $\overline{\mathcal{U}}_{\mathrm{Br}}^{s}$, the universal
enveloping bialgebra functor.

In view of the celebrated Milnor-Moore Theorem, see Remark \ref{rem:MM}, we
say that a category $\mathcal{M}$ is a Milnor-Moore category (MM-category
for short) whenever the unit of the adjunction $(\overline{\mathcal{U}}_{%
\mathrm{Br}}^{s},\mathcal{P}_{\mathrm{Br}}^{s})$ is a functorial isomorphism
(plus other conditions required for the existence of the functors involved).
One of the main results in the paper is Theorem \ref{teo:LambdaBr}, which
ensures that, for an MM-category $\mathcal{M}$, the functor $P_{\mathrm{Br}%
}^{s}$ has a monadic decomposition of monadic length at most two. Moreover,
in this case, we can identify the category $\left( \mathrm{Br}_{\mathcal{M}%
}^{s}\right) _{2}$ with $\mathrm{BrLie}_{\mathcal{M}}^{s}$ through an
equivalence $\Lambda _{\mathrm{Br}}:\left( \mathrm{Br}_{\mathcal{M}%
}^{s}\right) _{2}\rightarrow \mathrm{BrLie}_{\mathcal{M}}^{s}$. Hence
MM-categories, besides having an interest in their own, give us an
environment where the functor $P_{\mathrm{Br}}^{s}$ has a behaviour
completely analogous to the classical vector space situation we investigated
in \cite[Theorem 3.4]{AGM-MonadicLie1}. In the case of a symmetric
MM-category $\mathcal{M}$ the connection with Milnor-Moore Theorem becomes
more evident. In fact, in this case, we can apply Theorem \ref{teo:LambdaSym}
to obtain that the unit of the adjunction $\left( \overline{\mathcal{U}},%
\mathcal{P}\right) $ is a functorial isomorphism.

The next step is to provide meaningful examples of MM-categories. A first
result in this direction is Theorem \ref{teo:mainVS}, based on a result by
Kharchenko, which states that the category $\mathfrak{M}$ of vector spaces
over a field of characteristic $0$ is an MM-Category. Note that the Lie
algebras involved are not ordinary ones but they depend on a symmetric
braiding.

Much of the material developed in the paper (see e.g. Proposition \ref%
{pro:comdat1}, Theorem \ref{teo:Udata} and the construction of the
adjunctions used therein) is devoted to the proof of our central result
namely Theorem \ref{teo:mainNew} which allows us to lift the property of
being an MM-category whenever a suitable monoidal functor is given. A main
tool in this proof is the concept of commutation datum which we introduce
and investigate in Section \ref{sec:ComData}. We use this Theorem \ref%
{teo:mainNew} in the case of the forgetful functor $F:\mathcal{M}\rightarrow
\mathfrak{M}$ where $\mathcal{M}$ is a subcategory of $\mathfrak{M}$. The
goal is to provide, in this way, meaningful examples of MM-categories $%
\mathcal{M}$ and, in the case when $\mathcal{M}$ is symmetric, to recognize
the corresponding type of Lie algebras. A first example of MM-category
obtained in this way is the category of Yetter-Drinfeld modules which is
considered in Example \ref{ex:YD}. Subsection \ref{subs:Quasi-Bialgebras}
(resp. \ref{subs:DualQuasiBialgebras}) deals with the case when $\mathcal{M}$
is the category of modules (resp. comodules) over a quasi-bialgebra (resp.
over a dual quasi-bialgebra). We prove that the forgetful functor satisfies
the assumptions of Theorem \ref{teo:mainNew} if and only if the
quasi-bialgebra (resp. the dual quasi-bialgebra) is a deformation of a usual
bialgebra, see Lemma \ref{lem:gamma} (resp. Lemma \ref{lem:gammaDual}). As
particular cases of this situation we prove that the category $\widetilde{%
\mathcal{H}}\left( \mathfrak{M}\right) $ of \cite[Proposition 1.1]{CG} is an
MM-category, see Remark \ref{rem:HH}. Note that an object in $\mathrm{Lie}_{%
\mathcal{M}},$ for $\mathcal{M}=\widetilde{\mathcal{H}}\left( \mathfrak{M}%
\right) $, is nothing but a Hom-Lie algebra. In Remark \ref{rem:cotriang},
we recover $\left( H,R\right) $-Lie algebras in the sense of \cite[%
Definition 4.1]{BFM} by considering the category of comodules over a
co-triangular bialgebra $\left( H,R\right) $ regarded as a co-triangular
dual quasi-bialgebra with trivial reassociator. In particular, let $G$ be an
abelian group endowed with an anti-symmetric bicharacter $\chi :G\times
G\rightarrow \Bbbk \setminus \{0\}$ and extend $\chi $ by linearity to a $%
\Bbbk $-linear map $R:\Bbbk \left[ G\right] \otimes \Bbbk \left[ G\right]
\rightarrow \Bbbk $, where $\Bbbk \left[ G\right] $ denotes the group
algebra. Then $\left( \Bbbk \left[ G\right] ,R\right) $ is a co-triangular
bialgebra and, as a consequence, we recover $\left( G,\chi \right) $-Lie
color algebras in the sense of \cite[Example 10.5.14]{Mo}, in Example \ref%
{ex:colorLie}, and in particular Lie superalgebras in Example \ref%
{ex:superLie}.

The appendices contain general results regarding the existence of
(co)equalizers in the category of (co)algebras, bialgebras and their braided
analogue over a monoidal category. These results are applied to obtain
Proposition \ref{pro:BeckBr}, which is used in the proof of Theorem \ref%
{teo:LambdaBr}.

\section{Preliminaries \label{sec:preliminares}}

In this section, we shall fix some basic notation and terminology.

\begin{notation}
Throughout this paper $\Bbbk $ will denote a field. All vector spaces will
be defined over $\Bbbk $. The unadorned tensor product $\otimes $ will
denote the tensor product over $\Bbbk $ if not stated otherwise.
\end{notation}

\begin{claim}
\textbf{Monoidal Categories.} Recall that (see \cite[Chap. XI]{Kassel}) a
\emph{monoidal category}\textbf{\ }is a category $\mathcal{M}$ endowed with
an object $\mathbf{1}\in \mathcal{M}$ (called \emph{unit}), a functor $%
\otimes :\mathcal{M}\times \mathcal{M}\rightarrow \mathcal{M}$ (called \emph{%
tensor product}), and functorial isomorphisms $a_{X,Y,Z}:(X\otimes Y)\otimes
Z\rightarrow $ $X\otimes (Y\otimes Z)$, $l_{X}:\mathbf{1}\otimes
X\rightarrow X,$ $r_{X}:X\otimes \mathbf{1}\rightarrow X,$ for every $X,Y,Z$
in $\mathcal{M}$. The functorial morphism $a$ is called the \emph{%
associativity constraint }and\emph{\ }satisfies the \emph{Pentagon Axiom, }%
that is the equality
\begin{equation*}
(U\otimes a_{V,W,X})\circ a_{U,V\otimes W,X}\circ (a_{U,V,W}\otimes
X)=a_{U,V,W\otimes X}\circ a_{U\otimes V,W,X}
\end{equation*}%
holds true, for every $U,V,W,X$ in $\mathcal{M}.$ The morphisms $l$ and $r$
are called the \emph{unit constraints} and they obey the \emph{Triangle
Axiom, }that is $(V\otimes l_{W})\circ a_{V,\mathbf{1},W}=r_{V}\otimes W$,
for every $V,W$ in $\mathcal{M}$.

A \emph{monoidal functor}\label{MonFun} (also called strong monoidal in the
literature)
\begin{equation*}
(F,\phi _{0},\phi _{2}):(\mathcal{M},\otimes ,\mathbf{1},a,l,r\mathbf{%
)\rightarrow (}\mathcal{M}^{\prime }\mathfrak{,}\otimes ^{\prime },\mathbf{1}%
^{\prime },a^{\prime },l^{\prime },r^{\prime }\mathbf{)}
\end{equation*}%
between two monoidal categories consists of a functor $F:\mathcal{M}%
\rightarrow \mathcal{M}^{\prime },$ an isomorphism $\phi
_{2}(U,V):F(U)\otimes ^{\prime }F(V)\rightarrow F(U\otimes V),$ natural in $%
U,V\in \mathcal{M}$, and an isomorphism $\phi _{0}:\mathbf{1}^{\prime
}\rightarrow F(\mathbf{1})$ such that the diagram
\begin{equation*}
\xymatrixcolsep{63pt}\xymatrixrowsep{30pt}\xymatrix{ (F(U)\otimes'
F(V))\otimes' F(W) \ar[d]|{a'_{F(U),F(V),F(W)}} \ar[r]^{\phi_2(U,V)\otimes'
F(W)} & F(U\otimes V)\otimes' F(W) \ar[r]^{\phi_2(U\otimes V,W)} &
F((U\otimes V)\otimes W) \ar[d]|{F(a_{ U,V, W})} \\ F(U)\otimes'
(F(V)\otimes' F(W)) \ar[r]^{F(U)\otimes' \phi_2(V,W)} & F(U)\otimes'
F(V\otimes W) \ar[r]^{\phi_2(U,V\otimes W)} & F(U\otimes (V\otimes W)) }
\end{equation*}%
is commutative, and the following conditions are satisfied:
\begin{equation*}
{F(l_{U})}\circ {\phi _{2}(\mathbf{1},U)}\circ ({\phi _{0}\otimes }^{\prime }%
{F(U)})={l}^{\prime }{_{F(U)}},\text{\quad }{F(r_{U})}\circ {\phi _{2}(U,%
\mathbf{1})}\circ ({F(U)\otimes }^{\prime }{\phi _{0}})={r}^{\prime }{%
_{F(U)}.}
\end{equation*}%
The monoidal functor is called \emph{strict }if the isomorphisms $\phi
_{0},\phi _{2}$ are identities of $\mathcal{M}^{\prime }$.
\end{claim}

The notions of algebra, module over an algebra, coalgebra and comodule over
a coalgebra can be introduced in the general setting of monoidal categories.

From now on we will omit the associativity and unity constraints unless
needed to clarify the context.\medskip

Let $V$ be an object in a monoidal category $\left( \mathcal{M},\otimes ,%
\mathbf{1}\right) $. Define iteratively $V^{\otimes n}$ for all $n\in
\mathbb{N}$ by setting $V^{\otimes 0}:=\mathbf{1}$ for $n=0$ and $V^{\otimes
n}:=V^{\otimes \left( n-1\right) }\otimes V$ for $n>0.$

\begin{remark}
\label{rem: AlgMon} Let $\mathcal{M}$ be a monoidal category. Denote by $%
\mathrm{Alg}_{\mathcal{M}}$ the category of algebras in $\mathcal{M}$ and
their morphisms. Assume that $\mathcal{M}$ has denumerable coproducts and
that the tensor products (i.e. $M\otimes \left( -\right) :\mathcal{M}%
\rightarrow \mathcal{M}$ and $\left( -\right) \otimes M:\mathcal{M}%
\rightarrow \mathcal{M}$, for every object $M$ in $\mathcal{M}$) preserve
such coproducts. By \cite[Theorem 2, page 172]{MacLane}, the forgetful
functor
\begin{equation*}
\Omega :\mathrm{Alg}_{\mathcal{M}}\rightarrow \mathcal{M}
\end{equation*}%
has a left adjoint $T:\mathcal{M}\rightarrow \mathrm{Alg}_{\mathcal{M}}$. By
construction $\Omega TV=\oplus _{n\in \mathbb{N}}V^{\otimes n}$ for every $%
V\in \mathcal{M}$. For every $n\in \mathbb{N},$ we will denote by
\begin{equation*}
\alpha _{n}V:V^{\otimes n}\rightarrow \Omega TV
\end{equation*}%
the canonical injection. Given a morphism $f:V\rightarrow W$ in $\mathcal{M}$%
, we have that $Tf$ is uniquely determined by the following equality%
\begin{equation}
\Omega Tf\circ \alpha _{n}V=\alpha _{n}W\circ f^{\otimes n},\text{ for every
}n\in \mathbb{N}\text{.}  \label{form:Tf}
\end{equation}%
The multiplication $m_{\Omega TV}$ and the unit $u_{\Omega TV}$ are uniquely
determined by
\begin{eqnarray}
m_{\Omega TV}\circ \left( \alpha _{m}V\otimes \alpha _{n}V\right) &=&\alpha
_{m+n}V,\text{ for every }m,n\in \mathbb{N}\text{,}  \label{form:TVm} \\
u_{\Omega TV} &=&\alpha _{0}V.  \label{form:TVu}
\end{eqnarray}%
Note that (\ref{form:TVm}) should be integrated with the proper unit
constrains when $m$ or $n$ is zero.

The unit $\eta $ and the counit $\epsilon $ of the adjunction $\left(
T,\Omega \right) $ are uniquely determined, for all $V\in \mathcal{M}$ and $%
\left( A,m_{A},u_{A}\right) \in \mathrm{Alg}_{\mathcal{M}}$ by the following
equalities%
\begin{equation}
\eta V:=\alpha _{1}V\qquad \text{and}\qquad \Omega \epsilon \left(
A,m_{A},u_{A}\right) \circ \alpha _{n}A:=m_{A}^{n-1}\text{ for every }n\in
\mathbb{N}  \label{form:etaeps}
\end{equation}%
where $m_{A}^{n-1}:A^{\otimes n}\rightarrow A$ is the iterated
multiplication of $A$ defined by $m_{A}^{-1}:=u_{A},m_{A}^{0}:=\mathrm{Id}%
_{A}$ and, for $n\geq 2,$ $m_{A}^{n-1}=m_{A}(m_{A}^{n-2}\otimes A).$
\end{remark}

\begin{definition}
Recall that a \emph{monad} on a category $\mathcal{A}$ is a triple $\mathbb{Q%
}:=\left( Q,m,u\right) ,$ where $Q:\mathcal{A}\rightarrow \mathcal{A}$ is a
functor, $m:QQ\rightarrow Q$ and $u:\mathcal{A}\rightarrow Q$ are functorial
morphisms satisfying the associativity and the unitality conditions $m\circ
mQ=m\circ Qm$ and $m\circ Qu=\mathrm{Id}_{Q}=m\circ uQ.$ An\emph{\ algebra}
over a monad $\mathbb{Q}$ on $\mathcal{A}$ (or simply a $\mathbb{Q}$\emph{%
-algebra}) is a pair $\left( X,{\mu }\right) $ where $X\in \mathcal{A}$ and $%
{\mu }:QX\rightarrow X$ is a morphism in $\mathcal{A}$ such that ${\mu }%
\circ Q{\mu }={\mu }\circ mX$ and ${\mu }\circ uX=\mathrm{Id}_{X}.$ A \emph{%
morphism between two} $\mathbb{Q}$-\emph{algebras} $\left( X,{\mu }\right) $
and $\left( X^{\prime },{\mu }^{\prime }\right) $ is a morphism $%
f:X\rightarrow X^{\prime }$ in $\mathcal{A}$ such that ${\mu }^{\prime
}\circ Qf=f\circ {\mu }$. We will denote by ${_{\mathbb{Q}}\mathcal{A}}$ the
category of $\mathbb{Q}$-algebras and their morphisms. This is the so-called
\emph{Eilenberg-Moore category} of the monad $\mathbb{Q}$ (which is
sometimes also denoted by ${\mathcal{A}}^{\mathbb{Q}}$ in the literature).
When the multiplication and unit of the monad are clear from the context, we
will just write $Q$ instead of $\mathbb{Q}$.
\end{definition}

A monad $\mathbb{Q}$ on $\mathcal{A}$ gives rise to an adjunction $\left(
F,U\right) :=\left( {_{\mathbb{Q}}}F,{_{\mathbb{Q}}}U\right) $ where $U:{_{%
\mathbb{Q}}\mathcal{A\rightarrow A}}$ is the forgetful functor and $F:%
\mathcal{A}\rightarrow {_{\mathbb{Q}}\mathcal{A}}$ is the free functor.
Explicitly:
\begin{equation*}
U\left( X,{\mu }\right) :=X,{\mathcal{\quad }}Uf:=f{\mathcal{\qquad }}\text{%
and}{\mathcal{\qquad }F}X:=\left( QX,mX\right) ,{\mathcal{\quad }}Ff:=Qf.
\end{equation*}%
Note that $UF=Q.$ The unit of the adjunction $\left( F,U\right) $ is given
by the unit $u:\mathcal{A}\rightarrow UF=Q$ of the monad $\mathbb{Q}$. The
counit $\lambda :FU\rightarrow {_{\mathbb{Q}}\mathcal{A}}$ of this
adjunction is uniquely determined by the equality $U\left( \lambda \left( X,{%
\mu }\right) \right) ={\mu }$ for every $\left( X,{\mu }\right) \in {_{%
\mathbb{Q}}\mathcal{A}}.$ It is well-known that the forgetful functor $U:{_{%
\mathbb{Q}}\mathcal{A}}\rightarrow \mathcal{A}$ is faithful and reflects
isomorphisms (see e.g. \cite[Proposition 4.1.4]{Borceux2}).

Let $\left( L:\mathcal{B}\rightarrow \mathcal{A},R:\mathcal{A}\rightarrow
\mathcal{B}\right) $ be an adjunction with unit $\eta $ and counit $\epsilon
$. Then $\left( RL,R\epsilon L,\eta \right) $ is a monad on $\mathcal{B}$
and we can consider the so-called \emph{comparison functor} $K:\mathcal{A}%
\rightarrow {_{RL}\mathcal{B}}$ of the adjunction $\left( L,R\right) $ which
is defined by $KX:=\left( RX,R\epsilon X\right) $ and $Kf:=Rf.$ Note that $%
_{RL}U\circ K=R$.

\begin{definition}
An adjunction $(L:\mathcal{B}\rightarrow \mathcal{A},R:\mathcal{A}%
\rightarrow \mathcal{B})$ is called \emph{monadic} (tripleable in Beck's
terminology \cite[Definition 3, page 8]{Beck}) whenever the comparison
functor $K:\mathcal{A}\rightarrow {_{RL}}\mathcal{B}$ is an equivalence of
categories. A functor $R$ is called \emph{monadic} if it has a left adjoint $%
L$ such that the adjunction $(L,R)$ is monadic, see \cite[Definition 3',
page 8]{Beck}. In a similar way one defines \emph{comonadic} adjunctions and
functors using the Eilenberg-Moore category ${^{LR}\mathcal{A}}$ of
coalgebras over the comonad induces by $\left( L,R\right) .$
\end{definition}

The notion of an idempotent monad is tightly connected with the monadic
length of a functor.

\begin{definition}
(\cite[page 231]{AT}) A monad $\left( Q,m,u\right) $ is called \emph{%
idempotent} whenever $m$ is an isomorphism. An adjunction $\left( L,R\right)
$ is called \emph{idempotent} whenever the associated monad is idempotent.
\end{definition}

The interested reader can find results on idempotent monads in \cite{AT,MS}.
Here we just note that the fact that $\left( L,R\right) $ is idempotent is
equivalent to require that $\eta R$ is a functorial isomorphism.

\begin{definition}
\label{def:MonDec} (See \cite[Definition 2.7]{AGM-MonadicLie1}, \cite[%
Definition 2.1]{AHW} and \cite[Definitions 2.10 and 2.14]{MS}) Fix a $N\in
\mathbb{N}$. We say that a functor $R$ has a\emph{\ monadic decomposition of
monadic length }$N$\emph{\ }whenever there exists a sequence $\left(
R_{n}\right) _{n\leq N}$ of functors $R_{n}$ such that

1) $R_{0}=R$;

2) for $0\leq n\leq N$, the functor $R_{n}$ has a left adjoint functor $%
L_{n} $;

3) for $0\leq n\leq N-1$, the functor $R_{n+1}$ is the comparison functor
induced by the adjunction $\left( L_{n},R_{n}\right) $ with respect to its
associated monad;

4) $L_{N}$ is full and faithful while $L_{n}$ is not full and faithful for $%
0\leq n\leq N-1.$

Compare with the construction performed in \cite[1.5.5, page 49]{Manes-PhD}.

Note that for functor $R:\mathcal{A}\rightarrow \mathcal{B}$ having a
monadic decomposition of monadic length\emph{\ }$N$, we have a diagram
\begin{equation}  \label{diag:MonadicDec}
\xymatrixcolsep{2cm} \xymatrix{\mathcal{A}\ar@<.5ex>[d]^{R_0}&\mathcal{A}%
\ar@<.5ex>[d]^{R_1}\ar[l]_{\mathrm{Id}_{\mathcal{A}}}&\mathcal{A}%
\ar@<.5ex>[d]^{R_2}\ar[l]_{\mathrm{Id}_{\mathcal{A}}}&\cdots
\ar[l]_{\mathrm{Id}_{\mathcal{A}}}\quad\cdots&\mathcal{A}\ar@<.5ex>[d]^{R_N}%
\ar[l]_{\mathrm{Id}_{\mathcal{A}}}\\
\mathcal{B}_0\ar@<.5ex>@{.>}[u]^{L_0}&\mathcal{B}_1\ar@<.5ex>@{.>}[u]^{L_1}
\ar[l]_{U_{0,1}}&\mathcal{B}_2 \ar@<.5ex>@{.>}[u]^{L_2}
\ar[l]_{U_{1,2}}&\cdots\ar[l]_{U_{2,3}}\quad\cdots&\mathcal{B}_N%
\ar@<.5ex>@{.>}[u]^{L_N}\ar[l]_{U_{N-1,N}} }
\end{equation}

where $\mathcal{B}_{0}=\mathcal{B}$ and, for $1\leq n\leq N,$

\begin{itemize}
\item $\mathcal{B}_{n}$ is the category of $\left( R_{n-1}L_{n-1}\right) $%
-algebras ${_{R_{n-1}L_{n-1}}\mathcal{B}}_{n-1}$;

\item $U_{n-1,n}:\mathcal{B}_{n}\rightarrow \mathcal{B}_{n-1}$ is the
forgetful functor ${_{R_{n-1}L_{n-1}}}U$.
\end{itemize}

We will denote by $\eta _{n}:\mathrm{Id}_{\mathcal{B}_{n}}\rightarrow
R_{n}L_{n}$ and $\epsilon _{n}:L_{n}R_{n}\rightarrow \mathrm{Id}_{\mathcal{A}%
}$ the unit and counit of the adjunction $\left( L_{n},R_{n}\right) $
respectively for $0\leq n\leq N$. Note that one can introduce the forgetful
functor $U_{m,n}:\mathcal{B}_{n}\rightarrow \mathcal{B}_{m}$ for all $m\leq
n $ with $0\leq m,n\leq N$.
\end{definition}

\begin{proposition}[{\protect\cite[Proposition 2.9]{AGM-MonadicLie1}}]
\label{pro:idemmonad2}Let $\left( L:\mathcal{B}\rightarrow \mathcal{A},R:%
\mathcal{A}\rightarrow \mathcal{B}\right) $ be an idempotent adjunction.
Then $R:\mathcal{A}\rightarrow \mathcal{B}$ has a\emph{\ }monadic
decomposition of monadic length at most $1 $.
\end{proposition}

We refer to \cite[Remarks 2.8 and 2.10]{AGM-MonadicLie1} for further
comments on monadic decompositions.

\begin{definition}
\label{def:comparable}We say that a functor $R$ is\emph{\ comparable }%
whenever there exists a sequence $\left( R_{n}\right) _{n\in \mathbb{N}}$ of
functors $R_{n}$ such that $R_{0}=R$ and, for $n\in \mathbb{N}$,

1) the functor $R_{n}$ has a left adjoint functor $L_{n}$;

2) the functor $R_{n+1}$ is the comparison functor induced by the adjunction
$\left( L_{n},R_{n}\right) $ with respect to its associated monad.

In this case we have a diagram as (\ref{diag:MonadicDec}) but not
necessarily stationary. Hence we can consider the forgetful functors $%
U_{m,n}:\mathcal{B}_{n}\rightarrow \mathcal{B}_{m}$ for all $m\leq n$ with $%
m,n\in \mathbb{N}$.
\end{definition}

\begin{remark}
Fix a $N\in \mathbb{N}$. A functor $R$ having a\emph{\ }monadic
decomposition of monadic length $N$ is comparable, see \cite[Remark 2.10]%
{AGM-MonadicLie1}.
\end{remark}

By the proof of Beck's Theorem \cite[Proof of Theorem 1]{Beck} one gets the
following result.

\begin{lemma}
\label{lem: coequalizers}Let $\mathcal{A}$ be a category such that, for any
(reflexive) pair $\left( f,g\right) $ (\cite[3.6, page 98]{TTT}) where $%
f,g:X\rightarrow Y$ are morphisms in $\mathcal{A}$, one can choose a
specific coequalizer. Then the comparison functor $K:\mathcal{A}\rightarrow {%
_{RL}\mathcal{B}}$ of an adjunction $\left( L,R\right) $ is a right adjoint.
Thus any right adjoint $R:\mathcal{A}\rightarrow {\mathcal{B}}$ is
comparable.
\end{lemma}

\begin{lemma}
\label{lem:Cappuccio}Let $F:\mathcal{C}\rightarrow {\mathcal{B}}$ be a full
and faithful functor which is also injective on objects.

1) Let $G:\mathcal{A}\rightarrow {\mathcal{B}}$ be a functor such that any
object in ${\mathcal{B}}$ which is image of $G$ is also image of $F$. Then
there is a unique functor $\widehat{G}:{\mathcal{A}}\rightarrow \mathcal{C}$
such that $F\widehat{G}=G.$

2) Let $G,G^{\prime }:\mathcal{A}\rightarrow {\mathcal{B}}$ be functors as
in 1). For any natural transformation $\gamma :G\rightarrow G^{\prime }$
there is a unique natural transformation $\widehat{\gamma }:\widehat{G}%
\rightarrow \widehat{G^{\prime }}$ such that $F\widehat{\gamma }=\gamma .$
\end{lemma}

\section{Commutation Data \label{sec:ComData}}

\begin{definition}
\label{def:conservative}A functor is called \emph{conservative} if it
reflects isomorphisms.
\end{definition}

\begin{lemma}
\label{lem:zeta}Let $\left( L,R\right) $ and $\left( L^{\prime },R^{\prime
}\right) $ be adjunctions that fit into the following commutative diagram of
functors%
\begin{equation}
\xymatrixrowsep{15pt}\xymatrix{\mathcal{A}
\ar[r]^F\ar[d]_R&\mathcal{A}^{\prime }\ar[d]^{R^\prime}\\
\mathcal{B}\ar[r]^G&\mathcal{B}^{\prime }}  \label{diag:cd}
\end{equation}
Then there is a unique natural transformation $\zeta :L^{\prime
}G\longrightarrow FL$ such that
\begin{equation}
R^{\prime }\zeta \circ \eta ^{\prime }G=G\eta  \label{form:zeta1}
\end{equation}%
holds, namely%
\begin{equation}
\zeta :=\left( L^{\prime }G\overset{L^{\prime }G\eta }{\longrightarrow }%
L^{\prime }GRL=L^{\prime }R^{\prime }FL\overset{\epsilon ^{\prime }FL}{%
\longrightarrow }FL\right) .  \label{form:zeta}
\end{equation}%
Moreover we have that%
\begin{equation}
\epsilon ^{\prime }F=F\epsilon \circ \zeta R.  \label{form:zeta2}
\end{equation}
\end{lemma}

\begin{definition}
\label{def:zetadatum}We will say that $\left( F,G\right) :\left( L,R\right)
\rightarrow \left( L^{\prime },R^{\prime }\right) $ is a \emph{commutation
datum} if

1) $\left( L,R\right) $ and $\left( L^{\prime },R^{\prime }\right) $ are
adjunctions that fit into the commutative diagram \eqref{diag:cd}.

2) The natural transformation $\zeta :L^{\prime }G\longrightarrow FL$ of
Lemma \ref{lem:zeta} is a functorial isomorphism.

The map $\zeta $ will be called the \emph{canonical transformation} of the
datum.
\end{definition}

\begin{proposition}
\label{pro:compcomdat}Let $\left( F,G\right) :\left( L,R\right) \rightarrow
\left( L^{\prime },R^{\prime }\right) $ and $\left( F^{\prime },G^{\prime
}\right) :\left( L^{\prime },R^{\prime }\right) \rightarrow \left( L^{\prime
\prime },R^{\prime \prime }\right) $ be a commutation data. Then $\left(
F^{\prime }F,G^{\prime }G\right) :\left( L,R\right) \rightarrow \left(
L^{\prime \prime },R^{\prime \prime }\right) $ is a commutation datum.
\end{proposition}

\begin{proposition}
\label{pro:zetaIter} Let $\left( F,G\right) :\left( L,R\right) \rightarrow
\left( L^{\prime },R^{\prime }\right) $ be a commutation datum of functors
as in (\ref{diag:cd}). Assume also that $F$ preserves coequalizers of
reflexive pairs of morphisms in $\mathcal{A}$ and that the comparison
functors $R_{1}^{\prime }$ and $R_{1}$ have left adjoints $L_{1}^{\prime }$
and $L_{1}$ respectively. Then $G$ lifts to a functor $G_{1}:\mathcal{B}%
_{1}\rightarrow \mathcal{B}_{1}^{\prime }\ $such that $G_{1}\left( B,\mu
\right) :=\left( GB,G\mu \circ R^{\prime }\zeta B\right) $, $G_{1}\left(
f\right) =Gf$ and the following diagrams commute.
\begin{equation*}
\xymatrixrowsep{15pt}\xymatrix{\mathcal{B}_1
\ar[r]^{G_1}\ar[d]_U&\mathcal{B}^{\prime }_1\ar[d]^{U^\prime}\\
\mathcal{B}\ar[r]^G&\mathcal{B}^{\prime }}\qquad \xymatrixrowsep{15pt}%
\xymatrix{\mathcal{A} \ar[r]^F\ar[d]_{R_1}&\mathcal{A}^{\prime
}\ar[d]^{R^\prime_1}\\ \mathcal{B}_1\ar[r]^{G_1}&\mathcal{B}^{\prime }_1}
\end{equation*}%
Moreover $\left( F,G_{1}\right) :\left( L_{1},R_{1}\right) \rightarrow
\left( L_{1}^{\prime },R_{1}^{\prime }\right) $ is a commutation datum.

Furthermore the functor $G_{1}$ is conservative (resp. faithful) whenever $G$
is.

If $G$ is faithful then $G_{1}$ is full (resp. injective on objects)
whenever $G$ is.
\end{proposition}

\begin{proof}
Denote by $\zeta $ the canonical map of the datum $\left( F,G\right) :\left(
L,R\right) \rightarrow \left( L^{\prime },R^{\prime }\right) .$ Set $\lambda
:=R^{\prime }\zeta :\left( R^{\prime }L^{\prime }\right) G\rightarrow
R^{\prime }FL=G\left( RL\right) .$ By Lemma \ref{lem:zeta}, $\zeta $
fulfills (\ref{form:zeta2}). By (\ref{form:zeta1}), we have $\lambda \circ
\eta ^{\prime }G=G\eta $ and%
\begin{gather*}
GR\epsilon L\circ \lambda RL\circ R^{\prime }L^{\prime }\lambda =GR\epsilon
L\circ R^{\prime }\zeta RL\circ R^{\prime }L^{\prime }R^{\prime }\zeta
=R^{\prime }\left[ F\epsilon L\circ \zeta RL\circ L^{\prime }R^{\prime
}\zeta \right] \\
\overset{(\ref{form:zeta2})}{=}R^{\prime }\left[ \epsilon ^{\prime }FL\circ
L^{\prime }R^{\prime }\zeta \right] =R^{\prime }\left[ \zeta \circ \epsilon
^{\prime }L^{\prime }G\right] =\lambda \circ R^{\prime }\epsilon ^{\prime
}L^{\prime }G
\end{gather*}%
Hence we can apply \cite[Lemma 1]{Joh} to the case $"K"=R^{\prime }L^{\prime
},$ $"H"=RL$ and $"T"=G.$ Thus we get a functor $G_{1}:\mathcal{B}%
_{1}\rightarrow \mathcal{B}_{1}^{\prime }\ $such that $U^{\prime }G_{1}=GU.$
Explicitly $G_{1}\left( B,\mu \right) :=\left( GB,G\mu \circ R^{\prime
}\zeta B\right) $, $G_{1}\left( f\right) =Gf$. We have%
\begin{gather*}
G_{1}R_{1}A=G_{1}\left( RA,R\epsilon A\right) =\left( GRA,GR\epsilon A\circ
R^{\prime }\zeta RA\right) \\
=\left( R^{\prime }FA,R^{\prime }\left[ F\epsilon A\circ \zeta RA\right]
\right) \overset{(\ref{form:zeta2})}{=}\left( R^{\prime }FA,R^{\prime
}\epsilon ^{\prime }FA\right) =R_{1}^{\prime }FA
\end{gather*}%
and $G_{1}R_{1}f=GRf=R^{\prime }Ff=R_{1}^{\prime }Ff$ so that $%
G_{1}R_{1}=R_{1}^{\prime }F.$ By the proof of \cite[Theorem 1]{Beck}, if we
set $\pi :=\epsilon L_{1}\circ LU\eta _{1},$ we get the following
coequalizer of a reflexive pair of morphisms in $\mathcal{A}$.
\begin{equation*}
\xymatrix{LRLB\ar@<.5ex>[rr]^-{L\mu}\ar@<-.5ex>[rr]_-{\epsilon
LB}&&LB=LU\left( B,\mu \right)\ar[rr]^-{\pi \left( B,\mu \right) }
&&L_1(B,\mu)}
\end{equation*}

Since $F$ preserves coequalizers of reflexive pairs of morphisms in $%
\mathcal{A}$, we get the bottom fork in the diagram below is a coequalizer.%
\begin{equation}  \label{diag:zeta}
\xymatrixrowsep{.5cm} \xymatrix{ L^{\prime }R^{\prime }L^{\prime }GB
\ar@<.5ex>[rr]^{L^{\prime }\left( G\mu \circ R^{\prime }\zeta B\right) }
\ar@<-.5ex>[rr]_{\epsilon ^{\prime }L^{\prime}GB}\ar[d]_{\zeta RLB\circ
L^{\prime }R^{\prime }\zeta B}&& L^{\prime }GB\ar[d]^{\zeta B}
\ar[rr]^-{F\pi \left( B,\mu\right)\circ\zeta B}&&FL_{1}\left( B,\mu
\right)\ar[d]^{\mathrm{Id}_{FL_{1}\left( B,\mu \right)}}\\
FLRLB\ar@<.5ex>[rr]^{FL\mu}\ar@<-.5ex>[rr]_{F\epsilon LB}&&FLB\ar[rr]^{F\pi
\left( B,\mu \right) } &&FL_1(B,\mu)}
\end{equation}

We compute%
\begin{gather*}
FL\mu \circ \left( \zeta RLB\circ L^{\prime }R^{\prime }\zeta B\right)
=\zeta B\circ L^{\prime }G\mu \circ L^{\prime }R^{\prime }\zeta B=\zeta
B\circ L^{\prime }\left( G\mu \circ R^{\prime }\zeta B\right) , \\
F\epsilon LB\circ \left( \zeta RLB\circ L^{\prime }R^{\prime }\zeta B\right)
\overset{(\ref{form:zeta2})}{=}\epsilon ^{\prime }FLB\circ L^{\prime
}R^{\prime }\zeta B=\zeta B\circ \epsilon ^{\prime }L^{\prime }GB
\end{gather*}%
so that diagram \eqref{diag:zeta} serially commutes. Since, in this diagram,
the vertical arrows are isomorphisms, we get the upper fork is a coequalizer
too. In a similar way, if we set $\pi ^{\prime }:=\epsilon ^{\prime
}L_{1}^{\prime }\circ L^{\prime }U^{\prime }\eta _{1}^{\prime }$ we have the
coequalizer
\begin{equation*}
\xymatrix{L^\prime R^\prime L^\prime B^\prime \ar@<.5ex>[rr]^-{L^\prime
\mu^\prime }\ar@<-.5ex>[rr]_-{\epsilon^\prime L^\prime B^\prime }&&L^\prime
B^\prime \ar[rr]^-{\pi^\prime \left( B^\prime ,\mu^\prime \right) }
&&L^\prime _1(B^\prime ,\mu^\prime )}
\end{equation*}
For $\left( B^{\prime },\mu ^{\prime }\right) :=G_{1}\left( B,\mu \right) $
we get the coequalizer%
\begin{equation*}
\xymatrixcolsep{1cm} \xymatrix{ L^{\prime }R^{\prime }L^{\prime }GB
\ar@<.5ex>[rr]^{L^{\prime }\left( G\mu \circ R^{\prime }\zeta B\right) }
\ar@<-.5ex>[rr]_{\epsilon ^{\prime }L^{\prime }GB}&& L^{\prime }GB
\ar[rr]^-{\pi ^{\prime }G_{1}\left( B,\mu \right)}&&L_{1}^{\prime
}G_{1}\left( B,\mu \right)}
\end{equation*}
By the foregoing, $F\pi \left( B,\mu \right) \circ \zeta B$ coequalizes the
pair $\left( L^{\prime }\left( G\mu \circ R^{\prime }\zeta B\right)
,\epsilon ^{\prime }L^{\prime }GB\right) .$ By the universal property of
coequalizers, there is a unique morphism $\zeta _{1}\left( B,\mu \right)
:L_{1}^{\prime }G_{1}\left( B,\mu \right) \longrightarrow FL_{1}\left( B,\mu
\right) $ such that $\zeta _{1}\left( B,\mu \right) \circ \pi ^{\prime
}G_{1}\left( B,\mu \right) =F\pi \left( B,\mu \right) \circ \zeta B.$ By the
uniqueness of the coequalizers, $\zeta _{1}\left( B,\mu \right) $ is an
isomorphism.

Let us check that $\zeta _{1}\left( B,\mu \right) $ is natural. Let $%
f:\left( B,\mu \right) \rightarrow \left( B^{\prime },\mu ^{\prime }\right) $
in $\mathcal{B}_{1}$. Then {\small
\begin{gather*}
FL_{1}f\circ \zeta _{1}\left( B,\mu \right) \circ \pi ^{\prime }G_{1}\left(
B,\mu \right) =FL_{1}f\circ F\pi \left( B,\mu \right) \circ \zeta B=F\pi
\left( B^{\prime },\mu ^{\prime }\right) \circ FLUf\circ \zeta B \\
=F\pi \left( B^{\prime },\mu ^{\prime }\right) \circ \zeta B^{\prime }\circ
L^{\prime }GUf=\zeta _{1}\left( B^{\prime },\mu ^{\prime }\right) \circ \pi
^{\prime }G_{1}\left( B^{\prime },\mu ^{\prime }\right) \circ L^{\prime
}U^{\prime }G_{1}f=\zeta _{1}\left( B^{\prime },\mu ^{\prime }\right) \circ
L_{1}G_{1}f\circ \pi ^{\prime }G_{1}\left( B,\mu \right)
\end{gather*}%
} so that $FL_{1}f\circ \zeta _{1}\left( B,\mu \right) =\zeta _{1}\left(
B^{\prime },\mu ^{\prime }\right) \circ L_{1}G_{1}f$ and hence we get a
functorial isomorphism $\zeta _{1}:L_{1}^{\prime }G_{1}\longrightarrow
FL_{1}.$ We have%
\begin{eqnarray*}
\epsilon _{1}\circ \pi R_{1} &=&\epsilon _{1}\circ \epsilon L_{1}R_{1}\circ
LU\eta _{1}R_{1}=\epsilon \circ LR\epsilon _{1}\circ LU\eta
_{1}R_{1}=\epsilon \circ LU\left[ R_{1}\epsilon _{1}\circ \eta _{1}R_{1}%
\right] =\epsilon , \\
R\pi \circ \eta U &=&R\epsilon L_{1}\circ RLU\eta _{1}\circ \eta U=R\epsilon
L_{1}\circ \eta UR_{1}L_{1}\circ U\eta _{1}=R\epsilon L_{1}\circ \eta
RL_{1}\circ U\eta _{1}=U\eta _{1}
\end{eqnarray*}%
so that, we obtain that $\epsilon _{1}\circ \pi R_{1}=\epsilon $ and $R\pi
\circ \eta U=U\eta _{1}$ and similar equations for $\left( L^{\prime
},R^{\prime }\right) .$ We compute%
\begin{eqnarray*}
U^{\prime }\left( R_{1}^{\prime }\zeta _{1}\circ \eta _{1}^{\prime
}G_{1}\right) &=&R^{\prime }\zeta _{1}\circ R^{\prime }\pi ^{\prime
}G_{1}\circ \eta ^{\prime }U^{\prime }G_{1}\overset{\text{def. }\zeta _{1}%
\text{ }}{=}R^{\prime }F\pi \circ R^{\prime }\zeta U\circ \eta ^{\prime }GU
\\
&\overset{(\ref{form:zeta1})}{=}&R^{\prime }F\pi \circ G\eta U=G\left[ R\pi
\circ \eta U\right] =GU\eta _{1}=U^{\prime }G_{1}\eta _{1}
\end{eqnarray*}%
so that $R_{1}^{\prime }\zeta _{1}\circ \eta _{1}^{\prime }G_{1}=G_{1}\eta
_{1}.$ Let us check that $G_{1}$ is conservative whenever $G$ is. Let $%
f:\left( B,\mu \right) \rightarrow \left( B^{\prime },\mu ^{\prime }\right) $
in $\mathcal{B}_{1}$ be such that $G_{1}f$ is an isomorphism. Then $%
U^{\prime }G_{1}f=GUf$ is an isomorphism. Since $G$ and $U$ are conservative
(see \cite[Proposition 4.1.4, page 189]{Borceux2}), we get that $f$ is an
isomorphism.

If $G$ is faithful, from $U^{\prime }G_{1}=GU$ and the fact that $U$ is
faithful, we deduce that $G_{1}$ is faithful.

Assume $G$ is faithful and full. Let $f\in \mathcal{B}_{1}^{\prime }\left(
G_{1}\left( B,\mu \right) ,G_{1}\left( B^{\prime },\mu ^{\prime }\right)
\right) $. Then $U^{\prime }f\in \mathcal{B}^{\prime }\left( GB,GB^{\prime
}\right) $ so that there is $h\in \mathcal{B}\left( B,B^{\prime }\right) $
such that $Gh=U^{\prime }f.$ We have%
\begin{eqnarray*}
G\left( \mu ^{\prime }\circ RLh\right) \circ R^{\prime }\zeta B &=&G\mu
^{\prime }\circ GRLh\circ R^{\prime }\zeta B=G\mu ^{\prime }\circ R^{\prime
}FLh\circ R^{\prime }\zeta B \\
&=&G\mu ^{\prime }\circ R^{\prime }\zeta B^{\prime }\circ R^{\prime
}L^{\prime }Gh=G\mu ^{\prime }\circ R^{\prime }\zeta B^{\prime }\circ
R^{\prime }L^{\prime }U^{\prime }f \\
&=&U^{\prime }f\circ G\mu \circ R^{\prime }\zeta B=Gh\circ G\mu \circ
R^{\prime }\zeta B=G\left( h\circ \mu \right) \circ R^{\prime }\zeta B.
\end{eqnarray*}%
Since $\zeta $ is an isomorphism and $G$ is faithful, we get that $\mu
^{\prime }\circ RLh=h\circ \mu $ so that there is a unique morphism $k\in
\mathcal{B}_{1}\left( \left( B,\mu \right) ,\left( B^{\prime },\mu ^{\prime
}\right) \right) $ such that $Uk=h.$ Hence $U^{\prime }f=Gh=GUk=U^{\prime
}G_{1}k$ and hence $f=G_{1}k.$ Thus $G_{1}$ is faithful and full.

Assume $G$ is faithful and injective on objects. If $G_{1}\left( B,\mu
\right) =G_{1}\left( B^{\prime },\mu ^{\prime }\right) $ i.e. $\left(
GB,G\mu \circ R^{\prime }\zeta B\right) =\left( GB^{\prime },G\mu ^{\prime
}\circ R^{\prime }\zeta B^{\prime }\right) $ then $GB=GB^{\prime }$ and $%
G\mu \circ R^{\prime }\zeta B=G\mu ^{\prime }\circ R^{\prime }\zeta
B^{\prime }$. In view of the assumptions on $G$ and since $\zeta $ is an
isomorphism, we get $\left( B,\mu \right) =\left( B^{\prime },\mu ^{\prime
}\right) $ so that $G_{1}$ is faithful and injective on objects.
\end{proof}

\begin{lemma}
\label{lem:IdempSquare}Let $\left( L,R\right) $ and $\left( L^{\prime
},R^{\prime }\right) $ be adjunctions of functors as in (\ref{diag:cd}).
Assume that $R^{\prime }\zeta R$ is a functorial isomorphism where $\zeta
:L^{\prime }G\longrightarrow FL$ is the natural transformation of Lemma \ref%
{lem:zeta}. Assume also that $G$ is conservative.

1) Let $A\in \mathcal{A}$ be such that $\eta ^{\prime }R^{\prime }FA$ is an
isomorphism. Then $\eta RA$ is an isomorphism.

2) If the adjunction $\left( L^{\prime },R^{\prime }\right) $ is idempotent
then $\left( L,R\right) $ is idempotent.
\end{lemma}

\begin{proof}
1) Since $\eta ^{\prime }R^{\prime }FA=\eta ^{\prime }GRA$ is an isomorphism
and $R^{\prime }\zeta R$ is an isomorphism, we get that $R^{\prime }\zeta
RA\circ \eta ^{\prime }GRA$ is an isomorphism. By (\ref{form:zeta1}) this
means that $G\eta RA$ is an isomorphism. Since $G$ is conservative, we
conclude.

2) $\left( L,R\right) $ is idempotent if and only if $\eta R$ is a
functorial isomorphism and similarly for $\left( L^{\prime },R^{\prime
}\right) $. Thus $\left( L^{\prime },R^{\prime }\right) $ is idempotent if
and only if $\eta ^{\prime }R^{\prime }$ is a functorial isomorphism. If the
letter condition holds then $\eta ^{\prime }R^{\prime }F$ is a functorial
isomorphism and, by 1), so is $\eta R$ and hence $\left( L,R\right) $ is
idempotent.
\end{proof}

\begin{lemma}
\label{lem:faithcom}Let $\left( F,G\right) :\left( L,R\right) \rightarrow
\left( L^{\prime },R^{\prime }\right) $ be a commutation datum. If $G$ is
conservative and $\eta ^{\prime }$ is an isomorphism so is $\eta $.
\end{lemma}

\begin{proof}
By (\ref{form:zeta1}), we have $R^{\prime }\zeta \circ \eta ^{\prime
}G=G\eta $.
\end{proof}

\begin{corollary}
\label{coro:idempcom}Let $\left( F,G\right) :\left( L,R\right) \rightarrow
\left( L^{\prime },R^{\prime }\right) $ be a commutation datum. Assume also
that $F$ preserves coequalizers of reflexive pairs of morphisms in $\mathcal{%
A}$ and that $G$ is conservative. Assume that both $R$ and $R^{\prime }$ are
comparable. Let $N\in \mathbb{N}.$

1) Let $A\in \mathcal{A}$ be such that $\eta _{N}^{\prime }R_{N}^{\prime }FA$
is an isomorphism. Then $\eta _{N}R_{N}A$ is an isomorphism.

2) If $\left( L_{N}^{\prime },R_{N}^{\prime }\right) $ is idempotent so is $%
\left( L_{N},R_{N}\right) .$
\end{corollary}

\begin{proof}
Apply Proposition \ref{pro:zetaIter} and Lemma \ref{lem:IdempSquare}.
\end{proof}

Next lemma will be a useful tool to construct new commutation data.

\begin{lemma}
\label{lem:LiftAdj}Let $\left( L^{\prime },R^{\prime }\right) $ be an
adjunction and let $F$ and $G$ be full and faithful functors which are also
injective on objects and have domain and codomain as in the following
diagrams. Assume that any object in $\mathcal{A}^{\prime }$ which is image
of $L^{\prime }G$ is also image of $F$ and that any object in $\mathcal{B}%
^{\prime }$ which is image of $R^{\prime }F$ is also image of $G.$ Set $L:=%
\widehat{L^{\prime }G}$ and $R:=\widehat{R^{\prime }F}\ $with notation as in
Lemma \ref{lem:Cappuccio}. Then $L$ and $R$ are the unique functors which
make the following diagrams commute
\begin{equation*}
\xymatrixrowsep{15pt}\xymatrix{\mathcal{A} \ar[r]^{F}&\mathcal{A}^{\prime
}\\ \mathcal{B}\ar@{.>}[u]^L\ar[r]^G&\mathcal{B}^{\prime }\ar[u]_{L^\prime}}%
\qquad \xymatrixrowsep{15pt}\xymatrix{\mathcal{A}
\ar[r]^F\ar@{.>}[d]_{R}&\mathcal{A}^{\prime }\ar[d]^{R^\prime}\\
\mathcal{B}\ar[r]^{G}&\mathcal{B}^{\prime }}
\end{equation*}
Moreover $\left( L,R\right) $ is an adjunction with unit $\eta :\mathrm{Id}_{%
\mathcal{B}}\rightarrow RL$ and counit $\epsilon :LR\rightarrow \mathrm{Id}_{%
\mathcal{A}}$ which satisfy
\begin{equation}  \label{form:LiftAdj}
G\eta =\eta ^{\prime }G\qquad \text{and}\qquad F\epsilon =\epsilon ^{\prime
}F
\end{equation}%
where $\eta ^{\prime }$ and $\epsilon ^{\prime }$ are the corresponding unit
and counit of $\left( L^{\prime },R^{\prime }\right) .$ Moreover $\left(
F,G\right) :\left( L,R\right) \rightarrow \left( L^{\prime },R^{\prime
}\right) $ is a commutation datum and the canonical transformation $\zeta
:L^{\prime }G\rightarrow FL$ is $\mathrm{\mathrm{\mathrm{\mathrm{\mathrm{%
\mathrm{Id}}}}}}_{L^{\prime }G}$.
\end{lemma}

\begin{proof}
Apply Lemma \ref{lem:Cappuccio}.once observed that $RL=\widehat{R^{\prime
}L^{\prime }G}$, $LR=\widehat{L^{\prime }R^{\prime }F}$, $\widehat{G}=%
\mathrm{Id}_{\mathcal{B}}$ and $\widehat{F}=\mathrm{Id}_{\mathcal{A}}.$ Then
define $\eta :=\widehat{\eta ^{\prime }G}$ and $\epsilon :=\widehat{\epsilon
^{\prime }F}.$
\end{proof}

\section{Braided objects and adjunctions \label{sec:BrObjAdj}}

\begin{definition}
\label{def:braid}Let $\left( \mathcal{M},\otimes ,\mathbf{1}\right) $ be a
monoidal category (as usual we omit the brackets although we are not
assuming the constraints are trivial).

1) Let $V$ be an object in $\mathcal{M}$. A morphism $c=c_{V}:V\otimes
V\rightarrow V\otimes V$ is called a \emph{braiding }(see \cite[Definition
XIII.3.1]{Kassel} where it is called a Yang-Baxter operator) if it satisfies
the quantum Yang-Baxter equation%
\begin{equation}
\left( c\otimes V\right) \left( V\otimes c\right) \left( c\otimes V\right)
=\left( V\otimes c\right) \left( c\otimes V\right) \left( V\otimes c\right)
\label{ec: braided equation}
\end{equation}%
on $V\otimes V\otimes V.$ \textbf{We further assume that }$c$\textbf{\ is
invertible}. The pair $\left( V,c\right) $ will be called a \emph{braided
object in }$\mathcal{M}$. A morphism of braided objects $(V,c_{V})$ and $%
(W,c_{W})$ in $\mathcal{M}$ is a morphism $f:V\rightarrow W$ such that $%
c_{W}(f\otimes f)=(f\otimes f)c_{V}.$ This defines the category $\mathrm{Br}%
_{\mathcal{M}}$ of braided objects and their morphisms.

2) \cite{Ba} A quadruple $(A,m,u,c)$ is called a \emph{braided algebra} if

\begin{itemize}
\item $(A,m,u)$ is an algebra in $\mathcal{M}$;

\item $(A,c)$ is a braided object in $\mathcal{M}$;

\item $m$ and $u$ commute with $c$, that is the following conditions hold:
\begin{gather}
c(m\otimes A)=(A\otimes m)(c\otimes A)(A\otimes c),  \label{Br2} \\
c(A\otimes m)=(m\otimes A)\left( A\otimes c\right) (c\otimes A),  \label{Br3}
\\
c(u\otimes A)l_{A}^{-1}=\left( A\otimes u\right) r_{A}^{-1},\qquad
c(A\otimes u)r_{A}^{-1}=\left( u\otimes A\right) l_{A}^{-1}.  \label{Br4}
\end{gather}
\end{itemize}

A morphism of braided algebras is, by definition, a morphism of algebras
which, in addition, is a morphism of braided objects. This defines the
category $\mathrm{BrAlg}_{\mathcal{M}}$ of braided algebras and their
morphisms.

3) A quadruple $(C,\Delta ,\varepsilon ,c)$ is called a \emph{braided
coalgebra} if

\begin{itemize}
\item $(C,\Delta ,\varepsilon )$ is a coalgebra in $\mathcal{M}$;

\item $(C,c)$ is a braided object in $\mathcal{M}$;

\item $\Delta $ and $\varepsilon $ commute with $c$, that is the following
relations hold:
\begin{gather}
(\Delta \otimes C)c=(C\otimes c)(c\otimes C)(C\otimes \Delta ),  \label{Br5}
\\
(C\otimes \Delta )c=(c\otimes C)(C\otimes c)(\Delta \otimes C),  \label{Br6}
\\
l_{C}(\varepsilon \otimes C)c=r_{C}\left( C\otimes \varepsilon \right)
,\qquad r_{C}(C\otimes \varepsilon )c=l_{C}\left( \varepsilon \otimes
C\right) .  \label{Br7}
\end{gather}
\end{itemize}

A morphism of braided coalgebras is, by definition, a morphism of coalgebras
which, in addition, is a morphism of braided objects. This defines the
category $\mathrm{BrCoalg}_{\mathcal{M}}$ of braided coalgebras and their
morphisms.

4) \cite[Definition 5.1]{Ta} A sextuple $(B,m,u,\Delta ,\varepsilon ,c)$ is
a called a \emph{braided bialgebra} if

\begin{itemize}
\item $(B,m,u,c)$ is a braided algebra;

\item $(B,\Delta ,\varepsilon ,c)$ is a braided coalgebra;

\item the following relations hold:%
\begin{gather}
\Delta m =(m\otimes m)(B\otimes c\otimes B)(\Delta \otimes \Delta ), \qquad
\Delta u =(u\otimes u)\Delta _{\mathbf{1}},  \label{Br1} \\
\varepsilon m =m_{\mathbf{1}}\left( \varepsilon \otimes \varepsilon \right)
, \qquad \varepsilon u =\mathrm{Id}_{\mathbf{1}}.  \label{Br9}
\end{gather}
\end{itemize}

A morphism of braided bialgebras is both a morphism of braided algebras and
coalgebras. This defines the category $\mathrm{BrBialg}_{\mathcal{M}}$ of
braided bialgebras.

Recall that a braiding $c$ is called symmetric or a symmetry whenever $c^2=%
\mathrm{Id}$. Denote by $\mathrm{Br}_{\mathcal{M}}^{s},\mathrm{BrAlg}_{%
\mathcal{M}}^{s},\mathrm{BrCoalg}_{\mathcal{M}}^{s}$ and $\mathrm{BrBialg}_{%
\mathcal{M}}^{s}$ the full subcategories of the respective categories above
consisting of objects with symmetric braiding. Denote by%
\begin{eqnarray*}
\mathbb{I}_{\mathrm{Br}}^{s} &:&\mathrm{Br}_{\mathcal{M}}^{s}\rightarrow
\mathrm{Br}_{\mathcal{M}},\qquad \mathbb{I}_{\mathrm{BrAlg}}^{s}:\mathrm{%
BrAlg}_{\mathcal{M}}^{s}\rightarrow \mathrm{BrAlg}_{\mathcal{M}}, \\
\mathbb{I}_{\mathrm{BrCoalg}}^{s} &:&\mathrm{BrCoalg}_{\mathcal{M}%
}^{s}\rightarrow \mathrm{BrCoalg}_{\mathcal{M}},\qquad \mathbb{I}_{\mathrm{%
BrBialg}}^{s}:\mathrm{BrBialg}_{\mathcal{M}}^{s}\rightarrow \mathrm{BrBialg}%
_{\mathcal{M}}
\end{eqnarray*}%
the obvious inclusion functors. Note that they are full, faithful, injective
on objects and conservative.
\end{definition}

\begin{remark}
\label{rem:aureo}Let $\mathcal{M}$ be a monoidal category. Let $\mathcal{A}$
be one of the following categories $\mathrm{Br}_{\mathcal{M}}$, $\mathrm{%
BrAlg}_{\mathcal{M}}$, $\mathrm{BrCoalg}_{\mathcal{M}}$ and $\mathrm{BrBialg}%
_{\mathcal{M}},$ let $\mathcal{A}^{s}$ be the corresponding full subcategory
of objects with symmetric braiding and denote by $\mathbb{I}_{\mathcal{A}%
}^{s}:\mathcal{A}^{s}\rightarrow \mathcal{A}$ the obvious inclusion functor.
Let $\mathbb{D}_{\mathcal{A}}:\mathcal{A\rightarrow M}$ be the forgetful
functor.

1) Let $\overline{X}\in \mathcal{A},\overline{Y}^{s}\in \mathcal{A}^{s}$ and
let $\overline{\alpha }:\overline{X}\rightarrow \mathbb{I}_{\mathcal{A}}^{s}%
\overline{Y}^{s}$ be a morphism in $\mathcal{A}$ such that $\alpha :=\mathbb{%
D}_{\mathcal{A}}\overline{\alpha }$ is a monomorphism. Set $X:=\mathbb{D}_{%
\mathcal{A}}\overline{X}$ and $Y:=\mathbb{D}_{\mathcal{A}}\mathbb{I}_{%
\mathcal{A}}^{s}\overline{Y}^{s}.$ Since $\alpha $ is braided we have $%
\left( \alpha \otimes \alpha \right) c_{X}^{2}=c_{Y}^{2}\left( \alpha
\otimes \alpha \right) =\alpha \otimes \alpha $ where $c_{X}$ and $c_{Y}$
are the braiding of $X$ and $Y$ respectively. Assume that $\alpha \otimes
\alpha $ is a monomorphism. Then we obtain $c_{X}^{2}=\mathrm{Id}_{X\otimes
X}$ so that we can write $\overline{X}=\mathbb{I}_{\mathcal{A}}^{s}\overline{%
X}^{s}$ for some $\overline{X}^{s}\in \mathcal{A}^{s}$ and $\overline{\alpha
}$ is a morphism in $\mathcal{A}^{s}$. Since $\mathbb{D}_{\mathcal{A}}$
reflects monomorphisms, we have proved that $\mathcal{A}^{s}$ is closed in $%
\mathcal{A}$ for those subobjects in $\mathcal{A}$ which are preserved by $%
\mathbb{D}_{\mathcal{A}}$ and by $\left( -\right) ^{\otimes 2}\circ \mathbb{D%
}_{\mathcal{A}}$ where $\left( -\right) ^{\otimes 2}:\mathcal{M}\rightarrow
\mathcal{M}:V\mapsto V\otimes V$.

2) Dually $\mathcal{A}^{s}$ is closed in $\mathcal{A}$ for those quotients
in $\mathcal{A}$ which are preserved by $\mathbb{D}_{\mathcal{A}}$ and by $%
\left( -\right) ^{\otimes 2}\circ \mathbb{D}_{\mathcal{A}}$.
\end{remark}

\begin{claim}
\label{cl:BrBialg}Let $\mathcal{M}$ and $\mathcal{M}^{\prime }$ be monoidal
categories. Following \cite[Proposition 2.5]{AM-BraidedOb}, every monoidal
functor $\left( F,\phi _{0},\phi _{2}\right) :\mathcal{M}\rightarrow
\mathcal{M}^{\prime }$ induces in a natural way suitable functors $\mathrm{Br%
}F$, $\mathrm{Alg}F,$ $\mathrm{BrAlg}F$ and $\mathrm{BrBialg}F$ such that
the following diagrams commute
\begin{gather*}
\xymatrixrowsep{15pt}\xymatrix{\Br_\M \ar[r]^{\Br
F}\ar[d]_H&\mathrm{Br}_{\mathcal{M}^{\prime }}\ar[d]^{H^\prime}\\
\mathcal{M} \ar[r]^F&\mathcal{M}^{\prime }}\qquad \xymatrixrowsep{15pt}%
\xymatrix{\mathrm{Alg}_{\mathcal{M}}
\ar[r]^{\mathrm{Alg}F}\ar[d]_\Omega&\mathrm{Alg}_{\mathcal{M}^{\prime
}}\ar[d]^{\Omega^\prime}\\ \mathcal{M} \ar[r]^F&\mathcal{M}^{\prime }}\qquad %
\xymatrixrowsep{15pt}\xymatrixcolsep{25pt}\xymatrix{\mathrm{BrAlg}_{%
\mathcal{M}}
\ar[r]^{\mathrm{BrAlg}F}\ar[d]_{H_{\mathrm{Alg}}}&\mathrm{BrAlg}_{%
\mathcal{M}^{\prime }}\ar[d]^{{H^\prime_{\mathrm{Alg}}}}\\
\mathrm{Alg}_{\mathcal{M}}
\ar[r]^{\mathrm{Alg}F}&\mathrm{Alg}_{\mathcal{M}^{\prime }}} \\
\xymatrixrowsep{15pt}\xymatrixcolsep{25pt}\xymatrix{\mathrm{BrAlg}_{%
\mathcal{M}}
\ar[r]^{\mathrm{BrAlg}F}\ar[d]_{\Omega_{\mathrm{Br}}}&\mathrm{BrAlg}_{%
\mathcal{M}^{\prime }}\ar[d]^{{\Omega^\prime_{\mathrm{Br}}}}\\
\mathrm{Br}_{\mathcal{M}}
\ar[r]^{\mathrm{Br}F}&\mathrm{Br}_{\mathcal{M}^{\prime }}}\qquad %
\xymatrixrowsep{15pt}\xymatrixcolsep{35pt}\xymatrix{\mathrm{BrBialg}_{%
\mathcal{M}}
\ar[r]^{\mathrm{BrBialg}F}\ar[d]_{\mho_{\mathrm{Br}}}&\mathrm{BrBialg}_{%
\mathcal{M}^{\prime }}\ar[d]^{{\mho^\prime_{\mathrm{Br}}}}\\
\mathrm{BrAlg}_{\mathcal{M}}
\ar[r]^{\mathrm{BrAlg}F}&\mathrm{BrAlg}_{\mathcal{M}^{\prime }}}.
\end{gather*}%
where the vertical arrows denote the obvious forgetful functors. Moreover

\begin{enumerate}
\item The functors $H,\Omega ,H_{\mathrm{Alg}},\Omega _{\mathrm{Br}},\mho _{%
\mathrm{Br}}$ are conservative.

\item $\mathrm{Br}F,\mathrm{Alg}F,\mathrm{BrAlg}F$ and $\mathrm{BrBialg}F$
are equivalences (resp. isomorphisms or conservative) whenever $F$ is.

\item $F$ preserves symmetric objects (this follows by definition of the
braiding induced by $F$). Thus we can define $\mathrm{Br}^{s}F,\mathrm{BrAlg}%
^{s}F$ and $\mathrm{BrBialg}^{s}F$ such that%
\begin{equation}
\xymatrixrowsep{15pt}\xymatrixcolsep{25pt}\xymatrix{
\mathrm{Br}^{s}_\mathcal{M}\ar[r]^{\mathrm{Br}^{s}F}\ar[d]_{\mathbb{I}_{%
\mathrm{Br}}^{s}}&\mathrm{Br}^{s}_{\mathcal{M}^\prime}\ar[d]^{\mathbb{I}_{%
\mathrm{Br}}^{s}}\\ \mathrm{Br}_{\mathcal{M}}
\ar[r]^{\mathrm{Br}F}&\mathrm{Br}_{\mathcal{M}^\prime}}\quad %
\xymatrixrowsep{15pt}\xymatrixcolsep{30pt}\xymatrix{
\mathrm{BrAlg}^{s}_\mathcal{M}\ar[r]^{\mathrm{BrAlg}^{s}F}\ar[d]_{%
\mathbb{I}_{\mathrm{BrAlg}}^{s}}&\mathrm{BrAlg}^{s}_{\mathcal{M}^\prime}%
\ar[d]^{\mathbb{I}_{\mathrm{BrAlg}}^{s}}\\ \mathrm{BrAlg}_{\mathcal{M}}
\ar[r]^{\mathrm{BrAlg}F}&\mathrm{BrAlg}_{\mathcal{M}^\prime}} \quad %
\xymatrixrowsep{15pt}\xymatrixcolsep{40pt}\xymatrix{
\mathrm{BrBialg}^{s}_\mathcal{M}\ar[r]^{\mathrm{BrBialg}^{s}F}\ar[d]_{%
\mathbb{I}_{\mathrm{BrBialg}}^{s}}&\mathrm{BrBialg}^{s}_{\mathcal{M}^\prime}%
\ar[d]^{\mathbb{I}_{\mathrm{BrBialg}}^{s}}\\ \mathrm{BrBialg}_{\mathcal{M}}
\ar[r]^{\mathrm{BrBialg}F}&\mathrm{BrBialg}_{\mathcal{M}^\prime}}
\label{diag:BrsF}
\end{equation}
\end{enumerate}
\end{claim}

Next aim is to recall some meaningful adjunctions that will be investigated
in the paper.

\begin{claim}
\label{cl:TbrStrict} Let $\mathcal{M}$ be a monoidal category. Assume that $%
\mathcal{M}$ has denumerable coproducts and that the tensor functors
preserve such coproducts. In view of \cite[Proposition 3.1]{AM-BraidedOb},
the functor $\Omega _{\mathrm{Br}}$ has a left adjoint $T_{\mathrm{Br}}$ and
the following diagrams commute.
\begin{equation}
\xymatrixrowsep{15pt}\xymatrixcolsep{25pt}\xymatrix{\mathrm{BrAlg}_{%
\mathcal{M}} \ar[r]^{H_\mathrm{Alg}}&\mathrm{Alg}_{\mathcal{M}}\\
\mathrm{Br}_{\mathcal{M}}
\ar[u]^{T_{\mathrm{Br}}}\ar[r]^H&\mathcal{M}\ar[u]_T}. \qquad %
\xymatrixrowsep{15pt}\xymatrixcolsep{25pt}\xymatrix{\mathrm{BrAlg}_{%
\mathcal{M}}
\ar[r]^{H_\mathrm{Alg}}\ar[d]_{\Omega_{\mathrm{Br}}}&\mathrm{Alg}_{%
\mathcal{M}}\ar[d]^{\Omega}\\ \mathrm{Br}_{\mathcal{M}} \ar[r]^H&\mathcal{M}}
\label{form:HtildeTbrOmegaBr}
\end{equation}%
The unit $\eta _{\mathrm{Br}}$ and the counit $\epsilon _{\mathrm{Br}}$ are
uniquely determined by the following equations
\begin{equation}
H\eta _{\mathrm{Br}}=\eta H,\qquad H_{\mathrm{Alg}}\epsilon _{\mathrm{Br}%
}=\epsilon H_{\mathrm{Alg}},  \label{form:TbrStrict}
\end{equation}%
where $\eta $ and $\epsilon $ denote the unit and counit of the adjunction $%
\left( T,\Omega \right) $ of Remark \ref{rem: AlgMon}. Using Lemma \ref%
{lem:LiftAdj}, one shows that the adjunction $\left( T_{\mathrm{Br}},\Omega
_{\mathrm{Br}}\right) $ induces an adjunction $\left( T_{\mathrm{Br}%
}^{s},\Omega _{\mathrm{Br}}^{s}\right) $ such that the following diagrams
commute.
\begin{equation}
\xymatrixrowsep{15pt}\xymatrixcolsep{25pt}\xymatrix{\mathrm{BrAlg}^s_{%
\mathcal{M}}
\ar[r]^{\mathbb{I}_{\mathrm{BrAlg}}^{s}}&\mathrm{BrAlg}_{\mathcal{M}}\\
\mathrm{Br}^s_{\mathcal{M}}
\ar@{.>}[u]^{T^s_{\mathrm{Br}}}\ar[r]^{\mathbb{I}_{\mathrm{Br}}^{s}}&%
\mathrm{Br}_{\mathcal{M}}\ar[u]_{T_\mathrm{Br}}} \qquad \xymatrixrowsep{15pt}%
\xymatrixcolsep{25pt}\xymatrix{\mathrm{BrAlg}^s_{\mathcal{M}}
\ar[r]^{\mathbb{I}_{\mathrm{BrAlg}}^{s}}\ar@{.>}[d]_{\Omega^s_{%
\mathrm{Br}}}&\mathrm{BrAlg}_{\mathcal{M}}\ar[d]^{\Omega_{\mathrm{Br}}}\\
\mathrm{Br}^s_{\mathcal{M}}
\ar[r]^{\mathbb{I}_{\mathrm{Br}}^{s}}&\mathrm{Br}_\mathcal{M}}
\label{diag:Tbrs}
\end{equation}
The lemma can be applied by the following argument. It is clear that any
object in the image of $\Omega _{\mathrm{Br}}\mathbb{I}_{\mathrm{BrAlg}}^{s}$
is in the image of $\mathbb{I}_{\mathrm{Br}}^{s}$. Let $\left( M,c\right)
\in \mathrm{Br}_{\mathcal{M}}^{s}$ and set $\left(
A,m_{A},u_{A},c_{A}\right) :=T_{\mathrm{Br}}\mathbb{I}_{\mathrm{Br}%
}^{s}\left( M,c\right) $. Using \cite[(42)]{AM-BraidedOb}, we have $%
c_{A}\left( \alpha _{m}M\otimes \alpha _{n}M\right) =\left( \alpha
_{n}V\otimes \alpha _{m}M\right) c_{A}^{m,n}$ so that%
\begin{equation*}
c_{A}^{2}\left( \alpha _{m}M\otimes \alpha _{n}M\right) =c_{A}\left( \alpha
_{n}V\otimes \alpha _{m}M\right) c_{A}^{m,n}=\left( \alpha _{m}M\otimes
\alpha _{n}M\right) c_{A}^{n,m}c_{A}^{m,n}
\end{equation*}%
and $c_{A}^{n,m}c_{A}^{m,n}=\mathrm{Id}_{M^{\otimes \left( m+n\right) }}.$
The latter is proved by induction on $t=m+n\in \mathbb{N}$ using \cite[%
Proposition 2.7]{AM-BraidedOb}.

Thus $c_{A}^{2}\left( \alpha _{m}M\otimes \alpha _{n}M\right) =\left( \alpha
_{m}M\otimes \alpha _{n}M\right) $ for every $m,n\in \mathbb{N}$ and hence $%
c_{A}^{2}=\mathrm{Id}_{A\otimes A}$. Therefore $\left(
A,m_{A},u_{A},c_{A}\right) \in \mathrm{BrAlg}_{\mathcal{M}}^{s}$ and $T_{%
\mathrm{Br}}\mathbb{I}_{\mathrm{Br}}^{s}\left( M,c\right) =\mathbb{I}_{%
\mathrm{BrAlg}}^{s}\left( A,m_{A},u_{A},c_{A}\right) .$ Hence any object in
the image of $T_{\mathrm{Br}}\mathbb{I}_{\mathrm{Br}}^{s}$ is also in the
image of $\mathbb{I}_{\mathrm{BrAlg}}^{s}.$ Thus, by Lemma \ref{lem:LiftAdj}
we have the desired adjunction with unit $\eta _{\mathrm{Br}}^{s}:\mathrm{Id}%
_{\mathrm{Br}_{\mathcal{M}}^{s}}\rightarrow \Omega _{\mathrm{Br}}^{s}T_{%
\mathrm{Br}}^{s}$ and counit $\epsilon _{\mathrm{Br}}^{s}:T_{\mathrm{Br}%
}^{s}\Omega _{\mathrm{Br}}^{s}\rightarrow \mathrm{Id}_{\mathrm{BrAlg}_{%
\mathcal{M}}^{s}}$ which are uniquely defined by
\begin{equation}
\mathbb{I}_{\mathrm{BrAlg}}^{s}\epsilon _{\mathrm{Br}}^{s}=\epsilon _{%
\mathrm{Br}}\mathbb{I}_{\mathrm{BrAlg}}^{s}\qquad \text{and}\qquad \mathbb{I}%
_{\mathrm{Br}}^{s}\eta _{\mathrm{Br}}^{s}=\eta _{\mathrm{Br}}\mathbb{I}_{%
\mathrm{Br}}^{s}.  \label{form:TbrStricts}
\end{equation}%
Furthermore $\left( \mathbb{I}_{\mathrm{BrAlg}}^{s},\mathbb{I}_{\mathrm{Br}%
}^{s}\right) :\left( T_{\mathrm{Br}}^{s},\Omega _{\mathrm{Br}}^{s}\right)
\rightarrow \left( T_{\mathrm{Br}},\Omega _{\mathrm{Br}}\right) $ is a
commutation datum with canonical transformation given by the identity.
\end{claim}

\begin{definition}
\label{def:prim}Let $\mathcal{M}$ be a preadditive monoidal category with
equalizers. Assume that the tensor functors are additive. Let $\mathbb{C}%
:=\left( C,\Delta _{C},\varepsilon _{C},u_{C}\right) $ be a coalgebra $%
\left( C,\Delta _{C},\varepsilon _{C}\right) $ endowed with a coalgebra
morphism $u_{C}:\mathbf{1}\rightarrow C$. In this setting we always
implicitly assume that we can choose a specific equalizer%
\begin{equation}
\xymatrixcolsep{1cm}\xymatrix{ P\left( \mathbb{C}\right) \ar[r]^{\xi
\mathbb{C}} & C \ar@<.5ex>[rr]^{\Delta _{C}} \ar@<-.5ex>[rr]_{\left(
C\otimes u_{C}\right) r_{C}^{-1}+\left( u_{C}\otimes C\right)
l_{C}^{-1}}&&C\otimes C }  \label{diag:prim}
\end{equation}%
We will use the same symbol when $\mathbb{C}$ comes out to be enriched with
an extra structure such us when $\mathbb{C}$ will denote a bialgebra or a
braided bialgebra.
\end{definition}

We now investigate some properties of $T_{\mathrm{Br}}$.

\begin{claim}
\label{cl:TbarStrict} Let $\mathcal{M}$ be a preadditive monoidal category
with equalizers and denumerable coproducts. Assume that the tensor functors
are additive and preserve equalizers and denumerable coproducts. By \ref%
{cl:TbrStrict}, the forgetful functor $\Omega _{\mathrm{Br}}:\mathrm{BrAlg}_{%
\mathcal{M}}\rightarrow \mathrm{Br}_{\mathcal{M}}$ has a left adjoint $T_{%
\mathrm{Br}}:\mathrm{Br}_{\mathcal{M}}\rightarrow \mathrm{BrAlg}_{\mathcal{M}%
}.$ In view of \cite[Lemma 3.4]{AM-BraidedOb}, $T_{\mathrm{Br}}$ induces a
functor $\overline{T}_{\mathrm{Br}}$ such that
\begin{equation}
\xymatrixrowsep{12pt}\xymatrixcolsep{0.7cm}\xymatrix{\mathrm{BrBialg}_{%
\mathcal{M}} \ar[rr]^{\mho _{\mathrm{Br}}}&& \mathrm{BrAlg}_{\mathcal{M}} \\
&\mathrm{Br}_{\mathcal{M}}\ar[ul]^{\overline{T}_{\mathrm{Br}}}\ar[ru]_{
T_{\mathrm{Br}} }}  \label{form:OmegRibTbarBr}
\end{equation}%
Explicitly, for all $\left( V,c\right) \in \mathrm{Br}_{\mathcal{M}}$, we
can write $\overline{T}_{\mathrm{Br}}\left( V,c\right) $ in the form $%
(A,m_{A},u_{A},\Delta _{A},\varepsilon _{A},c_{A})$ where $\Delta
_{A}:A\rightarrow A\otimes A$ and $\varepsilon _{A}:A\rightarrow \mathbf{1}$
are unique algebra morphisms such that
\begin{eqnarray}
\Delta _{A}\circ \alpha _{1}V &=&\delta _{V}^{l}+\delta _{V}^{r},
\label{form:TBrdelta} \\
\varepsilon _{A}\circ \alpha _{1}V &=&0,  \label{form:TBreps}
\end{eqnarray}%
where $\delta _{V}^{l}:=\left( u_{A}\otimes \alpha _{1}V\right) \circ
l_{V}^{-1}$ and $\delta _{V}^{r}:=\left( \alpha _{1}V\otimes u_{A}\right)
\circ r_{V}^{-1}.$ Moreover%
\begin{equation}
\varepsilon _{A}\circ \alpha _{n}V=\delta _{n,0}\mathrm{Id}_{\mathbf{1}},%
\text{ for every }n\in \mathbb{N}\text{.}  \label{form:TBrepsGEN}
\end{equation}

In view of \cite[Theorem 3.5]{AM-BraidedOb}, the functor $\overline{T}_{%
\mathrm{Br}}$ has a right adjoint $P_{\mathrm{Br}}:\mathrm{BrBialg}_{%
\mathcal{M}}\rightarrow \mathrm{Br}_{\mathcal{M}},$ which is constructed in
\cite[Lemma 3.3]{AM-BraidedOb}. The unit $\overline{\eta }_{\mathrm{Br}}$
and the counit $\overline{\epsilon }_{\mathrm{Br}}$ are uniquely determined
by the following equalities%
\begin{eqnarray}
\xi \overline{T}_{\mathrm{Br}}\circ \overline{\eta }_{\mathrm{Br}} &=&\eta _{%
\mathrm{Br}},  \label{form:BarEta} \\
\epsilon _{\mathrm{Br}}\mho _{\mathrm{Br}}\circ T_{\mathrm{Br}}\xi &=&\mho _{%
\mathrm{Br}}\overline{\epsilon }_{\mathrm{Br}},  \label{form:BarEps}
\end{eqnarray}%
where $\left( V,c\right) \in \mathrm{Br}_{\mathcal{M}},\mathbb{B}\in \mathrm{%
BrBialg}_{\mathcal{M}}$ while $\eta _{\mathrm{Br}}$ and $\epsilon _{\mathrm{%
Br}}$ denote the unit and counit of the adjunction $\left( T_{\mathrm{Br}%
},\Omega _{\mathrm{Br}}\right) $ respectively. Moreover $\xi :P_{\mathrm{Br}%
}\rightarrow \Omega _{\mathrm{Br}}\mho _{\mathrm{Br}}$ is a natural
transformation induced by the canonical morphism in (\ref{diag:prim}).

Note that from \ref{cl:TbrStrict} it is clear that any object in the image
of $\overline{T}_{\mathrm{Br}}\mathbb{I}_{\mathrm{Br}}^{s}$ has symmetric
braiding and hence it is in the image of $\mathbb{I}_{\mathrm{BrBialg}}^{s}.$
Let $\mathbb{B}\in \mathrm{BrBialg}_{\mathcal{M}}^{s}$ and set $\left(
P,c_{P}\right) :=$ $P_{\mathrm{Br}}\mathbb{I}_{\mathrm{BrBialg}}^{s}\mathbb{B%
}$. Since the tensor functors preserve equalizers, we have that $\xi \mathbb{%
B\otimes }\xi \mathbb{B}$ is a monomorphism so that we can apply 1) in
Remark \ref{rem:aureo} to get that $\left( P,c_{P}\right) \in \mathrm{Br}_{%
\mathcal{M}}^{s}$. Thus any object in the image of $P_{\mathrm{Br}}\mathbb{I}%
_{\mathrm{BrBialg}}^{s}$ is also in the image of $\mathbb{I}_{\mathrm{Br}%
}^{s}.$ Hence, by Lemma \ref{lem:LiftAdj} we have an adjunction $\left(
\overline{T}_{\mathrm{Br}}^{s},P_{\mathrm{Br}}^{s}\right) $ such that the
diagrams
\begin{equation}
\xymatrixrowsep{15pt}\xymatrixcolsep{25pt}\xymatrix{\mathrm{BrBialg}^s_{%
\mathcal{M}}
\ar[r]^{\mathbb{I}_{\mathrm{BrBialg}}^{s}}&\mathrm{BrBialg}_{\mathcal{M}}\\
\mathrm{Br}^s_{\mathcal{M}}
\ar@{.>}[u]^{\overline{T}_{\mathrm{Br}}^s}\ar[r]^{\mathbb{I}_{%
\mathrm{Br}}^{s}}&\mathrm{Br}_{\mathcal{M}}\ar[u]_{\overline{T}_{%
\mathrm{Br}}}} \qquad \xymatrixrowsep{15pt}\xymatrixcolsep{25pt}%
\xymatrix{\mathrm{BrBialg}^s_{\mathcal{M}}
\ar[r]^{\mathbb{I}_{\mathrm{BrBialg}}^{s}}\ar@{.>}[d]_{P^s_{\mathrm{Br}}}&%
\mathrm{BrBialg}_{\mathcal{M}}\ar[d]^{P_{\mathrm{Br}}}\\
\mathrm{Br}^s_{\mathcal{M}}
\ar[r]^{\mathbb{I}_{\mathrm{Br}}^{s}}&\mathrm{Br}_\mathcal{M}}
\label{diag:PsTsBr}
\end{equation}
commute and the unit $\overline{\eta }_{\mathrm{Br}}^{s}:\mathrm{Id}_{%
\mathrm{Br}_{\mathcal{M}}^{s}}\rightarrow P_{\mathrm{Br}}^{s}\overline{T}_{%
\mathrm{Br}}^{s}$ and the counit $\overline{\epsilon }_{\mathrm{Br}}^{s}:%
\overline{T}_{\mathrm{Br}}^{s}P_{\mathrm{Br}}^{s}\rightarrow \mathrm{Id}_{%
\mathrm{BrBialg}_{\mathcal{M}}^{s}}$ are uniquely defined by
\begin{equation}
\mathbb{I}_{\mathrm{BrBialg}}^{s}\overline{\epsilon }_{\mathrm{Br}}^{s}=%
\overline{\epsilon }_{\mathrm{Br}}\mathbb{I}_{\mathrm{BrBialg}}^{s}\qquad
\text{and}\qquad \mathbb{I}_{\mathrm{Br}}^{s}\overline{\eta }_{\mathrm{Br}%
}^{s}=\overline{\eta }_{\mathrm{Br}}\mathbb{I}_{\mathrm{Br}}^{s}.
\label{form:barbrs}
\end{equation}%
Moreover $\left( \mathbb{I}_{\mathrm{BrBialg}}^{s},\mathbb{I}_{\mathrm{Br}%
}^{s}\right) :\left( \overline{T}_{\mathrm{Br}}^{s},P_{\mathrm{Br}%
}^{s}\right) \rightarrow \left( \overline{T}_{\mathrm{Br}},P_{\mathrm{Br}%
}\right) $ is a commutation datum with canonical transformation given by the
identity. Note that the functor $\mho _{\mathrm{Br}}$ induces a functor $%
\mho _{\mathrm{Br}}^{s}$ such that the following diagrams commute.%
\begin{equation}
\xymatrixrowsep{15pt}\xymatrixcolsep{25pt} \xymatrix{\BrBialg_{\M}^s
\ar[d]_{\mathbb{I}_{\BrBialg}^{s}}\ar[r]^{\mho _{\Br}^s}&\BrAlg _{\M}
^s\ar[d]^{\mathbb{I}_{\BrAlg}^s}\\ \BrBialg_{\M} \ar[r]^{\mho
_{\Br}}&\BrAlg_{\M}} \quad \xymatrix{\Br^s_\M\ar[dr]_{T_\Br^s}\ar[rr]^{%
\overline{T}_{\mathrm{Br}}^{s}}&&\BrBialg^s_\M\ar[dl]^{\mho _\Br^s}\\
&\BrAlg_\M^s}  \label{diag:OmRibs}
\end{equation}
Furthermore, by Lemma \ref{lem:Cappuccio}, the natural transformation $\xi
:P_{\mathrm{Br}}\rightarrow \Omega _{\mathrm{Br}}\mho _{\mathrm{Br}}$
induces a natural transformation $\xi :=\widehat{\xi \mathbb{I}_{\mathrm{%
BrBialg}}^{s}}:P_{\mathrm{Br}}^{s}\rightarrow \Omega _{\mathrm{Br}}^{s}\mho
_{\mathrm{Br}}^{s}$ such that $\mathbb{I}_{\mathrm{Br}}^{s}\xi =\xi \mathbb{I%
}_{\mathrm{BrBialg}}^{s}$.
\end{claim}

\begin{proposition}
\label{pro:comdat1}Let $\left( F,\phi _{0},\phi _{2}\right) :\mathcal{M}%
\rightarrow \mathcal{M}^{\prime }$ be a monoidal functor between monoidal
categories. Assume that $\mathcal{M}$ and $\mathcal{M}^{\prime }$ have
denumerable coproducts and that $F$ and the tensor functors preserve such
coproducts. Then both
\begin{equation*}
\left( \mathrm{Alg}F,F\right) :\left( T,\Omega \right) \rightarrow \left(
T^{\prime },\Omega ^{\prime }\right) \qquad \text{and}\qquad \left( \mathrm{%
BrAlg}F,\mathrm{Br}F\right) :\left( T_{\mathrm{Br}},\Omega _{\mathrm{Br}%
}\right) \rightarrow \left( T_{\mathrm{Br}}^{\prime },\Omega _{\mathrm{Br}%
}^{\prime }\right)
\end{equation*}
are commutation data.
\end{proposition}

\begin{proof}
First we deal with $\left( \mathrm{Alg}F,F\right) :\left( T,\Omega \right)
\rightarrow \left( T^{\prime },\Omega ^{\prime }\right) .$ By \ref%
{cl:BrBialg}, we have that $\Omega ^{\prime }\circ \mathrm{Alg}F=F\circ
\Omega $. By Remark \ref{rem: AlgMon}, we have that $\Omega $ and $\Omega
^{\prime }$ have left adjoints $T$ and $T^{\prime }$ respectively. The
structure morphisms $\phi _{0},\phi _{2}$ induce, for every $n\in \mathbb{N}$%
, the isomorphism $\widehat{\phi }_{n}V:\left( FV\right) ^{\otimes
n}\rightarrow F\left( V^{\otimes n}\right) $ given by
\begin{eqnarray*}
\widehat{\phi }_{0}V &:&=\phi _{0},\qquad \widehat{\phi }_{1}V:=\mathrm{Id}%
_{FV},\qquad \widehat{\phi }_{2}V:=\phi _{2}\left( V,V\right) ,\qquad \text{%
and, for }n>2 \\
\widehat{\phi }_{n}V &:&=\phi _{2}\left( V^{\otimes \left( n-1\right)
},V\right) \circ \left( \widehat{\phi }_{n-1}\otimes FV\right) .
\end{eqnarray*}

Using the naturality of $\phi _{2}$ and (\ref{form:TVm}) it is
straightforward to check, by induction on $n\in \mathbb{N}$, that%
\begin{equation}
m_{\left( \mathrm{Alg}F\right) TV}^{n-1}\circ \left( F\alpha _{1}V\right)
^{\otimes n}=F\alpha _{n}V\circ \widehat{\phi }_{n}V.  \label{form:multphi}
\end{equation}

Let $\zeta $ be the map of Lemma \ref{lem:zeta} i.e. $\zeta =\epsilon
^{\prime }\left( \mathrm{Alg}F\right) T\circ T^{\prime }F\eta $. We compute%
\begin{gather*}
\Omega ^{\prime }\zeta V\circ \alpha _{n}FV=\Omega ^{\prime }\epsilon
^{\prime }\left( \mathrm{Alg}F\right) TV\circ \Omega ^{\prime }T^{\prime
}F\eta V\circ \alpha _{n}FV \\
\overset{(\ref{form:etaeps})}{=}\Omega ^{\prime }\epsilon ^{\prime }\left(
\mathrm{Alg}F\right) TV\circ \Omega ^{\prime }T^{\prime }F\alpha _{1}V\circ
\alpha _{n}FV=\Omega ^{\prime }\epsilon ^{\prime }\left( \mathrm{Alg}%
F\right) TV\circ \alpha _{n}F\Omega TV\circ \left( F\alpha _{1}V\right)
^{\otimes n} \\
\overset{(\ref{form:etaeps})}{=}m_{\left( \mathrm{Alg}F\right)
TV}^{n-1}\circ \left( F\alpha _{1}V\right) ^{\otimes n}\overset{(\ref%
{form:multphi})}{=}F\alpha _{n}V\circ \widehat{\phi }_{n}V \\
=\left( \nabla _{t\in \mathbb{N}}F\alpha _{t}V\right) \circ j_{n}V\circ
\widehat{\phi }_{n}V=\left( \nabla _{t\in \mathbb{N}}F\alpha _{t}V\right)
\circ \left( \oplus _{t\in \mathbb{N}}\widehat{\phi }_{t}V\right) \circ
\alpha _{n}FV
\end{gather*}%
where $j_{n}V:F\left( V^{\otimes n}\right) \rightarrow \oplus _{t\in \mathbb{%
N}}F\left( V^{\otimes t}\right) $ denotes the canonical morphism. Since this
equality holds for an arbitrary $n\in \mathbb{N}$, we obtain $\Omega
^{\prime }\zeta V=\left( \nabla _{n\in \mathbb{N}}F\alpha _{n}V\right) \circ
\left( \oplus _{n\in \mathbb{N}}\widehat{\phi }_{n}V\right) .$ Now $\widehat{%
\phi }_{n}$ is an isomorphism by construction and $\nabla _{n\in \mathbb{N}%
}F\alpha _{n}V:\oplus _{n\in \mathbb{N}}F\left( V^{\otimes n}\right)
\rightarrow F\left( \oplus _{n\in \mathbb{N}}V^{\otimes n}\right) $ is an
isomorphism as $F$ preserves denumerable coproducts. Hence $\Omega ^{\prime
}\zeta V$ is an isomorphism. This clearly implies $\zeta V$ is an
isomorphism and hence $\left( \mathrm{Alg}F,F\right) :\left( T,\Omega
\right) \rightarrow \left( T^{\prime },\Omega ^{\prime }\right) $ is a
commutation datum$.$

Now, let us consider $\left( \mathrm{BrAlg}F,\mathrm{Br}F\right) :\left( T_{%
\mathrm{Br}},\Omega _{\mathrm{Br}}\right) \rightarrow \left( T_{\mathrm{Br}%
}^{\prime },\Omega _{\mathrm{Br}}^{\prime }\right) .$ By \ref{cl:TbrStrict},
the functor $\Omega _{\mathrm{Br}}:\mathrm{BrAlg}_{\mathcal{M}}\rightarrow
\mathrm{Br}_{\mathcal{M}}$ has a left adjoint $T_{\mathrm{Br}}:\mathrm{Br}_{%
\mathcal{M}}\rightarrow \mathrm{BrAlg}_{\mathcal{M}}$ and the (co)unit of
the adjunction obey (\ref{form:TbrStrict}). Moreover $H_{\mathrm{Alg}}T_{%
\mathrm{Br}}=TH.$ By \ref{cl:BrBialg}, we have $H^{\prime }\left( \mathrm{Br}%
F\right) =FH,$ $\Omega ^{\prime }\left( \mathrm{Alg}F\right) =F\Omega ,$ $H_{%
\mathrm{Alg}}^{\prime }\left( \mathrm{BrAlg}F\right) =\left( \mathrm{Alg}%
F\right) H_{\mathrm{Alg}}$ and $\Omega _{\mathrm{Br}}^{\prime }\left(
\mathrm{BrAlg}F\right) =\left( \mathrm{Br}F\right) \Omega _{\mathrm{Br}}.$
In view of Lemma (\ref{lem:zeta}) the diagrams
\begin{equation}  \label{diag:BrAlgF}
\xymatrixrowsep{15pt}\xymatrixcolsep{25pt}\xymatrix{\mathrm{BrAlg}_{%
\mathcal{M}}
\ar[r]^{\mathrm{BrAlg}F}\ar[d]_{\Omega_{\mathrm{Br}}}&\mathrm{BrAlg}_{%
\mathcal{M}^{\prime }}\ar[d]^{{\Omega^\prime_{\mathrm{Br}}}}\\
\mathrm{Br}_{\mathcal{M}}
\ar[r]^{\mathrm{Br}F}&\mathrm{Br}_{\mathcal{M}^{\prime }}}\qquad %
\xymatrixrowsep{15pt}\xymatrix{\mathrm{Alg}_{\mathcal{M}}
\ar[r]^{\mathrm{Alg}F}\ar[d]_\Omega&\mathrm{Alg}_{\mathcal{M}^{\prime
}}\ar[d]^{\Omega^\prime}\\ \mathcal{M} \ar[r]^F&\mathcal{M}^{\prime }}
\end{equation}
induce the maps $\zeta _{\mathrm{Br}}:T_{\mathrm{Br}}^{\prime }\left(
\mathrm{Br}F\right) \rightarrow \left( \mathrm{BrAlg}F\right) T_{\mathrm{Br}%
} $ and $\zeta :T^{\prime }F\rightarrow \left( \mathrm{Alg}F\right) T$
defined by
\begin{equation}
\zeta _{\mathrm{Br}}=\epsilon _{\mathrm{Br}}^{\prime }\left( \mathrm{BrAlg}%
F\right) T_{\mathrm{Br}}\circ T_{\mathrm{Br}}^{\prime }\left( \mathrm{Br}%
F\right) \eta _{\mathrm{Br}}\qquad \text{and}\qquad \zeta =\epsilon ^{\prime
}\left( \mathrm{Alg}F\right) T\circ T^{\prime }F\eta .
\label{form:zetaBr&zeta}
\end{equation}
One easily checks that%
\begin{equation}
H_{\mathrm{Alg}}^{\prime }\zeta _{\mathrm{Br}}=\zeta H.  \label{form:zetaBr}
\end{equation}%
By the first part of the proof, $\zeta $ is a functorial isomorphism so that
we get that $H_{\mathrm{Alg}}^{\prime }\zeta _{\mathrm{Br}}$ is a functorial
isomorphism too. Since $H_{\mathrm{Alg}}^{\prime }$ trivially reflects
isomorphisms, we get that $\zeta _{\mathrm{Br}}$ is a functorial isomorphism.
\end{proof}

\begin{proposition}
\label{pro:PrimFunct}Let $\mathcal{M}$ and $\mathcal{M}^{\prime }$ be
preadditive monoidal categories with equalizers. Assume that the tensor
functors are additive and preserve equalizers in both categories. For any
monoidal functor $\left( F,\phi _{0},\phi _{2}\right) :\mathcal{M}%
\rightarrow \mathcal{M}^{\prime }$ which preserves equalizers, the following
diagram commutes
\begin{equation}
\xymatrixrowsep{15pt}\xymatrixcolsep{35pt}\xymatrix{\mathrm{BrBialg}_{%
\mathcal{M}}
\ar[r]^{\mathrm{BrBialg}F}\ar[d]_{P_{\mathrm{Br}}}&\mathrm{BrBialg}_{%
\mathcal{M}^{\prime }}\ar[d]^{{P^\prime_{\mathrm{Br}}}}\\
\mathrm{Br}_{\mathcal{M}}
\ar[r]^{\mathrm{Br}F}&\mathrm{Br}_{\mathcal{M}^{\prime }}}
\label{diag:BrF-PBr}
\end{equation}%
where $\mathrm{BrBialg}F$ and $\mathrm{Br}F$ are the functors of \ref%
{cl:BrBialg}. Moreover we have
\begin{equation}
\xi ^{\prime }\left( \mathrm{BrBialg}F\right) =\left( \mathrm{Br}F\right)
\xi .  \label{form:comdat3}
\end{equation}%
Assume also that the categories $\mathcal{M}$ and $\mathcal{M}^{\prime }$
have denumerable coproducts and that $F$ and the tensor functors preserve
such coproducts. Then $\left( \mathrm{BrBialg}F,\mathrm{Br}F\right) :\left(
\overline{T}_{\mathrm{Br}},P_{\mathrm{Br}}\right) \rightarrow \left(
\overline{T}_{\mathrm{Br}}^{\prime },P_{\mathrm{Br}}^{\prime }\right) $ is a
commutation datum.
\end{proposition}

\begin{proof}
The first part is \cite[Proposition 3.6]{AM-BraidedOb}. Let us prove the
last assertion. Assume that the monoidal category $\mathcal{M}$ has
denumerable coproducts and that the tensor functors preserve such
coproducts. By \ref{cl:TbarStrict}, we have that $P_{\mathrm{Br}}$ and $P_{%
\mathrm{Br}}^{\prime }$ have left adjoints $\overline{T}_{\mathrm{Br}}$ and $%
\overline{T}_{\mathrm{Br}}^{\prime }$ respectively. By \ref{cl:BrBialg}, we
have $\mho _{\mathrm{Br}}^{\prime }\left( \mathrm{BrBialg}F\right) =\left(
\mathrm{BrAlg}F\right) \mho _{\mathrm{Br}}$ and $\Omega _{\mathrm{Br}%
}^{\prime }\left( \mathrm{BrAlg}F\right) =\left( \mathrm{Br}F\right) \Omega
_{\mathrm{Br}}.$ By (\ref{form:OmegRibTbarBr}), we have $\mho _{\mathrm{Br}}%
\overline{T}_{\mathrm{Br}}=T_{\mathrm{Br}}$. The commutative diagrams %
\eqref{diag:BrF-PBr} and \eqref{diag:BrAlgF}-left induce the natural
transformations $\overline{\zeta }_{\mathrm{Br}}:\overline{T}_{\mathrm{Br}%
}^{\prime }\left( \mathrm{Br}F\right) \rightarrow \left( \mathrm{BrBialg}%
F\right) \overline{T}_{\mathrm{Br}}$ and $\zeta _{\mathrm{Br}}:T_{\mathrm{Br}%
}^{\prime }\left( \mathrm{Br}F\right) \rightarrow \left( \mathrm{BrAlg}%
F\right) T_{\mathrm{Br}}$ of Lemma \ref{lem:zeta} i.e.
\begin{equation*}
\overline{\zeta }_{\mathrm{Br}}=\overline{\epsilon }_{\mathrm{Br}}^{\prime
}\left( \mathrm{BrBialg}F\right) \overline{T}_{\mathrm{Br}}\circ \overline{T}%
_{\mathrm{Br}}^{\prime }\left( \mathrm{Br}F\right) \overline{\eta }_{\mathrm{%
Br}}\qquad \text{and}\qquad \zeta _{\mathrm{Br}}=\epsilon _{\mathrm{Br}%
}^{\prime }\left( \mathrm{BrAlg}F\right) T_{\mathrm{Br}}\circ T_{\mathrm{Br}%
}^{\prime }\left( \mathrm{Br}F\right) \eta _{\mathrm{Br}}.
\end{equation*}
Using (\ref{form:BarEps}), (\ref{form:comdat3}) and (\ref{form:BarEta}), one
easily checks that $\mho _{\mathrm{Br}}^{\prime }\overline{\zeta }_{\mathrm{%
Br}}=\zeta _{\mathrm{Br}}.$ By Proposition \ref{pro:comdat1}, we know that $%
\zeta _{\mathrm{Br}}$ is a functorial isomorphism. Since $\mho _{\mathrm{Br}%
}^{\prime }$ is trivially conservative, we deduce that $\overline{\zeta }_{%
\mathrm{Br}}$ is a functorial isomorphism too.
\end{proof}

\section{Braided Categories \label{sec:BraidCat}}

\begin{claim}
\label{def braiding} A \emph{braided monoidal category} $(\mathcal{M}%
,\otimes ,\mathbf{1},a,l,r,c)$ is a monoidal category $(\mathcal{M},\otimes ,%
\mathbf{1})$ equipped with a \emph{braiding} $c$, that is an isomorphism $%
c_{U,V}:U\otimes V\rightarrow V\otimes U$, natural in $U,V\in \mathcal{M}$,
satisfying, for all $U,V,W\in \mathcal{M}$,
\begin{eqnarray*}
a_{V,W,U}\circ c_{U,V\otimes W}\circ a_{U,V,W} &=&(V\otimes c_{U,W})\circ
a_{V,U,W}\circ (c_{U,V}\otimes W), \\
a_{W,U,V}^{-1}\circ c_{U\otimes V,W}\circ a_{U,V,W}^{-1} &=&(c_{U,W}\otimes
V)\circ a_{U,W,V}^{-1}\circ (U\otimes c_{V,W}).
\end{eqnarray*}%
A braided monoidal category is called \emph{symmetric} if we further have $%
c_{V,U}\circ c_{U,V}=\mathrm{Id}_{U\otimes V}$ for every $U,V\in \mathcal{M}$%
.

A \emph{(symmetric) braided monoidal functor} is a monoidal functor $F:%
\mathcal{M}\rightarrow \mathcal{M}^{\prime }$ such that $F\left(
c_{U,V}\right) \circ \phi _{2}(U,V)=\phi _{2}(V,U)\circ c_{F\left( U\right)
,F\left( V\right) }^{\prime }.$ More details on these topics can be found in
\cite[Chapter XIII]{Kassel}.
\end{claim}

\begin{remark}
Given a braided monoidal category $(\mathcal{M},{\otimes },\mathbf{1},c)$
the category $\mathrm{Alg}_{\mathcal{M}}$ becomes monoidal where, for every $%
A,B\in \mathrm{Alg}_{\mathcal{M}}$ the multiplication and unit of $A\otimes
B $ are given by%
\begin{eqnarray*}
m_{A\otimes B} &:&=\left( m_{A}\otimes m_{B}\right) \circ \left( A\otimes
c_{B,A}\otimes B\right) :\left( A\otimes B\right) \otimes \left( A\otimes
B\right) \rightarrow A\otimes B, \\
u_{A\otimes B} &:&=\left( u_{A}\otimes u_{B}\right) \circ l_{\mathbf{1}%
}^{-1}:\mathbf{1}\rightarrow A\otimes B.
\end{eqnarray*}%
Moreover the forgetful functor $\mathrm{Alg}_{\mathcal{M}}\rightarrow
\mathcal{M}$ is a strict monoidal functor, cf. \cite[page 60]{Joyal-Street}.
\end{remark}

\begin{definition}
\label{cl: brdBialg} A \emph{bialgebra} in a braided monoidal category $(%
\mathcal{M},{\otimes },\mathbf{1},c)$ is a coalgebra $(B,\Delta ,\varepsilon
)$ in the monoidal category $\mathrm{Alg}_{\mathcal{M}}$. Equivalently a
bialgebra is a quintuple $\left( A,m,u,\Delta ,\varepsilon \right) $ where $%
\left( A,m,u\right) $ is an algebra in $\mathcal{M}$ and $\left( A,\Delta
,\varepsilon \right) $ is a coalgebra in $\mathcal{M}$ such that $\Delta $
and $\varepsilon $ are morphisms of algebras where $A\otimes A$ is an
algebra as in the previous remark. Denote by $\mathrm{Bialg}_{\mathcal{M}}$
the category of bialgebras in $\mathcal{M}$ and their morphisms, defined in
the expected way.
\end{definition}

\begin{claim}
\label{cl:defJ} Let $\mathcal{M}$ be a braided monoidal category. In view of
\cite[Proposition 4.4]{AM-BraidedOb}, there are obvious functors $J$, $J_{%
\mathrm{Alg}}$ and $J_{\mathrm{Bialg}}$ such that the diagrams%
\begin{equation}
\xymatrixrowsep{15pt}\xymatrixcolsep{35pt}\xymatrix{\mathrm{Bialg}_{%
\mathcal{M}}\ar[r]^{J_\mathrm{Bialg}}\ar[d]_{\mho}
&\mathrm{BrBialg}_{\mathcal{M}}\ar[d]^{\mho_{\mathrm{Br}}}\\
\mathrm{Alg}_{\mathcal{M}}
\ar[r]^{J_\mathrm{Alg}}&\mathrm{BrAlg}_{\mathcal{M}}}\qquad %
\xymatrixrowsep{15pt}\xymatrixcolsep{35pt}\xymatrix{\mathrm{Alg}_{%
\mathcal{M}}\ar[r]^{J_\mathrm{Alg}}\ar[d]_{\Omega}
&\mathrm{BrAlg}_{\mathcal{M}}\ar[d]^{\Omega_{\mathrm{Br}}}\\ \mathcal{M}
\ar[r]^{J}&\mathrm{Br}_{\mathcal{M}}}  \label{diag:JAlg-Bialg}
\end{equation}%
commute. In fact the functors $J$, $J_{\mathrm{Alg}}$ and $J_{\mathrm{Bialg}%
} $ add the evaluation of the braiding of $\mathcal{M}$ on the object on
which they act. Moreover they are full, faithful, injective on objects and
conservative.

Assume that $\mathcal{M}$ has denumerable coproducts and that the tensor
functors preserve such coproducts. Then, by \cite[Proposition 4.5]%
{AM-BraidedOb}, the following diagram%
\begin{equation}
\xymatrixrowsep{15pt}\xymatrixcolsep{35pt}\xymatrix{\mathrm{Alg}_{%
\mathcal{M}}\ar[r]^{J_\mathrm{Alg}} &\mathrm{BrAlg}_{\mathcal{M}}\\
\mathcal{M}\ar[u]_{T}
\ar[r]^{J}&\mathrm{Br}_{\mathcal{M}}\ar[u]^{T_{\mathrm{Br}}}}
\label{form:JT}
\end{equation}%
is commutative. When $\mathcal{M}$ is symmetric the functor $J,J_{\mathrm{Alg%
}}$ and $J_{\mathrm{Bialg}}$ factors through functors $J^{s},J_{\mathrm{Alg}%
}^{s}$ and $J_{\mathrm{Bialg}}^{s}$ i.e. the following diagrams commute
(apply Lemma \ref{lem:Cappuccio}).
\begin{equation}
\xymatrixrowsep{15pt}\xymatrixcolsep{8pt} \xymatrix{\M\ar[dr]_J%
\ar[rr]^{J^s}&&\Br_\M^s\ar[dl]^{\mathbb{I}_\Br^s}\\ &\Br_\M} \quad %
\xymatrix{\Alg_\M\ar[dr]_{J_\Alg}\ar[rr]^{J_\Alg^s}&&\BrAlg_\M^s\ar[dl]^{%
\mathbb{I}_\BrAlg^s}\\ &\BrAlg_\M} \quad \xymatrix{\Bialg_\M\ar[dr]_{J_%
\Bialg}\ar[rr]^{J_\Bialg^s}&&\BrBialg_\M^s\ar[dl]^{\mathbb{I}_\BrBialg^s}\\
&\BrBialg_\M}  \label{eq:Js}
\end{equation}
Note that they are full, faithful, injective on objects and conservative and
the following diagram commutes.%
\begin{equation}
\xymatrixrowsep{15pt}\xymatrixcolsep{15pt} \xymatrix{\Bialg_\M\ar[d]_\mho%
\ar[rr]^{J_\Bialg^s}&&\BrBialg_\M^s\ar[d]^{\mho_\Br^s}\\
\Alg_\M\ar[rr]^{J_\Alg^s}&&\BrBialg_\M^s}  \label{diag:JsAlgOrib}
\end{equation}
\end{claim}

\begin{claim}
\label{cl: Bialg} Let $\mathcal{M}$ be a preadditive braided monoidal
category with equalizers. Assume that the tensor functors are additive and
preserve equalizers. Define the functor%
\begin{equation*}
P:=H\circ P_{\mathrm{Br}}\circ J_{\mathrm{Bialg}}:\mathrm{Bialg}_{\mathcal{M}%
}\rightarrow \mathcal{M}
\end{equation*}%
For any $\mathbb{B}:=\left( B,m_{B},u_{B},\Delta _{B},\varepsilon
_{B}\right) \in \mathrm{Bialg}_{\mathcal{M}}$ one easily gets that $P\left(
\mathbb{B}\right) =P\left( B,\Delta _{B},\varepsilon _{B},u_{B}\right) ,$
see \cite[4.6]{AM-BraidedOb}. The canonical inclusion $\xi P\left( B,\Delta
_{B},\varepsilon _{B},u_{B}\right) :P\left( B,\Delta _{B},\varepsilon
_{B},u_{B}\right) \rightarrow B$ will be denoted by $\xi \mathbb{B}$. Thus
we have the equalizer%
\begin{equation*}
\xymatrixcolsep{1.5cm}\xymatrix{P\left( \mathbb{B}\right) \ar[r]^{\xi
\mathbb{B}} & B \ar@<.5ex>[rr]^{\Delta _{B}} \ar@<-.5ex>[rr]_{\left(
B\otimes u_{B}\right) r_{B}^{-1}+\left( u_{B}\otimes B\right)
l_{B}^{-1}}&&B\otimes B }
\end{equation*}%
By \cite[Proposition 4.7]{AM-BraidedOb}, we have a commutative diagram
\begin{equation}
\xymatrixrowsep{15pt}\xymatrixcolsep{35pt}\xymatrix{\mathrm{Bialg}_{%
\mathcal{M}}\ar[r]^{J_\mathrm{Bialg}}\ar[d]_{P}
&\mathrm{BrBialg}_{\mathcal{M}}\ar[d]^{P_{\mathrm{Br}}}\\ {\mathcal{M}}
\ar[r]^{J}&\mathrm{Br}_{\mathcal{M}}}  \label{form:JPbar}
\end{equation}%
where the horizontal arrows are the functors of \ref{cl:defJ}. Furthermore
\begin{equation}
\xi J_{\mathrm{Bialg}}=J\xi .  \label{form:xij}
\end{equation}%
Assume further that $\mathcal{M}$ has denumerable coproducts and that the
tensor functors preserve such coproducts. By Remark \ref{rem: AlgMon}, the
forgetful functor $\Omega :\mathrm{Alg}_{\mathcal{M}}\rightarrow \mathcal{M}$
has a left adjoint $T:\mathcal{M}\rightarrow \mathrm{Alg}_{\mathcal{M}}$.
Note that%
\begin{equation*}
\mathbb{I}_{\mathrm{BrAlg}}^{s}T_{\mathrm{Br}}^{s}J^{s}\overset{(\ref%
{diag:Tbrs})}{=}T_{\mathrm{Br}}\mathbb{I}_{\mathrm{Br}}^{s}J^{s}\overset{(%
\ref{eq:Js})}{=}T_{\mathrm{Br}}J\overset{(\ref{form:JT})}{=}J_{\mathrm{Alg}}T%
\overset{(\ref{eq:Js})}{=}\mathbb{I}_{\mathrm{BrAlg}}^{s}J_{\mathrm{Alg}%
}^{s}T
\end{equation*}%
and hence, since $\mathbb{I}_{\mathrm{BrAlg}}^{s}$ is both injective on
morphisms and objects, we get that the following diagram commutes%
\begin{equation}
\xymatrixrowsep{15pt}\xymatrixcolsep{15pt} \xymatrix{\Alg_\M\ar[rr]^{J^s_%
\Alg}&&\BrAlg_\M^s\\ \M\ar[u]^T\ar[rr]^{J^s}&&\Br_\M^s\ar[u]_{T_\Br^s}}
\label{diag:JsT}
\end{equation}%
In view of \cite[4.8]{AM-BraidedOb}, there is a functor%
\begin{equation*}
\overline{T}:\mathcal{M}\rightarrow \mathrm{Bialg}_{\mathcal{M}}
\end{equation*}%
such that the following diagrams commute.%
\begin{equation}
\xymatrixrowsep{15pt}\xymatrixcolsep{35pt}\xymatrix{\mathrm{Bialg}_{%
\mathcal{M}}\ar[r]^{J_\mathrm{Bialg}} &\mathrm{BrBialg}_{\mathcal{M}}\\
\mathcal{M}\ar[u]^{\overline{T}}
\ar[r]^{J}&\mathrm{Br}_{\mathcal{M}}\ar[u]_{\overline{T}_{\mathrm{Br}}}}
\label{form:JTbar}
\end{equation}%
\begin{equation}
\xymatrixrowsep{12pt}\xymatrixcolsep{0.7cm}\xymatrix{\mathrm{Bialg}_{%
\mathcal{M}} \ar[rr]^{\mho }&& \mathrm{Alg}_{\mathcal{M}} \\
&\mathcal{M}\ar[ul]^{\overline{T}}\ar[ru]_{ T }}  \label{form:OmegRibTbar}
\end{equation}%
By \cite[Theorem 4.9]{AM-BraidedOb}, the functor $\overline{T}$ is a left
adjoint of the functor $P:\mathrm{Bialg}_{\mathcal{M}}\rightarrow \mathcal{M}
$. The unit $\overline{\eta }$ and counit $\overline{\epsilon }$ of the
adjunction are uniquely determined by the following equalities%
\begin{gather}
\xi \overline{T}\circ \overline{\eta }=\eta , \qquad \epsilon \mho \circ
T\xi =\mho \overline{\epsilon },  \label{form:etabarVSeta}
\end{gather}%
where $\eta $ and $\epsilon $ denote the unit and counit of the adjunction $%
\left( T,\Omega \right) $ respectively. We have that
\begin{equation*}
\mathbb{I}_{\mathrm{BrBialg}}^{s}\overline{T}_{\mathrm{Br}}^{s}J^{s}\overset{%
(\ref{diag:PsTsBr})}{=}\overline{T}_{\mathrm{Br}}\mathbb{I}_{\mathrm{Br}%
}^{s}J^{s}\overset{(\ref{eq:Js})}{=}\overline{T}_{\mathrm{Br}}J\overset{(\ref%
{form:JTbar})}{=}J_{\mathrm{Bialg}}\overline{T}\overset{(\ref{eq:Js})}{=}%
\mathbb{I}_{\mathrm{BrBialg}}^{s}J_{\mathrm{Bialg}}^{s}\overline{T}
\end{equation*}%
and that%
\begin{equation*}
\mathbb{I}_{\mathrm{Br}}^{s}P_{\mathrm{Br}}^{s}J_{\mathrm{Bialg}}^{s}\overset%
{(\ref{diag:PsTsBr})}{=}P_{\mathrm{Br}}\mathbb{I}_{\mathrm{BrBialg}}^{s}J_{%
\mathrm{Bialg}}^{s}\overset{(\ref{eq:Js})}{=}P_{\mathrm{Br}}J_{\mathrm{Bialg}%
}\overset{(\ref{form:JPbar})}{=}JP\overset{(\ref{eq:Js})}{=}\mathbb{I}_{%
\mathrm{Br}}^{s}J^{s}P
\end{equation*}%
so that the following diagram commutes.
\begin{equation}
\xymatrixrowsep{15pt}\xymatrixcolsep{15pt} \xymatrix{\Bialg_M
\ar[rr]^{J_\Bialg^s}&& \BrBialg_M^s\\ \M
\ar[u]^{\overline{T}}\ar[rr]^{J^s}&& \Br_\M^s \ar[u]_{\overline{T}_\Br^s}}
\qquad \xymatrix{\Bialg_M \ar[d]_{P}\ar[rr]^{J_\Bialg^s}&&
\BrBialg_M^s\ar[d]^{P_\Br^s}\\ \M \ar[rr]^{J^s}&& \Br_\M^s }
\label{diag:JsTbar}
\end{equation}
\end{claim}

\begin{proposition}
\label{pro:Tbar}Let $\mathcal{M}$ be a preadditive braided monoidal category
with equalizers. Assume that the tensor functors are additive and preserve
equalizers. Assume further that $\mathcal{M}$ has denumerable coproducts and
that the tensor functors preserve such coproducts. Then the morphism $\zeta :%
\overline{T}_{\mathrm{Br}}J\longrightarrow J_{\mathrm{Bialg}}\overline{T}$
of Lemma \ref{lem:zeta} is $\mathrm{Id}_{\overline{T}_{\mathrm{Br}}J}$. In
particular $\left( J_{\mathrm{Bialg}},J\right) :\left( \overline{T},P\right)
\rightarrow \left( \overline{T}_{\mathrm{Br}},P_{\mathrm{Br}}\right) $ is a
commutation datum.
\end{proposition}

\begin{proof}
Consider the commutative diagram \ref{form:JPbar}. By Lemma \ref{lem:zeta},
then there is a unique natural transformation $\zeta :\overline{T}_{\mathrm{%
Br}}J\longrightarrow J_{\mathrm{Bialg}}\overline{T}$ such that $P_{\mathrm{Br%
}}\zeta \circ \overline{\eta }_{\mathrm{Br}}J=J\overline{\eta }.$ By \cite[%
Equality (75)]{AM-BraidedOb}, we also have $\overline{\eta }_{\mathrm{Br}}J=J%
\overline{\eta }.$ By uniqueness of $\zeta $, we have $\zeta =\mathrm{Id}_{%
\overline{T}_{\mathrm{Br}}J}.$
\end{proof}

\begin{proposition}
\label{pro:Tbars}Let $\mathcal{M}$ be a preadditive symmetric monoidal
category with equalizers. Assume that the tensor functors are additive and
preserve equalizers. Assume further that $\mathcal{M}$ has denumerable
coproducts and that the tensor functors preserve such coproducts. Then the
morphism $\zeta ^{s}:\overline{T}_{\mathrm{Br}}^{s}J^{s}\longrightarrow J_{%
\mathrm{Bialg}}^{s}\overline{T}$ of Lemma \ref{lem:zeta} is $\mathrm{Id}_{%
\overline{T}_{\mathrm{Br}}^{s}J^{s}}$. In particular $\left( J_{\mathrm{Bialg%
}}^{s},J^{s}\right) :\left( \overline{T},P\right) \rightarrow \left(
\overline{T}_{\mathrm{Br}}^{s},P_{\mathrm{Br}}^{s}\right) $ is a commutation
datum.
\end{proposition}

\begin{proof}
Consider the commutative diagram \ref{form:JPbar}. By Lemma \ref{lem:zeta},
then there is a unique natural transformation $\zeta ^{s}:\overline{T}_{%
\mathrm{Br}}^{s}J^{s}\longrightarrow J_{\mathrm{Bialg}}^{s}\overline{T}$
such that $P_{\mathrm{Br}}^{s}\zeta ^{s}\circ \overline{\eta }_{\mathrm{Br}%
}^{s}J^{s}=J^{s}\overline{\eta }.$ Now%
\begin{equation*}
\mathbb{I}_{\mathrm{Br}}^{s}\overline{\eta }_{\mathrm{Br}}^{s}J^{s}\overset{(%
\ref{form:barbrs})}{=}\overline{\eta }_{\mathrm{Br}}\mathbb{I}_{\mathrm{Br}%
}^{s}J^{s}\overset{(\ref{eq:Js})}{=}\overline{\eta }_{\mathrm{Br}}J\overset{%
(\ast )}{=}J\overline{\eta }\overset{(\ref{eq:Js})}{=}\mathbb{I}_{\mathrm{Br}%
}^{s}J^{s}\overline{\eta }
\end{equation*}%
where in (*) we used \cite[Equality (75)]{AM-BraidedOb}. Thus $\overline{%
\eta }_{\mathrm{Br}}^{s}J^{s}=J^{s}\overline{\eta }.$ By uniqueness of $%
\zeta ^{s}$, we have $\zeta ^{s}=\mathrm{Id}_{\overline{T}_{\mathrm{Br}%
}^{s}J^{s}}$ (note that we are using that the domain and codomain of $\zeta
^{s}$ coincide by (\ref{diag:JsTbar})).
\end{proof}

\begin{claim}
\label{cl:BialgF} Let $\mathcal{M}$ and $\mathcal{M}^{\prime }$ be braided
monoidal categories. Following \cite[Proposition 4.10]{AM-BraidedOb}, every
braided monoidal functor $\left( F,\phi _{0},\phi _{2}\right) :\mathcal{M}%
\rightarrow \mathcal{M}^{\prime }$ induces in a natural way a functor $%
\mathrm{Bialg}F$ and the following diagrams commute.%
\begin{equation*}
\xymatrixrowsep{15pt}\xymatrixcolsep{35pt}\xymatrix{\mathcal{M}\ar[r]^{F}%
\ar[d]_{J} &\mathcal{M}^\prime \ar[d]^{J^\prime}\\ {\mathrm{Br}_\mathcal{M}}
\ar[r]^{\mathrm{Br}F}&\mathrm{Br}_{\mathcal{M}^\prime}}\qquad %
\xymatrixrowsep{15pt}\xymatrixcolsep{35pt}\xymatrix{\mathrm{Bialg}_%
\mathcal{M}\ar[r]^{\mathrm{Bialg}F}\ar[d]_{J_\mathrm{Bialg}}
&\mathrm{Bialg}_{\mathcal{M}^\prime} \ar[d]^{J^\prime_\mathrm{Bialg}}\\
{\mathrm{BrBialg}_\mathcal{M}}
\ar[r]^{\mathrm{BrBialg}F}&\mathrm{BrBialg}_{\mathcal{M}^\prime}}\qquad %
\xymatrixrowsep{15pt}\xymatrixcolsep{35pt}\xymatrix{\mathrm{Bialg}_%
\mathcal{M}\ar[r]^{\mathrm{Bialg}F}\ar[d]_{\mho}
&\mathrm{Bialg}_{\mathcal{M}^\prime} \ar[d]^{\mho^\prime}\\
{\mathrm{Alg}_\mathcal{M}}
\ar[r]^{\mathrm{Alg}F}&\mathrm{Alg}_{\mathcal{M}^\prime}}
\end{equation*}%
Moreover

\begin{enumerate}
\item[1)] $\mathrm{Bialg}F$ is an equivalence (resp. category isomorphism or
conservative) whenever $F$ is.

\item[2)] If $F$ preserves equalizers, the following diagram commutes.
\begin{equation*}
\xymatrixrowsep{15pt}\xymatrixcolsep{35pt}\xymatrix{\mathrm{Bialg}_%
\mathcal{M}\ar[r]^{\mathrm{Bialg}F}\ar[d]_{P}
&\mathrm{Bialg}_{\mathcal{M}^\prime} \ar[d]^{P^\prime}\\
\mathcal{M}\ar[r]^{F}&\mathcal{M}^\prime}
\end{equation*}
\end{enumerate}
\end{claim}

\section{Lie algebras}

The following definition extends the classical notion of Lie algebra to a
monoidal category which is not necessarily braided. We expected this notion
to be well-known, but we could not find any reference.

\begin{definition}
\label{def:Lie} 1) Given an abelian monoidal category $\mathcal{M}$ a \emph{%
braided Lie algebra} in $\mathcal{M}$ consists of a tern $\left( M,c,\left[ -%
\right] :M\otimes M\rightarrow M\right) $ where $\left( M,c\right) \in
\mathrm{Br}_{\mathcal{M}}$ and the following equalities hold true:
\begin{gather}
\left[ -\right] =-\left[ -\right] \circ c\text{ (skew-symmetry);}
\label{Lie1} \\
\left[ -\right] \circ \left( M\otimes \left[ -\right] \right) \circ \left[
\mathrm{Id}_{\left( M\otimes M\right) \otimes M}+\left( M\otimes c\right)
\left( c\otimes M\right) +\left( c\otimes M\right) \left( M\otimes c\right) %
\right] =0\text{ (Jacobi condition);}  \label{Lie2} \\
c\circ \left( M\otimes \left[ -\right] \right) =\left( \left[ -\right]
\otimes M\right) \circ \left( M\otimes c\right) \circ \left( c\otimes
M\right) ;  \label{form:cbraid} \\
c\circ \left( \left[ -\right] \otimes M\right) =\left( M\otimes \left[ -%
\right] \right) \circ \left( c\otimes M\right) \circ \left( M\otimes
c\right) .  \label{form:cbraid2}
\end{gather}%
Of course one should take care of the associativity constraints, but as we
did before, we continue to omit them. A morphism of braided Lie algebras $%
\left( M,c,\left[ -\right] \right) $ and $\left( M^{\prime },c^{\prime },%
\left[ -\right] ^{\prime }\right) $ in $\mathcal{M}$ is a morphism $f:\left(
M,c\right) \rightarrow \left( M^{\prime },c^{\prime }\right) $ of braided
objects such that $f\circ \left[ -\right] =\left[ -\right] ^{\prime }\circ
(f\otimes f).$ This defines the category $\mathrm{BrLie}_{\mathcal{M}}$ of
braided Lie algebras in $\mathcal{M}$ and their morphisms. Denote by
\begin{equation*}
H_{\mathrm{BrLie}}:\mathrm{BrLie}_{\mathcal{M}}\rightarrow \mathrm{Br}_{%
\mathcal{M}}:\left( M,c,\left[ -\right] \right) \mapsto \left( M,c\right)
\end{equation*}%
the obvious functor forgetting the bracket and acting as the identity on
morphisms. Note that $H_{\mathrm{BrLie}}$ is faithful and conservative.

Denote by $\mathrm{BrLie}_{\mathcal{M}}^{s}$ the full subcategory $\mathrm{%
BrLie}_{\mathcal{M}}$ consisting of braided Lie algebras with symmetric
braiding. Denote by%
\begin{equation*}
\mathbb{I}_{\mathrm{BrLie}}^{s}:\mathrm{BrLie}_{\mathcal{M}}^{s}\rightarrow
\mathrm{BrLie}_{\mathcal{M}}
\end{equation*}%
the inclusion functor. It is clear that, by Lemma \ref{lem:Cappuccio}, the
functor $H_{\mathrm{BrLie}}$ induces a functor $H_{\mathrm{BrLie}}^{s}$ such
that the diagram%
\begin{equation}
\xymatrixrowsep{15pt}\xymatrixcolsep{15pt} \xymatrix{\BrLie_\M^s
\ar[d]_{\I_\BrLie^s}\ar[rr]^{H_\BrLie^s}&& \Br_\M^s\ar[d]^{\I_\Br^s}\\
\BrLie_\M \ar[rr]^{H_\BrLie}&& \Br_\M }  \label{diag:HBrLies}
\end{equation}%
commutes. Since $H_{\mathrm{BrLie}}$ and both vertical arrows are faithful
and conservative, the same is true for $H_{\mathrm{BrLie}}^{s}$.

2)\ Let $\mathcal{M}$ be an abelian braided monoidal category. A \emph{Lie
algebra} in $\mathcal{M}$ consists of a pair $\left( M,\left[ -\right]
:M\otimes M\rightarrow M\right) $ such that $\left( M,c_{M,M},\left[ -\right]
\right) \in \mathrm{BrLie}_{\mathcal{M}}$, where $c_{M,M}$ is the braiding $%
c $ of $\mathcal{M}$ evaluated on $M.$ A morphism of Lie algebras $\left( M,%
\left[ -\right] \right) $ and $\left( M^{\prime },\left[ -\right] ^{\prime
}\right) $ in $\mathcal{M}$ is a morphism $f:M\rightarrow M^{\prime }$ in $%
\mathcal{M}$ such that $f\circ \left[ -\right] =\left[ -\right] ^{\prime
}\circ (f\otimes f).$ This defines the category $\mathrm{Lie}_{\mathcal{M}}$
of Lie algebras in $\mathcal{M}$ and their morphisms. Note that there is a
full, faithful, injective on objects and conservative functor
\begin{equation*}
J_{\mathrm{Lie}}:\mathrm{Lie}_{\mathcal{M}}\rightarrow \mathrm{BrLie}_{%
\mathcal{M}}:\left( M,\left[ -\right] \right) \mapsto \left( M,c_{M,M},\left[
-\right] \right)
\end{equation*}%
which acts as the identity on morphisms. This notion already appeared in
\cite[c) page 82]{Manin} , where a Lie algebra in $\mathcal{M}$ is called an
$\mathcal{M}$-Lie algebra. Denote by
\begin{equation*}
H_{\mathrm{Lie}}:\mathrm{Lie}_{\mathcal{M}}\rightarrow \mathcal{M}:\left( M,%
\left[ -\right] \right) \mapsto M
\end{equation*}%
the obvious functor forgetting the bracket and acting as the identity on
morphisms. Note that $H_{\mathrm{BrLie}}J_{\mathrm{Lie}}=JH_{\mathrm{Lie}}$.

3) Let $\mathcal{M}$ be an abelian symmetric monoidal category. Given $%
\left( M,\left[ -\right] \right) \in \mathrm{Lie}_{\mathcal{M}}$ it is clear
that $\left( M,c_{M,M},\left[ -\right] \right) \in \mathrm{BrLie}_{\mathcal{M%
}}^{s}$ so that $J_{\mathrm{Lie}}$ factors through a functor $J_{\mathrm{Lie}%
}^{s}$ such that the following diagrams commute.
\begin{equation}
\xymatrixrowsep{15pt}\xymatrixcolsep{15pt} \xymatrix{\Lie_\M\ar[dr]_{J_%
\Lie^s}\ar[rr]^{J_\Lie}&&\BrLie_\M\\ &\BrLie_\M^s\ar[ur]_{\I_\BrLie^s}}
\qquad \xymatrix{\BrLie_\M^s \ar[rr]^{H_\BrLie^s}&& \Br_\M^s\\ \Lie_\M^s
\ar[u]^{J_\Lie^s}\ar[rr]^{H_\Lie}&& \M \ar[u]_{J^s}}  \label{diag:JLies}
\end{equation}
\end{definition}

\begin{remark}
We point out that $\mathrm{BrLie}_{\mathcal{M}}^{s}=\mathrm{YBLieAlg}(%
\mathcal{M})$ with the notations of \cite[Definition 2.5]{GV-LieMon} (note
that (\ref{form:cbraid}) follows from (\ref{form:cbraid2}) as we are in the
symmetric case).
\end{remark}

\begin{lemma}
\label{lem:Jac2}Let $\mathcal{M}$ be an abelian monoidal category. Consider
a tern $\left( M,c,\left[ -\right] :M\otimes M\rightarrow M\right) $ where $%
\left( M,c\right) \in \mathrm{Br}_{\mathcal{M}}$. If $c^{2}=\mathrm{Id}$ and
(\ref{Lie1}) holds, then we have that (\ref{Lie2}) is equivalent to
\begin{equation}
\left[ -\right] \circ \left( \left[ -\right] \otimes M\right) \circ \left[
\mathrm{Id}_{\left( V\otimes V\right) \otimes V}+\left( M\otimes c\right)
\left( c\otimes M\right) +\left( c\otimes M\right) \left( M\otimes c\right) %
\right] =0.  \label{Lie3}
\end{equation}
\end{lemma}

\begin{proof}
This proof is essentially the same as \cite[Lemma 2.9]{GV}.
\end{proof}

\begin{remark}
In view of Lemma \ref{lem:Jac2}, in the particular case when $\mathcal{M}$
is the category of vector spaces and $\left( M,c\right) \in \mathrm{Br}_{%
\mathcal{M}}$, conditions (\ref{Lie1}) and (\ref{Lie3}) encode the notion of
Lie algebra in the sense of Gurevich's \cite{Gure}.
\end{remark}

\begin{definition}
Let $\mathcal{M}$ a preadditive monoidal category with equalizers and
denumerable coproducts. Let $\left( M,c\right) \in \mathrm{Br}_{\mathcal{M}}$%
. For $\alpha _{2}M$ as in of Remark \ref{rem: AlgMon}, we set%
\begin{equation}
\theta _{\left( M,c\right) }:=\alpha _{2}M\circ \left( \mathrm{Id}_{M\otimes
M}-c\right) :M\otimes M\rightarrow \Omega TM.  \label{def:theta}
\end{equation}%
When $\mathcal{M}$ is braided and its braiding on $M$ is $c_{M,M}$ we will
simply write $\theta _{M}$ for $\theta _{\left( M,c_{M,M}\right) }$.
\end{definition}

\begin{definition}
\label{def:Lambda}Let $\mathcal{M}$ be a monoidal category. Let $\left(
A,m_{A},u_{A}\right) $ be an algebra in $\mathcal{M}$ and let $%
f:X\rightarrow A$ be a morphism in $\mathcal{M}$. We set%
\begin{equation}
\Lambda _{f}:=m_{A}\circ \left( m_{A}\otimes A\right) \circ \left( A\otimes
f\otimes A\right) :A\otimes X\otimes A\rightarrow A.  \label{def:Lambdaf}
\end{equation}%
When the category $\mathcal{M}$ is also abelian we can consider the
two-sided ideal of $A$ \emph{generated by }$f$ which is defined by $\left(
\left\langle f\right\rangle ,i_{f}\right) :=\mathrm{Im}\left( \Lambda
_{f}\right) $ and it has the following property (see e.g. \cite[Lemma 3.18]%
{AMS-Hoch}): for every algebra morphism $g:A\rightarrow B$ one has that $%
g\circ i_{f}=0$ if and only if $g\circ f=0.$
\end{definition}

\begin{remark}
Let $\mathcal{M}$ be an abelian monoidal category. Let $\left(
A,m_{A},u_{A}\right) $ be an algebra in $\mathcal{M}$.

1)\ Note that $\Lambda _{f}-\Lambda _{g}=\Lambda _{f-g}$ for every $%
f,g:X\rightarrow A$.

2)\ Assume that the tensor products preserve epimorphisms. Let $%
f:X\rightarrow A$ be a morphism in $\mathcal{M}$ and set $\left(
S,j:S\rightarrow A\right) :=\mathrm{Im}\left( f\right) .$ Define the ideal $%
\left( S\right) $ generated by $S$ by setting $\left( \left( S\right)
,i\right) :=\mathrm{Im}\left( \Lambda _{j}\right) $. Write $f=j\circ p$
where $p:X\rightarrow S$ is an epimorphism. We compute $\mathrm{Im}\left(
\Lambda _{f}\right) =\mathrm{Im}\left( \Lambda _{j\circ p}\right) =\mathrm{Im%
}\left( \Lambda _{j\circ p}\right) =\mathrm{Im}\left( \Lambda _{j}\circ
\left( A\otimes p\otimes A\right) \right) =\mathrm{Im}\left( \Lambda
_{j}\right) $ so that $\left( \left\langle f\right\rangle ,i_{f}\right) :=%
\mathrm{Im}\left( \Lambda _{f}\right) =\left( \left( S\right) ,i\right) .$
Therefore $\left\langle f\right\rangle =\left( \mathrm{Im}\left( f\right)
\right) .$
\end{remark}

Next aim is to construct suitable universal enveloping algebra type functors.

\begin{remark}
Let $\mathcal{M}$ an abelian monoidal category with denumerable coproducts.
Assume that the tensor functors preserve denumerable coproducts. Note that $%
\mathcal{M}$ has also finite coproducts as it has a zero object and
denumerable coproduct. Thus, by \cite[Proposition 3.3]{Stenstroem} the
tensor functors are additive as they preserve denumerable coproducts.
\end{remark}

\begin{proposition}
\label{pro:Ubr}Let $\mathcal{M}$ an abelian monoidal category with
denumerable coproducts. Assume that the tensor functors are right exact and
preserve denumerable coproducts.

Let $\left( M,c,\left[ -\right] :M\otimes M\rightarrow M\right) \in \mathrm{%
BrLie}_{\mathcal{M}}$ and set
\begin{equation*}
f:=f_{\left( M,c,\left[ -\right] \right) }:=\alpha _{1}M\circ \left[ -\right]
-\theta _{\left( M,c\right) }:M\otimes M\rightarrow \Omega TM.
\end{equation*}%
Let $\mathcal{U}_{\mathrm{Br}}\left( M,c,\left[ -\right] \right) :=R:=\Omega
TM/\left\langle f\right\rangle $ and let $p_{R}:\Omega TM\rightarrow R$
denote the canonical projection. Then there are morphisms $m_{R},u_{R},c_{R}$
such that $\left( R,m_{R},u_{R},c_{R}\right) \in \mathrm{BrAlg}_{\mathcal{M}%
} $ and $p_{R}$ is a morphism of braided algebras. This way we get a functor%
\begin{equation*}
\mathcal{U}_{\mathrm{Br}}:\mathrm{BrLie}_{\mathcal{M}}\rightarrow \mathrm{%
BrAlg}_{\mathcal{M}}.
\end{equation*}%
and the projections $p_{R}$ define a natural transformation $p:T_{\mathrm{Br}%
}H_{\mathrm{BrLie}}\rightarrow \mathcal{U}_{\mathrm{Br}}$. Moreover there is
a functor $\mathcal{U}_{\mathrm{Br}}^{s}:\mathrm{BrLie}_{\mathcal{M}%
}^{s}\rightarrow \mathrm{BrAlg}_{\mathcal{M}}^{s}$ such that the diagram
\begin{equation}
\xymatrixrowsep{15pt}\xymatrixcolsep{15pt} \xymatrix{\BrLie_M^s
\ar[d]_{\mathcal{U}_\Br^s}\ar[rr]^{\I_\BrLie^s}&&
\BrLie_M\ar[d]^{\mathcal{U}_\Br}\\ \BrAlg_M^s \ar[rr]^{\I_\BrAlg^s}&& \Alg_M
}  \label{diag:Ubrs}
\end{equation}%
commutes and there is a natural transformation $p^{s}:T_{\mathrm{Br}}^{s}H_{%
\mathrm{BrLie}}^{s}\rightarrow \mathcal{U}_{\mathrm{Br}}^{s}$ uniquely
defined by%
\begin{equation}
\mathbb{I}_{\mathrm{BrAlg}}^{s}p^{s}=p\mathbb{I}_{\mathrm{BrLie}}^{s}.
\label{form:ps}
\end{equation}
\end{proposition}

\begin{proof}
Set $\left( A,m_{A},u_{A},c_{A}\right) :=T_{\mathrm{Br}}\left( M,c\right) .$
We will use the equalities for the graded part $c_{A}^{m,n}$ of the braiding
$c_{A}$ which are in \cite[Proposition 2.7]{AM-BraidedOb}. Note that, by
\cite[(42)]{AM-BraidedOb}, we have that $c_{A}\circ \left( \alpha
_{m}M\otimes \alpha _{n}M\right) =\left( \alpha _{n}M\otimes \alpha
_{m}M\right) \circ c_{A}^{m,n}$ for every $m,n\in \mathbb{N}$. By induction
on $n\in \mathbb{N},$ using (\ref{form:cbraid}), one checks that
\begin{equation}
c_{A}^{n,1}\circ \left( M^{\otimes n}\otimes \left[ -\right] \right) =\left( %
\left[ -\right] \otimes M^{\otimes n}\right) \circ c_{A}^{n,2}.
\label{form:cbraidn}
\end{equation}
If we apply \cite[(32) and (34)]{AM-BraidedOb}, we get
\begin{equation*}
c_{A}^{l,n+m}\left( M^{\otimes l}\otimes c_{A}^{m,n}\right) =\left(
c_{A}^{m,n}\otimes M^{\otimes l}\right) c_{A}^{l,m+n}.
\end{equation*}%
If we apply this equality to the case $"\left( l,m,n\right) "=\left(
n,1,1\right) ,$ we obtain
\begin{equation}
c_{A}^{n,2}\left( M^{\otimes n}\otimes c\right) =\left( c\otimes M^{\otimes
n}\right) c_{A}^{n,2},\text{ for every }n\in \mathbb{N}\text{.}
\label{form:YBn}
\end{equation}

Since $\left\langle f\right\rangle $ is an ideal of $TM$, it is clear that $%
R $ is an algebra and $p_{R}$ is an algebra morphism. Consider the exact
sequence
\begin{equation*}
0\rightarrow \left\langle f\right\rangle \overset{i_{f}}{\rightarrow }A%
\overset{p_{R}}{\rightarrow }R\rightarrow 0
\end{equation*}%
If we apply to it the functor $A\otimes \left( -\right) $, we obtain the
exact sequence
\begin{equation*}
A\otimes \left\langle f\right\rangle \overset{A\otimes i_{f}}{\rightarrow }%
A\otimes A\overset{A\otimes p_{R}}{\rightarrow }A\otimes R\rightarrow 0
\end{equation*}%
We have that $\left( \left\langle f\right\rangle ,i_{f}\right) :=\mathrm{Im}%
\left( \Lambda _{f}\right) $ so that we can write $\Lambda _{f}=i_{f}\circ
p_{f}$ where $p_{f}:A\otimes X\otimes A\rightarrow \left\langle
f\right\rangle $ is an epimorphism. Since the tensor products preserve
epimorphisms, we have that $A\otimes p_{f}$ is an epimorphism so that $%
\left( p_{R}\otimes A\right) c_{A}\left( A\otimes i_{f}\right) =0$ if and
only if $\left( p_{R}\otimes A\right) c_{A}\left( A\otimes \Lambda
_{f}\right) =0.$ Using the definition of $c_{A},$ (\ref{form:cbraidn}) and (%
\ref{form:YBn}) one checks that $\left( p_{R}\otimes A\right) c_{A}\left(
A\otimes f\right) \left( \alpha _{n}M\otimes M\otimes M\right) =0.$ Since
this holds for every $n\in \mathbb{N}$ and the tensor products preserve the
denumerable coproducts, we get
\begin{equation}
\left( p_{R}\otimes A\right) c_{A}\left( A\otimes f\right) =0.
\label{form:U1}
\end{equation}%
Now using (\ref{Br3}) and (\ref{form:U1}) one gets $\left( p_{R}\otimes
A\right) c_{A}\left( A\otimes \Lambda _{f}\right) =0.$ Hence, by the
foregoing, we get $\left( p_{R}\otimes A\right) c_{A}\left( A\otimes
i_{f}\right) =0.$ Thus there is a unique morphism $c_{A,R}:A\otimes
R\rightarrow R\otimes A$ such that $c_{A,R}\circ \left( A\otimes
p_{R}\right) =\left( p_{R}\otimes A\right) \circ c_{A}$. Consider now the
exact sequence
\begin{equation*}
\left\langle f\right\rangle \otimes R\overset{i_{f}\otimes R}{\rightarrow }%
A\otimes R\overset{p_{R}\otimes R}{\rightarrow }R\otimes R\rightarrow 0.
\end{equation*}%
We will prove that $\left( R\otimes p_{R}\right) c_{A,R}\left( i_{f}\otimes
R\right) =0$.

This equality is equivalent to prove $\left( R\otimes p_{R}\right)
c_{A,R}\left( \Lambda _{f}\otimes R\right) =0.$ We have
\begin{eqnarray*}
&&\left( R\otimes p_{R}\right) c_{A,R}\left( \Lambda _{f}\otimes R\right)
\left( A\otimes M\otimes M\otimes A\otimes p_{R}\right) \\
&=&\left( R\otimes p_{R}\right) c_{A,R}\left( A\otimes p_{R}\right) \left(
\Lambda _{f}\otimes A\right) =\left( R\otimes p_{R}\right) \left(
p_{R}\otimes A\right) c_{A}\left( \Lambda _{f}\otimes A\right) \\
&=&\left( p_{R}\otimes R\right) \left( A\otimes p_{R}\right) c_{A}\left(
\Lambda _{f}\otimes A\right) .
\end{eqnarray*}%
Note that the latter term vanishes as $\left( A\otimes p_{R}\right)
c_{A}\left( \Lambda _{f}\otimes A\right) =0$ by a similar argument to the
one used to prove $\left( p_{R}\otimes A\right) c_{A}\left( A\otimes \Lambda
_{f}\right) =0$ and using (\ref{form:cbraid2}). Since $A\otimes M\otimes
M\otimes A\otimes p_{R}$ is an epimorphism, we get that $\left( R\otimes
p_{R}\right) c_{A,R}\left( \Lambda _{f}\otimes R\right) =0$ and hence there
is a unique morphism $c_{R}:R\otimes R\rightarrow R\otimes R$ such that $%
c_{R}\circ \left( p_{R}\otimes R\right) =\left( R\otimes p_{R}\right) \circ
c_{A,R}.$ We get%
\begin{equation*}
c_{R}\left( p_{R}\otimes p_{R}\right) =c_{R}\left( p_{R}\otimes R\right)
\left( A\otimes p_{R}\right) =\left( R\otimes p_{R}\right) c_{A,R}\left(
A\otimes p_{R}\right) =\left( R\otimes p_{R}\right) \left( p_{R}\otimes
A\right) c_{A}=\left( p_{R}\otimes p_{R}\right) c_{A}.
\end{equation*}

If we rewrite (\ref{form:cbraid}) and (\ref{form:cbraid2}) in terms of $%
c^{-1}$ we get that $\left( M,c^{-1}\right) $ fulfills (\ref{form:cbraid})
and (\ref{form:cbraid2}). Thus we can repeat the argument above obtaining a
morphism $c_{R}^{\prime }$ such that $c_{R}^{\prime }\left( p_{R}\otimes
p_{R}\right) =\left( p_{R}\otimes p_{R}\right) c_{A}^{-1}.$ It is easy to
check that $c_{R}^{\prime }$ is an inverse for $c_{R}.$ By Lemma \ref%
{lem:quot}, we get that $\left( R,c_{R}\right) $ is an object in $\mathrm{Br}%
_{\mathcal{M}}$ and $p_{R}$ becomes a morphism in $\mathrm{Br}_{\mathcal{M}}$
from $\left( A,c_{A}\right) $ to this object. We have%
\begin{eqnarray*}
c_{R}(m_{R}\otimes R)\left( p_{R}\otimes p_{R}\otimes p_{R}\right)
&=&c_{R}\left( p_{R}\otimes p_{R}\right) (m_{A}\otimes A)=\left(
p_{R}\otimes p_{R}\right) c_{A}(m_{A}\otimes A) \\
\overset{(\ref{Br2})}{=}\left( p_{R}\otimes p_{R}\right) (A\otimes
m)(c\otimes A)(A\otimes c) &=&(R\otimes m_{R})(c_{R}\otimes R)(R\otimes
c_{R})\left( p_{R}\otimes p_{R}\otimes p_{R}\right)
\end{eqnarray*}%
so that (\ref{Br2}) holds for $\left( R,m_{R},c_{R}\right) .$ Similarly one
proves (\ref{Br3}). Moreover%
\begin{gather*}
c_{R}(u_{R}\otimes R)l_{R}^{-1}p_{R}=c_{R}(u_{R}\otimes R)\left( \mathbf{1}%
\otimes p_{R}\right) l_{A}^{-1}=c_{R}(p_{R}u_{A}\otimes p_{R})l_{A}^{-1} \\
=\left( p_{R}\otimes p_{R}\right) c_{A}(u_{A}\otimes A)l_{A}^{-1}\overset{(%
\ref{Br4})}{=}\left( p_{R}\otimes p_{R}\right) \left( A\otimes u_{A}\right)
r_{A}^{-1}=\left( p_{R}\otimes u_{R}\right) r_{A}^{-1}=\left( R\otimes
u_{R}\right) r_{R}^{-1}p_{R}
\end{gather*}

and hence $c_{R}(u_{R}\otimes R)l_{R}^{-1}=\left( R\otimes u_{R}\right)
r_{R}^{-1}.$ Similarly one gets $c_{R}(R\otimes u_{R})r_{R}^{-1}=\left(
u_{R}\otimes R\right) l_{R}^{-1}.$ We have so proved that $\left(
R,m_{R},u_{R},c_{R}\right) \in \mathrm{BrAlg}_{\mathcal{M}}$. It is clear
that $p_{R}$ is a morphism of braided algebras.

Let $\nu :\left( M,c,\left[ -\right] \right) \rightarrow \left( M^{\prime
},c^{\prime },\left[ -\right] ^{\prime }\right) $ be a morphism of braided
Lie algebras. Consider the morphism of braided algebras $T_{\mathrm{Br}}\nu
:T_{\mathrm{Br}}\left( M,c\right) \rightarrow T_{\mathrm{Br}}\left(
M^{\prime },c^{\prime }\right) .$ Set $R^{\prime }:=\mathcal{U}_{\mathrm{Br}%
}\left( M^{\prime },c^{\prime },\left[ -\right] ^{\prime }\right) $ and
denote by $p_{R^{\prime }}$ the corresponding projection and set $f^{\prime
}:=f_{\left( M^{\prime },c^{\prime },\left[ -\right] ^{\prime }\right) }$.
We have
\begin{eqnarray*}
&&p_{R^{\prime }}\circ \Omega H_{\mathrm{Alg}}T_{\mathrm{Br}}H_{\mathrm{BrLie%
}}\nu \circ f\overset{(\ref{form:HtildeTbrOmegaBr})}{=}p_{R^{\prime }}\circ
\Omega THH_{\mathrm{BrLie}}\nu \circ \left( \alpha _{1}M\circ \left[ -\right]
-\theta _{\left( M,c\right) }\right) \\
&=&p_{R^{\prime }}\circ \Omega THH_{\mathrm{BrLie}}\nu \circ \alpha
_{1}M\circ \left[ -\right] -p_{R^{\prime }}\circ \Omega THH_{\mathrm{BrLie}%
}\nu \circ \alpha _{2}M\circ \left( \mathrm{Id}_{M\otimes M}-c\right) \\
&\overset{(\ref{form:Tf})}{=}&p_{R^{\prime }}\circ \alpha _{1}M^{\prime
}\circ HH_{\mathrm{BrLie}}\nu \circ \left[ -\right] -p_{R^{\prime }}\circ
\alpha _{2}M^{\prime }\circ \left( HH_{\mathrm{BrLie}}\nu \otimes HH_{%
\mathrm{BrLie}}\nu \right) \circ \left( \mathrm{Id}_{M\otimes M}-c\right) \\
&=&p_{R^{\prime }}\circ \alpha _{1}M^{\prime }\circ HH_{\mathrm{BrLie}}\nu
\circ \left[ -\right] -p_{R^{\prime }}\circ \alpha _{2}M^{\prime }\circ
\left( \mathrm{Id}_{M^{\prime }\otimes M^{\prime }}-c^{\prime }\right) \circ
\left( HH_{\mathrm{BrLie}}\nu \otimes HH_{\mathrm{BrLie}}\nu \right) \\
&=&p_{R^{\prime }}\circ \alpha _{1}M^{\prime }\circ \left[ -\right] ^{\prime
}\circ \left( HH_{\mathrm{BrLie}}\nu \otimes HH_{\mathrm{BrLie}}\nu \right)
-p_{R^{\prime }}\circ \theta _{\left( M^{\prime },c^{\prime }\right) }\circ
\left( HH_{\mathrm{BrLie}}\nu \otimes HH_{\mathrm{BrLie}}\nu \right) \\
&=&p_{R^{\prime }}\circ f^{\prime }\circ \left( HH_{\mathrm{BrLie}}\nu
\otimes HH_{\mathrm{BrLie}}\nu \right) =0.
\end{eqnarray*}

Since $p_{R^{\prime }}\circ \Omega H_{\mathrm{Alg}}T_{\mathrm{Br}}H_{\mathrm{%
BrLie}}\nu \circ $ is an algebra morphism we get $p_{R^{\prime }}\circ
\Omega H_{\mathrm{Alg}}T_{\mathrm{Br}}H_{\mathrm{BrLie}}\nu \circ i_{f}=0$
so that there is a unique morphism $\mathcal{U}_{\mathrm{Br}}\nu :\mathcal{U}%
_{\mathrm{Br}}\left( M,c,\left[ -\right] \right) \rightarrow \mathcal{U}_{%
\mathrm{Br}}\left( M^{\prime },c^{\prime },\left[ -\right] ^{\prime }\right)
$ such that $\mathcal{U}_{\mathrm{Br}}\nu \circ p_{R}=p_{R^{\prime }}\circ
T_{\mathrm{Br}}H_{\mathrm{BrLie}}\nu .$ It is easy to check that $\mathcal{U}%
_{\mathrm{Br}}\nu $ is a morphism of braided bialgebras. Since $T_{\mathrm{Br%
}}$ is a functor it is then clear that $\mathcal{U}_{\mathrm{Br}}$ becomes a
functor as well and that the projections define a natural transformation $%
p:T_{\mathrm{Br}}H_{\mathrm{BrLie}}\rightarrow \mathcal{U}_{\mathrm{Br}}$.

Let us construct $\mathcal{U}_{\mathrm{Br}}^{s}$. We already observed that
the functor $\mathbb{I}_{\mathrm{BrAlg}}^{s}$ is full, faithful and
injective on objects.

Let $\left( M,c,\left[ -\right] \right) \in \mathrm{BrLie}_{\mathcal{M}}^{s}$%
. Then, by Remark \ref{rem:aureo}-2), we get that $R=\mathcal{U}_{\mathrm{Br}%
}\left( M,c,\left[ -\right] \right) \in \mathrm{BrAlg}_{\mathcal{M}}^{s}$ as
$R$ is a quotient of $T_{\mathrm{Br}}H_{\mathrm{BrLie}}\left( M,c,\left[ -%
\right] \right) $ which is preserved by the required functors. Hence any
object which is image of $\mathcal{U}_{\mathrm{Br}}\mathbb{I}_{\mathrm{BrLie}%
}^{s}$ is also image of $\mathbb{I}_{\mathrm{BrAlg}}^{s}.$ By Lemma \ref%
{lem:Cappuccio}, there is a unique functor $\mathcal{U}_{\mathrm{Br}}^{s}:=%
\widehat{\mathcal{U}_{\mathrm{Br}}\mathbb{I}_{\mathrm{BrLie}}^{s}}$ such
that (\ref{diag:Ubrs}) commutes. We have%
\begin{equation}
T_{\mathrm{Br}}H_{\mathrm{BrLie}}\mathbb{I}_{\mathrm{BrLie}}^{s}\overset{(%
\ref{diag:HBrLies})}{=}T_{\mathrm{Br}}\mathbb{I}_{\mathrm{Br}}^{s}H_{\mathrm{%
BrLie}}^{s}\overset{(\ref{diag:Tbrs})}{=}\mathbb{I}_{\mathrm{BrAlg}}^{s}T_{%
\mathrm{Br}}^{s}H_{\mathrm{BrLie}}^{s}.  \label{form:THI}
\end{equation}%
By Lemma \ref{lem:Cappuccio}, we have $\widehat{T_{\mathrm{Br}}H_{\mathrm{%
BrLie}}\mathbb{I}_{\mathrm{BrLie}}^{s}}=T_{\mathrm{Br}}^{s}H_{\mathrm{BrLie}%
}^{s},$ $\widehat{\mathcal{U}_{\mathrm{Br}}\mathbb{I}_{\mathrm{BrLie}}^{s}}=%
\mathcal{U}_{\mathrm{Br}}^{s}$ and there is a unique natural transformation $%
p^{s}:=\widehat{p\mathbb{I}_{\mathrm{BrLie}}^{s}}:T_{\mathrm{Br}}^{s}H_{%
\mathrm{BrLie}}^{s}\rightarrow \mathcal{U}_{\mathrm{Br}}^{s}$ such that (\ref%
{form:ps}) holds.
\end{proof}

\begin{lemma}
\label{lem:Un}Let $\mathcal{M}$ a preadditive monoidal category with
denumerable coproducts. Assume that the tensor functors are additive and
preserve such coproducts. Let $\left( M,c,\left[ -\right] :M\otimes
M\rightarrow M\right) \in \mathrm{BrLie}_{\mathcal{M}}$, set $\left(
A,m_{A},u_{A},\Delta _{A},\varepsilon _{A},c_{A}\right) :=\overline{T}_{%
\mathrm{Br}}\left( M,c\right) $ and use the notations of \ref{cl:TbarStrict}%
. Then,
\begin{equation}
\Delta _{A}\circ \theta _{\left( M,c\right) }=\left[ \left( u_{A}\otimes
A\right) \circ l_{A}^{-1}+\left( A\otimes u_{A}\right) \circ r_{A}^{-1}%
\right] \circ \theta _{\left( M,c\right) }\text{ if }c^{2}=\mathrm{Id}%
_{M\otimes M};  \label{form:Un1}
\end{equation}%
\begin{equation}
\Delta _{A}\circ \alpha _{1}M=\left[ \left( u_{A}\otimes A\right) \circ
l_{A}^{-1}+\left( A\otimes u_{A}\right) \circ r_{A}^{-1}\right] \circ \alpha
_{1}M.  \label{form:Un2}
\end{equation}
\end{lemma}

\begin{proof}
Using, in the given order, (\ref{form:TVm}), the multiplicativity of $\Delta
_{A}$, (\ref{form:TBrdelta}), the definitions of $\delta _{M}^{l}$ and $%
\delta _{M}^{r},$ the equalities $c_{A}\circ \left( \alpha _{i}M\otimes
\alpha _{j}M\right) =\left( \alpha _{j}M\otimes \alpha _{j}M\right) \circ
c_{A}^{i,j}$ for $i,j\in \left\{ 1,2\right\} ,$ the equalities $%
c_{A}^{1,0}=l_{M}^{-1}r_{M}$,$c_{A}^{1,1}=c,c_{A}^{0,0}=l_{\mathbf{1}%
}^{-1}r_{\mathbf{1}}$ and $c_{A}^{0,1}=r_{M}^{-1}l_{M},$ the equalities (\ref%
{form:TVu}) and (\ref{form:TVm}), the equalities $r_{M}\otimes M=M\otimes
l_{M}$ and $r_{\mathbf{1}}\otimes M=\mathbf{1}\otimes l_{M},$ the equalities
$l_{M}^{-1}\otimes M=l_{M\otimes M}^{-1}$, $M\otimes l_{\mathbf{1}%
}^{-1}=r_{M}^{-1}\otimes \mathbf{1}$ and $M\otimes r_{M}^{-1}=r_{M\otimes
M}^{-1},$ the equalities $m_{\mathbf{1}}=r_{\mathbf{1}}=l_{\mathbf{1}},$ the
naturality of the unit constraints, $l_{M}^{-1}\otimes M=l_{M\otimes M}^{-1}$%
, $M\otimes r_{M}^{-1}=r_{M\otimes M}^{-1}$ and $r_{M}\otimes M=M\otimes
l_{M},$ the equality (\ref{form:TVu}) and the naturality of the unit
constraints one proves that%
\begin{equation*}
\Delta _{A}\circ \alpha _{2}M=\left[ \left( u_{A}\otimes A\right) \circ
l_{A}^{-1}+\left( A\otimes u_{A}\right) \circ r_{A}^{-1}\right] \circ \alpha
_{2}M+\left( \alpha _{1}M\otimes \alpha _{1}M\right) \circ \left( \mathrm{Id}%
_{M\otimes M}+c\right) .
\end{equation*}%
From this equality, composing with $\mathrm{Id}_{M\otimes M}-c$ on both
sides, we get (\ref{form:Un1}) holds true when $c^{2}=\mathrm{Id}_{M\otimes
M}$.

On the other hand, (\ref{form:Un2}) follows by (\ref{form:TBrdelta}), the
definitions of $\delta _{M}^{l}$ and $\delta _{M}^{r},$ the naturality of
the unit constraints.
\end{proof}

\begin{proposition}
\label{pro:BrBialgQuotient} Let $\left( B,m_{B},u_{B},\Delta
_{B},\varepsilon _{B},c_{B}\right) \in \mathrm{BrBialg}_{\mathcal{M}}$ be a
bialgebra in a monoidal category $\mathcal{M}$. Assume that the category $%
\mathcal{M}$ is abelian and the tensor functors are additive and right
exact. Let $\left( R,m_{R},u_{R},c_{R}\right) \in \mathrm{BrAlg}_{\mathcal{M}%
}$ and let $p_{R}:B\rightarrow R$ be an epimorphism which is a morphism of
braided algebras. Set $\left( I,i_{I}:I\rightarrow B\right) :=\mathrm{Ker}%
\left( p_{R}\right) $. Assume that%
\begin{gather}
\left( p_{R}\otimes p_{R}\right) \circ \Delta _{B}\circ i_{I}=0,
\label{form:coid1} \\
\varepsilon _{B}\circ i_{I}=0.  \label{form:coid2}
\end{gather}%
Then there are morphisms $\Delta _{R},\varepsilon _{R}$ such that $\left(
R,m_{R},u_{R},\Delta _{R},\varepsilon _{R},c_{R}\right) \in \mathrm{BrBialg}%
_{\mathcal{M}}$ and $p_{R}$ is a morphism of braided bialgebras.
\end{proposition}

\begin{proof}
Since $\left( R,p_{R}\right) =\mathrm{Coker}\left( i_{I}\right) ,$ by (\ref%
{form:coid1}), there is a unique morphism $\Delta _{R}:R\rightarrow R\otimes
R$ such that $\Delta _{R}\circ p_{R}=\left( p_{R}\otimes p_{R}\right) \circ
\Delta _{B}$ and, by (\ref{form:coid2}), there is a unique morphism $%
\varepsilon _{R}:R\rightarrow \mathbf{1}$ such that $\varepsilon _{R}\circ
p_{R}=\varepsilon _{B}.$ The rest of the proof is straightforward and relies
on the fact that $p_{R}\otimes p_{R}=\left( p_{R}\otimes R\right) \left(
A\otimes p_{R}\right) $ is an epimorphism by exactness of the tensor
functors.
\end{proof}

\begin{theorem}
\label{Teo:UbarBr}Let $\mathcal{M}$ an abelian monoidal category with
denumerable coproducts. Assume that the tensor functors are right exact and
preserve denumerable coproducts. Then there is a functor $\overline{\mathcal{%
U}}_{\mathrm{Br}}^{s}:\mathrm{BrLie}_{\mathcal{M}}^{s}\rightarrow \mathrm{%
BrBialg}_{\mathcal{M}}^{s}$ such that%
\begin{equation}
\xymatrixrowsep{15pt}\xymatrixcolsep{8pt} \xymatrix{\BrLie_\M^s
\ar[dr]_{\mathcal{U}_\Br^s}\ar[rr]^{\overline{\mathcal{U}}_\Br^s}&&
\BrBialg_\M^s\ar[dl]^{\mho_\Br^s}\\ &\BrAlg_\M^s }  \label{diag:UbrsBar}
\end{equation}%
Moreover there is a natural transformation $\overline{p}^{s}:\overline{T}_{%
\mathrm{Br}}^{s}H_{\mathrm{BrLie}}^{s}\rightarrow \overline{\mathcal{U}}_{%
\mathrm{Br}}^{s}$ uniquely defined by%
\begin{equation}
\mho _{\mathrm{Br}}\mathbb{I}_{\mathrm{BrBialg}}^{s}\overline{p}^{s}=p%
\mathbb{I}_{\mathrm{BrLie}}^{s}\qquad \text{and}\qquad \mho _{\mathrm{Br}%
}^{s}\overline{p}^{s}=p^{s}  \label{eq:P&pBars}
\end{equation}%
where $p:T_{\mathrm{Br}}H_{\mathrm{BrLie}}\rightarrow \mathcal{U}_{\mathrm{Br%
}}$ and $p^{s}:T_{\mathrm{Br}}^{s}H_{\mathrm{BrLie}}^{s}\rightarrow \mathcal{%
U}_{\mathrm{Br}}^{s}$ are the natural transformations of Proposition \ref%
{pro:Ubr}.
\end{theorem}

\begin{proof}
Let $\left( M,c,\left[ -\right] \right) \in \mathrm{BrLie}_{\mathcal{M}}^{s}$
and set $\left( A,m_{A},u_{A},\Delta _{A},\varepsilon _{A},c_{A}\right) :=%
\overline{T}_{\mathrm{Br}}^{s}\left( M,c\right) $ and $f:=f_{\left( M,c,%
\left[ -\right] \right) }.$ Set $\left( R,m_{R},u_{R},c_{R}\right) :=%
\mathcal{U}_{\mathrm{Br}}\left( M,c,\left[ -\right] \right) $ and let $p_{R}$
be the morphism in $\mathcal{M}$ underlying the canonical projection $%
p\left( M,c,\left[ -\right] \right) :T_{\mathrm{Br}}\left( M,c\right)
\rightarrow \mathcal{U}_{\mathrm{Br}}\left( M,c,\left[ -\right] \right) $.
By Proposition \ref{pro:Ubr}, we know that $p_{R}:A\rightarrow R$ is a
morphism of braided algebras. Using (\ref{form:Un1}) and (\ref{form:Un2}),
we get
\begin{equation}
\Delta _{A}\circ f=\left[ \left( u_{A}\otimes A\right) \circ
l_{A}^{-1}+\left( A\otimes u_{A}\right) \circ r_{A}^{-1}\right] \circ f
\label{form:Deltaf}
\end{equation}

Since $p_{R}$ is an algebra morphism and $p_{R}\circ i_{f}=0,$ we get that $%
p_{R}\circ f=0.$ We want to apply Proposition \ref{pro:BrBialgQuotient} to
the case $\left( I,i_{I}\right) =\left( \left\langle f\right\rangle
,i_{f}\right) .$ Since $\left( p_{R}\otimes p_{R}\right) \circ \Delta _{A}$
is an algebra morphism as a composition of algebra morphisms (use e.g. \cite[%
Proposition 2.2-3)]{AM-BraidedOb} to prove that $p_{R}\otimes p_{R}$ is an
algebra morphism and use (\ref{Br1}) to have that $\Delta _{A}$ is an
algebra morphism), we have that (\ref{form:coid1}) is equivalent to $\left(
p_{R}\otimes p_{R}\right) \circ \Delta _{A}\circ f=0$ and the latter holds
by (\ref{form:Deltaf}), unitality of $p_{R}$, naturality of the unit
constraints, and the equality $p_{R}\circ f=0$.

Since $\varepsilon _{A}$ is an algebra morphism, we have that (\ref%
{form:coid2}) if and only if $\varepsilon _{A}\circ f=0$ and the latter
holds by definition of $f$ and (\ref{form:TBrepsGEN}). Then, by Proposition %
\ref{pro:BrBialgQuotient}, there are morphisms $\Delta _{R},\varepsilon _{R}$
such that $\left( R,m_{R},u_{R},\Delta _{R},\varepsilon _{R},c_{R}\right)
\in \mathrm{BrBialg}_{\mathcal{M}}$ and $p_{R}$ is a morphism of braided
bialgebras. By Remark \ref{rem:aureo}-2) one easily checks that $\left(
R,m_{R},u_{R},\Delta _{R},\varepsilon _{R},c_{R}\right) \in \mathrm{BrBialg}%
_{\mathcal{M}}^{s}.$ We denote this datum by $\overline{\mathcal{U}}_{%
\mathrm{Br}}^{s}\left( M,c,\left[ -\right] \right) $. Let $\nu :\left( M,c,%
\left[ -\right] \right) \rightarrow \left( M^{\prime },c^{\prime },\left[ -%
\right] ^{\prime }\right) $ be a morphism in $\mathrm{BrLie}_{\mathcal{M}%
}^{s}.$ We know that $\widetilde{v}:=\Omega H_{\mathrm{Alg}}\mathcal{U}_{%
\mathrm{Br}}\nu :R\rightarrow R^{\prime }$ is a morphism in $\mathrm{BrAlg}_{%
\mathcal{M}}.$ Using that $p_{R}$ is comultiplicative and natural, and that $%
\Omega H_{\mathrm{Alg}}\mho _{\mathrm{Br}}\overline{T}_{\mathrm{Br}}H_{%
\mathrm{BrLie}}v$ is a coalgebra morphism one easily gets that $\left(
\widetilde{v}\otimes \widetilde{v}\right) \circ \Delta _{R}\circ
p_{R}=\Delta _{R^{\prime }}\circ \widetilde{v}\circ p_{R}$ and hence $%
\widetilde{v}$ is comultiplicative. A similar argument shows that $%
\widetilde{v}$ is also counitary and hence $\mathcal{U}_{\mathrm{Br}}\nu $
is a morphism in $\mathrm{BrBialg}_{\mathcal{M}}^{s}.$ This defines a
functor $\overline{\mathcal{U}}_{\mathrm{Br}}^{s}:\mathrm{BrLie}_{\mathcal{M}%
}^{s}\rightarrow \mathrm{BrBialg}_{\mathcal{M}}^{s}$ such that $\mho _{%
\mathrm{Br}}^{s}\circ \overline{\mathcal{U}}_{\mathrm{Br}}^{s}=\mathcal{U}_{%
\mathrm{Br}}^{s}.$ Since $p_{R}$ is a morphism of braided bialgebras and it
is natural in $R$ at the level of $\mathrm{BrAlg}_{\mathcal{M}},$ it is
clear that $\overline{p}^{s}$ such that $\mho _{\mathrm{Br}}\mathbb{I}_{%
\mathrm{BrBialg}}^{s}\overline{p}^{s}=p\mathbb{I}_{\mathrm{BrLie}}^{s}$
exists. Moreover we have%
\begin{equation*}
\mathbb{I}_{\mathrm{BrAlg}}^{s}p^{s}\overset{\left( \ref{form:ps}\right) }{=}%
p\mathbb{I}_{\mathrm{BrLie}}^{s}=\mho _{\mathrm{Br}}\mathbb{I}_{\mathrm{%
BrBialg}}^{s}\overline{p}^{s}\overset{(\ref{diag:OmRibs})}{=}\mathbb{I}_{%
\mathrm{BrAlg}}^{s}\mho _{\mathrm{Br}}^{s}\overline{p}^{s}
\end{equation*}%
and hence $p^{s}=\mho _{\mathrm{Br}}^{s}\overline{p}^{s}.$
\end{proof}

\section{Adjunctions for enveloping functors}

\begin{proposition}
\label{pro:LieFunct}Let $\mathcal{M}$ an abelian monoidal category with
denumerable coproducts. Assume that the tensor functors are right exact and
preserve denumerable coproducts. Then the functor $\mathcal{U}_{\mathrm{Br}%
}^{s}:\mathrm{BrLie}_{\mathcal{M}}^{s}\rightarrow \mathrm{BrAlg}_{\mathcal{M}%
}^{s}$ has a right adjoint $\mathcal{L}_{\mathrm{Br}}^{s}:\mathrm{BrAlg}_{%
\mathcal{M}}^{s}\rightarrow \mathrm{BrLie}_{\mathcal{M}}^{s}$ acting as the
identity on morphisms and defined on objects by $\mathcal{L}_{\mathrm{Br}%
}^{s}\left( B,m_{B},u_{B},c_{B}\right) :=\left( B,c_{B},\left[ -\right]
_{B}\right) ,$ where $\left[ -\right] _{B}:=m_{B}\circ \left( \mathrm{Id}%
_{B\otimes B}-c_{B}\right) $. The unit $\eta _{\mathrm{BrL}}^{s}:\mathrm{Id}%
_{\mathrm{BrLie}_{\mathcal{M}}^{s}}\rightarrow \mathcal{L}_{\mathrm{Br}}^{s}%
\mathcal{U}_{\mathrm{Br}}^{s}$ and the counit $\epsilon _{\mathrm{BrL}}^{s}:%
\mathcal{U}_{\mathrm{Br}}^{s}\mathcal{L}_{\mathrm{Br}}^{s}\rightarrow
\mathrm{Id}_{\mathrm{BrAlg}_{\mathcal{M}}^{s}}$ of the adjunction fulfill
\begin{equation}
\epsilon _{\mathrm{BrL}}^{s}\circ p^{s}\mathcal{L}_{\mathrm{Br}%
}^{s}=\epsilon _{\mathrm{Br}}^{s}\qquad \text{and}\qquad H_{\mathrm{BrLie}%
}^{s}\mathcal{L}_{\mathrm{Br}}^{s}p^{s}\circ \eta _{\mathrm{Br}}^{s}H_{%
\mathrm{BrLie}}^{s}=H_{\mathrm{BrLie}}^{s}\eta _{\mathrm{BrL}}^{s}.
\label{form:eps-etaBrL}
\end{equation}
\end{proposition}

\begin{proof}
The construction of the functor $\mathcal{L}_{\mathrm{Br}}^{s}$ is given in
\cite[Construction 2.16]{GV-LieMon} where $\mathrm{BrAlg}_{\mathcal{M}}^{s}$
plays the role of $\mathrm{YBAlg}(\mathcal{M})$ therein. Let us check that $%
\left( \mathcal{U}_{\mathrm{Br}}^{s},\mathcal{L}_{\mathrm{Br}}^{s}\right) $
is an adjunction.

Consider the natural transformation $p^{s}:T_{\mathrm{Br}}^{s}H_{\mathrm{%
BrLie}}^{s}\rightarrow \mathcal{U}_{\mathrm{Br}}^{s}$ of Proposition \ref%
{pro:Ubr}.

Note that $H_{\mathrm{BrLie}}^{s}\mathcal{L}_{\mathrm{Br}}^{s}\left(
B,m_{B},u_{B},c_{B}\right) =H_{\mathrm{BrLie}}^{s}\left( B,c_{B},\left[ -%
\right] _{B}\right) =\left( B,c_{B}\right) =\Omega _{\mathrm{Br}}^{s}\left(
B,m_{B},u_{B},c_{B}\right) $ and $H_{\mathrm{BrLie}}^{s}\mathcal{L}_{\mathrm{%
Br}}^{s}$ and $\Omega _{\mathrm{Br}}^{s}$ both act as the identity on
morphisms so that $H_{\mathrm{BrLie}}^{s}\mathcal{L}_{\mathrm{Br}%
}^{s}=\Omega _{\mathrm{Br}}^{s}.$ Then we have $p^{s}\mathcal{L}_{\mathrm{Br}%
}^{s}:T_{\mathrm{Br}}^{s}\Omega _{\mathrm{Br}}^{s}\rightarrow \mathcal{U}_{%
\mathrm{Br}}^{s}\mathcal{L}_{\mathrm{Br}}^{s}$. Consider $\epsilon _{\mathrm{%
Br}}^{s}:T_{\mathrm{Br}}^{s}\Omega _{\mathrm{Br}}^{s}\rightarrow \mathrm{Id}%
_{\mathrm{BrAlg}_{\mathcal{M}}^{s}}.$ Using the notation of Proposition \ref%
{pro:Ubr}, by means of (\ref{form:TbrStricts}), (\ref{form:TbrStrict}), (\ref%
{def:theta}) and (\ref{form:etaeps}) we get
\begin{equation*}
\Omega H_{\mathrm{Alg}}\mathbb{I}_{\mathrm{BrAlg}}^{s}\epsilon _{\mathrm{Br}%
}^{s}\left( B,m_{B},u_{B},c_{B}\right) \circ f_{\mathcal{L}_{\mathrm{Br}%
}^{s}\left( B,m_{B},u_{B},c_{B}\right) }=0.
\end{equation*}
Since $\epsilon _{\mathrm{Br}}^{s}$ is a morphism of braided algebras, by
construction of $\mathcal{U}_{\mathrm{Br}}^{s}\mathcal{L}_{\mathrm{Br}}^{s}$%
, the latter equality implies there is a unique morphism $\epsilon _{\mathrm{%
BrL}}^{s}:\mathcal{U}_{\mathrm{Br}}^{s}\mathcal{L}_{\mathrm{Br}%
}^{s}\rightarrow \mathrm{Id}_{\mathrm{BrAlg}_{\mathcal{M}}^{s}}$ such that $%
\epsilon _{\mathrm{BrL}}^{s}\circ p^{s}\mathcal{L}_{\mathrm{Br}%
}^{s}=\epsilon _{\mathrm{Br}}^{s}.$

Consider the morphism $H_{\mathrm{BrLie}}^{s}\mathcal{L}_{\mathrm{Br}%
}^{s}p^{s}\circ \eta _{\mathrm{Br}}^{s}H_{\mathrm{BrLie}}^{s}:H_{\mathrm{%
BrLie}}^{s}\rightarrow H_{\mathrm{BrLie}}^{s}\mathcal{L}_{\mathrm{Br}}^{s}%
\mathcal{U}_{\mathrm{Br}}^{s}.$ Let $\left( M,c_{M},\left[ -\right] \right)
\in \mathrm{BrLie}_{\mathcal{M}}^{s}$ and set $\nu :=H\mathbb{I}_{\mathrm{Br}%
}^{s}H_{\mathrm{BrLie}}^{s}\mathcal{L}_{\mathrm{Br}}^{s}p^{s}\left( M,c_{M},%
\left[ -\right] \right) \circ H\mathbb{I}_{\mathrm{Br}}^{s}\eta _{\mathrm{Br}%
}^{s}H_{\mathrm{BrLie}}^{s}\left( M,c_{M},\left[ -\right] \right) $, $\left(
R,m_{R},u_{R},c_{R}\right) :=\mathcal{U}_{\mathrm{Br}}^{s}\left( M,c_{M},%
\left[ -\right] \right) $ and $\left( A,m_{A},u_{A},c_{A}\right) :=T_{%
\mathrm{Br}}^{s}\left( M,c_{M}\right) $. Clearly $\nu :\left( M,c_{M}\right)
\rightarrow \left( R,c_{R}\right) $ is a morphism of braided objects. Using (%
\ref{diag:HBrLies}), (\ref{form:TbrStricts}), (\ref{form:ps}), (\ref%
{form:TbrStrict}), (\ref{form:etaeps}) and the equality $p_{R}=H\Omega _{%
\mathrm{Br}}p\mathbb{I}_{\mathrm{BrLie}}^{s}\left( M,c_{M},\left[ -\right]
\right) $ (which follows by definition of $p$ in Proposition \ref{pro:Ubr}),
we obtain that $\nu =p_{R}\circ \alpha _{1}M.$ By the latter formula, the
fact that $p_{R}$ is a braided morphisms, the definition of $c_{A}$ given by
\cite[(42)]{AM-BraidedOb}, the multiplicativity of $p_{R}$, using (\ref%
{form:TVm}), (\ref{def:theta}) and the formula $p_{R}\circ f_{\left( M,c_{M},%
\left[ -\right] \right) }=0$, we obtain $\left[ -\right] _{R}\circ \left(
\nu \otimes \nu \right) =\nu \circ \left[ -\right] .$ Since $\nu $ is the
morphism in $\mathcal{M}$ defining $H_{\mathrm{BrLie}}^{s}\mathcal{L}_{%
\mathrm{Br}}^{s}p^{s}\circ \eta _{\mathrm{Br}}^{s}H_{\mathrm{BrLie}}^{s}:H_{%
\mathrm{BrLie}}^{s}\rightarrow H_{\mathrm{BrLie}}^{s}\mathcal{L}_{\mathrm{Br}%
}^{s}\mathcal{U}_{\mathrm{Br}}^{s},$ we get that there is a unique natural
transformation $\eta _{\mathrm{BrL}}^{s}:\mathrm{Id}_{\mathrm{BrLie}_{%
\mathcal{M}}^{s}}\rightarrow \mathcal{L}_{\mathrm{Br}}^{s}\mathcal{U}_{%
\mathrm{Br}}^{s}$ such that $H_{\mathrm{BrLie}}^{s}\mathcal{L}_{\mathrm{Br}%
}^{s}p^{s}\circ \eta _{\mathrm{Br}}^{s}H_{\mathrm{BrLie}}^{s}=H_{\mathrm{%
BrLie}}^{s}\eta _{\mathrm{BrL}}^{s}.$ It is straightforward to check that
this gives rise to the claimed adjunction. Note that
\begin{equation}
H\mathbb{I}_{\mathrm{Br}}^{s}H_{\mathrm{BrLie}}^{s}\eta _{\mathrm{BrL}%
}^{s}\left( M,c_{M},\left[ -\right] \right) =v=p_{R}\circ \alpha _{1}M.
\label{form:etaBrLsW}
\end{equation}%
The latter equality will be used elsewhere.
\end{proof}

As a consequence of the construction of $\mathcal{U}_{\mathrm{Br}}$ we can
introduce an enveloping algebra functor $\mathcal{U}$ in the braided case.
We remark that in \cite[2.2]{GV-OnTheDuality} such a functor is just assumed
to exist and the functor $\mathcal{L}:\mathrm{Alg}_{\mathcal{M}}\rightarrow
\mathrm{Lie}_{\mathcal{M}}$ in the following result is also considered.

\begin{theorem}
\label{teo:U}Let $\mathcal{M}$ be an abelian symmetric monoidal category
with denumerable coproducts. Assume that the tensor functors are right exact
and preserve denumerable coproducts. There are unique functors $\mathcal{U}$
and $\mathcal{L}$ such that the following diagrams commute.%
\begin{equation}
\xymatrixrowsep{15pt}\xymatrixcolsep{15pt} \xymatrix{\Alg_M
\ar[rr]^{J_\Alg^s}&& \BrAlg_M^s\\ \Lie_M
\ar[u]^{\mathcal{U}}\ar[rr]^{J_\Lie^s}&& \BrLie_M^s
\ar[u]_{\mathcal{U}_\Br^s}} \qquad \xymatrix{\Alg_M
\ar[d]_{\mathcal{L}}\ar[rr]^{J_\Alg^s}&&
\BrAlg_M^s\ar[d]^{\mathcal{L}_\Br^s}\\ \Lie_M \ar[rr]^{J_\Lie^s}&&
\BrLie_M^s }  \label{diag:JsU}
\end{equation}%
Moreover $\left( \mathcal{U},\mathcal{L}\right) $ is an adjunction with unit
$\eta _{\mathrm{L}}:\mathrm{Id}_{\mathrm{Lie}_{\mathcal{M}}}\rightarrow
\mathcal{LU}$ and counit $\epsilon _{\mathrm{L}}:\mathcal{UL}\rightarrow
\mathrm{Id}_{\mathrm{Alg}_{\mathcal{M}}}$ defined by
\begin{equation}
J_{\mathrm{Alg}}^{s}\epsilon _{\mathrm{L}}=\epsilon _{\mathrm{BrL}}^{s}J_{%
\mathrm{Alg}}^{s}\qquad \text{and}\qquad J_{\mathrm{Lie}}^{s}\eta _{\mathrm{L%
}}=\eta _{\mathrm{BrL}}^{s}J_{\mathrm{Lie}}^{s},  \label{Form:EpsEtaL}
\end{equation}%
and $\left( J_{\mathrm{Alg}}^{s},J_{\mathrm{Lie}}^{s}\right) :\left(
\mathcal{U},\mathcal{L}\right) \rightarrow \left( \mathcal{U}_{\mathrm{Br}%
}^{s},\mathcal{L}_{\mathrm{Br}}^{s}\right) $ is a commutation datum with
canonical transformation given by the identity. The functors $\mathcal{U}$
and $\mathcal{L}$ can be described explicitly by $\mathcal{U}:=H_{\mathrm{Alg%
}}\mathcal{U}_{\mathrm{Br}}J_{\mathrm{Lie}}$ while $\mathcal{L}:\mathrm{Alg}%
_{\mathcal{M}}\rightarrow \mathrm{Lie}_{\mathcal{M}}$ acts as the identity
on morphisms and is defined on objects by $\mathcal{L}\left(
B,m_{B},u_{B}\right) :=\left( B,\left[ -\right] _{B}\right) ,$ where $\left[
-\right] _{B}:=m_{B}\circ \left( \mathrm{Id}_{B\otimes B}-c_{B,B}\right) $.
\end{theorem}

\begin{proof}
The existence and uniqueness of $\mathcal{U}$ and $\mathcal{L}$ as in the
statement follows by Lemma \ref{lem:LiftAdj}. It remains to prove the last
sentence. The equality $\mathcal{U}=H_{\mathrm{Alg}}\mathcal{U}_{\mathrm{Br}%
}J_{\mathrm{Lie}}$ follows by (\ref{diag:JsU}), (\ref{diag:Ubrs}) and (\ref%
{diag:JLies}). For $\left( B,m_{B},u_{B}\right) \in \mathrm{Alg}_{\mathcal{M}%
}$, by the foregoing, we have%
\begin{equation*}
J_{\mathrm{Lie}}^{s}\mathcal{L}\left( B,m_{B},u_{B}\right) \overset{(\ref%
{diag:JsU})}{=}\mathcal{L}_{\mathrm{Br}}^{s}J_{\mathrm{Alg}}^{s}\left(
B,m_{B},u_{B}\right) =\left( B,\left[ -\right] _{B},c_{B,B}\right)
\end{equation*}%
so that $\mathcal{L}\left( B,m_{B},u_{B}\right) =\left( B,\left[ -\right]
_{B}\right) .$ Since $J_{\mathrm{Lie}}^{s},\mathcal{L}_{\mathrm{Br}}^{s}$
and $J_{\mathrm{Alg}}^{s}$ act as the identity on morphisms so does $%
\mathcal{L}$.
\end{proof}

\begin{proposition}
\label{pro:PcalsBr}Let $\mathcal{M}$ be an abelian monoidal category with
denumerable coproducts. Assume that the tensor functors are right exact and
preserve denumerable coproducts. Then the functor $\overline{\mathcal{U}}_{%
\mathrm{Br}}^{s}:\mathrm{BrLie}_{\mathcal{M}}^{s}\rightarrow \mathrm{BrBialg}%
_{\mathcal{M}}^{s}$ has a right adjoint $\mathcal{P}_{\mathrm{Br}}^{s}:%
\mathrm{BrBialg}_{\mathcal{M}}^{s}\rightarrow \mathrm{BrLie}_{\mathcal{M}%
}^{s}$ such that the following diagram commutes
\begin{equation}
\xymatrixrowsep{15pt}\xymatrixcolsep{15pt} \xymatrix{&\BrBialg_\M^s
\ar[dl]_{\mathcal{P}_\Br^s}\ar[dr]^{P_\Br^s}\\ \BrLie_\M^s
\ar[rr]^{H_\BrLie^s}&& \Br_\M^s }  \label{diag:Pcal}
\end{equation}%
and the natural transformation $\xi :P_{\mathrm{Br}}^{s}\rightarrow \Omega _{%
\mathrm{Br}}^{s}\mho _{\mathrm{Br}}^{s}$ induces a natural transformation $%
\xi :\mathcal{P}_{\mathrm{Br}}^{s}\rightarrow \mathcal{L}_{\mathrm{Br}%
}^{s}\mho _{\mathrm{Br}}^{s}$ such that $H_{\mathrm{BrLie}}^{s}\xi =\xi $.
The unit $\overline{\eta }_{\mathrm{BrL}}^{s}:\mathrm{Id}_{\mathrm{BrLie}_{%
\mathcal{M}}^{s}}\rightarrow \mathcal{P}_{\mathrm{Br}}^{s}\overline{\mathcal{%
U}}_{\mathrm{Br}}^{s}$ and the counit $\overline{\epsilon }_{\mathrm{BrL}%
}^{s}:\overline{\mathcal{U}}_{\mathrm{Br}}^{s}\mathcal{P}_{\mathrm{Br}%
}^{s}\rightarrow \mathrm{Id}_{\mathrm{BrBialg}_{\mathcal{M}}^{s}}$ of the
adjunction satisfy
\begin{equation}
\xi \overline{\mathcal{U}}_{\mathrm{Br}}^{s}\circ \overline{\eta }_{\mathrm{%
BrL}}^{s}=\eta _{\mathrm{BrL}}^{s}\qquad \text{and}\qquad \epsilon _{\mathrm{%
BrL}}^{s}\mho _{\mathrm{Br}}^{s}\circ \mathcal{U}_{\mathrm{Br}}^{s}\xi =\mho
_{\mathrm{Br}}^{s}\overline{\epsilon }_{\mathrm{BrL}}^{s}.
\label{form:EtaEpsBarBrLs}
\end{equation}
\end{proposition}

\begin{proof}
Let $\mathbb{B}:=\left( B,m_{B},u_{B},\Delta _{B},\varepsilon
_{B},c_{B}\right) \in \mathrm{BrBialg}_{\mathcal{M}}^{s}$. Write $P_{\mathrm{%
Br}}^{s}\mathbb{B}=\left( P,c_{P}\right) .$ By \cite[Proposition 6.3(i)]%
{GV-LieMon}, there is a morphism $\left[ -\right] _{P}:=P\otimes
P\rightarrow P$ such that $\mathcal{P}_{\mathrm{Br}}^{s}\mathbb{B}:=\left(
P,c_{P},\left[ -\right] _{P}\right) \in \mathrm{BrLie}_{\mathcal{M}}^{s}$
and $\xi \mathbb{B}:\left( P,c_{P},\left[ -\right] _{P}\right) \rightarrow
\left( B,c_{B},\left[ -\right] _{B}\right) $ is a morphism in $\mathrm{BrLie}%
_{\mathcal{M}}^{s}$ where $\left[ -\right] _{B}:=m_{B}\circ \left( \mathrm{Id%
}_{B\otimes B}-c_{B}\right) .$ Clearly $\left[ -\right] _{P}$ is uniquely
determined by the compatibility with $\xi \mathbb{B}$. In this way we get a
functor $\mathcal{P}_{\mathrm{Br}}^{s}:\mathrm{BrBialg}_{\mathcal{M}%
}^{s}\rightarrow \mathrm{BrLie}_{\mathcal{M}}^{s}$ which acts as $P_{\mathrm{%
Br}}^{s}$ on morphisms. Let us check that there is a unique morphism $%
\overline{\eta }_{\mathrm{BrL}}^{s}:\mathrm{Id}_{\mathrm{BrLie}_{\mathcal{M}%
}^{s}}\rightarrow \mathcal{P}_{\mathrm{Br}}^{s}$ $\overline{\mathcal{U}}_{%
\mathrm{Br}}^{s}$ such that $\xi \overline{\mathcal{U}}_{\mathrm{Br}%
}^{s}\circ \overline{\eta }_{\mathrm{BrL}}^{s}=\eta _{\mathrm{BrL}}^{s}.$
Let $\left( M,c,\left[ -\right] \right) \in \mathrm{BrLie}_{\mathcal{M}}^{s}$%
, set $\left( R,m_{R},u_{R},\Delta _{R},\varepsilon _{R},c_{R}\right) :=%
\overline{\mathcal{U}}_{\mathrm{Br}}^{s}\left( M,c,\left[ -\right] \right) $
and set also $\left( A,m_{A},u_{A},\Delta _{A},\varepsilon _{A},c_{A}\right)
:=\overline{T}_{\mathrm{Br}}^{s}\left( M,c\right) $. Using that $p_{R}$ is
comultiplicative, the equality (\ref{form:TBrdelta}), unitality of $p_{R}$
and the naturality of the unit constraints, one easily checks that
\begin{equation*}
\nu :=H\mathbb{I}_{\mathrm{Br}}^{s}H_{\mathrm{BrLie}}^{s}\eta _{\mathrm{BrL}%
}^{s}\left( M,c,\left[ -\right] \right) \overset{(\ref{form:etaBrLsW})}{=}%
p_{R}\circ \alpha _{1}M:M\rightarrow R
\end{equation*}%
is equalized by the fork in (\ref{diag:prim}). Hence $v$ induces a morphism $%
v^{\prime }:M\rightarrow P\left( \overline{\mathcal{U}}_{\mathrm{Br}%
}^{s}\left( M,c,\left[ -\right] \right) \right) =:P$ such that $\xi
\overline{\mathcal{U}}_{\mathrm{Br}}^{s}\left( M,c,\left[ -\right] \right)
\circ v^{\prime }=\nu .$ One easily proves that $v^{\prime }$ defines a
natural transformation $\overline{\eta }_{\mathrm{BrL}}^{s}:\mathrm{Id}_{%
\mathrm{BrLie}_{\mathcal{M}}^{s}}\rightarrow \mathcal{P}_{\mathrm{Br}}^{s}$ $%
\overline{\mathcal{U}}_{\mathrm{Br}}^{s}$ such that $\xi \overline{\mathcal{U%
}}_{\mathrm{Br}}^{s}\circ \overline{\eta }_{\mathrm{BrL}}^{s}=\eta _{\mathrm{%
BrL}}^{s}.$ Let us check there is a natural transformation $\overline{%
\epsilon }_{\mathrm{BrL}}^{s}:\overline{\mathcal{U}}_{\mathrm{Br}}^{s}%
\mathcal{P}_{\mathrm{Br}}^{s}\rightarrow \mathrm{Id}_{\mathrm{BrBialg}_{%
\mathcal{M}}^{s}}$ such that $\epsilon _{\mathrm{BrL}}^{s}\mho _{\mathrm{Br}%
}^{s}\circ \mathcal{U}_{\mathrm{Br}}^{s}\xi =\mho _{\mathrm{Br}}^{s}%
\overline{\epsilon }_{\mathrm{BrL}}^{s}.$

Let $\mathbb{B}:=\left( B,m_{B},u_{B},\Delta _{B},\varepsilon
_{B},c_{B}\right) \in \mathrm{BrBialg}_{\mathcal{M}}^{s}$ and consider
\begin{equation*}
\gamma :=H\Omega _{\mathrm{Br}}\mathbb{I}_{\mathrm{BrAlg}}^{s}\left(
\epsilon _{\mathrm{BrL}}^{s}\mho _{\mathrm{Br}}^{s}\mathbb{B}\circ \mathcal{U%
}_{\mathrm{Br}}^{s}\xi \mathbb{B}\right) :R\rightarrow B
\end{equation*}%
where $\left( R,m_{R},u_{R},\Delta _{R},\varepsilon _{R},c_{R}\right) :=%
\overline{\mathcal{U}}_{\mathrm{Br}}^{s}\mathcal{P}_{\mathrm{Br}}^{s}\mathbb{%
B}.$ By definition $\gamma $ is a morphism of braided algebras and a direct
computation shows that $\gamma \circ p_{R}=H\Omega _{\mathrm{Br}}\mho _{%
\mathrm{Br}}\overline{\epsilon }_{\mathrm{Br}}\mathbb{I}_{\mathrm{BrBialg}%
}^{s}\mathbb{B}$, using the equality $p_{R}=H\Omega _{\mathrm{Br}}I_{\mathrm{%
BrAlg}}^{s}p^{s}P_{\mathrm{Br}}^{s}B$ and the equalities (\ref%
{form:eps-etaBrL}), (\ref{form:TbrStricts}), (\ref{diag:Tbrs}), (\ref%
{diag:OmRibs}), (\ref{form:BarEps}). Since $\overline{\epsilon }_{\mathrm{Br}%
}\mathbb{I}_{\mathrm{BrBialg}}^{s}\mathbb{B}$ is a morphism of braided
bialgebras and $p_{R}$ is an epimorphism and a morphism of braided
bialgebras, it is straightforward to prove that also $\gamma $ is. Hence
there is a unique morphism $\overline{\epsilon }_{\mathrm{BrL}}^{s}\mathbb{B}%
:\overline{\mathcal{U}}_{\mathrm{Br}}^{s}\mathcal{P}_{\mathrm{Br}}^{s}%
\mathbb{B}\rightarrow \mathbb{B}$ such that $H\Omega _{\mathrm{Br}}\mathbb{I}%
_{\mathrm{BrAlg}}^{s}\overline{\epsilon }_{\mathrm{BrL}}^{s}\mathbb{B}%
=\gamma .$ From the definition of $\gamma $ and the fact that $H\Omega _{%
\mathrm{Br}}\mathbb{I}_{\mathrm{BrAlg}}^{s}$ is faithful, we deduce $%
\epsilon _{\mathrm{BrL}}^{s}\mho _{\mathrm{Br}}^{s}\mathbb{B}\circ \mathcal{U%
}_{\mathrm{Br}}^{s}\xi \mathbb{B}=\mho _{\mathrm{Br}}^{s}\overline{\epsilon }%
_{\mathrm{BrL}}^{s}\mathbb{B}$. The naturality of the left-hand side of the
latter equality and the faithfulness of $\mho _{\mathrm{Br}}^{s}$ yield the
naturality of $\overline{\epsilon }_{\mathrm{BrL}}^{s}\mathbb{B}$. One
easily checks that the $\overline{\eta }_{\mathrm{BrL}}^{s}$ and $\overline{%
\epsilon }_{\mathrm{BrL}}^{s}$ make $( \overline{\mathcal{U}}_{\mathrm{Br}%
}^{s},\mathcal{P}_{\mathrm{Br}}^{s}) $ an adjunction.
\end{proof}

Next aim is to prove that, in the symmetric case, the functor $\mathcal{U}$
factors through a functor $\overline{\mathcal{U}}:\mathrm{Lie}_{\mathcal{M}%
}\rightarrow \mathrm{Bialg}_{\mathcal{M}}$ such that $\mho \circ \overline{%
\mathcal{U}}=\mathcal{U}$.

\begin{theorem}
\label{teo:Env}Let $\mathcal{M}$ an abelian symmetric monoidal category with
denumerable coproducts. Assume that the tensor functors are right exact and
preserve denumerable coproducts. Then there are unique functors $\overline{%
\mathcal{U}}$ and $\mathcal{P}$ such that the following diagrams commute%
\begin{equation}
\xymatrixrowsep{15pt}\xymatrixcolsep{15pt} \xymatrix{\Bialg_\M
\ar[rr]^{J_\Bialg^s}&& \BrBialg_\M^s\\ \Lie_\M
\ar[u]^{\overline{\U}}\ar[rr]^{J_\Lie^s}&& \BrLie_\M^s
\ar[u]_{\overline{\U}_\Br^s}} \quad \xymatrix{\Bialg_\M
\ar[d]_{\mathcal{P}}\ar[rr]^{J_\Bialg^s}&&
\BrBialg_\M^s\ar[d]^{\mathcal{P}_\Br^s}\\ \Lie_\M \ar[rr]^{J_\Lie^s}&&
\BrLie_\M^s } \quad \xymatrix{\Lie_\M \ar[dr]_{\U}\ar[rr]^{\overline{\U}}&&
\Bialg_\M\ar[dl]^{\mho}\\ &\Alg_\M }  \label{diag:Ubar}
\end{equation}%
where $\mathcal{U}$ is the functor of Theorem \ref{teo:U}. Moreover $\left(
\overline{\mathcal{U}},\mathcal{P}\right) $ is an adjunction with unit $%
\overline{\eta }_{\mathrm{L}}:\mathrm{Id}_{\mathrm{Lie}_{\mathcal{M}%
}}\rightarrow \mathcal{P}\overline{\mathcal{U}}$ and counit $\overline{%
\epsilon }_{\mathrm{L}}:\overline{\mathcal{U}}\mathcal{P}\rightarrow \mathrm{%
Id}_{\mathrm{Bialg}_{\mathcal{M}}}$ uniquely determined by
\begin{equation}
J_{\mathrm{Lie}}^{s}\overline{\eta }_{\mathrm{L}}=\overline{\eta }_{\mathrm{%
BrL}}^{s}J_{\mathrm{Lie}}^{s}\qquad \text{and}\qquad J_{\mathrm{Bialg}}^{s}%
\overline{\epsilon }_{\mathrm{L}}=\overline{\epsilon }_{\mathrm{BrL}}^{s}J_{%
\mathrm{Bialg}}^{s},  \label{Form:EpsEtaLbar}
\end{equation}%
and $( J_{\mathrm{Bialg}}^{s},J_{\mathrm{Lie}}^{s}) :( \overline{\mathcal{U}}%
,\mathcal{P}) \rightarrow ( \overline{\mathcal{U}}_{\mathrm{Br}}^{s},%
\mathcal{P}_{\mathrm{Br}}^{s}) $ is a commutation datum with canonical
transformation given by the identity. Furthermore there is a natural
transformation $\overline{p}:\overline{T}H_{\mathrm{Lie}}\rightarrow
\overline{\mathcal{U}}$ such that%
\begin{equation}
\overline{p}^{s}J_{\mathrm{Lie}}^{s}=J_{\mathrm{Bialg}}^{s}\overline{p}%
\qquad \text{and}\qquad \mho _{\mathrm{Br}}J_{\mathrm{Bialg}}\overline{p}%
=pJ_{\mathrm{Lie}}  \label{eq:PBar&pBars}
\end{equation}%
where $\overline{p}^{s}:\overline{T}_{\mathrm{Br}}^{s}H_{\mathrm{BrLie}%
}^{s}\rightarrow \overline{\mathcal{U}}_{\mathrm{Br}}^{s}$ is the natural
transformation of Theorem \ref{Teo:UbarBr} and $p:T_{\mathrm{Br}}H_{\mathrm{%
BrLie}}\rightarrow \mathcal{U}_{\mathrm{Br}}$ is the natural transformation
of Proposition \ref{pro:Ubr}. The natural transformation $\xi :\mathcal{P}_{%
\mathrm{Br}}^{s}\rightarrow \mathcal{L}_{\mathrm{Br}}^{s}\mho _{\mathrm{Br}%
}^{s}$ induces a natural transformation $\xi :\mathcal{P}\rightarrow
\mathcal{L}\mho $ such that $J_{\mathrm{Lie}}^{s}\xi =\xi J_{\mathrm{Bialg}%
}^{s}$.
\end{theorem}

\begin{proof}
The first part is a consequence of Lemma \ref{lem:LiftAdj}. The
commutativity of the third diagram of (\ref{diag:Ubar}) follows by (\ref%
{diag:JsAlgOrib}), (\ref{diag:Ubar}), (\ref{diag:UbrsBar}) and (\ref%
{diag:JsU}). By Lemma \ref{lem:Cappuccio}, there is a natural transformation
$\overline{p}:=\widehat{\overline{p}^{s}J_{\mathrm{Lie}}^{s}}:\overline{T}H_{%
\mathrm{Lie}}\rightarrow \overline{\mathcal{U}}$ such that $J_{\mathrm{Bialg}%
}^{s}\overline{p}=\overline{p}^{s}J_{\mathrm{Lie}}^{s}.$ Using (\ref{eq:Js})
(\ref{eq:PBar&pBars}), (\ref{eq:P&pBars}) and (\ref{diag:JLies}) we get $%
\mho _{\mathrm{Br}}J_{\mathrm{Bialg}}\overline{p}\overset{\left( \ref%
{diag:JLies}\right) }{=}pJ_{\mathrm{Lie}}.$ By Lemma \ref{lem:Cappuccio},
there is a natural transformation $\xi :=\widehat{\xi J_{\mathrm{Bialg}}^{s}}%
:\mathcal{P}\rightarrow \mathcal{L}\mho $ such that $J_{\mathrm{Lie}}^{s}\xi
=\xi J_{\mathrm{Bialg}}^{s}$.
\end{proof}

\begin{remark}
\label{rem:UbarNature}By Lemma \ref{lem:Cappuccio}, there is a natural
transformation $\overline{q}:=\widehat{\overline{p}^{s}J_{\mathrm{Lie}}^{s}}:%
\overline{T}H_{\mathrm{Lie}}\rightarrow \overline{\mathcal{U}}$ such that
\begin{equation}
J_{\mathrm{Bialg}}^{s}\overline{q}=\overline{p}^{s}J_{\mathrm{Lie}}^{s}.
\label{form:qbar}
\end{equation}
Using (\ref{form:qbar}), (\ref{eq:P&pBars}) and (\ref{diag:JLies}) one
checks that $\mho \overline{q}=H_{\mathrm{Alg}}pJ_{\mathrm{Lie}}$ where $p$
is the morphism of Proposition \ref{pro:Ubr}. This means that for every $%
\left( M,\left[ -\right] \right) \in \mathrm{Lie}_{\mathcal{M}}$ the
morphism $\overline{q}\left( M,\left[ -\right] \right) $ is really induced
by the canonical projection $p_{R}:\Omega TM\rightarrow R:=\mathcal{U}_{%
\mathrm{Br}}^{s}J_{\mathrm{Lie}}\left( M,\left[ -\right] \right) $ defining
in this lemma the universal enveloping algebra. Summing up, as a bialgebra
in $\mathcal{M}$ we have that $\overline{\mathcal{U}}\left( M,\left[ -\right]
\right) $ is a quotient of $\overline{T}H_{\mathrm{Lie}}\left( M,\left[ -%
\right] \right) =\overline{T}M$ via $\overline{q}\left( M,\left[ -\right]
\right) $ and the underlying algebra structure is the original one
underlying $\mathcal{U}_{\mathrm{Br}}^{s}J_{\mathrm{Lie}}\left( M,\left[ -%
\right] \right) .$
\end{remark}

\section{Stationary monadic decomposition}

\begin{theorem}
\label{teo:LambdaBr}Let $\mathcal{M}$ be an abelian monoidal category with
denumerable coproducts. Assume that the tensor functors are exact and
preserve denumerable coproducts.
\begin{equation}
\xymatrixcolsep{1.5cm}\xymatrixrowsep{0.40cm}\xymatrix{\mathrm{BrBialg}_{%
\mathcal{M}}^{s}\ar@<.5ex>[dd]^{P_{\mathrm{Br}}^{s}}&&\mathrm{BrBialg}_{%
\mathcal{M}}^{s}\ar@<.5ex>[dd]^{(P_{\mathrm{Br}}^{s})_1}|(.30)\hole\ar[ll]_{%
\mathrm{Id}_{\mathrm{BrBialg}_{\mathcal{M}}^{s}}}&&\mathrm{BrBialg}_{%
\mathcal{M}}^{s}\ar@<.5ex>[dd]^{(P_{\mathrm{Br}}^{s})_2}\ar[ll]_{%
\mathrm{Id}_{\mathrm{BrBialg}_{\mathcal{M}}^{s}}}\ar[dl]|{\mathrm{Id}_{%
\mathrm{BrBialg}_{\mathcal{M}}^{s}}}\\
&&&\mathrm{BrBialg}_{\mathcal{M}}^{s}\ar@<.5ex>[dd]^(.30){\mathcal{P}_{%
\mathrm{Br}}^{s}}\ar[ulll]^(.70){\mathrm{Id}_{\mathrm{BrBialg}_{%
\mathcal{M}}^{s}}}\\
\mathrm{Br}_{\mathcal{M}}^{s}\ar@<.5ex>@{.>}[uu]^{\overline{T}_{%
\mathrm{Br}}^{s}}&&(\mathrm{Br}_{\mathcal{M}}^{s})_1\ar@<.5ex>@{.>}[uu]^{(%
\overline{T}_{\mathrm{Br}}^{s})_1}|(.70)\hole
\ar[ll]_{U_{0,1}}&&(\mathrm{Br}_{\mathcal{M}}^{s})_2
\ar@<.5ex>@{.>}[uu]^{(\overline{T}_{\mathrm{Br}}^{s})_2}
\ar[ll]_(.30){U_{1,2}}|\hole\ar[dl]^{\Lambda_\mathrm{Br}} \\
&&&\mathrm{BrLie}_{\mathcal{M}}^{s}\ar@<.5ex>@{.>}[uu]^(.70){\overline{%
\mathcal{U}}^s_{\mathrm{Br}}}\ar[ulll]^{H_{\mathrm{BrLie}}^{s}}}
\label{diag:LambdaBr}
\end{equation}
The functor $P_{\mathrm{Br}}^{s}$ is comparable so that we can use the
notation of Definition \ref{def:comparable}. There is a functor $\Lambda _{%
\mathrm{Br}}:\left( \mathrm{Br}_{\mathcal{M}}^{s}\right) _{2}\rightarrow
\mathrm{BrLie}_{\mathcal{M}}^{s}$ such that $\Lambda _{\mathrm{Br}}\circ
\left( P_{\mathrm{Br}}^{s}\right) _{2}=\mathcal{P}_{\mathrm{Br}}^{s}$ and $%
H_{\mathrm{BrLie}}^{s}\circ \Lambda _{\mathrm{Br}}=U_{0,2}.$ Moreover there
exists a natural transformation $\overline{\chi }_{\mathrm{Br}}^{s}:%
\overline{\mathcal{U}}_{\mathrm{Br}}^{s}\Lambda _{\mathrm{Br}}\rightarrow (%
\overline{T}_{\mathrm{Br}}^{s})_{1}U_{1,2}$ such that
\begin{equation}
\overline{\chi }_{\mathrm{Br}}^{s}\circ \overline{p}^{s}\Lambda _{\mathrm{Br}%
}=\pi _{1}^{s}U_{1,2}  \label{form:chibarBrs}
\end{equation}%
where $\overline{p}^{s}$ is the natural transformation of Theorem \ref%
{Teo:UbarBr} and $\pi _{1}^{s}:\overline{T}_{\mathrm{Br}}^{s}U_{0,1}%
\rightarrow (\overline{T}_{\mathrm{Br}}^{s})_{1}$ is the canonical natural
transformation defining $(\overline{T}_{\mathrm{Br}}^{s})_{1}.$

Assume $\overline{\eta }_{\mathrm{BrL}}^{s}\Lambda _{\mathrm{Br}}$ is an
isomorphism.

\begin{itemize}
\item[1)] The adjunction $\left( \overline{\mathcal{U}}_{\mathrm{Br}}^{s},%
\mathcal{P}_{\mathrm{Br}}^{s}\right) $ is idempotent.

\item[2)] The adjunction $\left( (\overline{T}_{\mathrm{Br}}^{s})_{1},\left(
P_{\mathrm{Br}}^{s}\right) _{1}\right) $ is idempotent, we can choose $(%
\overline{T}_{\mathrm{Br}}^{s})_{2}:=(\overline{T}_{\mathrm{Br}%
}^{s})_{1}U_{1,2},$ $\pi _{2}^{s}=\mathrm{Id}_{(\overline{T}_{\mathrm{Br}%
}^{s})_{2}}$ and $(\overline{T}_{\mathrm{Br}}^{s})_{2}$ is full and faithful
i.e. $\left( \overline{\eta }_{\mathrm{Br}}^{s}\right) _{2}$ is an
isomorphism.

\item[3)] The functor $P_{\mathrm{Br}}^{s}$ has a monadic decomposition of
monadic length at most two.

\item[4)] $(\mathrm{Id}_{\mathrm{BrBialg}_{\mathcal{M}}^{s}},\Lambda _{%
\mathrm{Br}}):((\overline{T}_{\mathrm{Br}}^{s})_{2},\left( P_{\mathrm{Br}%
}^{s}\right) _{2})\rightarrow (\overline{\mathcal{U}}_{\mathrm{Br}}^{s},%
\mathcal{P}_{\mathrm{Br}}^{s})$ is a commutation datum whose canonical
transformation is $\overline{\chi }_{\mathrm{Br}}^{s}.$

\item[5)] The pair $\left( \left( P_{\mathrm{Br}}^{s}\right) _{2}\overline{%
\mathcal{U}}_{\mathrm{Br}}^{s},\Lambda _{\mathrm{Br}}\right) $ is an
adjunction with unit $\overline{\eta }_{\mathrm{BrL}}^{s}$ and counit $%
\left( \overline{\eta }_{\mathrm{Br}}^{s}\right) _{2}^{-1}\circ \left( P_{%
\mathrm{Br}}^{s}\right) _{2}\overline{\chi }_{\mathrm{Br}}^{s}$ so that $%
\Lambda _{\mathrm{Br}}$ is full and faithful. Hence $\overline{\eta }_{%
\mathrm{BrL}}^{s}$ is an isomorphism if and only if $\left( \left( P_{%
\mathrm{Br}}^{s}\right) _{2}\overline{\mathcal{U}}_{\mathrm{Br}}^{s},\Lambda
_{\mathrm{Br}}\right) $ is an equivalence of categories. In this case $%
\left( (\overline{T}_{\mathrm{Br}}^{s})_{2},\left( P_{\mathrm{Br}%
}^{s}\right) _{2}\right) $ identifies with $\left( \overline{\mathcal{U}}_{%
\mathrm{Br}}^{s},\mathcal{P}_{\mathrm{Br}}^{s}\right) $ via $\Lambda _{%
\mathrm{Br}}$.
\end{itemize}
\end{theorem}

\begin{proof}
By \ref{cl:TbarStrict} we have an adjunction $\left( \overline{T}_{\mathrm{Br%
}}^{s},P_{\mathrm{Br}}^{s}\right) .$ By Proposition \ref{pro:BeckBr}, the
right adjoint functor $R=P_{\mathrm{Br}}^{s}$ is comparable and we can use
the notation of Definition \ref{def:comparable}.

Let $M_{2}=\left( M_{1},\mu _{1}\right) \in \left( \mathrm{Br}_{\mathcal{M}%
}^{s}\right) _{2}$. Then we can write $M_{1}=\left( M_{0},\mu _{0}\right)
\in \left( \mathrm{Br}_{\mathcal{M}}^{s}\right) _{1}\ $and $M_{0}=\left(
M,c\right) \in \mathrm{Br}_{\mathcal{M}}^{s}.$ Let $\theta _{\left(
M,c\right) }:=\theta _{\mathbb{I}_{\mathrm{Br}}^{s}\left( M,c\right)
}:M\otimes M\rightarrow \Omega T\left( M\right) $ be defined as in (\ref%
{def:theta}) and set $\mathbb{A}:=\left( A,m_{A},u_{A},\Delta
_{A},\varepsilon _{A},c_{A}\right) :=\overline{T}_{\mathrm{Br}}M_{0}=%
\overline{T}_{\mathrm{Br}}\left( M,c\right) $. Since $c^{2}=\mathrm{Id}%
_{M\otimes M}$ we have that $\theta _{\left( M,c\right) }$ fulfills (\ref%
{form:Un1}). Thus there is a unique morphism $\overline{\theta }_{\left(
M,c\right) }:=\overline{\theta }_{\mathbb{I}_{\mathrm{Br}}^{s}\left(
M,c\right) }:M\otimes M\rightarrow P\left( \overline{T}_{\mathrm{Br}}\left(
M,c\right) \right) $ such that%
\begin{equation}
\xi \mathbb{A}\circ \overline{\theta }_{\left( M,c\right) }=\theta _{\left(
M,c\right) }.  \label{form:zetaBar}
\end{equation}%
Set%
\begin{equation*}
\left[ -\right] :=H\mathbb{I}_{\mathrm{Br}}^{s}\mu _{0}\circ \overline{%
\theta }_{\left( M,c\right) }:M\otimes M\rightarrow M.
\end{equation*}%
Let us check that $\left( M,c,\left[ -\right] \right) \in \mathrm{BrLie}_{%
\mathcal{M}}^{s}$. Now $\mu _{1}\circ \left( \overline{\eta }_{\mathrm{Br}%
}^{s}\right) _{1}M_{1}=\mathrm{Id}_{M_{1}}$ so that $\left( \overline{\eta }%
_{\mathrm{Br}}^{s}\right) _{1}M_{1}$ is a split monomorphism. Set $\mathbb{S}%
:=\left( S,m_{S},u_{S},\Delta _{S},\varepsilon _{S},c_{S}\right) :=\left(
\overline{T}_{\mathrm{Br}}^{s}\right) _{1}M_{1}.$ Thus $H\mathbb{I}_{\mathrm{%
Br}}^{s}U_{0,1}\left( \overline{\eta }_{\mathrm{Br}}^{s}\right)
_{1}M_{1}:M\rightarrow P\left( \mathbb{S}\right) $ is a split monomorphism
too. Let $\pi _{1}^{s}:\overline{T}_{\mathrm{Br}}^{s}U_{0,1}\rightarrow
\left( \overline{T}_{\mathrm{Br}}^{s}\right) _{1}$ be the canonical natural
transformation defining $\left( \overline{T}_{\mathrm{Br}}^{s}\right) _{1}.$
By construction one has
\begin{equation}
P_{\mathrm{Br}}^{s}\pi _{1}^{s}\circ \overline{\eta }_{\mathrm{Br}%
}^{s}U_{0,1}=U_{0,1}\left( \overline{\eta }_{\mathrm{Br}}^{s}\right) _{1}.
\label{form:pi1etaBar}
\end{equation}

We have%
\begin{equation}
H\xi \overline{T}_{\mathrm{Br}}\mathbb{I}_{\mathrm{Br}}^{s}\circ H\mathbb{I}%
_{\mathrm{Br}}^{s}\overline{\eta }_{\mathrm{Br}}^{s}=H\left( \xi \overline{T}%
_{\mathrm{Br}}\mathbb{I}_{\mathrm{Br}}^{s}\circ \mathbb{I}_{\mathrm{Br}}^{s}%
\overline{\eta }_{\mathrm{Br}}^{s}\right) \overset{(\ref{form:barbrs})}{=}%
H\left( \xi \overline{T}_{\mathrm{Br}}\mathbb{I}_{\mathrm{Br}}^{s}\circ
\overline{\eta }_{\mathrm{Br}}\mathbb{I}_{\mathrm{Br}}^{s}\right) \overset{(%
\ref{form:BarEta})}{=}H\eta _{\mathrm{Br}}\mathbb{I}_{\mathrm{Br}}^{s}%
\overset{(\ref{form:TbrStrict})}{=}\eta H\mathbb{I}_{\mathrm{Br}}^{s}.
\label{form:nonso1}
\end{equation}%
In particular, we have%
\begin{eqnarray*}
\xi \mathbb{A}\circ H\mathbb{I}_{\mathrm{Br}}^{s}\overline{\eta }_{\mathrm{Br%
}}^{s}M_{0} &=&\xi \mathbb{A}\circ H\mathbb{I}_{\mathrm{Br}}^{s}\overline{%
\eta }_{\mathrm{Br}}^{s}M_{0}=H\xi \overline{T}_{\mathrm{Br}}\mathbb{I}_{%
\mathrm{Br}}^{s}M_{0}\circ H\mathbb{I}_{\mathrm{Br}}^{s}\overline{\eta }_{%
\mathrm{Br}}^{s}M_{0} \\
\overset{\left( \ref{form:nonso1}\right) }{=}\eta H\mathbb{I}_{\mathrm{Br}%
}^{s}M_{0} &=&\eta M\overset{\left( \ref{form:etaeps}\right) }{=}\alpha _{1}M
\end{eqnarray*}%
so that%
\begin{equation}
\xi \mathbb{A}\circ H\mathbb{I}_{\mathrm{Br}}^{s}\overline{\eta }_{\mathrm{Br%
}}^{s}M_{0}=\alpha _{1}M  \label{form:nonso2}
\end{equation}%
We compute%
\begin{gather*}
\xi \mathbb{A}\circ \left[ -\right] _{P\left( \mathbb{A}\right) }\circ
\left( H\mathbb{I}_{\mathrm{Br}}^{s}\overline{\eta }_{\mathrm{Br}%
}^{s}M_{0}\otimes H\mathbb{I}_{\mathrm{Br}}^{s}\overline{\eta }_{\mathrm{Br}%
}^{s}M_{0}\right) =\left[ -\right] _{A}\circ \left( \xi \mathbb{A}\otimes
\xi \mathbb{A}\right) \circ \left( H\mathbb{I}_{\mathrm{Br}}^{s}\overline{%
\eta }_{\mathrm{Br}}^{s}M_{0}\otimes H\mathbb{I}_{\mathrm{Br}}^{s}\overline{%
\eta }_{\mathrm{Br}}^{s}M_{0}\right) \\
=m_{A}\circ \left( \mathrm{Id}_{A\otimes A}-c_{A}\right) \circ \left( \xi
\mathbb{A}\otimes \xi \mathbb{A}\right) \circ \left( H\mathbb{I}_{\mathrm{Br}%
}^{s}\overline{\eta }_{\mathrm{Br}}^{s}M_{0}\otimes H\mathbb{I}_{\mathrm{Br}%
}^{s}\overline{\eta }_{\mathrm{Br}}^{s}M_{0}\right) \\
\overset{(\ref{form:nonso2})}{=}m_{A}\circ \left( \mathrm{Id}_{A\otimes
A}-c_{A}\right) \circ \left( \alpha _{1}M\otimes \alpha _{1}M\right)
=m_{A}\circ \left( \alpha _{1}M\otimes \alpha _{1}M\right) \circ \left(
\mathrm{Id}_{M\otimes M}-c_{A}^{1,1}\right) \\
\overset{(\ref{form:TVm})}{=}\alpha _{2}M\circ \left( \mathrm{Id}_{M\otimes
M}-c\right) \overset{(\ref{def:theta})}{=}\theta _{\left( M,c\right) }%
\overset{(\ref{form:zetaBar})}{=}\xi \mathbb{A}\circ \overline{\theta }%
_{\left( M,c\right) }.
\end{gather*}%
Since $\xi \mathbb{A}$ is a monomorphism we get
\begin{equation}
\overline{\theta }_{\left( M,c\right) }=\left[ -\right] _{P\left( \mathbb{A}%
\right) }\circ \left( H\mathbb{I}_{\mathrm{Br}}^{s}\overline{\eta }_{\mathrm{%
Br}}^{s}M_{0}\otimes H\mathbb{I}_{\mathrm{Br}}^{s}\overline{\eta }_{\mathrm{%
Br}}^{s}M_{0}\right) .  \label{form:nonso3}
\end{equation}%
Moreover since $\pi _{1}^{s}M_{1}:\mathbb{A}=\overline{T}_{\mathrm{Br}%
}^{s}U_{0,1}M_{1}\rightarrow \left( \overline{T}_{\mathrm{Br}}^{s}\right)
_{1}M_{1}=\mathbb{S}$ is a morphism in $\mathrm{BrBialg}_{\mathcal{M}}^{s}$,
we have that $H\mathbb{I}_{\mathrm{Br}}^{s}P_{\mathrm{Br}}^{s}\pi
_{1}^{s}M_{1}\overset{(\ref{diag:Pcal})}{=}H\mathbb{I}_{\mathrm{Br}}^{s}H_{%
\mathrm{BrLie}}^{s}\mathcal{P}_{\mathrm{Br}}^{s}\pi _{1}^{s}M_{1}$ commutes
with lie brackets i.e.
\begin{equation}
\left[ -\right] _{P\left( \mathbb{S}\right) }\circ \left( H\mathbb{I}_{%
\mathrm{Br}}^{s}P_{\mathrm{Br}}^{s}\pi _{1}^{s}M_{1}\otimes H\mathbb{I}_{%
\mathrm{Br}}^{s}P_{\mathrm{Br}}^{s}\pi _{1}^{s}M_{1}\right) =H\mathbb{I}_{%
\mathrm{Br}}^{s}P_{\mathrm{Br}}^{s}\pi _{1}^{s}M_{1}\circ \left[ -\right]
_{P\left( \mathbb{A}\right) }  \label{form:nonso4}
\end{equation}%
Hence we get%
\begin{eqnarray*}
&&\left[ -\right] _{P\left( \mathbb{S}\right) }\circ \left( H\mathbb{I}_{%
\mathrm{Br}}^{s}U_{0,1}\left( \overline{\eta }_{\mathrm{Br}}^{s}\right)
_{1}M_{1}\otimes H\mathbb{I}_{\mathrm{Br}}^{s}U_{0,1}\left( \overline{\eta }%
_{\mathrm{Br}}^{s}\right) _{1}M_{1}\right) \\
&\overset{(\ref{form:pi1etaBar})}{=}&\left[ -\right] _{P\left( \mathbb{S}%
\right) }\circ \left( H\mathbb{I}_{\mathrm{Br}}^{s}P_{\mathrm{Br}}^{s}\pi
_{1}^{s}M_{1}\otimes H\mathbb{I}_{\mathrm{Br}}^{s}P_{\mathrm{Br}}^{s}\pi
_{1}^{s}M_{1}\right) \circ \left( H\mathbb{I}_{\mathrm{Br}}^{s}\overline{%
\eta }_{\mathrm{Br}}^{s}M_{0}\otimes H\mathbb{I}_{\mathrm{Br}}^{s}\overline{%
\eta }_{\mathrm{Br}}^{s}M_{0}\right) \\
&\overset{(\ref{form:nonso4})}{=}&H\mathbb{I}_{\mathrm{Br}}^{s}P_{\mathrm{Br}%
}^{s}\pi _{1}^{s}M_{1}\circ \left[ -\right] _{P\left( \mathbb{A}\right)
}\circ \left( H\mathbb{I}_{\mathrm{Br}}^{s}\overline{\eta }_{\mathrm{Br}%
}^{s}M_{0}\otimes H\mathbb{I}_{\mathrm{Br}}^{s}\overline{\eta }_{\mathrm{Br}%
}^{s}M_{0}\right) \\
&\overset{(\ref{form:nonso3})}{=}&H\mathbb{I}_{\mathrm{Br}}^{s}P_{\mathrm{Br}%
}^{s}\pi _{1}^{s}M_{1}\circ \overline{\theta }_{\left( M,c\right) }\overset{%
(\ast )}{=}H\mathbb{I}_{\mathrm{Br}}^{s}P_{\mathrm{Br}}^{s}\pi
_{1}^{s}M_{1}\circ H\mathbb{I}_{\mathrm{Br}}^{s}\overline{\eta }_{\mathrm{Br}%
}^{s}U_{0,1}M_{1}\circ H\mathbb{I}_{\mathrm{Br}}^{s}\mu _{0}\circ \overline{%
\theta }_{\left( M,c\right) } \\
&\overset{(\ref{form:pi1etaBar})}{=}&H\mathbb{I}_{\mathrm{Br}%
}^{s}U_{0,1}\left( \overline{\eta }_{\mathrm{Br}}^{s}\right) _{1}M_{1}\circ %
\left[ -\right]
\end{eqnarray*}%
where in (*) we used that $P_{\mathrm{Br}}^{s}\pi _{1}^{s}M_{1}\circ
\overline{\eta }_{\mathrm{Br}}^{s}U_{0,1}M_{1}\circ \mu _{0}=P_{\mathrm{Br}%
}^{s}\pi _{1}^{s}M_{1}$ which follows from $\pi _{1}^{s}M_{1}\circ \overline{%
T}_{\mathrm{Br}}^{s}\mu _{0}=\pi _{1}^{s}M_{1}\circ \overline{\epsilon }_{%
\mathrm{Br}}^{s}\overline{T}_{\mathrm{Br}}^{s}M_{0}$ (true by definition of $%
\pi _{1}$) and \cite[Lemma 3.3]{AGM-MonadicLie1}. We have so proved%
\begin{equation}
\left[ -\right] _{P\left( \mathbb{S}\right) }\circ \left( H\mathbb{I}_{%
\mathrm{Br}}^{s}U_{0,1}\left( \overline{\eta }_{\mathrm{Br}}^{s}\right)
_{1}M_{1}\otimes H\mathbb{I}_{\mathrm{Br}}^{s}U_{0,1}\left( \overline{\eta }%
_{\mathrm{Br}}^{s}\right) _{1}M_{1}\right) =H\mathbb{I}_{\mathrm{Br}%
}^{s}U_{0,1}\left( \overline{\eta }_{\mathrm{Br}}^{s}\right) _{1}M_{1}\circ %
\left[ -\right] .  \label{form:Immerg1}
\end{equation}%
Using the fact that $H\mathbb{I}_{\mathrm{Br}}^{s}U_{0,1}\left( \overline{%
\eta }_{\mathrm{Br}}^{s}\right) _{1}M_{1}$ is a monomorphism in $\mathcal{M}$
and
\begin{equation*}
\left( P\left( \mathbb{S}\right) ,c_{P\left( \mathbb{S}\right) },\left[ -%
\right] _{P\left( \mathbb{S}\right) }\right) =\mathcal{P}_{\mathrm{Br}}^{s}%
\mathbb{S}\in \mathrm{BrLie}_{\mathcal{M}}^{s},
\end{equation*}
one easily checks that $\Lambda _{\mathrm{Br}}\left( M_{2}\right) :=\left(
M,c,\left[ -\right] \right) \in \mathrm{BrLie}_{\mathcal{M}}^{s}$ and that $H%
\mathbb{I}_{\mathrm{Br}}^{s}U_{0,1}\left( \overline{\eta }_{\mathrm{Br}%
}^{s}\right) _{1}M_{1}:M\rightarrow P\left( \mathbb{S}\right) $ is a
morphism in $\mathrm{BrLie}_{\mathcal{M}}^{s}$. Let $\nu :M_{2}\rightarrow
M_{2}^{\prime }$ be a morphism in $\left( \mathrm{Br}_{\mathcal{M}%
}^{s}\right) _{2}.$ It is clearly a morphism of braided objects. Since, by (%
\ref{diag:Pcal}), we have $H\mathbb{I}_{\mathrm{Br}}^{s}P_{\mathrm{Br}}^{s}=H%
\mathbb{I}_{\mathrm{Br}}^{s}H_{\mathrm{BrLie}}^{s}\mathcal{P}_{\mathrm{Br}%
}^{s}$, then $H\mathbb{I}_{\mathrm{Br}}^{s}P_{\mathrm{Br}}^{s}\overline{T}_{%
\mathrm{Br}}^{s}U_{0,2}\nu $ commutes with Lie brackets and hence
\begin{eqnarray*}
&&\overline{\theta }_{\left( M^{\prime },c^{\prime }\right) }\circ \left( H%
\mathbb{I}_{\mathrm{Br}}^{s}U_{0,2}\nu \otimes H\mathbb{I}_{\mathrm{Br}%
}^{s}U_{0,2}\nu \right) \\
&\overset{(\ref{form:nonso3})}{=}&\left[ -\right] _{P\left( \mathbb{A}%
^{\prime }\right) }\circ \left( H\mathbb{I}_{\mathrm{Br}}^{s}\overline{\eta }%
_{\mathrm{Br}}^{s}M_{0}^{\prime }\otimes H\mathbb{I}_{\mathrm{Br}}^{s}%
\overline{\eta }_{\mathrm{Br}}^{s}M_{0}^{\prime }\right) \circ \left( H%
\mathbb{I}_{\mathrm{Br}}^{s}U_{0,2}\nu \otimes H\mathbb{I}_{\mathrm{Br}%
}^{s}U_{0,2}\nu \right) \\
&=&\left[ -\right] _{P\left( \mathbb{A}^{\prime }\right) }\circ \left( H%
\mathbb{I}_{\mathrm{Br}}^{s}P_{\mathrm{Br}}^{s}\overline{T}_{\mathrm{Br}%
}^{s}U_{0,2}\nu \otimes H\mathbb{I}_{\mathrm{Br}}^{s}P_{\mathrm{Br}}^{s}%
\overline{T}_{\mathrm{Br}}^{s}U_{0,2}\nu \right) \circ \left( H\mathbb{I}_{%
\mathrm{Br}}^{s}\overline{\eta }_{\mathrm{Br}}^{s}M_{0}\otimes H\mathbb{I}_{%
\mathrm{Br}}^{s}\overline{\eta }_{\mathrm{Br}}^{s}M_{0}\right) \\
&=&H\mathbb{I}_{\mathrm{Br}}^{s}P_{\mathrm{Br}}^{s}\overline{T}_{\mathrm{Br}%
}^{s}U_{0,2}\nu \circ \left[ -\right] _{P\left( \mathbb{A}\right) }\circ
\left( H\mathbb{I}_{\mathrm{Br}}^{s}\overline{\eta }_{\mathrm{Br}%
}^{s}M_{0}\otimes H\mathbb{I}_{\mathrm{Br}}^{s}\overline{\eta }_{\mathrm{Br}%
}^{s}M_{0}\right) \\
&\overset{(\ref{form:nonso3})}{=}&H\mathbb{I}_{\mathrm{Br}}^{s}P_{\mathrm{Br}%
}^{s}\overline{T}_{\mathrm{Br}}^{s}U_{0,2}\nu \circ \overline{\theta }%
_{\left( M,c\right) }
\end{eqnarray*}

so that $\overline{\theta }_{\left( M^{\prime },c^{\prime }\right) }\circ
\left( H\mathbb{I}_{\mathrm{Br}}^{s}U_{0,2}\nu \otimes H\mathbb{I}_{\mathrm{%
Br}}^{s}U_{0,2}\nu \right) =HP_{\mathrm{Br}}\overline{T}_{\mathrm{Br}}%
\mathbb{I}_{\mathrm{Br}}^{s}U_{0,2}\nu \circ \overline{\theta }_{\left(
M,c\right) }$. Using the latter equality, (\ref{diag:PsTsBr}) and that $\nu $
is a morphism in $\left( \mathrm{Br}_{\mathcal{M}}^{s}\right) _{2}$ we
obtain that $\left[ -\right] ^{\prime }\circ \left( H\mathbb{I}_{\mathrm{Br}%
}^{s}U_{0,2}\nu \otimes H\mathbb{I}_{\mathrm{Br}}^{s}U_{0,2}\nu \right) =H%
\mathbb{I}_{\mathrm{Br}}^{s}U_{0,2}\nu \circ \left[ -\right] $. Thus $\nu $
induces a morphism $\Lambda _{\mathrm{Br}}v\in \mathrm{BrLie}_{\mathcal{M}%
}^{s}$. It is clear that this defines a functor $\Lambda _{\mathrm{Br}%
}:\left( \mathrm{Br}_{\mathcal{M}}^{s}\right) _{2}\rightarrow \mathrm{BrLie}%
_{\mathcal{M}}^{s}$ acting as the identity on morphisms. Let $\mathbb{B}%
:=\left( B,m_{B},u_{B},\Delta _{B},\varepsilon _{B},c_{B}\right) \in \mathrm{%
BrBialg}_{\mathcal{M}}^{s}$. Set $M_{2}:=\left( P_{\mathrm{Br}}^{s}\right)
_{2}\mathbb{B}$. Then%
\begin{eqnarray*}
\left( M_{1},\mu _{1}\right) &:&=M_{2}=\left( \left( P_{\mathrm{Br}%
}^{s}\right) _{1}\mathbb{B},\left( P_{\mathrm{Br}}^{s}\right) _{1}\left(
\overline{\varepsilon }_{\mathrm{Br}}^{s}\right) _{1}\mathbb{B}\right) , \\
\left( M_{0},\mu _{0}\right) &:&=M_{1}=\left( P_{\mathrm{Br}}^{s}\right) _{1}%
\mathbb{B=}\left( P_{\mathrm{Br}}^{s}\mathbb{B},P_{\mathrm{Br}}^{s}\overline{%
\varepsilon }_{\mathrm{Br}}^{s}\mathbb{B}\right) , \\
\left( M,c\right) &:&=M_{0}=P_{\mathrm{Br}}^{s}\mathbb{B}
\end{eqnarray*}%
The bracket for this specific $M$ is
\begin{equation*}
\left[ -\right] :=H\mathbb{I}_{\mathrm{Br}}^{s}\mu _{0}\circ \overline{%
\theta }_{\left( M,c\right) }=H\mathbb{I}_{\mathrm{Br}}^{s}P_{\mathrm{Br}%
}^{s}\overline{\varepsilon }_{\mathrm{Br}}^{s}\mathbb{B}\circ \overline{%
\theta }_{P_{\mathrm{Br}}^{s}\mathbb{B}}.
\end{equation*}

It is straightforward to prove that $\xi \mathbb{B}\circ \left[ -\right] =%
\left[ -\right] _{B}\circ \left( \xi \mathbb{B\otimes }\xi \mathbb{B}\right)
=\xi \mathbb{B}\circ \left[ -\right] _{P(\mathbb{B})}$ so that $\left[ -%
\right] =\left[ -\right] _{P(\mathbb{B})}$ and hence%
\begin{equation*}
\Lambda _{\mathrm{Br}}\left( P_{\mathrm{Br}}^{s}\right) _{2}\mathbb{B}%
=\left( P_{\mathrm{Br}}^{s}\mathbb{B},\left[ -\right] _{P(\mathbb{B}%
)}\right) =\mathcal{P}_{\mathrm{Br}}^{s}\mathbb{B}.
\end{equation*}%
It is clear that the functors $\Lambda _{\mathrm{Br}}\left( P_{\mathrm{Br}%
}^{s}\right) _{2}$ and $\mathcal{P}_{\mathrm{Br}}^{s}$ coincide also on
morphisms so that we obtain $\Lambda _{\mathrm{Br}}\circ \left( P_{\mathrm{Br%
}}^{s}\right) _{2}=\mathcal{P}_{\mathrm{Br}}^{s}$. Let $M_{2}\in \left(
\mathrm{Br}_{\mathcal{M}}^{s}\right) _{2}.$ Then
\begin{equation*}
H_{\mathrm{BrLie}}^{s}\Lambda _{\mathrm{Br}}M_{2}=H_{\mathrm{BrLie}%
}^{s}\left( M,c,\left[ -\right] \right) =\left( M,c\right) =U_{0,2}M_{2}.
\end{equation*}%
Since $H_{\mathrm{BrLie}}^{s},\Lambda _{\mathrm{Br}}$ and $U_{0,2}$ act as
the identity on morphisms, we get $H_{\mathrm{BrLie}}^{s}\circ \Lambda _{%
\mathrm{Br}}=U_{0,2}.$

In view (\ref{form:nonso2}), naturality of $\xi ,$ the equality $(\ast )$
used above and (\ref{form:zetaBar}) we obtain%
\begin{equation*}
H\mathbb{I}_{\mathrm{Br}}^{s}\Omega _{\mathrm{Br}}^{s}\mho _{\mathrm{Br}%
}^{s}\pi _{1}^{s}M_{1}\circ \alpha _{1}M\circ \left[ -\right] =H\mathbb{I}_{%
\mathrm{Br}}^{s}\Omega _{\mathrm{Br}}^{s}\mho _{\mathrm{Br}}^{s}\pi
_{1}^{s}M_{1}\circ \theta _{\left( M,c\right) }.
\end{equation*}
Thus we get $H\mathbb{I}_{\mathrm{Br}}^{s}\Omega _{\mathrm{Br}}^{s}\mho _{%
\mathrm{Br}}^{s}\pi _{1}^{s}M_{1}\circ f_{\Lambda _{\mathrm{Br}}M_{2}}=0$.

Since $\pi _{1}^{s}M_{1}$ is an algebra map, we have $H\mathbb{I}_{\mathrm{Br%
}}^{s}\Omega _{\mathrm{Br}}^{s}\mho _{\mathrm{Br}}^{s}\pi _{1}^{s}M_{1}\circ
i_{f_{\Lambda _{\mathrm{Br}}M_{2}}}=0$ so that, by construction of $\mathcal{%
U}_{\mathrm{Br}}^{s}$ there is an braided algebra morphism $\chi _{\mathrm{Br%
}}^{s}M_{2}:\mathcal{U}_{\mathrm{Br}}^{s}\Lambda _{\mathrm{Br}%
}M_{2}\rightarrow \mho _{\mathrm{Br}}^{s}\left( \overline{T}_{\mathrm{Br}%
}^{s}\right) _{1}M_{1}$ such
\begin{equation*}
\chi _{\mathrm{Br}}^{s}M_{2}\circ p^{s}\Lambda _{\mathrm{Br}}M_{2}=\mho _{%
\mathrm{Br}}^{s}\pi _{1}^{s}M_{1}=\mho _{\mathrm{Br}}^{s}\pi
_{1}^{s}U_{1,2}M_{2}.
\end{equation*}%
By naturality of the other terms we obtain that also $\chi _{\mathrm{Br}%
}^{s}M_{2}$ is natural in $M_{2}$ so that we get%
\begin{equation*}
\chi _{\mathrm{Br}}^{s}\circ p^{s}\Lambda _{\mathrm{Br}}=\mho _{\mathrm{Br}%
}^{s}\pi _{1}^{s}U_{1,2}.
\end{equation*}%
By (\ref{eq:P&pBars}) we get $\chi _{\mathrm{Br}}^{s}\circ \mho _{\mathrm{Br}%
}^{s}\overline{p}^{s}\Lambda _{\mathrm{Br}}=\mho _{\mathrm{Br}}^{s}\pi
_{1}^{s}U_{1,2}.$ Since both $\overline{p}^{s}\Lambda _{\mathrm{Br}}$ and $%
\pi _{1}^{s}U_{1,2}$ are morphism of braided bialgebras and the underlying
morphism in $\mathcal{M}$ of $\overline{p}^{s}$ is $p$ which is an
epimorphism, one gets that $\chi _{\mathrm{Br}}^{s}$ is a morphism of
braided bialgebras too that will be denoted by $\overline{\chi }_{\mathrm{Br}%
}^{s}:\overline{\mathcal{U}}_{\mathrm{Br}}^{s}\Lambda _{\mathrm{Br}%
}\rightarrow \left( \overline{T}_{\mathrm{Br}}^{s}\right) _{1}U_{1,2}$. Thus
$\mho _{\mathrm{Br}}^{s}\overline{\chi }_{\mathrm{Br}}^{s}=\chi _{\mathrm{Br}%
}^{s}$ and hence
\begin{equation}
\overline{\chi }_{\mathrm{Br}}^{s}\circ \overline{p}^{s}\Lambda _{\mathrm{Br}%
}=\pi _{1}^{s}U_{1,2}.  \label{form:chiBar}
\end{equation}
A direct computation, shows that $\mathbb{I}_{\mathrm{Br}}^{s}\eta _{\mathrm{%
Br}}^{s}=\mathbb{I}_{\mathrm{Br}}^{s}\xi \overline{T}_{\mathrm{Br}}^{s}\circ
\mathbb{I}_{\mathrm{Br}}^{s}\overline{\eta }_{\mathrm{Br}}^{s}$ and hence
\begin{equation*}
\eta _{\mathrm{Br}}^{s}=\xi \overline{T}_{\mathrm{Br}}^{s}\circ \overline{%
\eta }_{\mathrm{Br}}^{s}.
\end{equation*}
Thus, using (\ref{diag:Pcal}), naturality of $\xi ,$ (\ref%
{form:EtaEpsBarBrLs}), (\ref{form:eps-etaBrL}), (\ref{eq:P&pBars}), (\ref%
{form:chiBar}), again naturality of $\xi $ and (\ref{form:pi1etaBar}) in the
given order, we get
\begin{equation*}
\xi \left( \overline{T}_{\mathrm{Br}}^{s}\right) _{1}U_{1,2}\circ H_{\mathrm{%
BrLie}}^{s}\left( \mathcal{P}_{\mathrm{Br}}^{s}\overline{\chi }_{\mathrm{Br}%
}^{s}\circ \overline{\eta }_{\mathrm{BrL}}^{s}\Lambda _{\mathrm{Br}}\right)
=\xi \left( \overline{T}_{\mathrm{Br}}^{s}\right) _{1}U_{1,2}\circ
U_{0,1}\left( \overline{\eta }_{\mathrm{Br}}^{s}\right) _{1}U_{1,2}.
\end{equation*}
Therefore, we obtain%
\begin{equation}
H_{\mathrm{BrLie}}^{s}\left( \mathcal{P}_{\mathrm{Br}}^{s}\overline{\chi }_{%
\mathrm{Br}}^{s}\circ \overline{\eta }_{\mathrm{BrL}}^{s}\Lambda _{\mathrm{Br%
}}\right) =U_{0,1}\left( \overline{\eta }_{\mathrm{Br}}^{s}\right)
_{1}U_{1,2}.  \label{form:chiBrsBar}
\end{equation}%
The latter is a split monomorphism. Since $H_{\mathrm{BrLie}}^{s}$ is
faithful, we get that the evaluation on objects of $\mathcal{P}_{\mathrm{Br}%
}^{s}\overline{\chi }_{\mathrm{Br}}^{s}\circ \overline{\eta }_{\mathrm{BrL}%
}^{s}\Lambda _{\mathrm{Br}}$ is a monomorphism.

Assume that $\overline{\eta }_{\mathrm{BrL}}^{s}\Lambda _{\mathrm{Br}}$ is
an isomorphism. Note that $\overline{\eta }_{\mathrm{BrL}}^{s}\Lambda _{%
\mathrm{Br}}$ isomorphism implies $\overline{\eta }_{\mathrm{BrL}%
}^{s}\Lambda _{\mathrm{Br}}\left( P_{\mathrm{Br}}^{s}\right) _{2}$
isomorphism. Since $\mathcal{P}_{\mathrm{Br}}^{s}=\Lambda _{\mathrm{Br}%
}\left( P_{\mathrm{Br}}^{s}\right) _{2}$ this means that $\overline{\eta }_{%
\mathrm{BrL}}^{s}\mathcal{P}_{\mathrm{Br}}^{s}$ is an isomorphism and hence
the adjunction $\left( \overline{\mathcal{U}}_{\mathrm{Br}}^{s},\mathcal{P}_{%
\mathrm{Br}}^{s}\right) $ is idempotent, cf. \cite[Proposition 2.8]{MS}.
Moreover, since $\overline{\eta }_{\mathrm{BrL}}^{s}\Lambda _{\mathrm{Br}}$
is an isomorphism, then the evaluation of $\mathcal{P}_{\mathrm{Br}}^{s}%
\overline{\chi }_{\mathrm{Br}}^{s}:\mathcal{P}_{\mathrm{Br}}^{s}\overline{%
\mathcal{U}}_{\mathrm{Br}}^{s}\Lambda _{\mathrm{Br}}\rightarrow \mathcal{P}_{%
\mathrm{Br}}^{s}\left( \overline{T}_{\mathrm{Br}}^{s}\right) _{1}U_{1,2}$ is
a monomorphism. Let $M_{2}\in \left( \mathrm{Br}_{\mathcal{M}}^{s}\right)
_{2}$ and consider the coequalizer%
\begin{equation*}
\xymatrix{ \overline{T}_\Br^s P_\Br^s \overline{T}_\Br^s M_0
\ar@<.5ex>[rr]^-{\overline{T}_\Br^s\mu_0 }
\ar@<-.5ex>[rr]_-{\overline{\epsilon}_\Br^s \overline{T}_\Br^s
M_0}&&\overline{T}_\Br^s M_0\ar[r]^{\pi_1^s M_1}&(\overline{T}_\Br^s)_1M_1 }
\end{equation*}

Then, from $\overline{\chi }_{\mathrm{Br}}^{s}\circ \overline{p}^{s}\Lambda
_{\mathrm{Br}}=\pi _{1}^{s}U_{1,2},$ we get $\overline{\chi }_{\mathrm{Br}%
}^{s}M_{2}\circ \overline{p}^{s}\Lambda _{\mathrm{Br}}M_{2}\circ \overline{T}%
_{\mathrm{Br}}^{s}\mu _{0}=\overline{\chi }_{\mathrm{Br}}^{s}M_{2}\circ
\overline{p}^{s}\Lambda _{\mathrm{Br}}M_{2}\circ \overline{\epsilon }_{%
\mathrm{Br}}^{s}\overline{T}_{\mathrm{Br}}^{s}M_{0}.$ If we apply $\mathcal{P%
}_{\mathrm{Br}}^{s}$, from the fact that $\mathcal{P}_{\mathrm{Br}}^{s}%
\overline{\chi }_{\mathrm{Br}}^{s}M_{2}$ is a monomorphism, we obtain%
\begin{equation*}
\mathcal{P}_{\mathrm{Br}}^{s}\left( \overline{p}^{s}\Lambda _{\mathrm{Br}%
}M_{2}\circ \overline{T}_{\mathrm{Br}}^{s}\mu _{0}\right) =\mathcal{P}_{%
\mathrm{Br}}^{s}\left( \overline{p}^{s}\Lambda _{\mathrm{Br}}M_{2}\circ
\overline{\epsilon }_{\mathrm{Br}}^{s}\overline{T}_{\mathrm{Br}%
}^{s}M_{0}\right) .
\end{equation*}

If we apply on both sides $H_{\mathrm{BrLie}}^{s},$ by (\ref{diag:Pcal}), we
obtain%
\begin{equation*}
P_{\mathrm{Br}}^{s}\left( \overline{p}^{s}\Lambda _{\mathrm{Br}}M_{2}\circ
\overline{T}_{\mathrm{Br}}^{s}\mu _{0}\right) =P_{\mathrm{Br}}^{s}\left(
\overline{p}^{s}\Lambda _{\mathrm{Br}}M_{2}\circ \overline{\epsilon }_{%
\mathrm{Br}}^{s}\overline{T}_{\mathrm{Br}}^{s}M_{0}\right) .
\end{equation*}

Since $\left( \overline{\mathcal{U}}_{\mathrm{Br}}^{s},\mathcal{P}_{\mathrm{%
Br}}^{s}\right) $ is idempotent, by \cite[Proposition 2.8]{MS}, we also have
that $\overline{\epsilon }_{\mathrm{BrL}}^{s}\overline{\mathcal{U}}_{\mathrm{%
Br}}^{s}$ is an isomorphism. Note that the arguments of $P_{\mathrm{Br}}^{s}$
in the above displayed equality are morphisms of the form $\overline{T}_{%
\mathrm{Br}}^{s}X\rightarrow Y$ for some objects $X,Y.$ Given two such
morphisms $f,g:\overline{T}_{\mathrm{Br}}^{s}X\rightarrow Y$ with $P_{%
\mathrm{Br}}^{s}f=P_{\mathrm{Br}}^{s}g$ we have%
\begin{eqnarray*}
f &=&f\circ \overline{\epsilon }_{\mathrm{Br}}^{s}\overline{T}_{\mathrm{Br}%
}^{s}X\circ \overline{T}_{\mathrm{Br}}^{s}\overline{\eta }_{\mathrm{Br}%
}^{s}X=\overline{\epsilon }_{\mathrm{Br}}^{s}Y\circ \overline{T}_{\mathrm{Br}%
}^{s}P_{\mathrm{Br}}^{s}f\circ \overline{T}_{\mathrm{Br}}^{s}\overline{\eta }%
_{\mathrm{Br}}^{s}X \\
&=&\overline{\epsilon }_{\mathrm{Br}}^{s}Y\circ \overline{T}_{\mathrm{Br}%
}^{s}P_{\mathrm{Br}}^{s}g\circ \overline{T}_{\mathrm{Br}}^{s}\overline{\eta }%
_{\mathrm{Br}}^{s}X=g\circ \overline{\epsilon }_{\mathrm{Br}}^{s}\overline{T}%
_{\mathrm{Br}}^{s}X\circ \overline{T}_{\mathrm{Br}}^{s}\overline{\eta }_{%
\mathrm{Br}}^{s}X=g.
\end{eqnarray*}%
In our case we get $\overline{p}^{s}\Lambda _{\mathrm{Br}}M_{2}\circ
\overline{T}_{\mathrm{Br}}^{s}\mu _{0}=\overline{p}^{s}\Lambda _{\mathrm{Br}%
}M_{2}\circ \overline{\epsilon }_{\mathrm{Br}}^{s}\overline{T}_{\mathrm{Br}%
}^{s}M_{0}.$ By the universal property of the coequalizer above, there is a
braided bialgebra morphism $\tau M_{2}:\left( \overline{T}_{\mathrm{Br}%
}^{s}\right) _{1}M_{1}\rightarrow \overline{\mathcal{U}}_{\mathrm{Br}%
}^{s}\Lambda _{\mathrm{Br}}M_{2}$ such that%
\begin{equation*}
\tau M_{2}\circ \pi _{1}^{s}M_{1}=\overline{p}^{s}\Lambda _{\mathrm{Br}%
}M_{2}.
\end{equation*}%
Note that, by Proposition \ref{pro:BeckBr}, the morphism $\pi _{1}^{s}M_{1}$
can be chosen in such a way to be a coequalizer when regarded as a morphism
in $\mathcal{M}$. We already observed that $\overline{p}^{s}$ is also an
epimorphism in $\mathcal{M}$. Using these facts one easily checks that $%
\overline{\chi }_{\mathrm{Br}}^{s}M_{2}$ and $\tau M_{2}$ are mutual
inverses and hence $\overline{\chi }_{\mathrm{Br}}^{s}:\overline{\mathcal{U}}%
_{\mathrm{Br}}^{s}\Lambda _{\mathrm{Br}}\rightarrow \left( \overline{T}_{%
\mathrm{Br}}^{s}\right) _{1}U_{1,2}$ is an isomorphism.

Therefore $U_{0,1}\left( \overline{\eta }_{\mathrm{Br}}^{s}\right)
_{1}U_{1,2}=H_{\mathrm{BrLie}}^{s}\left( \mathcal{P}_{\mathrm{Br}}^{s}%
\overline{\chi }_{\mathrm{Br}}^{s}\circ \overline{\eta }_{\mathrm{BrL}%
}^{s}\Lambda _{\mathrm{Br}}\right) \ $is an isomorphism. Since $U_{0,1}$
reflects is an isomorphism, we conclude that $\left( \overline{\eta }_{%
\mathrm{Br}}^{s}\right) _{1}U_{1,2}$ is an isomorphism. We have so proved
that the adjunction $\left( \left( \overline{T}_{\mathrm{Br}}^{s}\right)
_{1},\left( P_{\mathrm{Br}}^{s}\right) _{1}\right) $ is idempotent. Note
that in this case we can choose $\left( \overline{T}_{\mathrm{Br}%
}^{s}\right) _{2}:=\left( \overline{T}_{\mathrm{Br}}^{s}\right) _{1}U_{1,2}$
(and $\pi _{2}^{s}$ to be the identity) and it is full and faithful (cf.
\cite[Proposition 2.3]{AGM-MonadicLie1}) i.e. $\left( \overline{\eta }_{%
\mathrm{Br}}^{s}\right) _{2}$ is an isomorphism. By the quoted result we
also have $\left( \overline{\eta }_{\mathrm{Br}}^{s}\right)
_{1}U_{1,2}=U_{1,2}\left( \overline{\eta }_{\mathrm{Br}}^{s}\right) _{2}$ so
that%
\begin{equation*}
H_{\mathrm{BrLie}}^{s}\left( \mathcal{P}_{\mathrm{Br}}^{s}\overline{\chi }_{%
\mathrm{Br}}^{s}\circ \overline{\eta }_{\mathrm{BrL}}^{s}\Lambda _{\mathrm{Br%
}}\right) \overset{(\ref{form:chiBrsBar})}{=}U_{0,1}\left( \overline{\eta }_{%
\mathrm{Br}}^{s}\right) _{1}U_{1,2}=U_{0,1}U_{1,2}\left( \overline{\eta }_{%
\mathrm{Br}}^{s}\right) _{2}=H_{\mathrm{BrLie}}^{s}\Lambda _{\mathrm{Br}%
}\left( \overline{\eta }_{\mathrm{Br}}^{s}\right) _{2}
\end{equation*}%
and hence $\mathcal{P}_{\mathrm{Br}}^{s}\overline{\chi }_{\mathrm{Br}%
}^{s}\circ \overline{\eta }_{\mathrm{BrL}}^{s}\Lambda _{\mathrm{Br}}=\Lambda
_{\mathrm{Br}}\left( \overline{\eta }_{\mathrm{Br}}^{s}\right) _{2}.$ This
proves (\ref{form:zeta1}) holds i.e. that $(\mathrm{Id}_{\mathrm{BrBialg}_{%
\mathcal{M}}^{s}},\Lambda _{\mathrm{Br}}):(\left( \overline{T}_{\mathrm{Br}%
}^{s}\right) _{2},\left( P_{\mathrm{Br}}^{s}\right) _{2})\rightarrow (%
\overline{\mathcal{U}}_{\mathrm{Br}}^{s},\mathcal{P}_{\mathrm{Br}}^{s})$ is
a commutation datum whose canonical transformation is $\overline{\chi }_{%
\mathrm{Br}}^{s}$. Let us check that $\left( \left( P_{\mathrm{Br}%
}^{s}\right) _{2}\overline{\mathcal{U}}_{\mathrm{Br}}^{s},\Lambda _{\mathrm{%
Br}}\right) $ is an adjunction with unit and counit as in the statement. We
have%
\begin{gather*}
\Lambda _{\mathrm{Br}}\left( \left( \overline{\eta }_{\mathrm{Br}%
}^{s}\right) _{2}^{-1}\circ \left( P_{\mathrm{Br}}^{s}\right) _{2}\overline{%
\chi }_{\mathrm{Br}}^{s}\right) \circ \overline{\eta }_{\mathrm{BrL}%
}^{s}\Lambda _{\mathrm{Br}}=\Lambda _{\mathrm{Br}}\left( \overline{\eta }_{%
\mathrm{Br}}^{s}\right) _{2}^{-1}\circ \Lambda _{\mathrm{Br}}\left( P_{%
\mathrm{Br}}^{s}\right) _{2}\overline{\chi }_{\mathrm{Br}}^{s}\circ
\overline{\eta }_{\mathrm{BrL}}^{s}\Lambda _{\mathrm{Br}} \\
=\Lambda _{\mathrm{Br}}\left( \overline{\eta }_{\mathrm{Br}}^{s}\right)
_{2}^{-1}\circ \mathcal{P}_{\mathrm{Br}}^{s}\overline{\chi }_{\mathrm{Br}%
}^{s}\circ \overline{\eta }_{\mathrm{BrL}}^{s}\Lambda _{\mathrm{Br}}=\Lambda
_{\mathrm{Br}}\left( \overline{\eta }_{\mathrm{Br}}^{s}\right)
_{2}^{-1}\circ \Lambda _{\mathrm{Br}}\left( \overline{\eta }_{\mathrm{Br}%
}^{s}\right) _{2}=\Lambda _{\mathrm{Br}}.
\end{gather*}

Moreover, by (\ref{form:zeta2}) applied to our commutation datum, we have $%
\left( \overline{\epsilon }_{\mathrm{Br}}^{s}\right) _{2}\circ \overline{%
\chi }_{\mathrm{Br}}^{s}\left( P_{\mathrm{Br}}^{s}\right) _{2}=\overline{%
\epsilon }_{\mathrm{BrL}}^{s}$ so that%
\begin{eqnarray*}
&&\left( \left( \overline{\eta }_{\mathrm{Br}}^{s}\right) _{2}^{-1}\circ
\left( P_{\mathrm{Br}}^{s}\right) _{2}\overline{\chi }_{\mathrm{Br}%
}^{s}\right) \left( P_{\mathrm{Br}}^{s}\right) _{2}\overline{\mathcal{U}}_{%
\mathrm{Br}}^{s}\circ \left( P_{\mathrm{Br}}^{s}\right) _{2}\overline{%
\mathcal{U}}_{\mathrm{Br}}^{s}\overline{\eta }_{\mathrm{BrL}}^{s} \\
&=&\left( \overline{\eta }_{\mathrm{Br}}^{s}\right) _{2}^{-1}\left( P_{%
\mathrm{Br}}^{s}\right) _{2}\overline{\mathcal{U}}_{\mathrm{Br}}^{s}\circ
\left( P_{\mathrm{Br}}^{s}\right) _{2}\overline{\chi }_{\mathrm{Br}%
}^{s}\left( P_{\mathrm{Br}}^{s}\right) _{2}\overline{\mathcal{U}}_{\mathrm{Br%
}}^{s}\circ \left( P_{\mathrm{Br}}^{s}\right) _{2}\overline{\mathcal{U}}_{%
\mathrm{Br}}^{s}\overline{\eta }_{\mathrm{BrL}}^{s} \\
&=&\left( P_{\mathrm{Br}}^{s}\right) _{2}\left( \overline{\epsilon }_{%
\mathrm{Br}}^{s}\right) _{2}\overline{\mathcal{U}}_{\mathrm{Br}}^{s}\circ
\left( P_{\mathrm{Br}}^{s}\right) _{2}\overline{\chi }_{\mathrm{Br}%
}^{s}\left( P_{\mathrm{Br}}^{s}\right) _{2}\overline{\mathcal{U}}_{\mathrm{Br%
}}^{s}\circ \left( P_{\mathrm{Br}}^{s}\right) _{2}\overline{\mathcal{U}}_{%
\mathrm{Br}}^{s}\overline{\eta }_{\mathrm{BrL}}^{s} \\
&=&\left( P_{\mathrm{Br}}^{s}\right) _{2}\left[ \left( \overline{\epsilon }_{%
\mathrm{Br}}^{s}\right) _{2}\overline{\mathcal{U}}_{\mathrm{Br}}^{s}\circ
\overline{\chi }_{\mathrm{Br}}^{s}\left( P_{\mathrm{Br}}^{s}\right) _{2}%
\overline{\mathcal{U}}_{\mathrm{Br}}^{s}\circ \overline{\mathcal{U}}_{%
\mathrm{Br}}^{s}\overline{\eta }_{\mathrm{BrL}}^{s}\right] \\
&=&\left( P_{\mathrm{Br}}^{s}\right) _{2}\left[ \overline{\epsilon }_{%
\mathrm{BrL}}^{s}\overline{\mathcal{U}}_{\mathrm{Br}}^{s}\circ \overline{%
\mathcal{U}}_{\mathrm{Br}}^{s}\overline{\eta }_{\mathrm{BrL}}^{s}\right]
=\left( P_{\mathrm{Br}}^{s}\right) _{2}.
\end{eqnarray*}%
Note that the counit is an isomorphism so that $\Lambda _{\mathrm{Br}}$ is
full and faithful.

It is then clear that $\left( \left( P_{\mathrm{Br}}^{s}\right) _{2}%
\overline{\mathcal{U}}_{\mathrm{Br}}^{s},\Lambda _{\mathrm{Br}}\right) $ is
an equivalence of categories if and only if $\overline{\eta }_{\mathrm{BrL}%
}^{s}$ is an isomorphism (see e.g. \cite[Proposition 3.4.3]{Borceux1}).
\end{proof}

\begin{theorem}
\label{teo:LambdaSym}Let $\mathcal{M}$ be an abelian symmetric monoidal
category with denumerable coproducts. Assume that the tensor functors are
exact and preserve denumerable coproducts.
\begin{equation}
\xymatrixcolsep{1.5cm}\xymatrixrowsep{1cm}\xymatrix{\mathcal{M}_2%
\ar[r]^{J^s_2}\ar[d]_{\Lambda}&(\mathrm{Br}^s_\mathcal{M})_2\ar[d]^{\Lambda_%
\mathrm{Br}}\\
\mathrm{Lie}_{\mathcal{M}}\ar[r]^{J^s_\mathrm{Lie}}&\mathrm{BrLie}_{%
\mathcal{M}}^s}\quad \xymatrixcolsep{1cm}\xymatrixrowsep{0.40cm}%
\xymatrix{\mathrm{Bialg}_{\mathcal{M}}\ar@<.5ex>[dd]^{P}&&\mathrm{Bialg}_{%
\mathcal{M}}\ar@<.5ex>[dd]^{P_1}|(.30)\hole\ar[ll]_{\mathrm{Id}_{%
\mathrm{Bialg}_{\mathcal{M}}}}&&\mathrm{Bialg}_{\mathcal{M}}%
\ar@<.5ex>[dd]^{P_2}\ar[ll]_{\mathrm{Id}_{\mathrm{Bialg}_{\mathcal{M}}}}%
\ar[dl]|{\mathrm{Id}_{\mathrm{Bialg}_{\mathcal{M}}}}\\
&&&\mathrm{Bialg}_{\mathcal{M}}\ar@<.5ex>[dd]^(.30){\mathcal{P}}%
\ar[ulll]^(.70){\mathrm{Id}_{\mathrm{Bialg}_{\mathcal{M}}}}\\
\mathcal{M}\ar@<.5ex>@{.>}[uu]^{\overline{T}}&&\mathcal{M}_1%
\ar@<.5ex>@{.>}[uu]^{ \overline{T}_1}|(.70)\hole
\ar[ll]_{U_{0,1}}&&\mathcal{M}_2 \ar@<.5ex>@{.>}[uu]^{\overline{T}_2}
\ar[ll]_(.30){U_{1,2}}|\hole \ar[dl]^{\Lambda} \\
&&&\mathrm{Lie}_{\mathcal{M}}\ar@<.5ex>@{.>}[uu]^(.70){\overline{%
\mathcal{U}}}\ar[ulll]^{H_{\mathrm{Lie}}}}  \label{diag:Lambda}
\end{equation}
The functor $P$ is comparable so that we can use the notation of Definition %
\ref{def:comparable}. We have $H_{\mathrm{Lie}}\mathcal{P}=P$ and there is a
functor $\Lambda :\mathcal{M}_{2}\rightarrow \mathrm{Lie}_{\mathcal{M}}$
such that $\Lambda _{\mathrm{Br}}J_{2}^{s}=J_{\mathrm{Lie}}^{s}\Lambda $, $%
\Lambda \circ P_{2}=\mathcal{P}$ and $H_{\mathrm{Lie}}\circ \Lambda
=U_{0,2}. $ Moreover there exists a natural transformation $\overline{\chi }:%
\overline{\mathcal{U}}\Lambda \rightarrow \overline{T}_{1}U_{1,2}$ such that
such that
\begin{equation*}
J_{\mathrm{Bialg}}^{s}\overline{\chi }=\zeta _{1}^{s}U_{1,2}\circ \overline{%
\chi }_{\mathrm{Br}}^{s}J_{2}^{s},\qquad \overline{\chi }\circ \overline{p}%
\Lambda =\pi _{1}U_{1,2}
\end{equation*}%
where $\overline{p}$ is the natural transformation of Theorem \ref{teo:Env}
and $\pi _{1}:\overline{T}U_{0,1}\rightarrow \overline{T}_{1}$ is the
canonical natural transformation defining $\overline{T}_{1}.$

Assume $\overline{\eta }_{\mathrm{BrL}}^{s}\Lambda _{\mathrm{Br}}$ is an
isomorphism.

\begin{itemize}
\item[1)] The adjunction $(\overline{\mathcal{U}},\mathcal{P})$ is
idempotent.

\item[2)] The adjunction $\left( \overline{T}_{1},P_{1}\right) $ is
idempotent, we can choose $\overline{T}_{2}:=\overline{T}_{1}U_{1,2}$, $\pi
_{2}=\mathrm{Id}_{\overline{T}_{2}}$ and $\overline{T}_{2}$ is full and
faithful i.e. $\overline{\eta }_{2}$ is an isomorphism.

\item[3)] The functor $P$ has a monadic decomposition of monadic length at
most two.

\item[4)] $(\mathrm{Id}_{\mathrm{Bialg}_{\mathcal{M}}},\Lambda ):(\overline{T%
}_{2},P_{2})\rightarrow (\overline{\mathcal{U}},\mathcal{P})$ is a
commutation datum whose canonical transformation is $\overline{\chi }$.

\item[5)] The pair $\left( P_{2}\overline{\mathcal{U}},\Lambda \right) $ is
an adjunction with unit $\overline{\eta }_{\mathrm{L}}$ and counit $\left(
\overline{\eta }_{2}\right) ^{-1}\circ P_{2}\overline{\chi }$ so that $%
\Lambda $ is full and faithful. Hence $\overline{\eta }_{\mathrm{L}}$ is an
isomorphism if and only if $\left( P_{2}\overline{\mathcal{U}},\Lambda
\right) $ is an equivalence of categories. In this case $\left( \overline{T}%
_{2},P_{2}\right) $ identifies with $\left( \overline{\mathcal{U}},\mathcal{P%
}\right) $ via $\Lambda $.

\item[6)] If $\overline{\eta }_{\mathrm{BrL}}^{s}$ is an isomorphism so is $%
\overline{\eta }_{\mathrm{L}}$.
\end{itemize}
\end{theorem}

\begin{proof}
We have%
\begin{equation*}
J^{s}H_{\mathrm{Lie}}\mathcal{P}\overset{(\ref{diag:JLies})}{=}H_{\mathrm{%
BrLie}}^{s}J_{\mathrm{Lie}}^{s}\mathcal{P}\overset{(\ref{diag:Ubar})}{=}H_{%
\mathrm{BrLie}}^{s}\mathcal{P}_{\mathrm{Br}}^{s}J_{\mathrm{Bialg}}^{s}%
\overset{(\ref{diag:Pcal})}{=}P_{\mathrm{Br}}^{s}J_{\mathrm{Bialg}}^{s}%
\overset{(\ref{diag:JsTbar})}{=}J^{s}P
\end{equation*}%
so that $H_{\mathrm{Lie}}\mathcal{P}=P.$ By Proposition \ref{pro:Tbars}, $%
\left( J_{\mathrm{Bialg}}^{s},J^{s}\right) :\left( \overline{T},P\right)
\rightarrow \left( \overline{T}_{\mathrm{Br}}^{s},P_{\mathrm{Br}}^{s}\right)
$ is a commutation datum. Moreover, by Lemma \ref{lem:JBialgPres}, $J_{%
\mathrm{Bialg}}^{s}:\mathrm{Bialg}_{\mathcal{M}}\rightarrow \mathrm{BrBialg}%
_{\mathcal{M}}^{s}$ preserves coequalizers. By Proposition \ref{pro:BeckBr},
the right adjoint functor $R=P_{\mathrm{Br}}^{s}$ is comparable and we can
use the notation of Definition \ref{def:comparable}. By Lemma \ref%
{lem:BialgCoeq} and Lemma \ref{lem: coequalizers} we have that $P$ is also
comparable. Applying iteratively Proposition \ref{pro:zetaIter}, we get
functors $J_{n}^{s}:\mathcal{M}_{n}\rightarrow \left( \mathrm{Br}_{\mathcal{M%
}}^{s}\right) _{n},$ for all $n\in \mathbb{N}$, such that $J_{n}^{s}\circ
P_{n}=(P_{\mathrm{Br}}^{s})_{n}\circ J_{\mathrm{Bialg}}^{s}$. Let $M_{2}\in
\mathcal{M}_{2}$ and consider $\Lambda _{\mathrm{Br}}J_{2}^{s}M_{2}.$ Note
that, by construction we have
\begin{equation*}
J_{2}^{s}M_{2}=\left( J_{1}^{s}M_{1},J_{1}^{s}\mu _{1}\circ (P_{\mathrm{Br}%
}^{s})_{1}\zeta _{1}^{s}M_{1}\right) \qquad \text{and}\qquad
J_{1}^{s}M_{1}=\left( J^{s}M_{0},J^{s}\mu _{0}\circ P_{\mathrm{Br}}^{s}\zeta
_{0}^{s}M_{0}\right)
\end{equation*}%
where $\zeta _{i}^{s}:(\overline{T}_{\mathrm{Br}}^{s})_{i}J_{i}^{s}%
\rightarrow J_{\mathrm{Bialg}}^{s}\overline{T}_{i}$ for $i=0,1$ are the
canonical transformations of the respective commutation data. By
construction we have $\Lambda _{\mathrm{Br}}J_{2}^{s}M_{2}=\left(
M_{0},c_{M_{0},M_{0}},\left[ -\right] \right) $ where
\begin{equation*}
\left[ -\right] :=H\mathbb{I}_{\mathrm{Br}}^{s}J^{s}\mu _{0}\circ H\mathbb{I}%
_{\mathrm{Br}}^{s}P_{\mathrm{Br}}^{s}\zeta _{0}^{s}M_{0}\circ \overline{%
\theta }_{\left( M_{0},c_{M_{0},M_{0}}\right) }=\mu _{0}\circ H\mathbb{I}_{%
\mathrm{Br}}^{s}P_{\mathrm{Br}}^{s}\zeta _{0}^{s}M_{0}\circ \overline{\theta
}_{M_{0}}.
\end{equation*}

Now $\Lambda _{\mathrm{Br}}J_{2}^{s}M_{2}\in \mathrm{BrLie}_{\mathcal{M}%
}^{s} $ so that $\left( M_{0},c_{M_{0},M_{0}},\left[ -\right] \right) \in
\mathrm{BrLie}_{\mathcal{M}}$ i.e. $\left( M_{0},\left[ -\right] \right) \in
\mathrm{Lie}_{\mathcal{M}}$ and $\Lambda _{\mathrm{Br}}J_{2}^{s}M_{2}=J_{%
\mathrm{Lie}}^{s}\left( M_{0},\left[ -\right] \right) .$ Thus any object in
the image of $\Lambda _{\mathrm{Br}}J_{2}^{s}$ is also in the image of $J_{%
\mathrm{Lie}}^{s}.$ Thus, by Lemma \ref{lem:Cappuccio}, there is a unique
functor $\Lambda :\mathcal{M}_{2}\rightarrow \mathrm{Lie}_{\mathcal{M}}$
such that $\Lambda _{\mathrm{Br}}J_{2}^{s}=J_{\mathrm{Lie}}^{s}\Lambda .$
This equality implies that $\Lambda $ acts as the identity on morphisms and
that
\begin{equation*}
\Lambda M_{2}=\left( M_{0},\left[ -\right] \right) .
\end{equation*}%
Note that, by Proposition \ref{pro:Tbars}, we have $\zeta _{0}^{s}=\mathrm{Id%
}_{\overline{T}_{\mathrm{Br}}^{s}J^{s}}\ $so that we obtain%
\begin{equation*}
\left[ -\right] :=\mu _{0}\circ \overline{\theta }_{M_{0}}.
\end{equation*}%
We have%
\begin{equation*}
J_{\mathrm{Lie}}^{s}\Lambda P_{2}=\Lambda _{\mathrm{Br}}J_{2}^{s}P_{2}=%
\Lambda _{\mathrm{Br}}(P_{\mathrm{Br}}^{s})_{2}J_{\mathrm{Bialg}}^{s}\overset%
{(\ref{diag:LambdaBr})}{=}\mathcal{P}_{\mathrm{Br}}^{s}J_{\mathrm{Bialg}}^{s}%
\overset{(\ref{diag:Ubar})}{=}J_{\mathrm{Lie}}^{s}\mathcal{P}\text{.}
\end{equation*}%
Since $J_{\mathrm{Lie}}^{s}$ is both injective on morphisms and objects, we
get $\Lambda P_{2}=\mathcal{P}$. It is clear that $H_{\mathrm{Lie}}\Lambda
=U_{0,2}$. We have%
\begin{equation*}
J_{\mathrm{Bialg}}^{s}\overline{\mathcal{U}}\Lambda \overset{(\ref{diag:Ubar}%
)}{=}\overline{\mathcal{U}}_{\mathrm{Br}}^{s}J_{\mathrm{Lie}}^{s}\Lambda =%
\overline{\mathcal{U}}_{\mathrm{Br}}^{s}\Lambda _{\mathrm{Br}}J_{2}^{s}
\end{equation*}%
so that $\widehat{\overline{\mathcal{U}}_{\mathrm{Br}}^{s}\Lambda _{\mathrm{%
Br}}J_{2}^{s}}=\overline{\mathcal{U}}\Lambda .$ Thus, by Lemma \ref%
{lem:Cappuccio}, there is a natural transformation $\overline{\chi }:=%
\widehat{\zeta _{1}^{s}U_{1,2}\circ \overline{\chi }_{\mathrm{Br}%
}^{s}J_{2}^{s}}:\overline{\mathcal{U}}\Lambda \rightarrow \overline{T}%
_{1}U_{1,2}$ such that $J_{\mathrm{Bialg}}^{s}\overline{\chi }=\zeta
_{1}^{s}U_{1,2}\circ \overline{\chi }_{\mathrm{Br}}^{s}J_{2}^{s}$. We compute%
\begin{equation}
J_{\mathrm{Lie}}^{s}\overline{\eta }_{\mathrm{L}}\Lambda \overset{(\ref%
{Form:EpsEtaLbar})}{=}\overline{\eta }_{\mathrm{BrL}}^{s}J_{\mathrm{Lie}%
}^{s}\Lambda =\overline{\eta }_{\mathrm{BrL}}^{s}\Lambda _{\mathrm{Br}%
}J_{2}^{s}  \label{form:etabarL1}
\end{equation}%
so that%
\begin{gather*}
J^{s}H_{\mathrm{Lie}}\left( \mathcal{P}\overline{\chi }\circ \overline{\eta }%
_{\mathrm{L}}\Lambda \right) =J^{s}H_{\mathrm{Lie}}\mathcal{P}\overline{\chi
}\circ J^{s}H_{\mathrm{Lie}}\overline{\eta }_{\mathrm{L}}\Lambda \\
\overset{(\ref{diag:JLies})}{=}J^{s}P\overline{\chi }\circ H_{\mathrm{BrLie}%
}^{s}J_{\mathrm{Lie}}^{s}\overline{\eta }_{\mathrm{L}}\Lambda \overset{(\ref%
{diag:JsTbar}),(\ref{form:etabarL1})}{=}P_{\mathrm{Br}}^{s}J_{\mathrm{Bialg}%
}^{s}\overline{\chi }\circ H_{\mathrm{BrLie}}^{s}\overline{\eta }_{\mathrm{%
BrL}}^{s}\Lambda _{\mathrm{Br}}J_{2}^{s} \\
=P_{\mathrm{Br}}^{s}\zeta _{1}^{s}U_{1,2}\circ P_{\mathrm{Br}}^{s}\overline{%
\chi }_{\mathrm{Br}}^{s}J_{2}^{s}\circ H_{\mathrm{BrLie}}^{s}\overline{\eta }%
_{\mathrm{BrL}}^{s}\Lambda _{\mathrm{Br}}J_{2}^{s}\overset{(\ref{diag:Pcal})}%
{=}P_{\mathrm{Br}}^{s}\zeta _{1}^{s}U_{1,2}\circ H_{\mathrm{BrLie}}^{s}%
\mathcal{P}_{\mathrm{Br}}^{s}\overline{\chi }_{\mathrm{Br}%
}^{s}J_{2}^{s}\circ H_{\mathrm{BrLie}}^{s}\overline{\eta }_{\mathrm{BrL}%
}^{s}\Lambda _{\mathrm{Br}}J_{2}^{s} \\
=P_{\mathrm{Br}}^{s}\zeta _{1}^{s}U_{1,2}\circ H_{\mathrm{BrLie}}^{s}\left(
\mathcal{P}_{\mathrm{Br}}^{s}\overline{\chi }_{\mathrm{Br}}^{s}\circ
\overline{\eta }_{\mathrm{BrL}}^{s}\Lambda _{\mathrm{Br}}\right) J_{2}^{s}%
\overset{(\ref{form:chiBrsBar})}{=}U_{0,1}\left( P_{\mathrm{Br}}^{s}\right)
_{1}\zeta _{1}^{s}U_{1,2}\circ U_{0,1}\left( \overline{\eta }_{\mathrm{Br}%
}^{s}\right) _{1}U_{1,2}J_{2}^{s} \\
\overset{(\ref{form:chiBrsBar})}{=}U_{0,1}\left( P_{\mathrm{Br}}^{s}\right)
_{1}\zeta _{1}^{s}U_{1,2}\circ U_{0,1}\left( \overline{\eta }_{\mathrm{Br}%
}^{s}\right) _{1}J_{1}^{s}U_{1,2}=U_{0,1}\left( \left( P_{\mathrm{Br}%
}^{s}\right) _{1}\zeta _{1}^{s}\circ \left( \overline{\eta }_{\mathrm{Br}%
}^{s}\right) _{1}J_{1}^{s}\right) U_{1,2} \\
=U_{0,1}J_{1}^{s}\overline{\eta }_{1}U_{1,2}=J^{s}U_{0,1}\overline{\eta }%
_{1}U_{1,2}
\end{gather*}%
so that%
\begin{equation}
H_{\mathrm{Lie}}\left( \mathcal{P}\overline{\chi }\circ \overline{\eta }_{%
\mathrm{L}}\Lambda \right) =U_{0,1}\overline{\eta }_{1}U_{1,2}.
\label{form:etabarL2}
\end{equation}%
We have%
\begin{eqnarray*}
J_{\mathrm{Bialg}}^{s}\left( \overline{\chi }\circ \overline{p}\Lambda
\right) &=&J_{\mathrm{Bialg}}^{s}\overline{\chi }\circ J_{\mathrm{Bialg}}^{s}%
\overline{p}\Lambda \overset{(\ref{eq:PBar&pBars})}{=}\zeta
_{1}^{s}U_{1,2}\circ \overline{\chi }_{\mathrm{Br}}^{s}J_{2}^{s}\circ
\overline{p}^{s}J_{\mathrm{Lie}}^{s}\Lambda \\
&=&\zeta _{1}^{s}U_{1,2}\circ \overline{\chi }_{\mathrm{Br}%
}^{s}J_{2}^{s}\circ \overline{p}^{s}\Lambda _{\mathrm{Br}}J_{2}^{s}\overset{(%
\ref{form:chibarBrs})}{=}\zeta _{1}^{s}U_{1,2}\circ \pi
_{1}^{s}U_{1,2}J_{2}^{s} \\
&=&\zeta _{1}^{s}U_{1,2}\circ \pi _{1}^{s}J_{1}^{s}U_{1,2}=\left( \zeta
_{1}^{s}\circ \pi _{1}^{s}J_{1}^{s}\right) U_{1,2}\overset{(\ast )}{=}\left(
J_{\mathrm{Bialg}}^{s}\pi _{1}\circ \zeta _{0}^{s}\right) U_{1,2}=J_{\mathrm{%
Bialg}}^{s}\pi _{1}U_{1,2}
\end{eqnarray*}%
where (*) follows by construction of $\zeta _{1}^{s}$ (see the proof of
Proposition \ref{pro:zetaIter}). Thus we obtain $\overline{\chi }\circ
\overline{p}\Lambda =\pi _{1}U_{1,2}$.

Assume $\overline{\eta }_{\mathrm{BrL}}^{s}\Lambda _{\mathrm{Br}}$ is an
isomorphism. By Theorem \ref{teo:LambdaBr}, we have that $\overline{\chi }_{%
\mathrm{Br}}^{s}$ is an isomorphism. Thus, from $J_{\mathrm{Bialg}}^{s}%
\overline{\chi }=\zeta _{1}^{s}U_{1,2}\circ \overline{\chi }_{\mathrm{Br}%
}^{s}J_{2}^{s}$ and the fact that $\zeta _{1}^{s}$ is an isomorphism, we
deduce that $\overline{\chi }$ is an isomorphism too. Moreover, by (\ref%
{form:etabarL1}), we also have that $\overline{\eta }_{\mathrm{L}}\Lambda $
is an isomorphism. From this we get that $\overline{\eta }_{\mathrm{L}%
}\Lambda P_{2}$ is an isomorphism. Since $\Lambda P_{2}=\mathcal{P}$ we have
that $\overline{\eta }_{\mathrm{L}}\mathcal{P}$ is an isomorphism. By \cite[%
Proposition 2.8]{MS}, this means that the adjunction $(\overline{\mathcal{U}}%
,\mathcal{P})$ is idempotent.

Moreover, since $\overline{\eta }_{\mathrm{L}}\Lambda $ is an isomorphism,
by (\ref{form:etabarL2}), we deduce that $\overline{\eta }_{1}U_{1,2}$ is an
isomorphism i.e. $\left( \overline{T}_{1},P_{1}\right) $ is idempotent (cf.
\cite[Remark 2.2]{AGM-MonadicLie1}). Note that in this case we can choose $%
\overline{T}_{2}:=\overline{T}_{1}U_{1,2}$ and it is full and faithful (cf.
\cite[Proposition 2.3]{AGM-MonadicLie1}) i.e. $\overline{\eta }_{2}$ is an
isomorphism. The choice $\overline{T}_{2}:=\overline{T}_{1}U_{1,2}$ implies
we can choose the canonical projection $\pi _{2}:\overline{T}%
_{1}U_{1,2}\rightarrow \overline{T}_{2}$ to be the identity. In this case by
definition, $\overline{{\eta }}_{1}$ is given by the formula $\overline{\eta
}_{1}U_{1,2}=U_{1,2}\overline{\eta }_{2}.$ Thus the second term of \ref%
{form:etabarL2} becomes $U_{0,1}\overline{\eta }_{1}U_{1,2}=U_{0,1}U_{1,2}%
\overline{\eta }_{2}=U_{0,2}\overline{\eta }_{2}=H_{\mathrm{Lie}}\Lambda
\overline{\eta }_{2}$. Since $H_{\mathrm{Lie}}$ is faithful, by %
\eqref{form:etabarL2} we obtain $\mathcal{P}\overline{\chi }\circ \overline{%
\eta }_{\mathrm{L}}\Lambda =\Lambda \overline{\eta }_{2}$ which means that $(%
\mathrm{Id}_{\mathrm{Bialg}_{\mathcal{M}}},\Lambda ):(\overline{T}%
_{2},P_{2})\rightarrow (\overline{\mathcal{U}},\mathcal{P})$ is a
commutation datum whose canonical transformation is $\overline{\chi }$.

We already observed that $\Lambda _{\mathrm{Br}}J_{2}^{s}=J_{\mathrm{Lie}%
}^{s}\Lambda .$ Moreover, from $J_{n}^{s}\circ P_{n}=(P_{\mathrm{Br}%
}^{s})_{n}\circ J_{\mathrm{Bialg}}^{s}$, we deduce
\begin{equation*}
J_{2}^{s}\left( P_{2}\overline{\mathcal{U}}\right) =(P_{\mathrm{Br}%
}^{s})_{2}J_{\mathrm{Bialg}}^{s}\overline{\mathcal{U}}\overset{(\ref%
{diag:Ubar})}{=}\left( \left( P_{\mathrm{Br}}^{s}\right) _{2}\overline{%
\mathcal{U}}_{\mathrm{Br}}^{s}\right) J_{\mathrm{Lie}}^{s}.
\end{equation*}%
We know that $J_{\mathrm{Lie}}^{s}$ is full, faithful and injective on
objects. Since $J^{s}$ fulfils the same properties, by Proposition \ref%
{pro:zetaIter} applied to the commutation datum $\left( J_{\mathrm{Bialg}%
}^{s},J^{s}\right) :\left( \overline{T},P\right) \rightarrow \left(
\overline{T}_{\mathrm{Br}}^{s},P_{\mathrm{Br}}^{s}\right) $, we deduce that
the same is true for $J_{1}^{s}$ and hence, by same argument, also for $%
J_{2}^{s}$. Thus we can apply Lemma \ref{lem:LiftAdj} to the case $L^{\prime
}=\left( \left( P_{\mathrm{Br}}^{s}\right) _{2}\overline{\mathcal{U}}_{%
\mathrm{Br}}^{s}\right) ,R^{\prime }=\Lambda _{\mathrm{Br}},F=J_{2}^{s},G=J_{%
\mathrm{Lie}}^{s}$. Then $L=P_{2}\overline{\mathcal{U}}$ and $R=\Lambda ,$
the pair $\left( L,R\right) $ is an adjunction and the unit and counit of $%
\left( L,R\right) $ and $\left( L^{\prime },R^{\prime }\right) $ are related
by (\ref{form:LiftAdj}). Since $F$ and $G$ are both conservative, we get
that $\epsilon $ and $\eta $ are an isomorphism whenever $\epsilon ^{\prime
} $ and $\eta ^{\prime }=\overline{\eta }_{\mathrm{BrL}}^{s}$ are. By
Theorem \ref{teo:LambdaBr}, we know that $\Lambda _{\mathrm{Br}}$ is full
and faithful i.e. $\epsilon ^{\prime }$ is an isomorphism and hence $%
\epsilon $ is an isomorphism i.e. $\Lambda $ is full and faithful. It is
clear that $\left( P_{2}\overline{\mathcal{U}},\Lambda \right) $ is an
equivalence if and only if $\eta $ is an isomorphism. By (\ref%
{Form:EpsEtaLbar}), we have $J_{\mathrm{Lie}}^{s}\overline{\eta }_{\mathrm{L}%
}=\overline{\eta }_{\mathrm{BrL}}^{s}J_{\mathrm{Lie}}^{s}$ i.e. $G\overline{%
\eta }_{\mathrm{L}}=\eta ^{\prime }G.$ Thus, since $G$ is faithful, (\ref%
{form:LiftAdj}) implies $\eta =\overline{\eta }_{\mathrm{L}}.$ If we write %
\ref{form:zeta1} for the commutation datum $\left( J_{\mathrm{Bialg}%
}^{s},J_{2}^{s}\right) :\left( \overline{T}_{2},P_{2}\right) \rightarrow
\left( (\overline{T}_{\mathrm{Br}}^{s})_{2},(P_{\mathrm{Br}}^{s})_{2}\right)
,$ we get $(P_{\mathrm{Br}}^{s})_{2}\zeta _{2}^{s}\circ \left( \overline{%
\eta }_{\mathrm{Br}}^{s}\right) _{2}J_{2}^{s}=J_{2}^{s}\overline{\eta }_{2}$
(note that $\left( \overline{\eta }_{\mathrm{Br}}^{s}\right) _{2}$ is an
isomorphism by Theorem \ref{teo:LambdaBr}-2)). Using this equality we compute%
\begin{eqnarray*}
J_{2}^{s}\left( \left( \overline{\eta }_{2}\right) ^{-1}\circ P_{2}\overline{%
\chi }\right) &=&J_{2}^{s}\left( \overline{\eta }_{2}\right) ^{-1}\circ
J_{2}^{s}P_{2}\overline{\chi }=J_{2}^{s}\left( \overline{\eta }_{2}\right)
^{-1}\circ (P_{\mathrm{Br}}^{s})_{2}J_{\mathrm{Bialg}}^{s}\overline{\chi } \\
&=&J_{2}^{s}\left( \overline{\eta }_{2}\right) ^{-1}\circ (P_{\mathrm{Br}%
}^{s})_{2}\left( \zeta _{1}^{s}U_{1,2}\circ \overline{\chi }_{\mathrm{Br}%
}^{s}J_{2}^{s}\right) \\
&=&\left[ (P_{\mathrm{Br}}^{s})_{2}\zeta _{2}^{s}\circ \left( \overline{\eta
}_{\mathrm{Br}}^{s}\right) _{2}J_{2}^{s}\right] ^{-1}\circ (P_{\mathrm{Br}%
}^{s})_{2}\zeta _{1}^{s}U_{1,2}\circ (P_{\mathrm{Br}}^{s})_{2}\overline{\chi
}_{\mathrm{Br}}^{s}J_{2}^{s} \\
&=&\left( \overline{\eta }_{\mathrm{Br}}^{s}\right) _{2}^{-1}J_{2}^{s}\circ
(P_{\mathrm{Br}}^{s})_{2}\left( \zeta _{2}^{s}\right) ^{-1}\circ (P_{\mathrm{%
Br}}^{s})_{2}\zeta _{1}^{s}U_{1,2}\circ (P_{\mathrm{Br}}^{s})_{2}\overline{%
\chi }_{\mathrm{Br}}^{s}J_{2}^{s}.
\end{eqnarray*}

Now, by construction of $\zeta _{2}^{s}$ (see the proof of Proposition \ref%
{pro:zetaIter}), the fact that $\pi _{2}:\overline{T}_{1}U_{1,2}\rightarrow
\overline{T}_{2}$ is the identity and that also $\pi _{2}^{s}$ is the
identity (see Theorem \ref{teo:LambdaBr}-2)), we have that $\zeta
_{2}^{s}=\zeta _{1}^{s}U_{1,2}$ and hence
\begin{equation*}
F\left( \left( \overline{\eta }_{2}\right) ^{-1}\circ P_{2}\overline{\chi }%
\right) =J_{2}^{s}\left( \left( \overline{\eta }_{2}\right) ^{-1}\circ P_{2}%
\overline{\chi }\right) =\left( \overline{\eta }_{\mathrm{Br}}^{s}\right)
_{2}^{-1}J_{2}^{s}\circ (P_{\mathrm{Br}}^{s})_{2}\overline{\chi }_{\mathrm{Br%
}}^{s}J_{2}^{s}=\left( \left( \overline{\eta }_{\mathrm{Br}}^{s}\right)
_{2}^{-1}\circ (P_{\mathrm{Br}}^{s})_{2}\overline{\chi }_{\mathrm{Br}%
}^{s}\right) J_{2}^{s}=\epsilon ^{\prime }F.
\end{equation*}%
Thus, by (\ref{form:LiftAdj}), we get $\epsilon =\left( \overline{\eta }%
_{2}\right) ^{-1}\circ P_{2}\overline{\chi }.$
\end{proof}

\begin{definition}
\label{def:MM} An \textbf{MM-category} (Milnor-Moore-category) is an abelian
monoidal category $\mathcal{M}$ with denumerable coproducts such that the
tensor functors are exact and preserve denumerable coproducts and such that
the unit $\overline{\eta }_{\mathrm{BrL}}^{s}:\mathrm{Id}_{\mathrm{BrLie}_{%
\mathcal{M}}^{s}}\rightarrow \mathcal{P}_{\mathrm{Br}}^{s}\overline{\mathcal{%
U}}_{\mathrm{Br}}^{s}$ of the adjunction $\left( \overline{\mathcal{U}}_{%
\mathrm{Br}}^{s},\mathcal{P}_{\mathrm{Br}}^{s}\right) $ is a functorial
isomorphism i.e. the functor $\overline{\mathcal{U}}_{\mathrm{Br}}^{s}:%
\mathrm{BrLie}_{\mathcal{M}}^{s}\rightarrow \mathrm{BrBialg}_{\mathcal{M}%
}^{s}$ is full and faithful (see e.g. \cite[dual of Proposition 3.4.1, page
114]{Borceux1}).
\end{definition}

\begin{remark}
\label{rem:MM}1) The celebrated Milnor-Moore Theorem, cf. \cite[Theorem 5.18]%
{Milnor-Moore} states that, in characteristic zero, there is a category
equivalence between the category of Lie algebras and the category of
primitively generated bialgebras. The fact that the counit of the adjunction
involved is an isomorphism just encodes the fact that the bialgebras
considered are primitively generated. On the other hand the crucial point in
the proof is that the unit of the adjunction is an isomorphism.

In our wider context this translates to the unit of the adjunction $\left(
\overline{\mathcal{U}}_{\mathrm{Br}}^{s},\mathcal{P}_{\mathrm{Br}%
}^{s}\right) $ being a functorial isomorphism. From this the definition of
MM-category stems. Note that for an MM-category $\mathcal{M}$ we can apply
Theorem \ref{teo:LambdaBr} to obtain that the functor $P_{\mathrm{Br}}^{s}$
has a monadic decomposition of monadic length at most two. Moreover we can
identify the category $\left( \mathrm{Br}_{\mathcal{M}}^{s}\right) _{2}$
with $\mathrm{BrLie}_{\mathcal{M}}^{s}$.

2) In the case of a symmetric MM-category $\mathcal{M}$ the connection with
Milnor-Moore Theorem becomes more evident. In fact, in this case, we can
apply Theorem \ref{teo:LambdaSym} to obtain that the unit of the adjunction $%
\left( \overline{\mathcal{U}},\mathcal{P}\right) $ is a functorial
isomorphism.
\end{remark}

\section{Lifting the structure of MM-category}

We first prove a crucial result for braided vector spaces.

\begin{theorem}
\label{teo:mainVS} The category of vector spaces over a fixed field $\Bbbk $
of characteristic zero is an MM-category.
\end{theorem}

\begin{proof}
Let $\mathcal{M}=\mathfrak{M}$ be the category of vector spaces over $\Bbbk
. $ We have just to prove that $\overline{\eta }_{\mathrm{BrL}}^{s}$ is an
isomorphism. Let $\left( M,c,\left[ -\right] \right) \in \mathrm{BrLie}_{%
\mathcal{M}}^{s}$. Since we are working on vector spaces, we can express
explicitly the universal enveloping algebra $\overline{\mathcal{U}}_{\mathrm{%
Br}}^{s}\left( M,c,\left[ -\right] \right) $ with elements as follows
\begin{equation*}
\overline{\mathcal{U}}_{\mathrm{Br}}^{s}\left( M,c,\left[ -\right] \right) =%
\frac{\overline{T}_{\mathrm{Br}}^{s}\left( M,c\right) }{\left( \left[
x\otimes y\right] -x\otimes y+c\left( x\otimes y\right) \mid x,y\in M\right)
}.
\end{equation*}

By Lemma \ref{lem:Jac2}, $\left( M,\left[ -\right] \right) $ is a Lie $c$%
-algebra and $\overline{\mathcal{U}}_{\mathrm{Br}}^{s}\left( M,c,\left[ -%
\right] \right) $ coincides with the corresponding universal enveloping
algebra in the sense of \cite[Section 2.5]{Kharchenko-Connected}. Hence we
can apply \cite[Lemma 6.2]{Kharchenko-Connected} to conclude that the
canonical map from $M$ into the primitive part of $\overline{\mathcal{U}}_{%
\mathrm{Br}}^{s}\left( M,c,\left[ -\right] \right) $ is an isomorphism. In
our notation this means that
\begin{equation*}
H\mathbb{I}_{\mathrm{Br}}^{s}H_{\mathrm{BrLie}}^{s}\overline{\eta }_{\mathrm{%
BrL}}^{s}\left( M,c,\left[ -\right] \right) :M\rightarrow H\mathbb{I}_{%
\mathrm{Br}}^{s}H_{\mathrm{BrLie}}^{s}\mathcal{P}_{\mathrm{Br}}^{s}\overline{%
\mathcal{U}}_{\mathrm{Br}}^{s}\left( M,c,\left[ -\right] \right)
\end{equation*}%
is bijective. Note that $H,\mathbb{I}_{\mathrm{Br}}^{s}$ and $H_{\mathrm{%
BrLie}}^{s}$ are conservative by \ref{cl:BrBialg}, Definition \ref{def:braid}
and Definition \ref{def:Lie} respectively. Thus $H\mathbb{I}_{\mathrm{Br}%
}^{s}H_{\mathrm{BrLie}}^{s}$ is conservative and hence we get that $%
\overline{\eta }_{\mathrm{BrL}}^{s}\left( M,c,\left[ -\right] \right) $ is
an isomorphism for all $\left( M,c,\left[ -\right] \right) \in \mathrm{BrLie}%
_{\mathcal{M}}^{s}$. We have so proved that $\overline{\eta }_{\mathrm{BrL}%
}^{s}$ is an isomorphism.
\end{proof}

In the rest of this section we will deal with symmetric braided monoidal
categories $\mathcal{M}$ endowed with a faithful monoidal functor $W:%
\mathcal{M}\rightarrow \mathfrak{M}$ which is not necessarily braided. The
examples we will treat take $\mathcal{M}={\mathfrak{M}^{H}}$ for a dual
quasi-bialgebra $H$ or $\mathcal{M}={_{H}}\mathfrak{M}$ for a
quasi-bialgebra case. Note that in general the obvious forgetful functors
need not to be monoidal, see e.g. \cite[Example 9.1.4]{Maj} so that further
conditions will be required on $H$. Note that the results on ${\mathfrak{M}%
^{H}}$ and ${_{H}}\mathfrak{M}$ are not dual each other, unless $H$ is
finite-dimensional.

\begin{lemma}
\label{lem:BrLieF}Let $\mathcal{M}$ and $\mathcal{N}$ be monoidal
categories. Any monoidal functor $\left( F,\phi _{0},\phi _{2}\right) :%
\mathcal{M}\rightarrow \mathcal{N}$ induces a functor $\mathrm{BrLie}F:%
\mathrm{BrLie}_{\mathcal{M}}\rightarrow \mathrm{BrLie}_{\mathcal{N}}$ which
acts as $F$ on morphisms and such that $\mathrm{BrLie}F\left( M,c_{M},\left[
-\right] _{M}\right) :=\left( FM,c_{FM},\left[ -\right] _{FM}\right) $ where
$\left( FM,c_{FM}\right) =\mathrm{Br}F\left( M,c_{M}\right) $ and%
\begin{equation*}
\left[ -\right] _{FM}:=F\left[ -\right] _{M}\circ \phi _{2}\left( M,M\right)
:FM\otimes FM\rightarrow F\left( M\right) .
\end{equation*}%
Moreover the first diagram below commutes and there is a unique functor $%
\mathrm{BrLie}^{s}F$ such that the second diagram commutes.%
\begin{equation}
\xymatrixrowsep{15pt}\xymatrixcolsep{10pt} \xymatrix{\BrLie_\M
\ar[d]_{H_\BrLie}\ar[rr]^{\BrLie F}&& \BrLie_\N\ar[d]^{H_\BrLie}\\ \Br_\M
\ar[rr]^{\Br F}&& \Br_\N } \quad \xymatrix{\BrLie_\M^s
\ar[d]_{\I_\BrLie^s}\ar[rr]^{\BrLie^s F}&& \BrLie_\N^s\ar[d]^{\I_\BrLie^s}\\
\BrLie_\M \ar[rr]^{\BrLie F}&& \BrLie_\N } \quad \xymatrix{\BrLie_\M^s
\ar[d]_{H_\BrLie^s}\ar[rr]^{\BrLie^s F}&& \BrLie_\N^s\ar[d]^{H_\BrLie^s}\\
\Br_\M^s \ar[rr]^{\Br^s F}&& \Br_\N^s }  \label{diag:BrLieF}
\end{equation}%
Furthermore the functors $\mathrm{BrLie}F$ and $\mathrm{BrLie}^{s}F$ are
conservative whenever $F$ is.
\end{lemma}

\begin{proof}
It is straightforward.
\end{proof}

\begin{theorem}
\label{teo:Udata}Let $\mathcal{M}$ and $\mathcal{N}$ be monoidal categories.
Assume that both $\mathcal{M}$ and $\mathcal{N}$ are abelian with
denumerable coproducts, and that the tensor functors are exact and preserves
denumerable coproducts. Assume that there exists an exact monoidal functor $%
\left( F,\phi _{0},\phi _{2}\right) :\mathcal{M}\rightarrow \mathcal{N}$
which preserves denumerable coproducts. Then we have the following
commutation data with the respective canonical transformations%
\begin{eqnarray*}
\left( \mathrm{BrAlg}^{s}F,\mathrm{BrLie}^{s}F\right) &:&\left( \mathcal{U}_{%
\mathrm{Br}}^{s},\mathcal{L}_{\mathrm{Br}}^{s}\right) \rightarrow \left(
\mathcal{U}_{\mathrm{Br}}^{s},\mathcal{L}_{\mathrm{Br}}^{s}\right) ,\qquad
\zeta _{\mathrm{BrL}}^{s}:\mathcal{U}_{\mathrm{Br}}^{s}\left( \mathrm{BrLie}%
^{s}F\right) \rightarrow \left( \mathrm{BrAlg}^{s}F\right) \mathcal{U}_{%
\mathrm{Br}}^{s}, \\
\left( \mathrm{BrBialg}^{s}F,\mathrm{BrLie}^{s}F\right) &:&\left( \overline{%
\mathcal{U}}_{\mathrm{Br}}^{s},\mathcal{P}_{\mathrm{Br}}^{s}\right)
\rightarrow \left( \overline{\mathcal{U}}_{\mathrm{Br}}^{s},\mathcal{P}_{%
\mathrm{Br}}^{s}\right) ,\qquad \overline{\zeta }_{\mathrm{BrL}}^{s}:%
\overline{\mathcal{U}}_{\mathrm{Br}}^{s}\left( \mathrm{BrLie}^{s}F\right)
\rightarrow \left( \mathrm{BrBialg}^{s}F\right) \overline{\mathcal{U}}_{%
\mathrm{Br}}^{s}.
\end{eqnarray*}
\end{theorem}

\begin{proof}
A direct computation using (\ref{diag:BrLieF}) shows that%
\begin{equation*}
\mathbb{I}_{\mathrm{BrLie}}^{s}\left( \mathrm{BrLie}^{s}F\right) \mathcal{L}%
_{\mathrm{Br}}^{s}\left( B,m_{B},u_{B},c_{B}\right) =\mathbb{I}_{\mathrm{%
BrLie}}^{s}\mathcal{L}_{\mathrm{Br}}^{s}\left( \mathrm{BrAlg}^{s}F\right)
\left( B,m_{B},u_{B},c_{B}\right) .
\end{equation*}
Since both functors act as $F$ on morphisms, we get $\mathbb{I}_{\mathrm{%
BrLie}}^{s}\left( \mathrm{BrLie}^{s}F\right) \mathcal{L}_{\mathrm{Br}}^{s}=%
\mathbb{I}_{\mathrm{BrLie}}^{s}\mathcal{L}_{\mathrm{Br}}^{s}\left( \mathrm{%
BrAlg}^{s}F\right) $. Since $\mathbb{I}_{\mathrm{BrLie}}^{s}$ is both
injective on morphisms and objects we obtain
\begin{equation*}
\left( \mathrm{BrLie}^{s}F\right) \mathcal{L}_{\mathrm{Br}}^{s}=\mathcal{L}_{%
\mathrm{Br}}^{s}\left( \mathrm{BrAlg}^{s}F\right) .
\end{equation*}%
Now, using in the given order (\ref{diag:HBrLies}), (\ref{diag:BrLieF}),
again (\ref{diag:HBrLies}), (\ref{form:comdat3}) and again (\ref{diag:BrLieF}%
), we get the equality $\mathbb{I}_{\mathrm{Br}}^{s}H_{\mathrm{BrLie}%
}^{s}\left( \mathrm{BrLie}^{s}F\right) \xi =\mathbb{I}_{\mathrm{Br}}^{s}H_{%
\mathrm{BrLie}}^{s}\xi \left( \mathrm{BrBialg}^{s}F\right) .$ Then one shows
that $\left( \mathrm{BrLie}^{s}F\right) \xi $ and $\xi \left( \mathrm{BrBialg%
}^{s}F\right) $ has the same domain and codomain. Thus, from $\mathbb{I}_{%
\mathrm{Br}}^{s}H_{\mathrm{BrLie}}^{s}\left( \mathrm{BrLie}^{s}F\right) \xi =%
\mathbb{I}_{\mathrm{Br}}^{s}H_{\mathrm{BrLie}}^{s}\xi \left( \mathrm{BrBialg}%
^{s}F\right) $ we deduce that%
\begin{equation*}
\left( \mathrm{BrLie}^{s}F\right) \xi =\xi \left( \mathrm{BrBialg}%
^{s}F\right) .
\end{equation*}%
Consider the natural transformation $\overline{\zeta }_{\mathrm{BrL}}^{s}:%
\overline{\mathcal{U}}_{\mathrm{Br}}^{s}\left( \mathrm{BrLie}^{s}F\right)
\rightarrow \left( \mathrm{BrBialg}^{s}F\right) \overline{\mathcal{U}}_{%
\mathrm{Br}}^{s}$ of Lemma \ref{lem:zeta}. By definition%
\begin{equation*}
\overline{\zeta }_{\mathrm{BrL}}^{s}:=\overline{\epsilon }_{\mathrm{BrL}%
}^{s}\left( \mathrm{BrBialg}^{s}F\right) \overline{\mathcal{U}}_{\mathrm{Br}%
}^{s}\circ \overline{\mathcal{U}}_{\mathrm{Br}}^{s}\left( \mathrm{BrLie}%
^{s}F\right) \overline{\eta }_{\mathrm{BrL}}^{s}.
\end{equation*}%
It is straightforward to check that%
\begin{equation}
\mho _{\mathrm{Br}}^{s}\overline{\zeta }_{\mathrm{BrL}}^{s}=\zeta _{\mathrm{%
BrL}}^{s}  \label{form:zetaBrLBars}
\end{equation}%
where $\zeta _{\mathrm{BrL}}^{s}:\mathcal{U}_{\mathrm{Br}}^{s}\left( \mathrm{%
BrLie}^{s}F\right) \rightarrow \left( \mathrm{BrAlg}^{s}F\right) \mathcal{U}%
_{\mathrm{Br}}^{s}$ is the canonical morphism of Lemma \ref{lem:zeta}, and
also
\begin{equation*}
\zeta _{\mathrm{BrL}}^{s}\circ p^{s}\left( \mathrm{BrLie}^{s}F\right)
=\left( \mathrm{BrAlg}^{s}F\right) p^{s}\circ \zeta _{\mathrm{Br}}^{s}H_{%
\mathrm{BrLie}}^{s}.
\end{equation*}
Let $\left( M,c_{M},\left[ -\right] _{M}\right) \in \mathrm{BrLie}_{\mathcal{%
M}}^{s}.$ Then we have that $\left( M\otimes M,c_{M\otimes M}\right) \in
\mathrm{BrLie}_{\mathcal{M}}^{s}$ where $c_{M\otimes M}:=\left( M\otimes
c_{M}\otimes M\right) \left( c_{M}\otimes c_{M}\right) \left( M\otimes
c_{M}\otimes M\right) .$ It is easy to check that $\left[ -\right] :M\otimes
M\rightarrow M$ and $\theta _{\left( M,c_{M}\right) }:M\otimes M\rightarrow
\Omega TM$ induce morphisms of braided objects%
\begin{equation*}
\left[ -\right] ^{s}:\left( M\otimes M,c_{M\otimes M}\right) \rightarrow
\left( M,c_{M}\right) \qquad \text{and}\qquad \theta _{\left( M,c_{M}\right)
}^{s}:\left( M\otimes M,c_{M\otimes M}\right) \rightarrow \Omega _{\mathrm{Br%
}}^{s}T_{\mathrm{Br}}^{s}\left( M,c_{M}\right)
\end{equation*}%
such that
\begin{equation*}
H\mathbb{I}_{\mathrm{Br}}^{s}\left[ -\right] ^{s}=\left[ -\right] \qquad
\text{and}\qquad H\mathbb{I}_{\mathrm{Br}}^{s}\theta _{\left( M,c_{M}\right)
}^{s}=\theta _{\left( M,c_{M}\right) }.
\end{equation*}
Let us check that the following is a coequalizer in $\mathrm{BrAlg}_{%
\mathcal{M}}^{s}$%
\begin{equation}
\xymatrixrowsep{15pt}\xymatrixcolsep{50pt} \xymatrix{
T_{\mathrm{Br}}^{s}\left( M\otimes M,c_{M\otimes M}\right)
\ar@<.5ex>[rr]^-{T_{\mathrm{Br}}^{s}\left[ -\right] ^{s} }
\ar@<-.5ex>[rr]_-{\epsilon _{\mathrm{Br}}^{s}T_{\mathrm{Br}}^{s}\left(
M,c_{M}\right) \circ T_{\mathrm{Br}}^{s}\theta _{\left( M,c_{M}\right) }^{s}
}&&T_{\mathrm{Br}}^{s}\left( M,c_{M}\right)\ar[r]^-{p^{s}\left(
M,c_{M},\left[ -\right] _{M}\right) }&\mathcal{U}_{\mathrm{Br}}^{s}\left(
M,c_M,\left[ -\right] \right) _M}.  \label{diag:ps}
\end{equation}
Apply $H_{\mathrm{Alg}}\mathbb{I}_{\mathrm{BrAlg}}^{s}$ to this diagram we
get the diagram%
\begin{equation}
\xymatrixcolsep{40pt} \xymatrix{T\left( M\otimes M\right)
\ar@<.5ex>[rr]^-{T\left[ -\right] } \ar@<-.5ex>[rr]_-{\epsilon TM\circ
T\theta _{\left( M,c_{M}\right) } }&&T\left(
M\right)\ar[rr]^-{H_{\mathrm{Alg}}p\mathbb{I}_{\mathrm{BrLie}}^{s}\left(
M,c_{M},\left[ -\right] _{M}\right)
}&&H_{\mathrm{Alg}}\mathcal{U}_{\mathrm{Br}}\mathbb{I}_{\mathrm{BrLie}}^{s}%
\left( M,c,\left[ -\right] \right) }.  \label{diag:p}
\end{equation}
which can be checked to be a coequalizer in $\mathrm{Alg}_{\mathcal{M}}$. By
Lemma \ref{lem:HAlgRefCo} we have that $H_{\mathrm{Alg}}$ reflects
coequalizers and by \cite[Proposition 2.9.9]{Borceux1}, we have that $%
\mathbb{I}_{\mathrm{BrAlg}}^{s}$ reflects coequalizers. Thus (\ref{diag:ps})
is also a coequalizer. By Lemma \ref{lem:BBialgWpres}, since $F$ preserves
coequalizers, we get that $\mathrm{Alg}F$ preserves the coequalizer (\ref%
{diag:p}). Denote by $\mathrm{Alg}F(\ref{diag:p})$ the coequalizer obtained
in this way. Now, with the same notation, $\mathrm{Alg}F(\ref{diag:p})$ can
also be obtained as $H_{\mathrm{Alg}}\mathbb{I}_{\mathrm{BrAlg}}^{s}(\mathrm{%
BrAlg} ^{s}F)\eqref{diag:ps}$ (this is straightforward). Since we already
observed that both $H_{\mathrm{Alg}}$ and $\mathbb{I}_{\mathrm{BrAlg}}^{s}$
reflect coequalizers, we deduce that $(\mathrm{BrAlg}^{s}F)\eqref{diag:ps}$
is a coequalizer too. This coequalizer appears in the second line of the
diagram
\begin{equation*}
\xymatrixcolsep{35pt}\xymatrix{ T_{\mathrm{Br}}^{s}\left( FM\otimes
FM,c_{FM\otimes FM}\right)\ar@{.>}[d]
\ar@<.5ex>[rr]^-{T_{\mathrm{Br}}^{s}\left[ -\right] _{FM}^{s} }
\ar@<-.5ex>[rr]_-{\epsilon _{\mathrm{Br}}^{s}T_{\mathrm{Br}}^{s}\left(
FM,c_{FM}\right) \circ T_{\mathrm{Br}}^{s}\theta _{\left( FM,c_{FM}\right)
}^{s} }&&T_{\mathrm{Br}}^{s}H_{\mathrm{BrLie}}^{s}\left(
\mathrm{BrLie}^{s}F\right) \overline{M}\ar[d]^{\zeta
_{\mathrm{Br}}^{s}H_{\mathrm{BrLie}}^{s}\overline{M}}\ar[rr]^-{p^{s}\left(
\mathrm{BrLie}^{s}F\right) \overline{M}
}&&\mathcal{U}_{\mathrm{Br}}^{s}\left( \mathrm{BrLie}^{s}F\right)
\overline{M} \ar[d]^{\zeta _{\mathrm{BrL}}^{s}\overline{M}} \\ \overline{F}
T_{\mathrm{Br}}^{s}\left( M\otimes M,c_{M\otimes M}\right) \ar@<.5ex>[rr]^-{
\overline{F} T_{\mathrm{Br}}^{s}\left[ -\right] ^{s}}
\ar@<-.5ex>[rr]_-{\overline{F} \left( \epsilon
_{\mathrm{Br}}^{s}T_{\mathrm{Br}}^{s}\left( M,c_{M}\right) \circ
T_{\mathrm{Br}}^{s}\theta _{\left( M,c_{M}\right) }^{s}\right) }&&
\overline{F} T_{\mathrm{Br}}^{s}\left( M,c_{M}\right)\ar[rr]^-{ \overline{F}
p^{s}\overline{M} }&& \overline{F} \mathcal{U}_{\mathrm{Br}}^{s}\overline{M}
}
\end{equation*}%
where, for sake of shortness, we set $\overline{M}:=\left( M,c_{M},\left[ -%
\right] _{M}\right) $ and $\overline{F}:=\mathrm{BrAlg}^{s}F$. One proves
that the morphism
\begin{gather*}
T_{\mathrm{Br}}^{s}\left( FM\otimes FM,c_{FM\otimes FM}\right) \overset{T_{%
\mathrm{Br}}^{s}\phi _{2}\left( M,M\right) }{\longrightarrow }T_{\mathrm{Br}%
}^{s}\left( F\left( M\otimes M\right) ,c_{F\left( M\otimes M\right) }\right)
= \\
=T_{\mathrm{Br}}^{s}\left( \mathrm{Br}^{s}F\right) \left( M\otimes
M,c_{M\otimes M}\right) \overset{\zeta _{\mathrm{Br}}^{s}\left( M\otimes
M,c_{M\otimes M}\right) }{\longrightarrow }\left( \mathrm{BrAlg}^{s}F\right)
T_{\mathrm{Br}}^{s}\left( M\otimes M,c_{M\otimes M}\right)
\end{gather*}

is an isomorphism (we just point out that, as one easily checks, the
morphism $\phi _{2}\left( M,M\right) $ is a braided morphism so that the
morphism above is well-defined) and it completes the diagram above on the
left making it a serially commutative diagram. The fact it is serially
commutative depends on the following equality that can be easily checked%
\begin{equation}
\mathbb{I}_{\mathrm{BrAlg}}^{s}\zeta _{\mathrm{Br}}^{s}=\zeta _{\mathrm{Br}}%
\mathbb{I}_{\mathrm{Br}}^{s}.  \label{form:zetas}
\end{equation}
Now, by (\ref{form:zetas}) we have $\mathbb{I}_{\mathrm{BrAlg}}^{s}\zeta _{%
\mathrm{Br}}^{s}=\zeta _{\mathrm{Br}}\mathbb{I}_{\mathrm{Br}}^{s}.$ On the
other hand, by Proposition \ref{pro:comdat1} (here we use the fact that $F$
preserves denumerable coproducts), we know that $\zeta _{\mathrm{Br}}$ is a
functorial isomorphism. Since $\mathbb{I}_{\mathrm{BrAlg}}^{s}$ is
conservative, we deduce that $\zeta _{\mathrm{Br}}^{s}$ is a functorial
isomorphism. Thus, by the uniqueness of coequalizers (note that the first
line in the diagram above is just (\ref{diag:ps}) applied to $\left( \mathrm{%
BrLie}^{s}F\right) \left( M,c_{M},\left[ -\right] _{M}\right) =\left(
FM,c_{FM},\left[ -\right] _{FM}\right) $ instead of $\left( M,c_{M},\left[ -%
\right] _{M}\right) )$, we get that $\zeta _{\mathrm{BrL}}^{s}\left( M,c,%
\left[ -\right] \right) $ is an isomorphism too. Thus $\zeta _{\mathrm{BrL}%
}^{s}$ is a functorial isomorphism.

By (\ref{form:zetaBrLBars}) we have $\mho _{\mathrm{Br}}^{s}\overline{\zeta }%
_{\mathrm{BrL}}^{s}=\zeta _{\mathrm{BrL}}^{s}$ so that $\overline{\zeta }_{%
\mathrm{BrL}}^{s}$ is a functorial isomorphism too.
\end{proof}

\begin{theorem}
\label{teo:mainNew}Let $\mathcal{M}$ be an abelian monoidal category with
denumerable coproducts and such that the tensor functors are exact and
preserve denumerable coproducts. Let $\mathcal{N}$ be an MM-category and
assume that there exists a conservative (see \ref{def:conservative}) and
exact monoidal functor $\left( F,\phi _{0},\phi _{2}\right) :\mathcal{M}%
\rightarrow \mathcal{N}$ which preserves denumerable coproducts. Then $%
\mathcal{M}$ is an MM-category.
\end{theorem}

\begin{proof}
By Theorem \ref{teo:Udata}, we have the following commutation datum%
\begin{equation*}
\left( \mathrm{BrBialg}^{s}F,\mathrm{BrLie}^{s}F\right) :\left( \overline{%
\mathcal{U}}_{\mathrm{Br}}^{s},\mathcal{P}_{\mathrm{Br}}^{s}\right)
\rightarrow \left( \overline{\mathcal{U}}_{\mathrm{Br}}^{s},\mathcal{P}_{%
\mathrm{Br}}^{s}\right) .
\end{equation*}%
By Lemma \ref{lem:BrLieF}, we know that $\mathrm{BrLie}^{s}F$ is
conservative as $F$ is. By Lemma \ref{lem:faithcom}, we have that the unit $%
\overline{\eta }_{\mathrm{BrL}}^{s}:\mathrm{Id}_{\mathrm{BrLie}_{\mathcal{M}%
}^{s}}\rightarrow \mathcal{P}_{\mathrm{Br}}^{s}\overline{\mathcal{U}}_{%
\mathrm{Br}}^{s}$ is a functorial isomorphism.
\end{proof}

\begin{theorem}
\label{teo:LieSymNew}Let $\mathfrak{M}$ be the category of vector spaces
over a field $\Bbbk $ with $\mathrm{char}\Bbbk =0.$ Let $\mathcal{M}$ be an
abelian monoidal category with denumerable coproducts, such that the tensor
functors are exact and preserve denumerable coproducts. Assume that there
exists a conservative and exact monoidal functor $\left( F,\phi _{0},\phi
_{2}\right) :\mathcal{M}\rightarrow \mathfrak{M}$ which preserves
denumerable coproducts. Then $\mathcal{M}$ is an MM-category.
\end{theorem}

\begin{proof}
By Theorem \ref{teo:mainVS} we have $\mathfrak{M}$ is an MM-category. We
conclude by Theorem \ref{teo:mainNew}.
\end{proof}

\section{Examples of MM-categories}

\begin{example}
\label{ex:YD} Let $\Bbbk $ be a field with $\mathrm{char}\left( \Bbbk
\right) =0$. Let $H$ be any Hopf algebra over $\Bbbk $ of and consider the
monoidal category of Yetter-Drinfeld modules $\left( {_{H}^{H}\mathcal{YD}}%
,\otimes _{\Bbbk },\Bbbk \right) $. Then the forgetful functor $F:\left( {%
_{H}^{H}\mathcal{YD}},\otimes _{\Bbbk },\Bbbk \right) \rightarrow \left(
\mathfrak{M},\otimes _{\Bbbk },\Bbbk \right) $ is monoidal. One can prove by
hand that ${_{H}^{H}\mathcal{YD}}$ is abelian with denumerable coproducts.
The tensor functors are clearly exact and preserve denumerable coproducts in
${_{H}^{H}\mathcal{YD}}$ as this is the case in $\mathfrak{M}$. Furthermore $%
F$ is clearly conservative and exact and preserves denumerable coproducts.
By \ref{teo:LieSymNew}, we conclude that $\left( {_{H}^{H}\mathcal{YD}}%
,\otimes _{\Bbbk },\Bbbk \right) $ is an MM-category. Note that, by Theorem
\cite{Pa}, this category, with respect to its standard pre-braiding, is not
symmetric unless $H=\Bbbk $.
\end{example}

\subsection{Quasi-Bialgebras\label{subs:Quasi-Bialgebras}}

The following definition is not the original one given in \cite[page 1421]%
{Drinfeld-QuasiHopf}. We adopt the more general form of \cite[Remark 1, page
1423]{Drinfeld-QuasiHopf} (see also \cite[Proposition XV.1.2]{Kassel}) in
order to comprise the case of Hom-Lie algebras. Later on, for dual
quasi-bialgebras, we will take the simplified respective definition from the
very beginning having no meaningful example to treat in the full generality.

\begin{definition}
A \emph{quasi-bialgebra} is a datum $\left( H,m,u,\Delta ,\varepsilon ,\phi
,\lambda ,\rho \right) $ where $\left( H,m,u\right) $ is an associative
algebra, $\Delta :H\rightarrow H\otimes H$ and $\varepsilon :H\rightarrow
\Bbbk $ are algebra maps, $\lambda ,\rho \in H$ are invertible elements, $%
\phi \in H\otimes H\otimes H$ is a counital $3$-cocycle i.e. it is an
invertible element and satisfies%
\begin{eqnarray*}
\left( H\otimes H\otimes \Delta \right) \left( \phi \right) \cdot \left(
\Delta \otimes H\otimes H\right) \left( \phi \right) &=&\left( 1_{H}\otimes
\phi \right) \cdot \left( H\otimes \Delta \otimes H\right) \left( \phi
\right) \cdot \left( \phi \otimes 1_{H}\right) , \\
\left( H\otimes \varepsilon \otimes H\right) \left( \phi \right) &=&\rho
\otimes \lambda ^{-1}.
\end{eqnarray*}%
Moreover $\Delta $ is required to be quasi-coassociative and counitary i.e.
to satisfy%
\begin{gather*}
\left( H\otimes \Delta \right) \Delta \left( h\right) =\phi \cdot \left(
\Delta \otimes H\right) \Delta \left( h\right) \cdot \phi ^{-1}, \\
\left( \varepsilon \otimes H\right) \Delta \left( h\right) =\lambda
^{-1}h\lambda ,\qquad \left( H\otimes \varepsilon \right) \Delta \left(
h\right) =\rho ^{-1}h\rho .
\end{gather*}%
A morphism of quasi-bialgebras $\Xi :\left( H,m,u,\Delta ,\varepsilon ,\phi
,\lambda ,\rho \right) \rightarrow \left( H^{\prime },m^{\prime },u^{\prime
},\Delta ^{\prime },\varepsilon ^{\prime },\phi ^{\prime },\lambda ^{\prime
},\rho ^{\prime }\right) $ (see \cite[page 371]{Kassel}) is an algebra
homomorphism $\Xi :\left( H,m,u\right) \rightarrow \left( H^{\prime
},m^{\prime },u^{\prime }\right) $ such that $(\Xi \otimes \Xi )\Delta
=\Delta ^{\prime }\Xi ,$ $\varepsilon ^{\prime }\Xi =\varepsilon ,$ $\left(
\Xi \otimes \Xi \otimes \Xi \right) \left( \phi \right) =\phi ^{\prime },$ $%
\Xi \left( \lambda \right) =\lambda ^{\prime }$ and $\Xi \left( \rho \right)
=\rho ^{\prime }.$ It is an isomorphism of quasi-bialgebras if, in addition,
it is invertible. We will adopt the standard notation%
\begin{equation*}
\phi ^{1}\otimes \phi ^{2}\otimes \phi ^{3}:=\phi \,\text{(summation
understood).}
\end{equation*}%
In the case when $\phi $ is not trivial and $\lambda =\rho =1_{H},$ we call $%
H$ an \emph{ordinary quasi-bialgebra. }If further $\phi $ is trivial we then
land at the classical concept of bialgebra.

A quasi-subbialgebra of a quasi-bialgebra $H^{\prime }$ is a quasi-bialgebra
$H$ such that $H$ is a vector subspace of $H^{\prime }$ and the canonical
inclusion is a morphism of dual quasi-bialgebras.
\end{definition}

Let $(H,m,u,\Delta ,\varepsilon ,\phi ,\lambda ,\rho )$ be a
quasi-bialgebra. It is well-known, see \cite[page 285 and Proposition XV.1.2]%
{Kassel}, that the category ${_{H}}\mathfrak{M}$ of left $H$-modules becomes
a monoidal category as follows. Given a left $H$-module $V$, we denote by $%
\mu =\mu _{V}^{l}:H\otimes V\rightarrow V,\mu (h\otimes v)=hv$, its left $H$%
-action. The tensor product of two left $H$-modules $V$ and $W$ is a module
via diagonal action i.e. $h\left( v\otimes w\right) =h_{1}v\otimes h_{2}w.$
The unit is $\Bbbk ,$ which is regarded as a left $H$-module via the trivial
action i.e. $hk=\varepsilon \left( h\right) k$, for all $h\in H,k\in \Bbbk $%
. The associativity and unit constraints are defined, for all $V,W,Z\in {_{H}%
}\mathfrak{M}$ and $v\in V,w\in W,z\in Z,$ by $a_{V,W,Z}((v\otimes w)\otimes
z):=\phi ^{1}v\otimes (\phi ^{2}w\otimes \phi ^{3}z)$, $l_{V}(1\otimes
v):=\lambda v$ and $r_{V}(v\otimes 1):=\rho v.$ This monoidal category will
be denoted by $({_{H}}\mathfrak{M},\otimes ,\Bbbk ,a,l,r).$ Given an
invertible element $\alpha \in H\otimes H,$ we can construct a new
quasi-bialgebra $H_{\alpha }=\left( H,m,u,\Delta _{\alpha },\varepsilon
,\phi _{\alpha },\lambda _{\alpha },\rho _{\alpha }\right) $ where
\begin{eqnarray*}
\Delta _{\alpha }\left( h\right) &=&\alpha \cdot \Delta \left( h\right)
\cdot \alpha ^{-1},\quad \lambda _{\alpha }=\lambda \cdot \left( \varepsilon
_{H}\otimes H\right) \left( \alpha ^{-1}\right) ,\quad \rho _{\alpha }=\rho
\cdot \left( H\otimes \varepsilon _{H}\right) \left( \alpha ^{-1}\right) , \\
\phi _{\alpha } &=&\left( 1_{H}\otimes \alpha \right) \cdot \left( H\otimes
\Delta \right) \left( \alpha \right) \cdot \phi \cdot \left( \Delta \otimes
H\right) \left( \alpha ^{-1}\right) \cdot \left( \alpha ^{-1}\otimes
1_{H}\right) .
\end{eqnarray*}

\begin{definition}
We refer to \cite[Proposition XV.2.2]{Kassel} but with a different
terminology (cf. \cite[page 1439]{Drinfeld-QuasiHopf}). A quasi-bialgebra $%
\left( H,m,u,\Delta ,\varepsilon ,\phi ,\lambda ,\rho \right) $ is called
\emph{quasi-triangular} whenever there exists an invertible element $R\in
H\otimes H$ such that, for every $h\in H,$ one has%
\begin{eqnarray*}
\left( \Delta \otimes H\right) \left( R\right) &=&\left[
\begin{array}{c}
\left( \phi ^{2}\otimes \phi ^{3}\otimes \phi ^{1}\right) \left(
R^{1}\otimes 1\otimes R^{2}\right) \left( \phi ^{1}\otimes \phi ^{3}\otimes
\phi ^{2}\right) ^{-1} \\
\left( 1\otimes R^{1}\otimes R^{2}\right) \left( \phi ^{1}\otimes \phi
^{2}\otimes \phi ^{3}\right)%
\end{array}%
\right] \\
\left( H\otimes \Delta \right) \left( R\right) &=&\left[
\begin{array}{c}
\left( \phi ^{3}\otimes \phi ^{1}\otimes \phi ^{2}\right) ^{-1}\left(
R^{1}\otimes 1\otimes R^{2}\right) \left( \phi ^{2}\otimes \phi ^{1}\otimes
\phi ^{3}\right) \\
\left( R^{1}\otimes R^{2}\otimes 1\right) \left( \phi ^{1}\otimes \phi
^{2}\otimes \phi ^{3}\right) ^{-1}%
\end{array}%
\right] \\
\Delta ^{cop}\left( h\right) &=&R\Delta \left( h\right) R^{-1}
\end{eqnarray*}%
where $\phi :=\phi ^{1}\otimes \phi ^{2}\otimes \phi ^{3},$ $R=R^{1}\otimes
R^{2}.$ A \emph{morphism of quasi-triangular quasi-bialgebras} is a morphism
$\Xi :H\rightarrow H^{\prime }$ of quasi-bialgebras such that $\left( \Xi
\otimes \Xi \right) \left( R\right) =R^{\prime }.$
\end{definition}

By \cite[Proposition XV.2.2]{Kassel}, ${_{H}}\mathfrak{M}=({_{H}}\mathfrak{M}%
,\otimes ,\Bbbk ,a,l,r)$ is braided if and only if there is an invertible
element $R\in H\otimes H$ such that $\left( H,m,u,\Delta ,\varepsilon ,\phi
,\lambda ,\rho ,R\right) $ is quasi-triangular. Note that the braiding is
given, for all $X,Y\in {_{H}}\mathfrak{M},$ by%
\begin{equation*}
c_{X,Y}:X\otimes Y\rightarrow Y\otimes X:x\otimes y\mapsto R^{2}y\otimes
R^{1}x.
\end{equation*}%
Moreover ${_{H}}\mathfrak{M}$ is symmetric if and only if we further assume $%
R^{2}\otimes R^{1}=R^{-1}.$ Such a quasi-bialgebra will be called a \emph{%
triangular quasi-bialgebra}. A morphism of triangular quasi-bialgebras is
just a morphism of the underlying quasi-triangular quasi-bialgebras
structures.

Given an invertible element $\alpha \in H\otimes H,$ if $H$ is
(quasi-)triangular so is $H_{\alpha }$ with respect to $R_{\alpha }=\left(
\alpha ^{2}\otimes \alpha ^{1}\right) R\alpha ^{-1}$, where $\alpha :=\alpha
^{1}\otimes \alpha ^{2}.$

Let $\left( H,m,u,\Delta ,\varepsilon ,\phi ,\lambda ,\rho \right) $ be a
quasi-bialgebra. We want to apply Theorem \ref{teo:LieSymNew} to the case $%
\mathcal{M}={_{H}}\mathfrak{M}.$ Let $F:{_{H}}\mathfrak{M}\rightarrow
\mathfrak{M}$ be the forgetful functor. We need a monoidal $\left( F,\psi
_{0},\psi _{2}\right) :({_{H}}\mathfrak{M},\otimes ,\Bbbk ,a,l,r)\rightarrow
\mathfrak{M.}$

\begin{lemma}
\label{lem:gamma}Let $\left( H,m,u,\Delta ,\varepsilon ,\phi ,\lambda ,\rho
\right) $ be a quasi-bialgebra. Let $F:{_{H}}\mathfrak{M}\rightarrow
\mathfrak{M}$ be the forgetful functor. The following are equivalent.

$\left( 1\right) $ There is a natural transformation $\psi _{2}$ such that $%
\left( F,\mathrm{Id}_{\Bbbk },\psi _{2}\right) :({_{H}}\mathfrak{M},\otimes
,\Bbbk ,a,l,r)\rightarrow \mathfrak{M}$ is monoidal.

$\left( 2\right) $ There is an invertible element $\alpha \in H\otimes H$
such that $H_{\alpha }$ is an ordinary bialgebra.

$\left( 3\right) $ There is an invertible element $\alpha \in H\otimes H$
such that
\begin{gather}
\phi =\left( H\otimes \Delta \right) \left( \alpha ^{-1}\right) \cdot \left(
1_{H}\otimes \alpha ^{-1}\right) \cdot \left( \alpha \otimes 1_{H}\right)
\cdot \left( \Delta \otimes H\right) \left( \alpha \right) ,
\label{form:triavial1} \\
\left( \varepsilon _{H}\otimes H\right) \left( \alpha \right) =\lambda
,\qquad \left( H\otimes \varepsilon _{H}\right) \left( \alpha \right) =\rho .
\label{form:triavial2}
\end{gather}%
Moreover, if $\left( 2\right) $ holds, we can choose $\psi _{2}\left(
V,W\right) \left( v\otimes w\right) :=\alpha ^{-1}\left( v\otimes w\right) $.
\end{lemma}

\begin{proof}
$\left( 1\right) \Leftrightarrow \left( 2\right) $ Cf. \cite[Proposition 1]%
{ABM-HomLie}. $\left( 2\right) \Leftrightarrow \left( 3\right) $ We have
that $H_{\alpha }$ is an ordinary bialgebra if and only if $\phi _{\alpha
}=1_{H}\otimes 1_{H}\otimes 1_{H},\lambda _{\alpha }=1_{H}$ and $\rho
_{\alpha }=1_{H},$ if and only if $\alpha $ fulfills the equations in $%
\left( 3\right) $.
\end{proof}

Note that $F:{_{H}}\mathfrak{M}\rightarrow \mathfrak{M}$ is clearly
conservative and preserves equalizers, epimorphisms and coequalizers.
Furthermore we need ${_{H}}\mathfrak{M}$ to be braided.

\begin{theorem}
\label{teo:LieQuasiTriang}Let $\left( H,m,u,\Delta ,\varepsilon ,\phi
,\lambda ,\rho \right) $ be a quasi-bialgebra such that (\ref{form:triavial1}%
) and (\ref{form:triavial2}) hold for some invertible element $\gamma \in
H\otimes H$. Let $\mathcal{M}$ be the monoidal category $({_{H}}\mathfrak{M}%
,\otimes ,\Bbbk ,a,l,r)$ of left modules over $H$. Assume $\mathrm{char}%
\Bbbk =0.$ Then $\mathcal{M}$ is an MM-category. In particular, if $\left(
H,m,u,\Delta ,\varepsilon ,\phi ,\lambda ,\rho \right) $ is endowed with a
triangular structure, then $\mathcal{M}$ is a symmetric MM-category.
\end{theorem}

\begin{proof}
First note that $\mathcal{M}$ is a Grothendieck category. In $\mathcal{M}$
the tensor products are exact and preserve denumerable coproducts. We can
apply Theorem \ref{teo:LieSymNew} to the monoidal functor $\left( F,\mathrm{%
Id}_{\Bbbk },\psi _{2}\right) :({_{H}}\mathfrak{M},\otimes ,\Bbbk
,a,l,r)\rightarrow \mathfrak{M}$ of Lemma \ref{lem:gamma}. Then $\mathcal{M}$
is an MM-category.

If $\left( H,m,u,\Delta ,\varepsilon ,\phi ,\lambda ,\rho \right) $ is
endowed with a triangular structure, by the foregoing $\mathcal{M}$ is also
symmetric monoidal.
\end{proof}

\begin{example}
Let $H$ be a bialgebra over a field $\Bbbk $ of characteristic zero. Then $H$
is a quasi-bialgebra with $\phi ,\lambda ,\rho $ trivial. Note that (\ref%
{form:triavial1}) and (\ref{form:triavial2}) hold for $\gamma =\varepsilon
_{H}\otimes \varepsilon _{H}.$ Thus, by Theorem \ref{teo:LieQuasiTriang},
the monoidal category ${_{H}}\mathfrak{M}$ of left modules over $H$ is an
MM-category.
\end{example}

\begin{example}
Examples of triangular quasi-bialgebra structures on the group algebra $%
\Bbbk \left[ G\right] $ over a torsion-free abelian group $G$ are
investigated in \cite[Proposition 3]{ABM-HomLie}. Consider the particular
case when $G=\left\langle g\right\rangle $ is the group $\mathbb{Z}$ in
multiplicative notation, where $g$ is a generator. Let $\left( q,a,b\right)
\in (\Bbbk \backslash \{0\})\times \mathbb{Z}\times \mathbb{Z}.$ In view of
\cite[Proposition 3]{ABM-HomLie}, we have the triangular quasi-bialgebra
\begin{equation*}
\Bbbk \lbrack \left\langle g\right\rangle ]_{q}^{a,b}=\left( \Bbbk \lbrack
\left\langle g\right\rangle ],\Delta ,\varepsilon ,\phi ,\lambda ,\rho
,R\right)
\end{equation*}%
on the group algebra $\Bbbk \lbrack \left\langle g\right\rangle ]$ which is
defined by
\begin{eqnarray*}
\Delta \left( g\right) &=&g\otimes g,\qquad \varepsilon \left( g\right)
=1,\qquad \phi =g^{a}\otimes 1_{H}\otimes g^{b} \\
\lambda &=&qg^{-b},\qquad \rho =qg^{a},\qquad R=g^{a+b}\otimes g^{-a-b}.
\end{eqnarray*}%
In order to apply Theorem \ref{teo:LieQuasiTriang} in case $H=\Bbbk \lbrack
\left\langle g\right\rangle ]_{q}^{a,b}$, we must check that (\ref%
{form:triavial1}) and (\ref{form:triavial2}) hold for some invertible
element $\gamma \in H\otimes H$. By \cite[Theorem 2]{ABM-HomLie}, one has $%
\Bbbk \lbrack \left\langle g\right\rangle ]_{q}^{a,b}=\Bbbk \lbrack
\left\langle g\right\rangle ]_{\alpha }$ where $\alpha :=q^{-1}g^{-a}\otimes
g^{b}$ and $\Bbbk \lbrack \left\langle g\right\rangle ]$ is the usual
bialgebra structure on the group algebra regarded as a trivial triangular
quasi-bialgebras (i.e. $\phi ,\lambda ,\rho ,R$ are all trivial). Set $%
\gamma :=\alpha ^{-1}=qg^{a}\otimes g^{-b}.$ Then $H_{\gamma }=\Bbbk
\left\langle g\right\rangle $ which is an ordinary bialgebra so that, by
Lemma \ref{lem:gamma} we have that (\ref{form:triavial1}) and (\ref%
{form:triavial2}) hold for our $\gamma $. Hence by Theorem \ref%
{teo:LieQuasiTriang} the symmetric monoidal category $({_{H}}\mathfrak{M}%
,\otimes ,\Bbbk ,a,l,r)$ of left modules over $H$ is an MM-category.
\end{example}

\begin{definition}
Let $\mathcal{C}$ be an ordinary category. Following \cite[Section 1]{CG},
we associate to $\mathcal{C}$ a new category $\mathcal{H}\left( \mathcal{C}%
\right) $ as follows. Objects are pairs $\left( M,f_{M}\right) $ with $M\in
\mathcal{C}$ and $f_{M}\in \mathrm{Aut}_{\mathcal{C}}(M).$ A morphism $\xi
:\left( M,f_{M}\right) \rightarrow \left( N,f_{N}\right) $ is a morphism $%
\xi :M\rightarrow N$ in $\mathcal{C}$ such that $f_{N}\circ \xi =\xi \circ
f_{M}.$ The category $\mathcal{H}\left( \mathcal{C}\right) $ is called the%
\emph{\ Hom-category} associated to $\mathcal{C}$.
\end{definition}

\begin{example}
\label{ex:HH} Take $\mathcal{C}:=\mathfrak{M.}$ In view of \cite[Theorem 4]%
{ABM-HomLie}, to each datum $\left( q,a,b\right) \in (\Bbbk \backslash
\{0\})\times \mathbb{Z}\times \mathbb{Z}$ one associates a monoidal category
\begin{equation*}
\mathcal{H}_{q}^{a,b}(\mathfrak{M})=(\mathcal{H}(\mathfrak{M}),\otimes
,(\Bbbk ,f_{\Bbbk }),a,l,r)
\end{equation*}%
which consists of the category $\mathcal{H}(\mathfrak{M})$ equipped with a
suitable braided (actually symmetric) monoidal structure. By \cite[Theorem 4]%
{ABM-HomLie} there is a strict symmetric monoidal category isomorphism
\begin{equation*}
\left( W,w_{0},w_{2}\right) :{_{\Bbbk \lbrack \left\langle g\right\rangle
]_{q}^{a,b}}}\mathfrak{M}\rightarrow \mathcal{H}_{q}^{a,b}\left( \mathfrak{M}%
\right) .
\end{equation*}%
The underlying functor $W:{_{\Bbbk \lbrack \langle g\rangle ]}}\mathfrak{M}%
\rightarrow \mathcal{H}\left( \mathfrak{M}\right) $ is given on objects by%
\begin{equation*}
W\left( X,\mu _{X}:\Bbbk \lbrack \left\langle g\right\rangle ]\otimes
X\rightarrow X\right) =\left( X,f_{X}:X\rightarrow X\right) ,
\end{equation*}%
where $f_{X}\left( x\right) :=\mu _{X}\left( g\otimes x\right) ,$ for all $%
x\in X$, and on morphisms by $W\xi =\xi $.

Composing $W^{-1}$ with the forgetful functor ${_{\Bbbk \lbrack \left\langle
g\right\rangle ]_{q}^{a,b}}}\mathfrak{M}\rightarrow \mathfrak{M}$ we get a
monoidal functor $\mathcal{H}_{q}^{a,b}\left( \mathfrak{M}\right)
\rightarrow \mathfrak{M}$ to which we can apply Theorem \ref{teo:LieSymNew}
to get that $\mathcal{H}_{q}^{a,b}\left( \mathfrak{M}\right) $ is an
MM-category.
\end{example}

\begin{remark}
\label{rem:HH} By \cite[Proposition 5]{ABM-HomLie}, $\mathcal{M}:=\mathcal{H}%
_{1}^{1,-1}\left( \mathfrak{M}\right) $ is the symmetric braided monoidal
category $\widetilde{\mathcal{H}}\left( \mathfrak{M}\right) $ of \cite[%
Proposition 1.1]{CG}. Thus, by the foregoing, $\widetilde{\mathcal{H}}\left(
\mathfrak{M}\right) $ is an MM-category. By \cite[page 2236]{CG}, an object
in $\left( M,\left[ -\right] \right) \in \mathrm{Lie}_{\mathcal{M}}$ is
nothing but a Hom-Lie algebra. By Remark \ref{rem:UbarNature}, $\overline{%
\mathcal{U}}\left( M,\left[ -\right] \right) $ as a bialgebra is a quotient
of $\overline{T}M$. The morphism giving the projection is induced by the
canonical projection $p_{R}:\Omega TM\rightarrow R:=\mathcal{U}_{\mathrm{Br}%
}^{s}J_{\mathrm{Lie}}\left( M,\left[ -\right] \right) $ defining the
universal enveloping algebra. At algebra level we have%
\begin{eqnarray*}
\mho \overline{\mathcal{U}}\left( M,\left[ -\right] \right) &=&\frac{TM}{%
\left( f_{J_{\mathrm{Lie}}^{s}\left( M,\left[ -\right] \right) }\left(
x\otimes y\right) |x,y\in M\right) }=\frac{TM}{\left( \left[ x,y\right]
-\theta _{\left( M,c_{M,M}\right) }\left( x\otimes y\right) |x,y\in M\right)
} \\
&=&\frac{TM}{\left( \left[ x,y\right] -x\otimes y+c_{M,M}\left( x\otimes
y\right) |x,y\in M\right) }=\frac{TM}{\left( \left[ x,y\right] -x\otimes
y+y\otimes x|x,y\in M\right) }
\end{eqnarray*}%
which is the Hom-version of the universal enveloping algebra, see \cite[%
Section 8]{CG}. Note that, as a by-product, we have that $\overline{\eta }_{%
\mathrm{L}}:\mathrm{Id}_{\mathrm{Lie}_{\mathcal{M}}}\rightarrow \mathcal{P}%
\overline{\mathcal{U}}$ is an isomorphism so that $\left( M,\left[ -\right]
\right) \cong \mathcal{P}\overline{\mathcal{U}}\left( M,\left[ -\right]
\right) .$
\end{remark}

\subsection{Dual Quasi-Bialgebras\label{subs:DualQuasiBialgebras}}

First, observe that dual quasi-bialgebras can be understood as a dual
version of quasi-bialgebras just in the finite-dimensional case. In fact,
for an infinite-dimensional quasi-bialgebra $H$ (as in the case for $H=\Bbbk
\mathbb{Z}$ considered above) it is not true that the dual is a dual
quasi-bialgebra so that the results in the two settings are independent, in
general.

\begin{definition}
A \emph{dual quasi-bialgebra} is a datum $(H,m,u,\Delta ,\varepsilon ,\omega
)$ where $(H,\Delta ,\varepsilon )$ is a coassociative coalgebra, $%
m:H\otimes H\rightarrow H$ and $u:\Bbbk \rightarrow H$ are coalgebra maps
called multiplication and unit respectively, we set $1_{H}:=u(1_{\Bbbk })$, $%
\omega :H\otimes H\otimes H\rightarrow \Bbbk $ is a unital $3$-cocycle i.e.
it is convolution invertible and satisfies%
\begin{eqnarray}
\omega \left( H\otimes H\otimes m\right) \ast \omega \left( m\otimes
H\otimes H\right) &=&\left( \varepsilon \otimes \omega \right) \ast \omega
\left( H\otimes m\otimes H\right) \ast \left( \omega \otimes \varepsilon
\right)  \label{eq:3-cocycle} \\
\text{and}\quad \omega \left( h\otimes k\otimes l\right) &=&\varepsilon
\left( h\right) \varepsilon \left( k\right) \varepsilon \left( l\right)
\qquad \text{whenever}\qquad 1_{H}\in \{h,k,l\}.
\label{eq:qusi-unitairity cocycle}
\end{eqnarray}

Moreover $m$ is quasi-associative and unitary i.e. it satisfies%
\begin{gather*}
m\left( H\otimes m\right) \ast \omega =\omega \ast m\left( m\otimes H\right)
, \\
m\left( 1_{H}\otimes h\right) =h\qquad \text{and}\qquad m\left( h\otimes
1_{H}\right) =h,\qquad \text{for all }h\in H.
\end{gather*}%
The map $\omega $ is called \textit{the }\emph{reassociator} of the dual
quasi-bialgebra.

A morphism of dual quasi-bialgebras $\Xi :\left( H,m,u,\Delta ,\varepsilon
,\omega \right) \rightarrow \left( H^{\prime },m^{\prime },u^{\prime
},\Delta ^{\prime },\varepsilon ^{\prime },\omega ^{\prime }\right) $ is a
coalgebra homomorphism $\Xi :\left( H,\Delta ,\varepsilon \right)
\rightarrow \left( H^{\prime },\Delta ^{\prime },\varepsilon ^{\prime
}\right) $ such that $m^{\prime }(\Xi \otimes \Xi )=\Xi m,$ $\Xi u=u^{\prime
}$ and $\omega ^{\prime }\left( \Xi \otimes \Xi \otimes \Xi \right) =\omega
. $ It is an isomorphism of dual quasi-bialgebras if, in addition, it is
invertible.

A dual quasi-subbialgebra of a dual quasi-bialgebra $H^{\prime }$ is a
quasi-bialgebra $H$ such that $H$ is a vector subspace of $H^{\prime }$ and
the canonical inclusion is a morphism of dual quasi-bialgebras.
\end{definition}

Let $(H,m,u,\Delta ,\varepsilon ,\omega )$ be a dual quasi-bialgebra. It is
well-known that the category $\mathfrak{M}^{H}$ of right $H$-comodules
becomes a monoidal category as follows. Given a right $H$-comodule $V$, we
denote by $\rho =\rho _{V}^{r}:V\rightarrow V\otimes H,\rho (v)=v_{0}\otimes
v_{1}$, its right $H$-coaction. The tensor product of two right $H$%
-comodules $V$ and $W$ is a comodule via diagonal coaction i.e. $\rho \left(
v\otimes w\right) =v_{0}\otimes w_{0}\otimes v_{1}w_{1}.$ The unit is $\Bbbk
,$ which is regarded as a right $H$-comodule via the trivial coaction i.e. $%
\rho \left( \Bbbk \right) =\Bbbk \otimes 1_{H}$. The associativity and unit
constraints are defined, for all $U,V,W\in \mathfrak{M}^{H}$ and $u\in
U,v\in V,w\in W,\Bbbk \in \Bbbk ,$ by $a_{U,V,W}((u\otimes v)\otimes
w):=u_{0}\otimes (v_{0}\otimes w_{0})\omega (u_{1}\otimes v_{1}\otimes
w_{1}),$ $l_{U}(\Bbbk \otimes u):=\Bbbk u$ and $r_{U}(u\otimes \Bbbk
):=u\Bbbk .$ This monoidal category will be denoted by $(\mathfrak{M}%
^{H},\otimes ,\Bbbk ,a,l,r).$

Let $\left( H,m,u,\Delta ,\varepsilon ,\omega \right) $ be a dual
quasi-bialgebra. Let $v:H\otimes H\rightarrow \Bbbk $ be a \emph{gauge
transformation} i.e. a convolution invertible map such that $v\left(
1_{H}\otimes h\right) =\varepsilon \left( h\right) =v\left( h\otimes
1_{H}\right) $ for all $h\in H.$ Then $H^{v}:=\left( H,m^{v},u,\Delta
,\varepsilon ,\omega ^{v}\right) $ is also a dual quasi-bialgebra where
\begin{eqnarray}
m^{v} &:&=v\ast m\ast v^{-1}  \label{form: twisted mult} \\
\omega ^{v} &:&=\left( \varepsilon \otimes {v}\right) \ast {v}\left(
H\otimes m\right) \ast \omega \ast {v}^{-1}\left( m\otimes H\right) \ast
\left( {v}^{-1}\otimes \varepsilon \right) .  \label{form: reassociator}
\end{eqnarray}

\begin{definition}
A dual quasi-bialgebra $\left( H,m,u,\Delta ,\varepsilon ,\omega \right) $
is called \emph{quasi-co-triangular} whenever there exists $R\in \mathrm{Reg}%
\left( H^{\otimes 2},\Bbbk \right) $ such that%
\begin{align*}
R\left( m\otimes H\right) & =\left[
\begin{array}{c}
\omega \tau _{H\otimes H,H}\ast R\left( H\otimes l_{H}\right) \left(
H\otimes \varepsilon \otimes H\right) \\
\ast \omega ^{-1}\left( H\otimes \tau _{H,H}\right) \ast m_{\Bbbk }\left(
\varepsilon \otimes R\right) \ast \omega%
\end{array}%
\right] , \\
R\left( H\otimes m\right) & =\left[
\begin{array}{c}
\omega ^{-1}\tau _{H,H\otimes H}\ast R\left( H\otimes l_{H}\right) \left(
H\otimes \varepsilon \otimes H\right) \\
\ast \omega \left( \tau _{H,H}\otimes H\right) \ast m_{\Bbbk }\left(
R\otimes \varepsilon \right) \ast \omega ^{-1}%
\end{array}%
\right] , \\
m\tau _{H,H}\ast R& =R\ast m.
\end{align*}
\end{definition}

By \cite[Exercice 9.2.9, page 437]{Maj} \cite[dual to Proposition XIII.1.4,
page 318]{Kassel}, $\mathfrak{M}^{H}=\left( \mathfrak{M}^{H},\otimes _{\Bbbk
},\Bbbk ,a,l,r\right) $ is braided if and only if there is a map $R\in
\mathrm{Reg}\left( H^{\otimes 2},\Bbbk \right) $ such that $\left(
H,m,u,\Delta ,\varepsilon ,\omega ,R\right) $ is quasi-co-triangular. Note
that the braiding is given, for all $X,Y\in \mathfrak{M}^{H},$ by%
\begin{equation*}
c_{X,Y}:X\otimes Y\rightarrow Y\otimes X:x\otimes y\mapsto \sum
y_{\left\langle 0\right\rangle }\otimes x_{\left\langle 0\right\rangle
}R\left( x_{\left\langle 1\right\rangle }\otimes y_{\left\langle
1\right\rangle }\right) .
\end{equation*}%
Moreover $\mathfrak{M}^{H}$ is symmetric if and only if $c_{Y,X}\circ
c_{X,Y}=\mathrm{Id}_{X\otimes Y}$ for all $X,Y\in \mathfrak{M}^{H}$ i.e. if
and only if
\begin{equation*}
\sum x_{\left\langle 0\right\rangle }\otimes y_{\left\langle 0\right\rangle
}R\left( y_{\left\langle 1\right\rangle }\otimes x_{\left\langle
1\right\rangle }\right) R\left( x_{\left\langle 2\right\rangle }\otimes
y_{\left\langle 2\right\rangle }\right) =x\otimes y.
\end{equation*}%
This is equivalent to require that
\begin{equation}
R\left( h_{\left\langle 1\right\rangle }\otimes l_{\left\langle
1\right\rangle }\right) R\left( l_{\left\langle 2\right\rangle }\otimes
h_{\left\langle 2\right\rangle }\right) =\varepsilon _{H}\left( h\right)
\varepsilon _{H}\left( l\right) ,\text{ for every }h,l\in H.
\label{form:cotriang}
\end{equation}%
Such a dual quasi-bialgebra will be called a \emph{co-triangular dual
quasi-bialgebra}.

Let $(H,m,u,\Delta ,\varepsilon ,\omega )$ be a dual quasi-bialgebra. We
want to apply Theorem \ref{teo:LieSymNew} to the case $\mathcal{M}=\mathfrak{%
M}^{H}.$ We need a monoidal functor $\left( F,\phi _{0},\phi _{2}\right) :(%
\mathfrak{M}^{H},\otimes ,\Bbbk ,a,l,r)\rightarrow \mathfrak{M.}$ Take $F:%
\mathfrak{M}^{H}\rightarrow \mathfrak{M}$ to be the forgetful functor. Note
that $F$ is clearly conservative and preserves equalizers, epimorphisms and
coequalizers. Note also that we will further need $\mathfrak{M}^{H}$ to be
braided.

\begin{lemma}
\label{lem:gammaDual}Let $(H,m,u,\Delta ,\varepsilon ,\omega )$ be a dual
quasi-bialgebra. Let $F:\mathfrak{M}^{H}\rightarrow \mathfrak{M}$ be the
forgetful functor. The following are equivalent.

$\left( 1\right) $ There is a natural transformation $\psi _{2}$ such that $%
\left( F,\mathrm{Id}_{\Bbbk },\phi _{2}\right) :(\mathfrak{M}^{H},\otimes
,\Bbbk ,a,l,r)\rightarrow \mathfrak{M}$ is monoidal.

$\left( 2\right) $ There is a gauge transformation $v:H\otimes H\rightarrow
\Bbbk $ such that $H^{v}$ is an ordinary bialgebra.

$\left( 3\right) $ There is a gauge transformation $v:H\otimes H\rightarrow
\Bbbk $ such that
\begin{equation}
\omega ={v}^{-1}\left( H\otimes m\right) \ast \left( \varepsilon \otimes {v}%
^{-1}\right) \ast \left( {v}\otimes \varepsilon \right) \ast {v}\left(
m\otimes H\right)  \label{form:triavialOmega}
\end{equation}%
Moreover, if $\left( 2\right) $ holds, we can choose $\phi _{2}\left(
V,W\right) \left( x\otimes y\right) =x_{0}\otimes y_{0}v^{-1}\left(
x_{1}\otimes y_{1}\right) $.
\end{lemma}

\begin{proof}
It is similar to the one of Lemma \ref{lem:gamma}.
\end{proof}

\begin{lemma}
\label{lem:unitalQT}\cite[cf. Lemma 2.2.2]{Maj} Let $\left( H,m,u,\Delta
,\varepsilon ,\omega ,R\right) $ be a quasi-co-triangular dual
quasi-bialgebra. Then $R$ is unital i.e. $R\left( 1_{H}\otimes h\right)
=\varepsilon \left( h\right) =R\left( h\otimes 1_{H}\right) $ for all $h\in
H.$
\end{lemma}

\begin{theorem}
\label{teo:LieCoquasiTriang}Let $(H,m,u,\Delta ,\varepsilon ,\omega )$ be a
dual quasi-bialgebra such that $\omega $ fulfills (\ref{form:triavialOmega})
for some gauge transformation $\gamma :H\otimes H\rightarrow \Bbbk $. Let $%
\mathcal{M}$ be the monoidal category $\left( \mathfrak{M}^{H},\otimes
_{\Bbbk },\Bbbk ,a,l,r\right) $ of right comodules over $H$. Assume $\mathrm{%
char}\Bbbk =0.$ Then $\mathcal{M}$ is an MM-category. In particular, if $%
(H,m,u,\Delta ,\varepsilon ,\omega )$ is endowed with a co-triangular
structure, then $\mathcal{M}$ is a symmetric MM-category.
\end{theorem}

\begin{proof}
It is analogous to the proof of Theorem \ref{teo:LieQuasiTriang} but using,
from Lemma \ref{lem:gammaDual}, the functor $\left( F,\mathrm{Id}_{\Bbbk
},\phi _{2}\right) :(\mathfrak{M}^{H},\otimes ,\Bbbk ,a,l,r)\rightarrow
\mathfrak{M}$.
\end{proof}

\begin{example}
Let $H$ be a bialgebra over a field $\Bbbk $ of characteristic zero. Then $H$
is a dual quasi-bialgebra with reassociator $\omega =\varepsilon _{H}\otimes
\varepsilon _{H}\otimes \varepsilon _{H}$. Note that $\omega $ fulfills (\ref%
{form:triavialOmega}) for $\gamma =\varepsilon _{H}\otimes \varepsilon _{H}.$
Thus, by Theorem \ref{teo:LieCoquasiTriang}, the monoidal category $%
\mathfrak{M}^{H}$ of right comodules over $H$ is an MM-category. In
particular, for $H=\Bbbk \left[ \mathbb{N}\right] ,$ the monoid bialgebra
over the naturals, defined by taking $\Delta n=n\otimes n$ and $\varepsilon
\left( n\right) =1$ for every $n\in \mathbb{N}$, then the category $%
\mathfrak{M}^{H}$ is the category of $\mathbb{N}$-graded vector spaces $%
V=\oplus _{n\in \mathbb{N}}V_{n}$ with monoidal structure having tensor
product given by $\left( V\otimes W\right) _{n}=\oplus _{i=0}^{n}\left(
V_{i}\otimes W_{n-i}\right) $ and unit $\Bbbk $ concentrated in degree $0$.
The constraints are the same of vector spaces. The category $\mathfrak{M}%
^{H} $ is braided with respect to the canonical flip (this can be seen by
showing that $R=\varepsilon _{H}\otimes \varepsilon _{H}$ turns $H$ into a
co-triangular bialgebra, see remark below).
\end{example}

\begin{remark}
\label{rem:cotriang}Let $\left( H,m,u,\Delta ,\varepsilon ,\omega ,R\right) $
be a co-triangular dual quasi-bialgebra. Assume that $\omega $ fulfills (\ref%
{form:triavialOmega}) for $\gamma =\varepsilon _{H}\otimes \varepsilon _{H}.$
This means $\omega =\varepsilon _{H}\otimes \varepsilon _{H}\otimes
\varepsilon _{H}$ and $\left( H,m,u,\Delta ,\varepsilon ,R\right) $ is a
co-triangular bialgebra i.e. for every $x,y,z\in H$ we have%
\begin{gather*}
R\left( xy\otimes z\right) =R\left( x\otimes z_{1}\right) R\left( y\otimes
z_{2}\right) ,\quad R\left( x\otimes yz\right) =R\left( x_{1}\otimes
z\right) R\left( x_{2}\otimes y\right) , \\
y_{1}x_{1}R\left( x_{2}\otimes y_{2}\right) =R\left( x_{1}\otimes
y_{1}\right) x_{2}y_{2}.
\end{gather*}%
Let $\left( M,\left[ -\right] \right) \in \mathrm{Lie}_{\mathcal{M}}.$ Then %
\ref{Lie1} and \ref{Lie3} become%
\begin{gather*}
\left[ x,y\right] =-\sum \left[ y_{\left\langle 0\right\rangle
},x_{\left\langle 0\right\rangle }\right] R\left( x_{\left\langle
1\right\rangle }\otimes y_{\left\langle 1\right\rangle }\right) , \\
\sum \left[ \left[ x,y\right] ,z\right] +\sum \left[ \left[ y_{\left\langle
0\right\rangle },z_{\left\langle 0\right\rangle }\right] ,x_{\left\langle
0\right\rangle }\right] R\left( x_{\left\langle 1\right\rangle }\otimes
y_{\left\langle 1\right\rangle }z_{\left\langle 1\right\rangle }\right)
+\sum \left[ \left[ z_{\left\langle 0\right\rangle },x_{\left\langle
0\right\rangle }\right] ,y_{\left\langle 0\right\rangle }\right] R\left(
x_{\left\langle 1\right\rangle }y_{\left\langle 1\right\rangle }\otimes
z_{\left\langle 1\right\rangle }\right) =0.
\end{gather*}%
This means that $\left( M,\left[ -\right] \right) $ is an $\left( H,R\right)
$-Lie algebra in the sense of \cite[Definition 4.1]{BFM}. By Remark \ref%
{rem:UbarNature}, $\overline{\mathcal{U}}\left( M,\left[ -\right] \right) $
as a bialgebra is a quotient of $\overline{T}M$. The morphism giving the
projection is induced by the canonical projection $p_{R}:\Omega
TM\rightarrow R:=\mathcal{U}_{\mathrm{Br}}^{s}J_{\mathrm{Lie}}\left( M,\left[
-\right] \right) $ defining the universal enveloping algebra. At algebra
level we have%
\begin{gather*}
\mathcal{U}\left( M,\left[ -\right] \right) \overset{(\ref{diag:Ubar})}{=}%
\mho \overline{\mathcal{U}}\left( M,\left[ -\right] \right) =\frac{TM}{%
\left( f_{J_{\mathrm{Lie}}^{s}\left( M,\left[ -\right] \right) }\left(
x\otimes y\right) |x,y\in M\right) } \\
=\frac{TM}{\left( \left[ x,y\right] -\theta _{\left( M,c_{M,M}\right)
}\left( x\otimes y\right) |x,y\in M\right) }=\frac{TM}{\left( \left[ x,y%
\right] -x\otimes y+c_{M,M}\left( x\otimes y\right) |x,y\in M\right) } \\
=\frac{TM}{\left( \left[ x,y\right] -x\otimes y+\sum y_{\left\langle
0\right\rangle }\otimes x_{\left\langle 0\right\rangle }R\left(
x_{\left\langle 1\right\rangle }\otimes y_{\left\langle 1\right\rangle
}\right) |x,y\in M\right) }
\end{gather*}%
which is the universal enveloping algebra of our $\left( H,R\right) $-Lie
algebra, see e.g. \cite[(2.6)]{FM}. Note that, as a by-product, we have that
$\overline{\eta }_{\mathrm{L}}:\mathrm{Id}_{\mathrm{Lie}_{\mathcal{M}%
}}\rightarrow \mathcal{P}\overline{\mathcal{U}}$ is an isomorphism so that $%
\left( M,\left[ -\right] \right) \cong \mathcal{P}\overline{\mathcal{U}}%
\left( M,\left[ -\right] \right) .$
\end{remark}

\begin{example}
\label{ex:colorLie} Let $\Bbbk $ be a field with $\mathrm{char}(\Bbbk )=0$
and let $G$ be an abelian group endowed with an anti-symmetric bicharacter $%
\chi :G\times G\rightarrow \Bbbk \setminus \{0\}$, i.e. for all $g,h,k\in G,$
we have:
\begin{equation*}
\chi (g,hk)=\chi (g,h)\chi (g,k),\quad \chi (gh,k)=\chi (g,k)\chi
(h,k),\quad \chi (g,h)\chi (h,g)=1.
\end{equation*}%
Extend $\chi $ by linearity to a $\Bbbk $-linear map $R:\Bbbk \left[ G\right]
\otimes \Bbbk \left[ G\right] \rightarrow \Bbbk $, where $\Bbbk \left[ G%
\right] $ denotes the group algebra. Then $\left( \Bbbk \left[ G\right]
,R\right) $ is a co-triangular bialgebra, cf. \cite[Example 2.2.5]{Maj}.
Hence, we can apply Theorem \ref{teo:LieCoquasiTriang} and Remark \ref%
{rem:cotriang} to $H=\Bbbk \left[ G\right] $. Note that the category $\left(
\mathfrak{M}^{H},\otimes _{\Bbbk },\Bbbk ,a,l,r,c\right) $ consists of $G$%
-graded modules $V=\oplus _{g\in G}V_{g}.$ Given $G$-graded modules $V\ $and
$W,$ their tensor product $V\otimes W$ is graded with $\left( V\otimes
W\right) _{g}:=\oplus _{hl=g}\left( V_{h}\otimes W_{l}\right) .$ The
braiding is given on homogeneous elements by%
\begin{equation*}
c_{V,W}:V\otimes W\rightarrow W\otimes V,\quad c_{V,W}(v\otimes w)=w\otimes
v\chi (|v|,|w|),
\end{equation*}%
where $|v|$ denotes the degree of $v.$ In this case a $\left( H,R\right) $%
-lie algebra $\left( V,\left[ -,-\right] \right) $ in the sense of \cite[%
Definition 4.1]{BFM} means
\begin{gather*}
\left[ x,y\right] =-\left[ y,x\right] \chi \left( |x|,|y|\right) , \\
\left[ \left[ x,y\right] ,z\right] +\sum \left[ \left[ y,z\right] ,x\right]
\chi \left( \left\vert x\right\vert ,\left\vert y\right\vert \left\vert
z\right\vert \right) +\sum \left[ \left[ z,x\right] ,y\right] \chi \left(
\left\vert x\right\vert \left\vert y\right\vert ,\left\vert z\right\vert
\right) =0.
\end{gather*}%
Multiplying by $\chi \left( \left\vert z\right\vert ,\left\vert x\right\vert
\right) $ the two sides of the second equality, we get the equivalent
\begin{equation*}
\left[ \left[ x,y\right] ,z\right] \chi \left( \left\vert z\right\vert
,\left\vert x\right\vert \right) +\sum \left[ \left[ y,z\right] ,x\right]
\chi \left( \left\vert x\right\vert ,\left\vert y\right\vert \right) +\sum %
\left[ \left[ z,x\right] ,y\right] \chi \left( \left\vert y\right\vert
,\left\vert z\right\vert \right) =0.
\end{equation*}%
This means that $\left( V,\left[ -,-\right] \right) $ is a $\left( G,\chi
\right) $-Lie color algebra in the sense of \cite[Example 10.5.14]{Mo}. Note
that the braiding defined in \cite[page 200]{Mo} is $c_{V,W}^{\prime
}(v\otimes w)=w\otimes v\chi (|w|,|v|)=w\otimes v\chi ^{-1}(|v|,|w|)$ so
that we should say more precisely that $\left( V,\left[ -,-\right] \right) $
is a $\left( G,\chi ^{-1}\right) $-Lie color algebra. The corresponding
enveloping algebra is%
\begin{equation*}
\mathcal{U}\left( M,\left[ -\right] \right) =\frac{TM}{\left( \left[ x,y%
\right] -x\otimes y+y\otimes x\chi \left( |x|,|y|\right) \mid x,y\in V\text{
homogeneous}\right) }.
\end{equation*}
\end{example}

\begin{example}
\label{ex:superLie} Lie superalgebras are a particular instance of the
construction above. One has to take $G=\mathbb{Z}_{2}$ and consider the
anti-symmetric bicharacter $\chi :G\times G\rightarrow \Bbbk \setminus \{0\}$
defined by $\chi \left( \overline{a},\overline{b}\right) :=\left( -1\right)
^{ab}$ for all $a,b\in \mathbb{Z}.$
\end{example}

\begin{example}
Let $G:=\left( \mathbb{Z},+,0\right) $. Let $\Bbbk $ be a field and let $%
q\in \Bbbk \setminus \left\{ 0\right\} $. Then it is easy to check that $%
\left\langle -,-\right\rangle :G\times G\rightarrow \Bbbk ,\left\langle
a,b\right\rangle :=q^{ab}$ is a bicharacter of $G$.
\end{example}

\begin{remark}
Let $\Bbbk $ be a field with $\mathrm{char}\left( \Bbbk \right) =0$. Let $H$
be a finite-dimensional Hopf algebra. By \cite[Proposition 6]{RT}, the
category of Yetter-Drinfeld modules ${_{H}^{H}\mathcal{YD}}$ and ${_{H}%
\mathcal{YD}}^{H}$ are isomorphic. Moreover, by \cite[Proposition 10.6.16]%
{Mo}, the ${_{H}\mathcal{YD}}^{H}$ can be identified with the category ${%
_{D\left( H\right) }}\mathfrak{M}$ of left modules over the Drinfeld double $%
D\left( H\right) .$ Now ${_{D\left( H\right) }}\mathfrak{M}\cong \mathfrak{M}%
^{D\left( H\right) ^{\ast }}$ and $D\left( H\right) ^{\ast }$ is a
quasi-co-triangular bialgebra. Thus we can identify ${_{H}^{H}\mathcal{YD}}$
with $\mathfrak{M}^{D\left( H\right) ^{\ast }}$. One is tempted to apply
Theorem \ref{teo:LieCoquasiTriang}. Unfortunately, $D\left( H\right) $ is
never triangular (whence $D\left( H\right) ^{\ast }$ is never co-triangular)
in view of \cite{Pa}, unless $H=\Bbbk $.
\end{remark}

\appendix

\section{(Co)equalizers and (co)monadicity}

\begin{definition}
\cite[page 112]{MacLane} Let $\mathcal{I}$ be a small category. Recall that
a functor $V:\mathcal{A}\rightarrow \mathcal{B}$ \textbf{creates limits}
\textbf{for a functor} $F:\mathcal{I}\rightarrow \mathcal{A}$ if in case $VF$
has a limit $\left( X,\left( \tau _{I}:X\rightarrow VFI\right) _{I\in
\mathcal{I}}\right) $, then there is exactly one pair $\left( L,\left(
\sigma _{I}:L\rightarrow FI\right) _{I\in \mathcal{I}}\right) $ which is a
limit of $F$ and such that $VL=X,$ $V\sigma _{I}=\tau _{I}$ for every $I\in
\mathcal{I}$. We just say that $V:\mathcal{A}\rightarrow \mathcal{B}$
\textbf{creates limits} if it creates limits for all functors $F:\mathcal{I}%
\rightarrow \mathcal{A}$ and for all small category $\mathcal{I}$. Similarly
one defines creation of colimits.
\end{definition}

\begin{lemma}
\label{lem:OmegaCreates}Let $\mathcal{M}$ be a monoidal category. Then the
functor $\Omega :\mathrm{Alg}_{\mathcal{M}}\rightarrow \mathcal{M}$ creates
limits and the functor $\mho :\mathrm{Coalg}_{\mathcal{M}}\rightarrow
\mathcal{M}$ creates colimits.
\end{lemma}

\begin{proof}
It is straightforward.
\end{proof}

\begin{claim}
\label{cl:CoeqAlg} Let $\mathcal{M}$ be a monoidal category. Assume that $%
\mathcal{M}$ has coequalizers and that the tensor functors preserve them. It
is well-known that $\mathrm{Alg}_{\mathcal{M}}$ has coequalizers, see e.g.
\cite[Proposition 2.1.5]{AEM}. Given an algebra morphism $\alpha
:E\rightarrow A$, consider $\Lambda _{\alpha }:=m_{A}^{2}\circ \left(
A\otimes \alpha \otimes A\right) $ of (\ref{def:Lambdaf}) where $%
m_{A}^{2}:A\otimes A\otimes A\rightarrow A$ is the iterated multiplication.
The coequalizer of algebra morphisms $\alpha ,\beta :E\rightarrow A$ is, as
an object in $\mathcal{M}$, the coequalizer $\left( B,\pi :A\rightarrow
B\right) $ of $\left( \Lambda _{\alpha },\Lambda _{\beta }\right) $ in $%
\mathcal{M}$ and the algebra structure is the unique one making $\pi $ an
algebra morphism.
\end{claim}

\begin{lemma}
\label{lem:BialgCoeq}Let $\mathcal{M}$ be a monoidal category.

1) If $\mathcal{M}$ has coequalizers then $\mathrm{Coalg}_{\mathcal{M}}$ has
coequalizers, and $\mho :\mathrm{Coalg}_{\mathcal{M}}\rightarrow \mathcal{M}$
preserves coequalizers. Moreover if the tensor products preserve the
coequalizers in $\mathcal{M}$, then $\mathrm{Alg}_{\mathcal{M}}$ has
coequalizers.

2) If $\mathcal{M}$ has equalizers then $\mathrm{Alg}_{\mathcal{M}}$ has
equalizers, and $\Omega :\mathrm{Alg}_{\mathcal{M}}\rightarrow \mathcal{M}$
preserves equalizers. Moreover if the tensor products preserve the
equalizers in $\mathcal{M}$, then $\mathrm{Coalg}_{\mathcal{M}}$ has
equalizers.

3) If $\mathcal{M}$ is braided, it has coequalizers and the tensor products
preserve them, then $\mathrm{Bialg}_{\mathcal{M}}$ has coequalizers and $%
\mho :\mathrm{Bialg}_{\mathcal{M}}\rightarrow \mathrm{Alg}_{\mathcal{M}}$
preserves coequalizers.

4)\ If $\mathcal{M}$ is braided, it has equalizers and the tensor products
preserve them, then $\mathrm{Bialg}_{\mathcal{M}}$ has equalizers and $%
\Omega :\mathrm{Bialg}_{\mathcal{M}}\rightarrow \mathrm{Coalg}_{\mathcal{M}}$
preserves equalizers.
\end{lemma}

\begin{proof}
1) The first part follows by Lemma \ref{lem:OmegaCreates} and uniqueness of
coequalizers in $\mathrm{Coalg}_{\mathcal{M}}$. By \ref{cl:CoeqAlg}, $%
\mathrm{Alg}_{\mathcal{M}}$ has coequalizers. 2) is dual to 1).

3) Note that $\mathrm{Bialg}_{\mathcal{M}}=\mathrm{Coalg}_{\mathcal{N}}$ for
$\mathcal{N}:=\mathrm{Alg}_{\mathcal{M}}$. By 1) we have that $\mathcal{N}$
has coequalizers and then $\mathrm{Coalg}_{\mathcal{N}}$ has coequalizers,
and $\mho :\mathrm{Coalg}_{\mathcal{N}}\rightarrow \mathcal{N}$ preserves
coequalizers. 4) is dual to 3).
\end{proof}

\begin{lemma}
\label{lem:JBialgPres}Let $\mathcal{M}$ be a braided monoidal category.
Assume that $\mathcal{M}$ is abelian and that the tensor functors preserves
equalizers, coequalizers.

1) Let $\alpha :J_{\mathrm{Bialg}}D\rightarrow E$ be a morphism in $\mathrm{%
BrBialg}_{\mathcal{M}}.$ Then there is a bialgebra $Q\in \mathrm{Bialg}_{%
\mathcal{M}}$ a morphism $\pi :D\rightarrow Q$ in $\mathrm{Bialg}_{\mathcal{M%
}}$ and a morphism $\sigma :J_{\mathrm{Bialg}}Q\rightarrow E$ in $\mathrm{%
BrBialg}_{\mathcal{M}}$ such that $\alpha =\sigma \circ J_{\mathrm{Bialg}%
}\left( \pi \right) $ and $\sigma $ and $\pi $ are a monomorphism and an
epimorphism respectively when regarded as morphism in $\mathcal{M}$.

2) The functor $J_{\mathrm{Bialg}}:\mathrm{Bialg}_{\mathcal{M}}\rightarrow
\mathrm{BrBialg}_{\mathcal{M}}$ preserves coequalizers.

3) Assume that $\mathcal{M}$ is symmetric. Then $J_{\mathrm{Bialg}}^{s}:%
\mathrm{Bialg}_{\mathcal{M}}\rightarrow \mathrm{BrBialg}_{\mathcal{M}}^{s}$
preserves coequalizers.
\end{lemma}

\begin{proof}
1) Denote by $D$ and $E$ the underlying objects in $\mathcal{M}$ of $D$ and $%
E.$ Since $\mathcal{M}$ is abelian we can factor $\alpha :D\rightarrow E$ as
the composition of a monomorphism $\sigma :Q\rightarrow E$ and an
epimorphism $\pi :D\rightarrow Q$ in $\mathcal{M}$ where $Q$ is the image of
$\alpha $ in $\mathcal{M}$.

It is straightforward to check that $Q$ fulfils the require properties.

2) By \ref{cl:defJ}, we have $J_{\mathrm{Bialg}}\left( B,m_{B},u_{B},\Delta
_{B},\varepsilon _{B}\right) =\left( B,m_{B},u_{B},\Delta _{B},\varepsilon
_{B},c_{B,B}\right) $ and $J_{\mathrm{Bialg}}\left( f\right) =f.$

Let $\left( e_{0},e_{1}\right) $ from $\left( B,m_{B},u_{B},\Delta
_{B},\varepsilon _{B}\right) $ to $\left( D,m_{D},u_{D},\Delta
_{D},\varepsilon _{D}\right) $ be a pair of morphisms in $\mathrm{Bialg}_{%
\mathcal{M}}$. Assume that this pair has coequalizer $\left( E,p\right) $ in
$\mathrm{Bialg}_{\mathcal{M}}$
\begin{equation*}
\xymatrix{ B \ar@<.5ex>[rr]^-{e_{0} } \ar@<-.5ex>[rr]_-{e_{1}
}&&D\ar[r]^{p}&E }
\end{equation*}%
Let us check that $J_{\mathrm{Bialg}}$ preserves this coequalizer. Let $%
\alpha :J_{\mathrm{Bialg}}D\rightarrow Z$ be a morphism in $\mathrm{BrBialg}%
_{\mathcal{M}}$ such that $\alpha e_{0}=\alpha e_{1}.$ By 1) we write $%
\alpha =\sigma \circ J_{\mathrm{Bialg}}\left( \pi \right) $. Since $\sigma $
is a monomorphism in $\mathcal{M}$, we have that $\pi e_{0}=\pi e_{1}.$
Since the coequalizer $\left( E,p\right) $ is in $\mathrm{Bialg}_{\mathcal{M}%
},$ there is a unique morphism $\overline{\pi }:E\rightarrow Q$ in $\mathrm{%
Bialg}_{\mathcal{M}}$ such that $\overline{\pi }\circ p=\pi .$ Set $%
\overline{\alpha }:=\sigma \overline{\pi }:E\rightarrow Z$ as morphisms in $%
\mathcal{M}$. Then $\overline{\alpha }p=\sigma \overline{\pi }p=\sigma \pi
=\alpha .$ Moreover $\sigma $ and $\overline{\pi }$ commute with
(co)multiplications and (co)units and
\begin{equation*}
\left( \overline{\alpha }\otimes \overline{\alpha }\right) c_{E,E}=\left(
\sigma \otimes \sigma \right) \left( \overline{\pi }\otimes \overline{\pi }%
\right) c_{E,E}=\left( \sigma \otimes \sigma \right) c_{Q,Q}\left( \overline{%
\pi }\otimes \overline{\pi }\right) =c_{Z}\left( \sigma \otimes \sigma
\right) \left( \overline{\pi }\otimes \overline{\pi }\right) =c_{Z}\left(
\overline{\alpha }\otimes \overline{\alpha }\right) .
\end{equation*}%
We have so proved that $\overline{\alpha }$ is a morphism in $\mathrm{BrBialg%
}_{\mathcal{M}}$ from $J_{\mathrm{Bialg}}E$ to $Z.$

Let $\beta :J_{\mathrm{Bialg}}E\rightarrow Z$ in $\mathrm{BrBialg}_{\mathcal{%
M}}$ be such that $\beta p=\alpha $ as morphisms in $\mathrm{BrBialg}_{%
\mathcal{M}}.$ Then $\beta p=\overline{\alpha }p$ as morphisms in $\mathcal{M%
}$. Since $\left( E,p\right) $ is a coequalizer in $\mathrm{Bialg}_{\mathcal{%
M}}$ and $\mathcal{M}$ has coequalizers (it is abelian) we have that $\left(
E,p\right) $ can be constructed as a suitable coequalizer in $\mathcal{M}$
(cf. the proof of Lemma \ref{lem:BialgCoeq}) so that $p$ is an epimorphism
in $\mathcal{M}$. Hence we get $\beta =\overline{\alpha }$ as morphisms in $%
\mathcal{M}$ whence in $\mathrm{BrBialg}_{\mathcal{M}}$.

3) By 2) $J_{\mathrm{Bialg}}:\mathrm{Bialg}_{\mathcal{M}}\rightarrow \mathrm{%
BrBialg}_{\mathcal{M}}$ preserves the coequalizers. Since $J_{\mathrm{Bialg}%
}=\mathbb{I}_{\mathrm{BrBialg}}^{s}\circ J_{\mathrm{Bialg}}^{s}$ we get that
$\mathbb{I}_{\mathrm{BrBialg}}^{s}\circ J_{\mathrm{Bialg}}^{s}$ preserves
coequalizers. Since $\mathbb{I}_{\mathrm{BrBialg}}^{s}$ is both full and
faithful, it reflects colimits (see the dual of \cite[Proposition 2.9.9]%
{Borceux1}) so that $J_{\mathrm{Bialg}}^{s}$ preserves coequalizers.
\end{proof}

The following result can be obtained mimicking the proof of $\left( 1\right)
\Rightarrow \left( 2\right) $ in \cite[Theorem 4.6.2]{Borceux2}. For the
reader's sake we write here a proof in the specific case we are concerned.

\begin{theorem}
Let $\mathcal{M}$ be a monoidal category.

\begin{itemize}
\item[1)] \label{teo: Monoids}If the forgetful functor $\Omega :\mathrm{Alg}%
_{\mathcal{M}}\rightarrow \mathcal{M}$ has a left adjoint, then $\Omega $ is
monadic. In fact the comparison functor is a category isomorphism.

\item[2)] \label{teo:Comonoids}If the forgetful functor $\mho :\mathrm{Coalg}%
_{\mathcal{M}}\rightarrow \mathcal{M}$ has a right adjoint, then $\mho $ is
comonadic. In fact the comparison functor is a category isomorphism.
\end{itemize}
\end{theorem}

\begin{proof}
1) We will apply Theorem \cite[Theorem 2.1]{BLV} (which is a form of Beck's
Theorem). First, in order to prove that $\Omega $ is monadic, we have to
check that $\Omega $ is conservative and that for any reflexive pair of
morphisms in $\mathrm{Alg}_{\mathcal{M}}$ whose image by $\Omega $ has a
split coequalizer has a coequalizer which is preserved by $\Omega .$ Clearly
if $f$ is a morphism in $\mathrm{Alg}_{\mathcal{M}}$ such that $\Omega f$ is
an isomorphism then the inverse of $\Omega f$ is a morphism of monoids
whence it gives rise to an inverse of $f$ in $\mathrm{Alg}_{\mathcal{M}}$.
Thus $\Omega $ is conservative.

Let $\left( d_{0},d_{1}\right) $ from $A$ to $A^{\prime }$ be a reflexive
pair as above. Then there exists $C\in \mathcal{M}$ and a morphism $c:\Omega
A^{\prime }\rightarrow C$ such that
\begin{equation*}
\xymatrix{ \Omega A \ar@<.5ex>[rr]^-{\Omega d_{0} } \ar@<-.5ex>[rr]_-{\Omega
d_{1} }&&\Omega A^{\prime }\ar[r]^{c}&C }
\end{equation*}%
is a split coequalizer, whence preserved by any functor in particular by $%
F_{n}:\mathcal{M}\rightarrow \mathcal{M}$, the functor defined by $%
F_{n}:=\left( -\right) ^{\otimes n}$ i.e. the $n$th tensor power functor.
Then we have a commutative diagram with exact rows%
\begin{equation*}
\xymatrix{ \Omega A \otimes \Omega A \ar[d]_{m_{\Omega
A}}\ar@<.5ex>[rr]^-{\Omega d_{0}\otimes \Omega d_{0}}
\ar@<-.5ex>[rr]_-{\Omega d_{1}\otimes \Omega d_{1} }&&\Omega A^{\prime
}\otimes\Omega A^{\prime }\ar[d]_{m_{\Omega A^\prime}}\ar[r]^-{c\otimes
c}&C\otimes C\\ \Omega A \ar@<.5ex>[rr]^-{\Omega d_{0} }
\ar@<-.5ex>[rr]_-{\Omega d_{1} }&&\Omega A^{\prime }\ar[r]^{c}&C}
\end{equation*}%
By the universal property of coequalizers there is a unique morphism $%
m_{C}:C\otimes C\rightarrow C$ in $\mathcal{M}$ such that $m_{C}\circ \left(
c\otimes c\right) =c\circ m_{\Omega A^{\prime }}.$ One easily checks that $%
Q:=\left( C,m_{C},u_{C}\right) \in \mathrm{Alg}_{\mathcal{M}}$ where $%
u_{C}:=c\circ u_{\Omega A^{\prime }}.$ Moreover $c$ gives rise to a morphism
$q:A^{\prime }\rightarrow Q$ in $\mathrm{Alg}_{\mathcal{M}}$ such that $%
\Omega q=c.$ Since $\Omega $ is faithful, it is straightforward to check
that $\left( Q,q\right) $ is the coequalizer of $\left( d_{0},d_{1}\right) $
in $\mathcal{M}_{m}$. Thus $\Omega $ is monadic.

Let us check that the comparison functor is indeed a category isomorphism.
It suffices to check that for any isomorphism $f:\Omega X\rightarrow B$ in
the category $\mathcal{M}$ there exists a unique pair $\left(
A,g:X\rightarrow A\right) ,$ where $A$ is an object in $\mathrm{Alg}_{%
\mathcal{M}}$ and $g$ a morphism in $\mathrm{Alg}_{\mathcal{M}}$, such that $%
\Omega A=B$ and $\Omega g=f$. This is trivial (just induce on $B$ the monoid
structure of $X$ via $f$).

2) It is dual to 1).
\end{proof}

\begin{example}
Let $\Bbbk $ be a field. Let $\mathfrak{M}$ be the category of vector spaces
over $\Bbbk $.

1) By \cite[Theorem 6.4.1]{Sw}, the forgetful functor $\mho :\mathrm{Coalg}_{%
\mathfrak{M}}\rightarrow \mathfrak{M}$ has a right adjoint given by the
cofree coalgebra functor.

2) By \cite[Theorem 2.3]{Agore}, the forgetful functor $\mho :\mathrm{Bialg}%
_{\mathfrak{M}}\rightarrow \mathrm{Alg}_{\mathfrak{M}}$ has a right adjoint.

In both cases, by Theorem \ref{teo:Comonoids}, we have that $\mho $ is
comonadic and that the comparison functor is a category isomorphism.
\end{example}

\begin{lemma}
\label{lem:Vallette}Let $\mathcal{M}$ be a monoidal category. Assume that
the tensor functors preserve coequalizers of reflexive pairs in $\mathcal{M}$%
. Given two coequalizers%
\begin{equation*}
\xymatrix{X_{1} \ar@<.5ex>[rr]^-{f_1 } \ar@<-.5ex>[rr]_-{g_1
}&&Y_1\ar[r]^-{p_1}&Z_1 } \qquad \xymatrix{X_{2} \ar@<.5ex>[rr]^-{f_2 }
\ar@<-.5ex>[rr]_-{g_2 }&&Y_2\ar[r]^-{p_2}&Z_2 }
\end{equation*}%
in $\mathcal{M}$, where $\left( f_{1},g_{1}\right) $ and $\left(
f_{2},g_{2}\right) $ are reflexive pairs of morphisms in $\mathcal{M}$, we
have that
\begin{equation*}
\xymatrix{X_{1}\otimes X_2 \ar@<.5ex>[rr]^-{f_1\otimes f_2 }
\ar@<-.5ex>[rr]_-{g_1\otimes g_2 }&&Y_1\otimes Y_2\ar[r]^-{p_1\otimes
p_2}&Z_1\otimes Z_2 }
\end{equation*}%
is a coequalizer too.
\end{lemma}

\begin{proof}
See \cite[Proposition 2]{Vallette} (where we can drop the assumption on
abelianity as the result follows by \cite[Lemma 0.17]{Joh-Topos} where this
condition is not used).
\end{proof}

\begin{proposition}
\label{pro:CreAlg}Let $\mathcal{M}$ be a monoidal category. Assume that the
tensor functors preserve coequalizers of reflexive pairs in $\mathcal{M}$.
Then the forgetful functor $\Omega :\mathrm{Alg}{\mathcal{M}}\rightarrow
\mathcal{M}$ creates coequalizers of those pairs $\left( f,g\right) $ in $%
\mathrm{Alg}{\mathcal{M}}$ for which $\left( \Omega f,\Omega g\right) $ is a
reflexive pair.
\end{proposition}

\begin{proof}
Let $f,g:\left( A,m_{A},u_{A}\right) \rightarrow \left( B,m_{B},u_{B}\right)
$ be a pair of morphism in $\mathrm{Alg}{\mathcal{M}}$ that fits into a
coequalizer
\begin{equation*}
\xymatrix{ A \ar@<.5ex>[rr]^-{\Omega f } \ar@<-.5ex>[rr]_-{\Omega g
}&&B\ar[r]^-{p}&C }
\end{equation*}%
in $\mathcal{M}$ such that $\left( \Omega f,\Omega g\right) $ is a reflexive
pair. By Lemma \ref{lem:Vallette}, we have the following coequalizer
\begin{equation*}
\xymatrix{ A\otimes A \ar@<.5ex>[rr]^-{\Omega f\otimes \Omega f }
\ar@<-.5ex>[rr]_-{\Omega g\otimes \Omega g }&&B\otimes B\ar[r]^-{p\otimes
p}&C\otimes C }
\end{equation*}%
We have%
\begin{equation*}
p\circ m_{B}\circ \left( \Omega f\otimes \Omega f\right) =p\circ \Omega
f\circ m_{A}=p\circ \Omega g\circ m_{A}=p\circ m_{B}\circ \left( \Omega
g\otimes \Omega g\right) .
\end{equation*}%
The universal property of the latter coequalizer entails there is a unique
morphism $m_{C}:C\otimes C\rightarrow C$ such that $m_{C}\circ \left(
p\otimes p\right) =p\circ m_{B}.$ Set $u_{C}:=p\circ u_{B}.$ It is easy to
check that $\left( C,m_{C},u_{C}\right) \in \mathrm{Alg}{\mathcal{M}}$, that
$p$ becomes an algebra morphism from $\left( B,m_{B},u_{B}\right) $ to $%
\left( C,m_{C},u_{C}\right) $ and that
\begin{equation*}
\xymatrix{ (A,m_{A},u_{A}) \ar@<.5ex>[rr]^-{f } \ar@<-.5ex>[rr]_-{ g
}&&(B,m_{B},u_{B})\ar[r]^-{p}&(C,m_{C},u_{C}) }
\end{equation*}%
is a coequalizer in $\mathrm{Alg}{\mathcal{M}}$.
\end{proof}

\begin{corollary}
\label{coro:OmegaPres}Let $\mathcal{M}$ be a monoidal category with
coequalizers of reflexive pairs. Assume these coequalizers are preserved by
the tensor functors in $\mathcal{M}$. Then $\mathrm{Alg}{\mathcal{M}}$ has
coequalizers of reflexive pairs and they are preserved by the forgetful
functor $\Omega :\mathrm{Alg}{\mathcal{M}}\rightarrow \mathcal{M}$.
\end{corollary}

\begin{proof}
It follows by Proposition \ref{pro:CreAlg} and uniqueness of coequalizers in
$\mathrm{Alg}{\mathcal{M}}$.
\end{proof}

\section{Braided (co)equalizers}

\begin{lemma}
\label{lem:cInv}Let $\mathcal{M}$ be a monoidal category. We have functors%
\begin{eqnarray*}
\mathrm{Br}_{\mathcal{M}} &\longrightarrow &\mathrm{Br}_{\mathcal{M}}:\left(
V,c\right) \rightarrow \left( V,c^{-1}\right) ,f\mapsto f \\
\mathrm{BrAlg}_{\mathcal{M}} &\rightarrow &\mathrm{BrAlg}_{\mathcal{M}%
}:\left( A,m,u,c\right) \rightarrow \left( A,m,u,c^{-1}\right) ,f\mapsto f
\end{eqnarray*}
\end{lemma}

\begin{proof}
It is straightforward.
\end{proof}

\begin{lemma}
\label{lem:quotdual}Let $\mathcal{M}$ be a monoidal category and let $\left(
V,c_{V}\right) $ be an object in $\mathrm{Br}_{\mathcal{M}}$. Assume there
is a morphism $d:D\rightarrow V$ in $\mathcal{M}$ and a morphism $%
c_{D}:D\otimes D\rightarrow D\otimes D$ such that $\left( d\otimes d\right)
c_{D}=c_{V}\left( d\otimes d\right) \ $and $d\otimes d\otimes d$ is a
monomorphism.

1)Assume that $c_{D}$ is an isomorphism. Then $\left( D,c_{D}\right) $ is an
object in $\mathrm{Br}_{\mathcal{M}}$ and $d$ becomes a morphism in $\mathrm{%
Br}_{\mathcal{M}}$ from $\left( D,c_{D}\right) $ to $\left( V,c_{V}\right) $.

2) Assume that $d\otimes d$ is a monomorphism. If $\left( V,c_{V}\right) \in
\mathrm{Br}_{\mathcal{M}}^{s}$ then $\left( D,c_{D}\right) \in \mathrm{Br}_{%
\mathcal{M}}^{s}$.
\end{lemma}

\begin{proof}
Using $\left( d\otimes d\right) c_{D}=c_{V}\left( d\otimes d\right) $ and
the quantum Yang-Baxter equation for $c_{V}$ one gets
\begin{equation*}
\left( d\otimes d\otimes d\right) \left( c_{D}\otimes D\right) \left(
D\otimes c_{D}\right) \left( c_{D}\otimes D\right) =\left( d\otimes d\otimes
d\right) \left( D\otimes c_{D}\right) \left( c_{D}\otimes D\right) \left(
D\otimes c_{D}\right) .
\end{equation*}%
Since $d\otimes d\otimes d$ is a monomorphism we get that $c_{D}$ satisfies
the quantum Yang-Baxter equation.

1) Since $c_{D}$ is an isomorphism it is clear that $\left( D,c_{D}\right)
\in \mathrm{Br}_{\mathcal{M}}$ and that $d:\left( D,c_{D}\right) \rightarrow
\left( V,c_{V}\right) $ is a morphism in $\mathrm{Br}_{\mathcal{M}}.$

2) Since $\left( d\otimes d\right) c_{D}^{2}=c_{V}^{2}\left( d\otimes
d\right) =d\otimes d$ and $d\otimes d$ is a monomorphism we get $c_{D}^{2}=%
\mathrm{Id}_{D\otimes D}$ so that we can apply 1).
\end{proof}

\begin{lemma}
\label{lem: BrVec}Let $\mathcal{M}$ be a monoidal category and let $H:%
\mathrm{Br}_{\mathcal{M}}\longrightarrow \mathcal{M}\ $be the forgetful
functor. Let $\left( e_{0},e_{1}\right) $ be a pair of morphisms in $\mathrm{%
Br}_{\mathcal{M}}$ such $\left( He_{0},He_{1}\right) $ is a coreflexive pair
of morphisms in $\mathcal{M}$. Assume that $\left( He_{0},He_{1}\right) $
has an equalizer which is preserved by the tensor functors. Then $\left(
e_{0},e_{1}\right) $ has an equalizer in $\mathrm{Br}_{\mathcal{M}}$ which
is preserved by $H$. The same statement holds when we replace $\mathrm{Br}_{%
\mathcal{M}}$ by $\mathrm{Br}_{\mathcal{M}}^{s}$ and $H$ by the
corresponding forgetful functor.
\end{lemma}

\begin{proof}
Let $\left( e_{0},e_{1}\right) $ from $\left( V,c_{V}\right) $ to $\left(
W,c_{W}\right) $ a coreflexive pair of morphisms in $\mathrm{Br}_{\mathcal{M}%
}$. We denote $\left( He_{0},He_{1}\right) $ by $\left( e_{0},e_{1}\right) $
to simplify the notation. By definition, there exists a morphism $%
p:W\rightarrow V$ in $\mathcal{M}$ such that $p\circ e_{0}=\mathrm{Id}%
_{V}=p\circ e_{1}$. Consider the equalizer%
\begin{equation*}
\xymatrix{D \ar[r]^-{d} & V \ar@<.5ex>[rr]^-{e_0 } \ar@<-.5ex>[rr]_-{e_1
}&&W }
\end{equation*}

By the dual version of Lemma \ref{lem:Vallette}, we have the following
equalizer%
\begin{equation*}
\xymatrix{\duplica{D} \ar[r]^-{\duplica{d}} & \duplica{V}
\ar@<.5ex>[rr]^-{\duplica{e_0} } \ar@<-.5ex>[rr]_-{\duplica{e_1}
}&&\duplica{W} }
\end{equation*}%
We have%
\begin{equation*}
\left( e_{0}\otimes e_{0}\right) c_{V}\left( d\otimes d\right) =c_{W}\left(
e_{0}\otimes e_{0}\right) \left( d\otimes d\right) =c_{W}\left( e_{1}\otimes
e_{1}\right) \left( d\otimes d\right) =\left( e_{1}\otimes e_{1}\right)
c_{V}\left( d\otimes d\right) .
\end{equation*}%
Hence there is a unique morphism $c_{D}:D\otimes D\rightarrow D\otimes D$
such that $\left( d\otimes d\right) c_{D}=c_{V}\left( d\otimes d\right) .$

Since $\left( V,c_{V}^{-1}\right) $ and $\left( W,c_{W}^{-1}\right) $ are
also braided objects, and $e_{0},e_{1}$ are also morphisms from $\left(
V,c_{V}^{-1}\right) $ to $\left( W,c_{W}^{-1}\right) $, as above we can
construct a morphism $\gamma _{D}:D\otimes D\rightarrow D\otimes D$ such
that $\left( d\otimes d\right) \gamma _{D}=c_{V}^{-1}\left( d\otimes
d\right) .$ We have $\left( d\otimes d\right) c_{D}\gamma _{D}=c_{V}\left(
d\otimes d\right) \gamma _{D}=c_{V}c_{V}^{-1}\left( d\otimes d\right)
=d\otimes d$ and hence $c_{D}\gamma _{D}=\mathrm{Id}_{D\otimes D}.$
Similarly $\gamma _{D}c_{D}=\mathrm{Id}_{D\otimes D}.$ Thus $c_{D}$ is
invertible. Since $d\otimes d\otimes d=\left( d\otimes V\otimes V\right)
\left( D\otimes d\otimes V\right) \left( D\otimes D\otimes d\right) $ we
have that $d\otimes d\otimes d$ is a monomorphism. Thus we can apply Lemma %
\ref{lem:quotdual} to get that $\left( D,c_{D}\right) $ is an object in $%
\mathrm{Br}_{\mathcal{M}}$ and $d:\left( D,c_{D}\right) \rightarrow \left(
V,c_{V}\right) $ is a morphism in $\mathrm{Br}_{\mathcal{M}}.$ It is
straightforward to check that
\begin{equation*}
\xymatrix{( D,c_{D}) \ar[r]^-{d} & ( V,c_{V}) \ar@<.5ex>[rr]^-{e_0 }
\ar@<-.5ex>[rr]_-{e_1 }&&( W,c_{W}) }
\end{equation*}%
is an equalizer in $\mathrm{Br}_{\mathcal{M}}$. Consider now the case of $%
\mathrm{Br}_{\mathcal{M}}^{s}$ so that $\left( e_{0},e_{1}\right) $ as above
is a pair in $\mathrm{Br}_{\mathcal{M}}^{s}$. Since $d$ is a monomorphism,
by Remark \ref{rem:aureo}, we get that $\left( D,c_{D}\right) \in \mathrm{Br}%
_{\mathcal{M}}^{s}$ and $d$ becomes a morphism in this category. Since $%
\mathrm{Br}_{\mathcal{M}}^{s}$ is a full subcategory of $\mathrm{Br}_{%
\mathcal{M}}$ we have that $\mathbb{I}_{\mathrm{Br}}^{s}:\mathrm{Br}_{%
\mathcal{M}}^{s}\rightarrow \mathrm{Br}_{\mathcal{M}}$ is full and faithful
and hence it reflects equalizers (see \cite[Proposition 2.9.9]{Borceux1}) so
that the above equalizer obtained in $\mathrm{Br}_{\mathcal{M}}$ is indeed
an equalizer in $\mathrm{Br}_{\mathcal{M}}^{s}$.
\end{proof}

\begin{lemma}
\label{lem:quot}Let $\mathcal{M}$ be a monoidal category and let $\left(
W,c_{W}\right) $ be an object in $\mathrm{Br}_{\mathcal{M}}$. Assume there
is a morphism $d:W\rightarrow D$ in $\mathcal{M}$ and a morphism $%
c_{D}:D\otimes D\rightarrow D\otimes D$ such that $c_{D}\left( d\otimes
d\right) =\left( d\otimes d\right) c_{W}\ $and $d\otimes d\otimes d$ is an
epimorphism.

1) Assume that $c_{D}$ is an isomorphism. Then $\left( D,c_{D}\right) $ is
an object in $\mathrm{Br}_{\mathcal{M}}$ and $d$ becomes a morphism in $%
\mathrm{Br}_{\mathcal{M}}$ from $\left( W,c_{W}\right) $ to this object.

2) Assume that $d\otimes d$ is an epimorphism. If $\left( V,c_{V}\right) \in
\mathrm{Br}_{\mathcal{M}}^{s}$ then $\left( D,c_{D}\right) \in \mathrm{Br}_{%
\mathcal{M}}^{s}$.
\end{lemma}

\begin{proof}
It is dual to Lemma \ref{lem:quotdual}.
\end{proof}

\begin{lemma}
\label{lem: BrCoeq}Let $\mathcal{M}$ be a monoidal category and let $H:%
\mathrm{Br}_{\mathcal{M}}\longrightarrow \mathcal{M}\ $be the forgetful
functor. Let $\left( e_{0},e_{1}\right) $ be a pair of morphisms in $\mathrm{%
Br}_{\mathcal{M}}$ such $\left( He_{0},He_{1}\right) $ is a reflexive pair
of morphisms in $\mathcal{M}$. Assume that $\left( He_{0},He_{1}\right) $
has a coequalizer which is preserved by the tensor functors. Then $\left(
e_{0},e_{1}\right) $ has an coequalizer in $\mathrm{Br}_{\mathcal{M}}$ which
is preserved by $H$.

The same statement holds when we replace $\mathrm{Br}_{\mathcal{M}}$ by $%
\mathrm{Br}_{\mathcal{M}}^{s}$ and $H$ by the corresponding forgetful
functor.
\end{lemma}

\begin{proof}
It is dual to Lemma \ref{lem: BrVec}.
\end{proof}

\begin{lemma}
\label{lem:HAlgRefCo}Let $\mathcal{M}$ be a monoidal category. Assume that $%
\mathcal{M}$ has coequalizers and that the tensor functors preserve them.
Then the functor $H_{\mathrm{Alg}}:\mathrm{BrAlg}_{\mathcal{M}}\rightarrow
\mathrm{Alg}_{\mathcal{M}}$ reflects coequalizers.
\end{lemma}

\begin{proof}
Let
\begin{equation*}
\xymatrix{ ( A,c_{A}) \ar@<.5ex>[rr]^-{\alpha } \ar@<-.5ex>[rr]_-{\beta }&&(
B,c_{B})\ar[r]^-{p}&( D,c_{D}) }
\end{equation*}%
be a diagram of morphisms and objects in $\mathrm{BrAlg}_{\mathcal{M}}$
which is sent by $H_{\mathrm{Alg}}$ to a coequalizer in $\mathrm{Alg}_{%
\mathcal{M}}$. Since $H_{\mathrm{Alg}}$ is faithful we have that $p\alpha
=p\beta $ as morphisms in $\mathrm{BrAlg}_{\mathcal{M}}.$ Let $\lambda
:\left( B,c_{B}\right) \rightarrow \left( E,c_{E}\right) $ be a morphism in $%
\mathrm{BrAlg}_{\mathcal{M}}$ such that $\lambda \alpha =\lambda \beta $.
Then $H_{\mathrm{Alg}}\lambda \circ H_{\mathrm{Alg}}\alpha =H_{\mathrm{Alg}%
}\lambda \circ H_{\mathrm{Alg}}\beta $ so that there is a unique algebra
morphism $\lambda ^{\prime }:D\rightarrow E$ such that $\lambda ^{\prime
}\circ H_{\mathrm{Alg}}p=H_{\mathrm{Alg}}\lambda .$ We have%
\begin{gather*}
c_{E}\left( \Omega \lambda ^{\prime }\otimes \Omega \lambda ^{\prime
}\right) \left( \Omega H_{\mathrm{Alg}}p\otimes \Omega H_{\mathrm{Alg}%
}p\right) =c_{E}\left( \Omega H_{\mathrm{Alg}}\lambda \otimes \Omega H_{%
\mathrm{Alg}}\lambda \right) =\left( \Omega H_{\mathrm{Alg}}\lambda \otimes
\Omega H_{\mathrm{Alg}}\lambda \right) c_{B} \\
=\left( \Omega \lambda ^{\prime }\otimes \Omega \lambda ^{\prime }\right)
\left( \Omega H_{\mathrm{Alg}}p\otimes \Omega H_{\mathrm{Alg}}p\right)
c_{B}=\left( \Omega \lambda ^{\prime }\otimes \Omega \lambda ^{\prime
}\right) c_{D}\left( \Omega H_{\mathrm{Alg}}p\otimes \Omega H_{\mathrm{Alg}%
}p\right) .
\end{gather*}%
By \ref{cl:CoeqAlg}, we have that $\left( H_{\mathrm{Alg}}\alpha ,H_{\mathrm{%
Alg}}\beta \right) $ has a coequalizer in $\mathrm{Alg}_{\mathcal{M}}$ which
is a regular epimorphism in $\mathcal{M}$. By the uniqueness of coequalizers
in $\mathrm{Alg}_{\mathcal{M}}$, we get that $\Omega H_{\mathrm{Alg}}p$ is
also regular epimorphism in $\mathcal{M}$. By the assumption on the tensor
products, we get that $\Omega H_{\mathrm{Alg}}p\otimes \Omega H_{\mathrm{Alg}%
}p$ is an epimorphism in $\mathcal{M}$. Thus the computation above implies $%
c_{E}\left( \Omega \lambda ^{\prime }\otimes \Omega \lambda ^{\prime
}\right) =\left( \Omega \lambda ^{\prime }\otimes \Omega \lambda ^{\prime
}\right) c_{D}$ so that there is a morphism $\lambda ^{\prime \prime
}:\left( D,c_{D}\right) \rightarrow \left( E,c_{E}\right) $ in $\mathrm{BrAlg%
}_{\mathcal{M}}$ such that $H_{\mathrm{Alg}}\lambda ^{\prime \prime
}=\lambda ^{\prime }$. Since $H_{\mathrm{Alg}}$ is faithful we get $\lambda
^{\prime \prime }\circ p=\lambda .$ Also the uniqueness follows by the fact
that $H_{\mathrm{Alg}}$ is faithful.
\end{proof}

\begin{lemma}
\label{lem: BrAlgCoeq}Let $\mathcal{M}$ be a monoidal category. Let $\left(
e_{0},e_{1}\right) :\mathbb{A}\rightarrow \mathbb{B}$ be a pair of morphisms
in $\mathrm{BrAlg}_{\mathcal{M}}$ such $\left( \Omega H_{\mathrm{Alg}%
}e_{0},\Omega H_{\mathrm{Alg}}e_{1}\right) $ is a reflexive pair of
morphisms in $\mathcal{M}$. Assume that $\mathcal{M}$ has coequalizers and
that the tensor functors preserve them. Then $\left( e_{0},e_{1}\right) $
has a coequalizer $\left( \mathbb{C},p:\mathbb{B}\rightarrow \mathbb{C}%
\right) $ in $\mathrm{BrAlg}_{\mathcal{M}}$ which is preserved both by the
functor $H_{\mathrm{Alg}}:\mathrm{BrAlg}_{\mathcal{M}}\rightarrow \mathrm{Alg%
}_{\mathcal{M}}$ and the functor $\Omega H_{\mathrm{Alg}}$ (in particular $%
\Omega H_{\mathrm{Alg}}p$ is a regular epimorphism in $\mathcal{M}$ in the
sense of \cite[Definition 4.3.1]{Borceux1}).

The same statement holds when we replace $\mathrm{BrAlg}_{\mathcal{M}}$ by $%
\mathrm{BrAlg}_{\mathcal{M}}^{s}$ and $H_{\mathrm{Alg}}$ by the
corresponding forgetful functor.
\end{lemma}

\begin{proof}
Let $\left( A,m_{A},u_{A},c_{A}\right) :=\mathbb{A}$ and let $\left(
B,m_{B},u_{B},c_{B}\right) :=\mathbb{B}$. By Proposition \ref{pro:CreAlg},
we have that $\left( H_{\mathrm{Alg}}e_{0},H_{\mathrm{Alg}}e_{1}\right) $
has a coequalizer $\left( \left( C,m_{C},u_{C}\right) ,p:\left(
B,m_{B},u_{B}\right) \rightarrow \left( C,m_{C},u_{C}\right) \right) $ in $%
\mathrm{Alg}_{\mathcal{M}}$ and it is preserved by $\Omega $. Thus, we have
the following coequalizer in $\mathcal{M}$%
\begin{equation*}
\xymatrix{ A \ar@<.5ex>[rr]^-{e_0} \ar@<-.5ex>[rr]_-{e_1}&&B\ar[r]^-{p}&C }
\end{equation*}%
where $e_{0},e_{1}$ and $p$ denotes the same morphisms regarded as morphisms
in $\mathcal{M}$(hence, by construction, $p$ is a regular epimorphism in $%
\mathcal{M}$). By Lemma \ref{lem:Vallette}, we have the following coequalizer%
\begin{equation*}
\xymatrix{ \duplica{A} \ar@<.5ex>[rr]^-{\duplica{e_0}}
\ar@<-.5ex>[rr]_-{\duplica{e_1}}&&\duplica{B}\ar[r]^-{\duplica{p}}&%
\duplica{C} }
\end{equation*}%
We have
\begin{equation*}
\left( p\otimes p\right) c_{B}\left( e_{0}\otimes e_{0}\right) =\left(
p\otimes p\right) \left( e_{0}\otimes e_{0}\right) c_{A}=\left( p\otimes
p\right) \left( e_{1}\otimes e_{1}\right) c_{A}=\left( p\otimes p\right)
c_{B}\left( e_{1}\otimes e_{1}\right)
\end{equation*}%
so that there is a unique morphism $c_{C}:C\otimes C\rightarrow C\otimes C$
such that
\begin{equation*}
c_{C}\left( p\otimes p\right) =\left( p\otimes p\right) c_{B}.
\end{equation*}%
Now, by Lemma \ref{lem:cInv}, we have that $e_{0},e_{1}:\left(
A,m_{A},u_{A},c_{A}^{-1}\right) \rightarrow \left(
B,m_{B},u_{B},c_{B}^{-1}\right) $ are morphisms of braided objects. Hence
the same argument we used above proves that there is a unique morphism $%
\widetilde{c}_{C}:C\otimes C\rightarrow C\otimes C$ such that $\widetilde{c}%
_{C}\left( p\otimes p\right) =\left( p\otimes p\right) c_{B}^{-1}.$ We have $%
\widetilde{c}_{C}c_{C}\left( p\otimes p\right) =\widetilde{c}_{C}\left(
p\otimes p\right) c_{B}=\left( p\otimes p\right) c_{B}^{-1}c_{B}=p\otimes p$
and hence $\widetilde{c}_{C}c_{C}=\mathrm{Id}_{C\otimes C}.$ Similarly $c_{C}%
\widetilde{c}_{C}=\mathrm{Id}_{C\otimes C}$ so that $c_{C}$ is invertible.
By Lemma \ref{lem:quot}, $\left( C,c_{C}\right) $ is an object in $\mathrm{Br%
}_{\mathcal{M}}$ and $p:\left( B,c_{B}\right) \rightarrow \left(
C,c_{C}\right) $ is a morphism in $\mathrm{Br}_{\mathcal{M}}.$ It is
straightforward to check that $\left( C,m_{C},u_{C},c_{C}\right) $ is a
braided algebra (whence $p$ is a braided algebra morphism) and that $\left(
\left( C,m_{C},u_{C},c_{C}\right) ,p:\left( B,m_{B},u_{B},c_{B}\right)
\rightarrow \left( C,m_{C},u_{C},c_{C}\right) \right) $ is the coequalizer
of $\left( e_{0},e_{1}\right) .$

Consider now the case of $\mathrm{BrAlg}_{\mathcal{M}}^{s}$ so that $\left(
e_{0},e_{1}\right) $ as above is a pair in $\mathrm{BrAlg}_{\mathcal{M}}^{s}$%
. Since $p$ is an epimorphism, by Remark \ref{rem:aureo}, we get that $%
\left( C,m_{C},u_{C},c_{C}\right) \in \mathrm{BrAlg}_{\mathcal{M}}^{s}$ and $%
p$ becomes a morphism in this category. Since $\mathrm{BrAlg}_{\mathcal{M}%
}^{s}$ is a full subcategory of $\mathrm{BrAlg}_{\mathcal{M}}$ we have that $%
\mathbb{I}_{\mathrm{BrAlg}}^{s}:\mathrm{BrAlg}_{\mathcal{M}}^{s}\rightarrow
\mathrm{BrAlg}_{\mathcal{M}}$ is full and faithful and hence it reflects
coequalizers (dual to \cite[Proposition 2.9.9]{Borceux1}) so that the above
coequalizer obtained in $\mathrm{BrAlg}_{\mathcal{M}}$ is indeed a
coequalizer in $\mathrm{BrAlg}_{\mathcal{M}}^{s}$.
\end{proof}

\begin{proposition}
\label{pro: BrBiAlgCoeqLift} Let $\mathcal{M}$ be a monoidal category such
that the tensor functors preserve coequalizers. Let $\left(
e_{0},e_{1}\right) :\mathbb{A}\rightarrow \mathbb{B}$ be a pair of morphisms
in $\mathrm{BrBialg}_{\mathcal{M}}$ such $\left( \mho _{\mathrm{Br}%
}e_{0},\mho _{\mathrm{Br}}e_{1}\right) $ has a coequalizer in $\mathrm{BrAlg}%
_{\mathcal{M}}$ which is preserved by the functor $H_{\mathrm{Alg}}:\mathrm{%
BrAlg}_{\mathcal{M}}\rightarrow \mathrm{Alg}_{\mathcal{M}}$ and which is a
regular epimorphism when regarded in $\mathcal{M}$. Then $\left(
e_{0},e_{1}\right) $ has a coequalizer in $\mathrm{BrBialg}_{\mathcal{M}}$
which is preserved by the functor $\mho _{\mathrm{Br}}:\mathrm{BrBialg}_{%
\mathcal{M}}\rightarrow \mathrm{BrAlg}_{\mathcal{M}}$.

The same statement holds when we replace $\mathrm{BrBialg}_{\mathcal{M}},%
\mathrm{BrAlg}_{\mathcal{M}}$ and $\mho _{\mathrm{Br}}$ by $\mathrm{BrAlg}_{%
\mathcal{M}}^{s},\mathrm{BrAlg}_{\mathcal{M}}^{s}$ and $\mho _{\mathrm{Br}%
}^{s}$ respectively and we replace $H_{\mathrm{Alg}}$ by the corresponding
forgetful functor.
\end{proposition}

\begin{proof}
Let $\left( A,m_{A},u_{A},\Delta _{A},\varepsilon _{A},c_{A}\right) $ be the
domain of $e_{0}$ and let $\left( B,m_{B},u_{B},\Delta _{B},\varepsilon
_{B},c_{B}\right) $ be its codomain. Now, the pair $\left( \mho _{\mathrm{Br}%
}e_{0},\mho _{\mathrm{Br}}e_{1}\right) $ has a coequalizer in $\mathrm{BrAlg}%
_{\mathcal{M}},$ say
\begin{equation*}
\left( \left( C,m_{C},u_{C},c_{C}\right) ,p:\left(
B,m_{B},u_{B},c_{B}\right) \rightarrow \left( C,m_{C},u_{C},c_{C}\right)
\right) ,
\end{equation*}%
which is preserved by the functor $H_{\mathrm{Alg}}:\mathrm{BrAlg}_{\mathcal{%
M}}\rightarrow \mathrm{Alg}_{\mathcal{M}}$ and such that $p$ is a regular
epimorphism in $\mathcal{M}$. Denote by $e_{0},e_{1}$ and $p$ the same
morphisms regarded as morphisms in $\mathrm{Alg}_{\mathcal{M}}$. By \cite[%
Lemma 2.3]{AM-BraidedOb}, $(A\otimes A,m_{A\otimes A},u_{A\otimes A}),\in
\mathrm{Alg}_{\mathcal{M}}$, where $m_{A\otimes A}:=\left( m_{A}\otimes
m_{A}\right) \left( A\otimes c_{A}\otimes B\right) $ and $u_{A\otimes
A}:=\left( u_{A}\otimes u_{A}\right) \Delta _{\mathbf{1}}$. Similarly $%
(C\otimes C,m_{C\otimes C},u_{C\otimes C})\in \mathrm{Alg}_{\mathcal{M}}$.
Since $\left( B,m_{B},u_{B},\Delta _{B},\varepsilon _{B},c_{B}\right) $ is a
braided bialgebra, we have that $\Delta _{B}:\left( B,m_{B},u_{B}\right)
\rightarrow (A\otimes A,m_{A\otimes A},u_{A\otimes A})$ is an algebra map.
Moreover, by Proposition \cite[3) of Proposition 2.2]{AM-BraidedOb}, we have
that $p\otimes p$ is a morphism in $\mathrm{Alg}_{\mathcal{M}}.$ Thus $%
\left( p\otimes p\right) \Delta _{B}$ is an algebra map. Since $H_{\mathrm{%
Alg}}:\mathrm{BrAlg}_{\mathcal{M}}\rightarrow \mathrm{Alg}_{\mathcal{M}}$
preserves the coequalizer of $\left( \mho _{\mathrm{Br}}e_{0},\mho _{\mathrm{%
Br}}e_{1}\right) $ the first row in the following diagram is a coequalizer
in $\mathrm{Alg}_{\mathcal{M}}$.
\begin{equation*}
\xymatrix{ A\ar[d]_{\Delta_A} \ar@<.5ex>[rr]^-{e_0}
\ar@<-.5ex>[rr]_-{e_1}&&B\ar[d]_{\Delta_B} \ar[r]^-{p}&C \\ \duplica{A}
\ar@<.5ex>[rr]^-{\duplica{e_0}}
\ar@<-.5ex>[rr]_-{\duplica{e_1}}&&\duplica{B}\ar[r]^-{\duplica{p}}&%
\duplica{C}}
\end{equation*}%
Since the same diagram serially commutes, by the universal property of the
coequalizer in $\mathrm{Alg}_{\mathcal{M}},$ we get that there is a unique
algebra morphism $\Delta _{C}:\left( C,m_{C},u_{C}\right) \rightarrow
(C\otimes C,m_{C\otimes C},u_{C\otimes C})$ such that $\Delta _{C}p=\left(
p\otimes p\right) \Delta _{B}.$ Denote by $\Delta _{C}$ the same morphism
regarded as a morphism in $\mathcal{M}$. Since $p$ is an epimorphism in $%
\mathcal{M}$, one easily checks that $\Delta _{C}$ is coassociative using
coassociativity of $\Delta _{B}$. Since the diagram
\begin{equation*}
\xymatrixrowsep{15pt} \xymatrix{ A\ar[drr]_{\varepsilon_A}
\ar@<.5ex>[rr]^(.7){e_0} \ar@<-.5ex>[rr]_(.7){e_1}&&B\ar[d]^{\varepsilon_B}
\ar[r]^-{p}&C\\ &&\mathbf{1}}
\end{equation*}%
serially commutes, we get that there is a unique algebra morphism $%
\varepsilon _{C}:\left( C,m_{C},u_{C}\right) \rightarrow (\mathbf{1},m_{%
\mathbf{1}},u_{\mathbf{1}})$ such that $\varepsilon _{C}p=\varepsilon _{B}.$
Denote by $\varepsilon _{C}$ the same morphism regarded as a morphism in $%
\mathcal{M}$. Since $p$ is an epimorphism in $\mathcal{M}$, one easily
checks that $\Delta _{C}$ is counitary using counitarity of $\Delta _{B}$.
Hence $\left( C,\Delta _{C},\varepsilon _{C}\right) $ is a coalgebra in $%
\mathcal{M}$ and $p:\left( B,\Delta _{B},\varepsilon _{B}\right) \rightarrow
\left( C,\Delta _{C},\varepsilon _{C}\right) $ is a coalgebra map.

Since $p$ is a regular epimorphism in $\mathcal{M}$ we have $p\otimes p$ is
an epimorphism too by the assumption on the tensor products. Using this
fact, that $\left( B,\Delta _{B},\varepsilon _{B},c_{B}\right) $ is a
braided coalgebra and that $p$ is a coalgebra morphism compatible with the
braiding, one easily checks that $\left( C,\Delta _{C},\varepsilon
_{C},c_{C}\right) $ is a braided coalgebra too and hence $p$ a morphism of
these braided coalgebras. Summing up $p:\left( B,m_{B},u_{B},\Delta
_{B},\varepsilon _{B},c_{B}\right) \rightarrow \left( C,m_{C},u_{C},\Delta
_{C},\varepsilon _{C},c_{C}\right) $ is a morphism of braided bialgebras in $%
\mathcal{M}$. Using the fact that $p$ is an epimorphism in $\mathcal{M},$
one easily checks it is the searched coequalizer. The symmetric case can be
treated analogously.
\end{proof}

\begin{corollary}
\label{coro: BrBiAlgCoeq}Let $\mathcal{M}$ be a monoidal category. Let $%
\left( e_{0},e_{1}\right) $ be a pair of morphisms in $\mathrm{BrBialg}_{%
\mathcal{M}}$ such $\left( \Gamma e_{0},\Gamma e_{1}\right) $ is a reflexive
pair of morphisms in $\mathcal{M}$ where $\Gamma :=\Omega H_{\mathrm{Alg}%
}\mho _{\mathrm{Br}}:\mathrm{BrBialg}_{\mathcal{M}}\rightarrow \mathcal{M}$
denotes the forgetful functor. Assume that $\mathcal{M}$ has coequalizers
and that the tensor functors preserve them. Then $\left( e_{0},e_{1}\right) $
has a coequalizer in $\mathrm{BrBialg}_{\mathcal{M}}$ which is preserved by
the functors $\mho _{\mathrm{Br}}:\mathrm{BrBialg}_{\mathcal{M}}\rightarrow
\mathrm{BrAlg}_{\mathcal{M}}$, $H_{\mathrm{Alg}}\mho _{\mathrm{Br}}:\mathrm{%
BrBialg}_{\mathcal{M}}\rightarrow \mathrm{Alg}_{\mathcal{M}}$ and $\Gamma ,$
and which is a regular epimorphism when regarded in $\mathcal{M}$.

The same statement holds when we replace $\mathrm{BrBialg}_{\mathcal{M}},%
\mathrm{BrAlg}_{\mathcal{M}}$ and $\mho _{\mathrm{Br}}$ by $\mathrm{BrAlg}_{%
\mathcal{M}}^{s},\mathrm{BrAlg}_{\mathcal{M}}^{s}$ and $\mho _{\mathrm{Br}%
}^{s}$ respectively and we replace $H_{\mathrm{Alg}}$ by the corresponding
forgetful functor.
\end{corollary}

\begin{proof}
The pair $\left( \mho _{\mathrm{Br}}e_{0},\mho _{\mathrm{Br}}e_{1}\right) $
fulfills the requirements of Lemma \ref{lem: BrAlgCoeq} so that $\left( \mho
_{\mathrm{Br}}e_{0},\mho _{\mathrm{Br}}e_{1}\right) $ has a coequalizer in $%
\mathrm{BrAlg}_{\mathcal{M}}$ which is preserved by the functors $H_{\mathrm{%
Alg}}:\mathrm{BrAlg}_{\mathcal{M}}\rightarrow \mathrm{Alg}_{\mathcal{M}}$
and $\Omega H_{\mathrm{Alg}}$ (and which is a regular epimorphism when
regarded in $\mathcal{M}$). Hence we can apply Proposition \ref{pro:
BrBiAlgCoeqLift} to conclude.
\end{proof}

\begin{lemma}
\label{lem:BBialgWpres}Let $\mathcal{M}$ and $\mathcal{N}$ be monoidal
categories. Assume that both $\mathcal{M}$ and $\mathcal{N}$ have
coequalizers and that the tensor functors preserve them. Assume that there
exists a monoidal functor $\left( F,\phi _{0},\phi _{2}\right) :\mathcal{M}%
\rightarrow \mathcal{N}$ which preserves coequalizers. Then

1) $\mathrm{Alg}F:\mathrm{Alg}_{\mathcal{M}}\rightarrow \mathrm{Alg}_{%
\mathcal{N}}$ preserves coequalizers.

2) $\mathrm{BrBialg}F:\mathrm{BrBialg}_{\mathcal{M}}\rightarrow \mathrm{%
BrBialg}_{\mathcal{N}}$ preserves coequalizers of reflexive pairs of
morphisms. The same statement holds when we replace $\mathrm{BrBialg}\ $by $%
\mathrm{BrBialg}^{s}$ everywhere.
\end{lemma}

\begin{proof}
1) In view of \ref{cl:CoeqAlg}, the coequalizer of the pair $\left( \alpha
,\beta \right) $ of algebra morphisms $E\rightarrow A$ is, as an object in $%
\mathcal{M}$, the coequalizer $\left( B,\pi :A\rightarrow B\right) $ of $%
\left( \Lambda _{\alpha },\Lambda _{\beta }\right) $ in $\mathcal{M}$ and
the algebra structure is the unique one making $\pi $ an algebra morphism.
Since $F$ preserves coequalizers, we get the coequalizer in $\mathcal{N}$
\begin{equation*}
\xymatrix{F( A\otimes E\otimes A) \ar@<.5ex>[rr]^-{F(\Lambda_\alpha) }
\ar@<-.5ex>[rr]_-{F(\Lambda_\beta) }&&FA\ar[r]^-{F\pi}&FB }
\end{equation*}%
Note that, since $\mathrm{Alg}F$ is a functor, we have that $FA$, $FB$ are
algebras and $F\pi $ is an algebra morphism.

Using the definition of $\Lambda _{\alpha },$ the naturality of $\phi _{2}$,
the equality $m_{FA}=Fm_{A}\circ \phi _{2}\left( A,A\right) $ and the
definition of $\Lambda _{F\alpha }$ one proves that $F\left( \Lambda
_{\alpha }\right) \circ \phi _{2}\left( A\otimes E,A\right) \circ \left(
\phi _{2}\left( A,E\right) \otimes FA\right) =\Lambda _{F\alpha }$ and
similarly with $\beta $ in place of $\alpha $. Since $\phi _{2}\left(
A\otimes E,A\right) \circ \left( \phi _{2}\left( A,E\right) \otimes
FA\right) $ is an isomorphism, we get the coequalizer%
\begin{equation*}
\xymatrix{FA\otimes FE\otimes FA \ar@<.5ex>[rr]^-{\Lambda_{F\alpha} }
\ar@<-.5ex>[rr]_-{\Lambda_{F\beta} }&&FA\ar[r]^-{F\pi}&FB }.
\end{equation*}%
By construction we get that $\left( FB,F\pi \right) $ is the coequalizer of $%
\left( \Lambda _{F\alpha },\Lambda _{F\beta }\right) $ in $\mathcal{N}$.
Since, as observed, $FA$ and $FB$ are algebras and $F\pi $ is an algebra
morphism, we conclude that $\left( FB,F\pi \right) $ is the coequalizer of $%
\left( F\alpha ,F\beta \right) $ in $\mathcal{N}$ (apply \ref{cl:CoeqAlg}
again).

2) Consider a coequalizer of a reflexive pair in $\mathrm{BrBialg}_{\mathcal{%
M}}$
\begin{equation}
\xymatrix{ B \ar@<.5ex>[rr]^-{e_{0} } \ar@<-.5ex>[rr]_-{e_{1}
}&&D\ar[r]^-{d}&E }  \label{form:Coeqpres}
\end{equation}%
If we apply the forgetful functor $\Gamma :=\Omega H_{\mathrm{Alg}}\mho _{%
\mathrm{Br}}:\mathrm{BrBialg}_{\mathcal{M}}\rightarrow \mathcal{M}$ to the
pair, we get a reflexive pair in $\mathcal{M}$. By Corollary \ref{coro:
BrBiAlgCoeq}, $\left( e_{0},e_{1}\right) $ has a coequalizer in $\mathrm{%
BrBialg}_{\mathcal{M}}$ (different from (\ref{form:Coeqpres}), in principle)
which is preserved by the functor $H_{\mathrm{Alg}}\mho _{\mathrm{Br}}:%
\mathrm{BrBialg}_{\mathcal{M}}\rightarrow \mathrm{Alg}_{\mathcal{M}}$. By
uniqueness of coequalizers, we get that the coequalizer (\ref{form:Coeqpres}%
) is preserved by $H_{\mathrm{Alg}}\mho _{\mathrm{Br}}$ and hence, by 1), it
is preserved by $\left( \mathrm{Alg}F\right) H_{\mathrm{Alg}}\mho _{\mathrm{%
Br}}:\mathrm{BrBialg}_{\mathcal{M}}\rightarrow \mathrm{Alg}_{\mathcal{N}}$.
Hence $\left( FE,Fd\right) $ is the coequalizer of $\left(
Fe_{0},Fe_{1}\right) $ in $\mathrm{Alg}_{\mathcal{N}}$.

Note that $\left( Fe_{0},Fe_{1}\right) $ is a reflexive pair of morphisms in
$\mathrm{BrBialg}_{\mathcal{N}}.$ By Corollary \ref{coro: BrBiAlgCoeq}, $%
\left( Fe_{0},Fe_{1}\right) $ has a coequalizer $\left( E^{\prime },\pi
:FD\rightarrow E^{\prime }\right) $ in $\mathrm{BrBialg}_{\mathcal{N}}$
which is preserved by the functor $H_{\mathrm{Alg}}^{\prime }\mho _{\mathrm{%
Br}}^{\prime }:\mathrm{BrBialg}_{\mathcal{N}}\rightarrow \mathrm{Alg}_{%
\mathcal{N}}$. By uniqueness of coequalizers in $\mathrm{Alg}_{\mathcal{N}},$
there is an algebra isomorphism $\xi :E^{\prime }\rightarrow FE$ such that $%
\xi \circ \pi =Fd.$ Since $\mathrm{BrBialg}F$ is a functor we have that $FE$
is a braided bialgebra and $Fd$ is a morphism in $\mathrm{BrBialg}_{\mathcal{%
N}}$. Now, by construction $\pi $ is a suitable coequalizer in $\mathcal{N}$
(which further inherits a proper braided bialgebra structure) so that, by
assumption it is preserved by the tensor functors. Hence both $\pi $ and $%
\pi \otimes \pi $ are epimorphisms in $\mathcal{N}$. Using these properties
one proves that $\xi :E^{\prime }\rightarrow FE$ is a morphism in $\mathrm{%
BrBialg}_{\mathcal{N}}$.

Since $\xi $ is invertible, we obtain that $\left( FE,Fd\right) $ is the
coequalizer of $\left( Fe_{0},Fe_{1}\right) $ in $\mathrm{BrBialg}_{\mathcal{%
N}}$ i.e. $\mathrm{BrBialg}F:\mathrm{BrBialg}_{\mathcal{M}}\rightarrow
\mathrm{BrBialg}_{\mathcal{N}}$ preserves coequalizers of reflexive pairs of
morphisms. The symmetric case follows analogously once observed that $F$
preserves symmetric objects, see \ref{cl:BrBialg}.
\end{proof}

\begin{proposition}
\label{pro:BeckBr}Let $\mathcal{M}$ be a monoidal category. Assume that $%
\mathcal{M}$ has a coequalizers and that the tensor functors preserve them.
Consider a right adjoint functor $R:\mathrm{BrBialg}_{\mathcal{M}%
}\rightarrow \mathcal{B}$ into an arbitrary category $\mathcal{B}$. Then the
comparison functor $R_{1}$ has a left adjoint $L_{1}\ $which is uniquely
determined by the following properties.

1) For every object $\left( B,\mu \right) \in {_{RL}\mathcal{B}}$, there is
a morphism $\pi \left( B,\mu \right) :LB\rightarrow L_{1}\left( B,\mu
\right) $ such that
\begin{equation}
\xymatrix{ \Gamma LRLB \ar@<.5ex>[rr]^-{\Gamma L\mu }
\ar@<-.5ex>[rr]_-{\Gamma \epsilon LB }&&\Gamma LB\ar[rr]^-{\Gamma \pi \left(
B,\mu \right)}&&\Gamma L_{1}\left( B,\mu \right) }  \label{coeq:GammaPi}
\end{equation}%
is a coequalizer in $\mathcal{M}$, where $\Gamma :=\Omega H_{\mathrm{Alg}%
}\mho _{\mathrm{Br}}:\mathrm{BrBialg}_{\mathcal{M}}\rightarrow \mathcal{M}$
denotes the forgetful functor.

2) The bialgebra structure of $\Gamma L_{1}\left( B,\mu \right) $ is
uniquely determined by the fact that $\Gamma \pi \left( B,\mu \right) $ is a
morphism of braided bialgebras in $\mathcal{M}$.

3) $R$ is comparable.

4) The statements above still hold true when $\mathrm{BrBialg}_{\mathcal{M}%
}^{s}$ replaces $\mathrm{BrBialg}_{\mathcal{M}}$.
\end{proposition}

\begin{proof}
By Beck's Theorem, it suffices to check that for every $\left( B,\mu \right)
\in {_{RL}\mathcal{B}}$ the fork $\left( L\mu ,\epsilon LB\right) $ has a
coequalizer in $\mathrm{BrBialg}_{\mathcal{M}},$ where $L$ denotes the left
adjoint of $R$ and $\epsilon :LR\rightarrow \mathrm{Id}_{\mathrm{BrBialg}_{%
\mathcal{M}}}$ the counit of the adjunction. Now $L\mu \circ L\eta B=\mathrm{%
Id}_{LB}=\epsilon LB\circ L\eta B$ where $\eta :\mathrm{Id}_{\mathcal{B}%
}\rightarrow RL$ is the unit of the adjunction. Thus $\left( L\mu ,\epsilon
LB\right) $ is a reflexive pair of morphisms in $\mathrm{BrBialg}_{\mathcal{M%
}}$. Therefore $\left( \Gamma L\mu ,\Gamma \epsilon LB\right) $ is a
reflexive pair of morphisms in $\mathcal{M}$. By Corollary \ref{coro:
BrBiAlgCoeq}, the pair $\left( L\mu ,\epsilon LB\right) $ has a coequalizer
in $\mathrm{BrBialg}_{\mathcal{M}}$ which is preserved both by the functors $%
\mho _{\mathrm{Br}}:\mathrm{BrBialg}_{\mathcal{M}}\rightarrow \mathrm{BrAlg}%
_{\mathcal{M}}$, $H_{\mathrm{Alg}}\mho _{\mathrm{Br}}:\mathrm{BrBialg}_{%
\mathcal{M}}\rightarrow \mathrm{Alg}_{\mathcal{M}}$ and $\Gamma :=\Omega H_{%
\mathrm{Alg}}\mho _{\mathrm{Br}}:\mathrm{BrBialg}_{\mathcal{M}}\rightarrow
\mathcal{M}$. By construction the coequalizer of $\left( L\mu ,\epsilon
LB\right) $ is $\left( L_{1}\left( B,\mu \right) ,\pi \left( B,\mu \right)
:LB\rightarrow L_{1}\left( B,\mu \right) \right) .$ Furthermore (\ref%
{coeq:GammaPi}) is a coequalizer in $\mathcal{M}\ $and the bialgebra
structure of $\Gamma L_{1}\left( B,\mu \right) $ is uniquely determined by
the fact that $\Gamma \pi \left( B,\mu \right) $ is a morphism of braided
bialgebras in $\mathcal{M}$. By Lemma \ref{lem: coequalizers}, $R$ is
comparable.

The symmetric case follows analogously.
\end{proof}


\begin{thebibliography}{AMS1}
\bibitem[Ag]{Agore} A. L. Agore, \emph{Categorical Constructions for Hopf
Algebras}. Comm. Algebra \textbf{39} (2011), no. 4, 1476--1481.

\bibitem[ABM]{ABM-HomLie} A. Ardizzoni, D. Bulacu and C. Menini, \emph{%
Quasi-bialgebra Structures and Torsion-free Abelian Groups}, Bull. Math.
Soc. Sci. Math. Roumanie (N.S.), \textbf{56}(104) No. 3 (2013) 247--265.

\bibitem[AEM]{AEM} A. Ardizzoni, L. El Kaoutit and C. Menini, \emph{%
Coendomorphism Bialgebroid and Chain Complexes}, preprint.
(arXiv:1004.4572v1).

\bibitem[AGM]{AGM-MonadicLie1} A. Ardizzoni, J. G\'{o}mez-Torrecillas and C.
Menini, \emph{Monadic Decompositions and Classical Lie Theory}, \emph{Appl.
Categor. Struct.}, Online First. (doi:10.1007/s10485-013-9326-7)

%\bibitem[AMa]{Aguiar-Mahajan} M. Aguiar, S. Mahajan, \emph{Monoidal
%functors, species and Hopf algebras}. With forewords by Kenneth Brown and
%Stephen Chase and Andr\'{e} Joyal. CRM Monograph Series, \textbf{29}.
%American Mathematical Society, Providence, RI, 2010.

\bibitem[AHW]{AHW} J. Ad\'{a}mek, H. Herrlich, W. Tholen, \emph{Monadic
decompositions}. J. Pure Appl. Algebra \textbf{59} (1989), no. 2, 111--123.

\bibitem[AM]{AM-BraidedOb} A. Ardizzoni and C. Menini, \emph{Adjunctions and
Braided Objects}, \emph{J. Algebra Appl.} To appear.
(doi:10.1142/S0219498814500194)

\bibitem[AMS1]{AMS-splitting} A. Ardizzoni, C. Menini and D. \c{S}tefan,
\emph{A Monoidal Approach to Splitting Morphisms of Bialgebras}, Trans.
Amer. Math. Soc., \textbf{359} (2007), 991--1044.

\bibitem[AMS2]{AMS-Hoch} A. Ardizzoni, C. Menini and D. \c{S}tefan, \emph{%
Hochschild Cohomology And `Smoothness' In Monoidal Categories}, J. Pure
Appl. Algebra, \textbf{208} (2007), no. 1, 297--330.

\bibitem[AT]{AT} H. Appelgate, M. Tierney, \emph{Categories with models}.
1969 Sem. on Triples and Categorical Homology Theory (ETH, z\"{u}rich,
1966/67) pp. 156--244 Springer, Berlin.

\bibitem[Ba]{Ba} J.C. Baez, \emph{Hochschild homology in a braided tensor
category}, Trans. Amer. Math. Soc. \textbf{334} (1994), 885--906.

\bibitem[Be]{Beck} J.M. Beck, \emph{Triples, algebras and cohomology},
Reprints in Theory and Applications of Categories \textbf{2} (2003), 1--59.

\bibitem[Bo1]{Borceux1} F. Borceux, \emph{Handbook of categorical algebra.
1. Basic category theory.} Encyclopedia of Mathematics and its Applications,
\textbf{50}. Cambridge University Press, Cambridge, 1994.

\bibitem[Bo2]{Borceux2} F. Borceux, \emph{Handbook of categorical algebra.
2. Categories and structures.} Encyclopedia of Mathematics and its
Applications, \textbf{51}. Cambridge University Press, Cambridge, 1994.

\bibitem[BFM]{BFM} Y. Bahturin, D. Fischman and S. Montgomery, \emph{%
Bicharacters, twistings and Scheunert's theorem for Hopf algebras}, J.
Algebra \textbf{236} (2001), 246--276.

\bibitem[BLV]{BLV} A. Brugui\`{e}res, S. Lack, A. Virelizier, \emph{Hopf
monads on monoidal categories}. Adv. Math. \textbf{227} (2011), no. 2,
745--800.

\bibitem[BW]{TTT} M. Barr, C. Wells, \emph{Toposes, triples and theories}.
Corrected reprint of the 1985 original. Repr. Theory Appl. Categ. No.
\textbf{12} (2005).

\bibitem[CG]{CG} S. Caenepeel, I. Goyvaerts, \emph{Monoidal Hom-Hopf algebras%
}. Comm. Algebra \textbf{39} (2011), no. 6, 2216--2240.

%\bibitem[CMZ]{CMZ} S. Caenepeel, G. Militaru, S. Zhu, \emph{Frobenius and
%separable functors for generalized module categories and nonlinear equations}%
%. Lecture Notes in Mathematics, 1787. Springer-Verlag, Berlin, 2002.

\bibitem[Dr]{Drinfeld-QuasiHopf} V. G. Drinfeld, \emph{Quasi-Hopf algebras}.
(Russian) Algebra i Analiz \textbf{1} (1989), no. 6, 114--148; translation
in Leningrad Math. J. \textbf{1} (1990), no. 6, 1419--1457.

\bibitem[FM]{FM} D. Fischman, S. Montgomery, \emph{A Schur double
centralizer theorem for cotriangular Hopf algebras and generalized Lie
algebras}. J. Algebra \textbf{168} (1994), no. 2, 594--614.

\bibitem[Gu]{Gure} D.I. Gurevich, \emph{Generalized translation operators on
Lie groups}. Izv. Akad. Nauk Arm. SSR, Ser. Mat. \textbf{18}(4), 305--317
(1983).

\bibitem[GV1]{GV} I. Goyvaerts and J. Vercruysse, \emph{A note on the
categorisation of Lie algebras}. Lie Theory and its applications in
physics\textquotedblright , Springer proceedings in Mathematics and
Statistics 36, 2013.

\bibitem[GV2]{GV-LieMon} I. Goyvaerts and J. Vercruysse, \emph{Lie monads
and dualities}, preprint. (arXiv:1302.6869v1)

\bibitem[GV3]{GV-OnTheDuality} I. Goyvaerts and J. Vercruysse, \emph{On the
duality of generalized Lie and Hopf algebras}, preprint. (arXiv:1305.7447v2)

\bibitem[Jo1]{Joh} P. T. Johnstone, \emph{Adjoint lifting theorems for
categories of algebras}. Bull. London Math. Soc. \textbf{7} (1975), no. 3,
294--297.

\bibitem[Jo2]{Joh-Topos} P. T. Johnstone, \emph{Topos theory}. London
Mathematical Society Monographs, Vol. \textbf{10}. Academic Press [Harcourt
Brace Jovanovich, Publishers], London-New York, 1977.

\bibitem[JS]{Joyal-Street} A. Joyal, R. Street, \emph{Braided tensor
categories}. Adv. Math. \textbf{102} (1993), no. 1, 20--78.

\bibitem[Ka]{Kassel} C. Kassel, \emph{Quantum groups}. Graduate Texts in
Mathematics, \textbf{155}. Springer-Verlag, New York, 1995.

\bibitem[Kh]{Kharchenko-Connected} V. K. Kharchenko, \emph{Connected braided
Hopf algebras}. J. Algebra \textbf{307} (2007), no. 1, 24--48.

\bibitem[Maj]{Maj} S. Majid, \emph{Foundations of quantum group theory}.
Cambridge University Press, Cambridge, 1995.

\bibitem[Ma]{Manes-PhD} E. G. Manes, \emph{A TRIPLE MISCELLANY: SOME ASPECTS
OF THE THEORY OF ALGEBRAS OVER A TRIPLE}. Thesis (Ph.D.)--Wesleyan
University. 1967.

\bibitem[Man]{Manin} Y. I. Manin, \emph{Quantum groups and noncommutative
geometry}. Universit\'{e} de Montr\'{e}al, Centre de Recherches Math\'{e}%
matiques, Montreal, QC, 1988.

\bibitem[McL]{MacLane} S. Mac Lane, \emph{Categories for the working
mathematician.} Second edition. Graduate Texts in Mathematics, \textbf{5}.
Springer-Verlag, New York, 1998.

%\bibitem[Me]{Mesab} B. Mesablishvili, \emph{Monads of effective type and
%comonadicity}. Theory Appl. Categ. \textbf{16} (2006), No. 1, 1--45.

\bibitem[MM]{Milnor-Moore} J.W. Milnor, J.C. Moore, \emph{On the structure
of Hopf algebras}. Ann. of Math. (2) \textbf{81} (1965), 211--264.

\bibitem[Mo]{Mo} S. Montgomery, \emph{Hopf Algebras and their actions on
rings, }CMBS Regional Conference Series in Mathematics \textbf{82}, 1993.

\bibitem[MS]{MS} J. L. MacDonald, A. Stone, \emph{The tower and regular
decomposition}. Cahiers Topologie G\'{e}om. Diff\'{e}rentielle \textbf{23}
(1982), no. 2, 197--213.

%\bibitem[Pa1]{Pareigis-Cat} B. Pareigis, \emph{Categories and functors}.
%Translated from the German. Pure and Applied Mathematics, Vol. \textbf{39}
%Academic Press, New York-London 1970.

\bibitem[Pa]{Pa} B. Pareigis, \emph{Symmetric Yetter-Drinfeld categories are
trivial}. J. Pure Appl. Algebra \textbf{155} (2001), no. 1, 91.

\bibitem[RT]{RT} D. E. Radford, J. Towber, \emph{Yetter-Drinfeld categories
associated to an arbitrary bialgebra}. J. Pure Appl. Algebra \textbf{87}
(1993), no. 3, 259--279.

%\bibitem[NVV]{Nastasescu/alt:1989} C. N\u{a}st\u{a}sescu, M. van den Bergh,
%and F. Van Oystaeyen, \emph{Separable functors applied to graded rings}, J.
%Algebra \textbf{123} (1989), 397-413.

%\bibitem[Ra]{Rafael} M. D. Rafael, \emph{Separable Functors Revisited},
%Comm. Algebra \textbf{18} (1990), 1445--1459.

\bibitem[St]{Stenstroem} B. Stenstr\"{o}m, \emph{Rings of quotients} Die
Grundlehren der Mathematischen Wissenschaften, Band \textbf{217}. An
introduction to methods of ring theory. Springer-Verlag, New
York-Heidelberg, 1975.

\bibitem[Sw]{Sw} M. E. Sweedler, \emph{Hopf Algebras}. Mathematics Lecture
Note Series W. A. Benjamin, Inc., New York 1969.

\bibitem[Ta]{Ta} M. Takeuchi, \emph{Survey of braided Hopf algebras},
Contemp. Math. \textbf{267} (2000), 301--324.

\bibitem[Va]{Vallette} B. Vallette, \emph{Free monoid in monoidal abelian
categories}. Appl. Categ. Structures \textbf{17} (2009), no. 1, 43--61.
\end{thebibliography}
\end{document}